\documentclass[a4paper, 11pt, leqno]{article}

\usepackage{bm}
\usepackage{geometry}
\usepackage{amssymb,amsmath}

\usepackage[utf8]{inputenc}
\usepackage[british]{babel}
\usepackage{csquotes}
\usepackage{amsthm,thmtools}
\usepackage{tikz}
    \usetikzlibrary{cd}
    \usetikzlibrary{arrows.meta, positioning, quotes}

\usepackage{mathtools}
\usepackage{titlesec}
\usepackage{url}
\usepackage[normalem]{ulem} 

\usepackage{afterpage}
\usepackage[shortlabels]{enumitem} 
\setlist{nolistsep}
\setlist[enumerate]{label={\upshape(\roman*)}}
\usepackage[fulladjust]{marginnote}
\usepackage{stmaryrd} 

\usepackage[
    backend=biber,
    style=trad-abbrv,
    sorting=nyt,
    sortcites=true,
]{biblatex}
\DeclareFieldFormat{issn}{}

\renewbibmacro*{series+number:emphcond}{%
    \newunit\newblock
    \printfield[noformat]{series}\isdot%
    \newcommaunit%
    \printfield[noformat]{number}%
    \newunit}
\renewcommand{\volumenumberdelim}{\addnbspace}
\renewbibmacro*{volume+number+pages+eid}{%
    \printfield{volume}%
    \setunit*{\volumenumberdelim}%
    \printfield{number}%
    \setunit{\addcolon\space}%
    \printfield{pages}%
    \newcommaunit
    \printfield{eid}}

\usepackage{setspace}
\usepackage{graphicx}
\usepackage{array}
\usepackage{quiver}

\usepackage{hyperref} 
\hypersetup{
    bookmarksnumbered,
    colorlinks=true,
    linkcolor=cyan!50!blue,
    citecolor=cyan!50!blue,
    urlcolor=cyan!50!blue,
}

\usepackage{cleveref} 

\crefdefaultlabelformat{#2{\upshape#1}#3}

\crefformat{equation}{{}#2{\upshape(#1)}#3}
\crefrangeformat{equation}{{}{\upshape#3(#1)#4--#5(#2)#6}}
\crefmultiformat{equation}{{}#2{\upshape(#1)}#3}{ and #2{\upshape(#1)}#3}{, #2{\upshape(#1)}#3}{, and~#2{\upshape(#1)}#3}

\crefformat{enumi}{{}#2{\upshape#1}#3}
\crefrangeformat{enumi}{{}{\upshape#3#1#4--#5#2#6}}
\crefmultiformat{enumi}{{}{\upshape#2#1#3}}{ and {\upshape#2#1#3}}{, {\upshape#2#1#3}}{, and~{\upshape#2#1#3}}

\crefformat{para}{{}#2{\upshape\S#1}#3}
\crefrangeformat{para}{{}{\upshape#3\S\S#1#4--#5#2#6}}
\crefmultiformat{para}{{}{\upshape#2\S\S#1#3}}{ and {\upshape#2#1#3}}{, {\upshape#2#1#3}}{, and~{\upshape#2#1#3}}

\usepackage{lmodern}
\usepackage[mathcal]{eucal}

\tikzset{
    symbol/.style={
    draw=none,
    every to/.append style={
    edge node={node [sloped, allow upside down, auto=false]{$#1$}}}
    }
}

\usetikzlibrary{decorations.markings}

    \makeatletter
    \tikzcdset{
      open/.code     = {\tikzcdset{hook, circled};},
      closed/.code   = {\tikzcdset{hook, slashed};},
      open'/.code    = {\tikzcdset{hook', circled};},
      closed'/.code  = {\tikzcdset{hook', slashed};},
      circled/.code  = {\tikzcdset{markwith = {\draw (0,0) circle (.375ex);}};},
      slashed/.code  = {\tikzcdset{markwith = {\draw[-] (-.4ex,-.4ex) -- (.4ex,.4ex);}};},
      markwith/.code ={
        \pgfutil@ifundefined%
        {tikz@library@decorations.markings@loaded}%
        {\pgfutil@packageerror{tikz-cd}{You need to say %
          \string\usetikzlibrary{decorations.markings} to use arrows with markings}{}}{}%
        \pgfkeysalso{/tikz/postaction = {
          /tikz/decorate,
          /tikz/decoration={markings, mark = at position 0.5 with {#1}}}
        }
      },
    }
    \makeatother

\declaretheorem[numberwithin=subsection, name=Theorem]{theorem}
\declaretheorem[sibling=theorem, style=definition, name=Definition]{definition}
\declaretheorem[sibling=theorem, name=Lemma]{lemma}
\declaretheorem[sibling=theorem, name=Corollary]{corollary}
\declaretheorem[sibling=theorem, name=Proposition]{proposition}
\declaretheorem[sibling=theorem, name=Conjecture]{conjecture}
\declaretheorem[sibling=theorem, style=definition, name=Remark]{remark}

\declaretheorem[sibling=theorem, style=definition, name=Example]{example}
\declaretheorem[sibling=theorem, style=definition, name=Question]{question}

\declaretheorem[style=definition,numbered=no,name=Definition]{definition*}
\declaretheorem[style=theorem,numbered=no,name=Conjecture]{conjecture*}

\makeatletter
\declaretheoremstyle[
    spaceabove=9pt, spacebelow=9pt,
    postheadspace=.5em,
    headfont=\normalfont\bfseries,
    notefont=\normalfont\bfseries\boldmath,
    notebraces={}{.},
    headpunct={},
    headformat={\NUMBER.\@\NOTE}]{para}
\makeatother
\declaretheorem[sibling=theorem, style=para, refname={\S,\S\S}]{para}

\numberwithin{equation}{theorem}

\titleformat{\subsection}{\large\bfseries\boldmath}{\thesubsection}{1em}{}

\tikzset{
    labl/.style={anchor=south, rotate=90, inner sep=.5mm}
}
\tikzset{
    lablh/.style={anchor=south, inner sep=.5mm}
}

\usepackage{tocloft}
\newcommand{\authorinforule}{\noindent\rule{0.38\textwidth}{0.4pt}}

\newlength{\authorwidth}
\newcommand{\authorinfo}[3]{%
    \setlength{\leftskip}{1.5em}
    \setlength{\parindent}{0em}
    \setstretch{1.1}%
    \par%
    {\small%
    \makebox[\authorwidth][l]{#1}%
    \texttt{#2}%
    \\
    #3.}
    \vspace{6pt}\par
}

\newcommand{\bQ}{{\mathbb Q}}

\newcommand{\cF}{{\mathcal F}}

\newcommand{\cH}{{\mathcal H}}

\newcommand{\cX}{{\mathcal X}}

\makeatletter
\newcommand{\colim@}[2]{%
  \vtop{\m@th\ialign{##\cr
    \hfil$#1\operator@font colim$\hfil\cr
    \noalign{\nointerlineskip\kern1.5\ex@}#2\cr
    \noalign{\nointerlineskip\kern-\ex@}\cr}}%
}
\newcommand{\colim}{%
  \mathop{\mathpalette\colim@{\rightarrowfill@\scriptscriptstyle}}\nmlimits@
}
\renewcommand{\varprojlim}{%
  \mathop{\mathpalette\varlim@{\leftarrowfill@\scriptscriptstyle}}\nmlimits@
}
\renewcommand{\varinjlim}{%
  \mathop{\mathpalette\varlim@{\rightarrowfill@\scriptscriptstyle}}\nmlimits@
}
\makeatother

\newcommand{\id}{\mathrm{id}}

\DeclareMathOperator{\coker}{coker}

\DeclareMathOperator{\ev}{ev}

\DeclareMathOperator{\Hom}{Hom}
\DeclareMathOperator{\Map}{Map}

\DeclareMathOperator{\GL}{\text{GL}}

\DeclareMathOperator{\Spec}{Spec}

\DeclareMathOperator{\Filt}{Filt}
\DeclareMathOperator{\Grad}{Grad}
\let\ev\relax
\DeclareMathOperator{\ev}{ev}
\DeclareMathOperator{\gr}{gr}

\DeclareMathOperator{\rank}{rank}

\DeclareMathOperator{\rk}{rk}

\DeclareMathOperator{\Ext}{Ext}

\DeclareMathOperator{\Bun}{Bun} 

\renewcommand{\phi}{\varphi}

\DeclareMathOperator{\crk}{crk}

\makeatletter 
\DeclareFontEncoding{LS1}{}{}
\DeclareFontSubstitution{LS1}{stix}{m}{n}
\DeclareMathAlphabet{\mathcaldos}{LS1}{stixscr}{m}{n}
\makeatother

\newcommand{\GIT}{{/\! \! /}}

\DeclareMathOperator{\vdim}{vdim}

\newcommand{\Dc}{{\mathsf{D}}_{\mathsf{c}}} 
\newcommand{\Dh}{{\mathsf{D}}_{\mathsf{H}}} 
\newcommand{\Dhc}{{\mathsf{D}}_{\mathsf{H}, \mathsf{c}}} 

\title{\textbf{Cohomology of symmetric stacks}}
\author{\clap{Chenjing Bu, Ben Davison, Andrés Ibáñez Núñez, Tasuki Kinjo, Tudor P\u{a}durariu}}
\date{February 2025}

\begin{document}

\maketitle

\begin{abstract}
    We construct decompositions of:
    \begin{enumerate}[leftmargin=3.3em]
    \item[(1)] the cohomology of smooth stacks,
    \item[(2)] the Borel--Moore homology of $0$-shifted symplectic stacks, and
    \item[(3)] the vanishing cycle cohomology of $(-1)$-shifted symplectic stacks,
    \end{enumerate}
    assuming a good moduli space exists and the tangent space has a pointwise orthogonal structure.
    These conditions are satisfied by many stacks of interest, including moduli stacks of semistable $G$-bundles and (twisted) $G$-Higgs bundles on curves, $G$-character stacks of oriented closed 2-manifolds and various 3-manifolds, and moduli stacks of semistable coherent sheaves on Calabi--Yau threefolds and K3 surfaces with generic polarization.
    As a special case, we prove a PBW-type theorem for cohomological Hall algebras of $3$-Calabi--Yau categories with commutative orientation data, a strong form of the cohomological integrality conjecture for such categories.
    
    We define the \emph{BPS cohomology} as the primary summand of the decomposition.
    When the stack is smooth, the BPS cohomology coincides with the intersection cohomology of the good moduli space, generalizing a theorem of Meinhardt--Reineke.
    Using the BPS cohomology for singular spaces, we propose a formulation of the topological mirror symmetry conjecture for the stack of $G$-Higgs bundles generalizing the work of Hausel and Thaddeus for type A groups, 
    and a version of Langlands duality for character stacks of compact oriented 3-manifolds, following Ben-Zvi--Gunningham--Jordan--Safronov.
\end{abstract}

\newpage
\tableofcontents

\newpage
\section{Introduction}

Throughout the paper, schemes and stacks are defined over the complex number field. Also, we work with sheaves and cohomology with rational coefficients if not otherwise specified.

\subsection{Motivations}

\begin{para}
    This paper concerns the cohomology $\mathrm{H}^*(\mathcal{M},\mathcal{F})$ of a wide range of moduli stacks appearing across geometry and topology, where the appropriate choice of coefficient constructible complex $\mathcal{F}$ generally depends on what kind of moduli stack $\mathcal{M}$ is.  
    A basic feature of the cohomology that we will consider is that it is infinite-dimensional, reflecting the fact that it is non-zero in infinitely many degrees.  
    Our main results state that these infinite-dimensional vector spaces can be decomposed and repackaged in terms of finite-dimensional representations of certain groups, which are analogues of Weyl groups defined intrinsically with respect to $\mathcal{M}$.  
    The underlying vector spaces of these representations are connected to intersection cohomology of singular moduli schemes, enumerative geometry of Calabi--Yau threefolds, non-abelian Hodge theory and the geometric Langlands program for $2$- and $3$-manifolds.  
    
    For simplicity, we start by concentrating on the case in which $\mathcal{M}$ is smooth, and $\mathcal{F}$ is the constant sheaf with rational coefficients. 
    The main results for singular spaces will be explained in \Cref{para-intro-cohint--1} and afterwards, especially in connection with Donaldson--Thomas theory (\Crefrange{para-BPS-invariant}{para-intro-BPS-Lie}), the geometric Langlands conjecture for $3$-manifolds (\Cref{para-intro-3mfd}), and non-abelian Hodge theory (\Cref{para-NAHT}).
\end{para}

\begin{para}[Historical background: Atiyah--Bott recursion]
    The topology of the moduli space of principal bundles on a smooth projective curve has been studied from several angles for the last fifty years.
    Among the milestones in this topic are the papers of Harder and Narasimhan \cite{_Harder_Onthecohomologygroupsofmodulispacesofvectorbundlesoncurves} and Atiyah and Bott \cite{_Atiyah_TheYangMillsEquationsoverRiemannSurfaces} published in 1975 and 1983 respectively, which present a beautiful recursive formula for the Poincar\'e polynomial of the moduli space $\mathrm{Bun}_{\mathrm{GL}_r}(C)_{d}^{\mathrm{ss}}$ of semistable vector bundles of coprime rank $r$ and degree $d$ on a smooth projective curve $C$.
    The computation, following the strategy of Atiyah and Bott, can be decomposed into the following three steps:
    \begin{enumerate}
        \item \label{itemref-atiha-bott-full} Computing the Poincar\'e power series for the moduli stack $\mathcal{B}\mathrm{un}_{\mathrm{GL}_r}(C)_d$ of vector bundles on $C$ of rank $r$ and degree $d$.
        \item \label{itemref-atiha-bott-semistable-stack} Computing the Poincar\'e power series for the moduli stack $\mathcal{B}\mathrm{un}_{\mathrm{GL}_r}(C)_d^{\mathrm{ss}}$ of semistable vector bundles on $C$ of rank $r$ and degree $d$. 
        \item \label{itemref-atiha-bott-semistable-space} Computing the Poincar\'e polynomial for the good moduli space $\mathrm{Bun}_{\mathrm{GL}_r}(C)_d^{\mathrm{ss}}$ of semistable vector bundles on $C$ of rank $r$ and degree $d$. 
    \end{enumerate}
    
    The computation \Cref{itemref-atiha-bott-full} was carried out in \cite[\S 1, \S 2]{_Atiyah_TheYangMillsEquationsoverRiemannSurfaces} by proving that the cohomology ring $\mathcal{B}\mathrm{un}_{\mathrm{GL}_r}(C)_d$ is freely generated by the tautological classes.

    The computation \Cref{itemref-atiha-bott-semistable-stack},  the main step of the recursion, is done in \cite[Theorem 10.10]{_Atiyah_TheYangMillsEquationsoverRiemannSurfaces} by proving that the Harder--Narasimhan stratification of $\mathcal{B}\mathrm{un}_{\mathrm{GL}_r}(C)_d$ splits its cohomology into a direct sum of the cohomology of the strata.

    For \Cref{itemref-atiha-bott-semistable-space}, the computation was completed in \cite[(9.3)]{_Atiyah_TheYangMillsEquationsoverRiemannSurfaces} by observing that the morphism $\mathcal{B}\mathrm{un}_{\mathrm{GL}_r}(C)_d^{\mathrm{ss}} \to \mathrm{Bun}_{\mathrm{GL}_r}(C)_d^{\mathrm{ss}}$ is a trivial $\mathbb{G}_{\mathrm{m}}$-gerbe.

    We note that the methods for the computation \Cref{itemref-atiha-bott-full} and \Cref{itemref-atiha-bott-semistable-stack} generalize to arbitrary $d$ without the assumption $\gcd(r, d) = 1$ and extend further to arbitrary connected reductive groups $G$ with any choice of degree.
    On the other hand, the argument for \Cref{itemref-atiha-bott-semistable-space} is special to the coprime case, since the map $\mathcal{B}\mathrm{un}_{G}(C)_d^{\mathrm{ss}} \to \mathrm{Bun}_{G}(C)_d^{\mathrm{ss}}$ being a gerbe is true only for such a choice. Also, it is known that a coprime choice does not exist outside groups of type A, as shown by \textcite[Proposition 7.8]{ramanathan-stable}.
\end{para}

\begin{para}

    Since the moduli space $\mathrm{Bun}_{G}(C)_d^{\mathrm{ss}}$ is not smooth in general, the intersection cohomology behaves better than the ordinary cohomology.
    This naturally leads to the following question:

\end{para}

\begin{question}\label{question-cohomology-g-bundle}
    What is the relation between the cohomology $\mathrm{H}^*(\mathcal{B}\mathrm{un}_{G}(C)_d^{\mathrm{ss}})$ and the intersection cohomology $\mathrm{IH}^*(\mathrm{Bun}_{G}(C)_d^{\mathrm{ss}})$?
\end{question}

\begin{para}
    This question, in the case of $G = \mathrm{GL}_r$, was first addressed by \textcite{kirwan1986homology} in 1986 as an application of her partial desingularization of GIT quotients (see \Cref{para-kirwan-blowup}).
    Remarkably, she determined the intersection Betti numbers of $\mathrm{Bun}_{\mathrm{GL}_2}(C)_0^{\mathrm{ss}}$, 
    and her method in principle provides a recursive formula relating $\mathrm{H}^*(\mathcal{B}\mathrm{un}_{\mathrm{GL}_r}(C)_d^{\mathrm{ss}})$ and $\mathrm{IH}^*(\mathrm{Bun}_{\mathrm{GL}_r}(C)_d^{\mathrm{ss}})$ for higher ranks.
    On the other hand, because of the inductive nature of her method, it does not (at least a priori) provide a closed formula relating these cohomology groups.
    
    \Cref{question-cohomology-g-bundle} for $G = \mathrm{GL}_r$ was revisited independently by \textcite{meinhardt2015donaldson} and by \textcite{mozgovoy2015intersection} around 2015.
    For simplicity, we assume the genus satisfies $g(C) \geq 2$. They obtained the following closed formula relating $\mathrm{H}^*(\mathcal{B}\mathrm{un}_{\mathrm{GL}_r}(C)_d^{\mathrm{ss}})$ and $\mathrm{IH}^*(\mathrm{Bun}_{\mathrm{GL}_r}(C)_d^{\mathrm{ss}})$:
    \begin{multline}\label{eq-cohint-intro-gln-bundles}
        \mathrm{H}^*(\mathcal{B}\mathrm{un}_{\mathrm{GL}_{r}}(C)_d^{\mathrm{ss}})[n_r] \\[1ex]
        \cong \bigoplus_{\substack{(r, d) = (r_1, d_1) + \cdots + (r_l, d_l) \\ d_{i} / r_{i} = d/r}} \Bigl( ( \mathrm{IH}^*(\mathrm{Bun}_{\mathrm{GL}_{r_1}}(C)_{d_1}^{\mathrm{ss}}) \otimes \mathrm{H}^*(\mathrm{B} \mathbb{G}_{\mathrm{m}})[- 1] ) \otimes \cdots  \\[-2ex]
          \cdots \otimes ( \mathrm{IH}^*(\mathrm{Bun}_{\mathrm{GL}_{r_l}}(C)_{d_l}^{\mathrm{ss}}) \otimes \mathrm{H}^*(\mathrm{B} \mathbb{G}_{\mathrm{m}})[-1] ) \Bigr)^{\mathrm{Aut}(r_1, \ldots, r_l)}.
    \end{multline}
    Here, we set $n_r \coloneqq \dim \mathcal{B}\mathrm{un}_{\mathrm{GL}_{r}}(C)_d^{\mathrm{ss}}$, the intersection complex is shifted to be perverse, and $\mathrm{Aut}(r_1, \ldots, r_l)$ denotes the group of automorphisms of the multiset $(r_1, \ldots, r_l)$.

    As an application of our main results, explained in \Cref{ssec:main_results}, we will provide an answer to  \Cref{question-cohomology-g-bundle}, thereby generalizing \Cref{eq-cohint-intro-gln-bundles} from $\mathrm{GL}_r$ to arbitrary connected reductive groups $G$.
    To explain this, we introduce the following notations.
    For a Levi subgroup $L \subset G$, we let $W_G(L) \coloneqq \mathrm{N}_G(L) / L$ denote the relative Weyl group, and $\mathrm{Z}(L)^{\circ} \subset L$ denote the neutral component of the centre.
    For a degree $d \in \pi_0(\mathcal{B}\mathrm{un}_G(C)) \cong \pi_1(G)$,
    we say that a Levi subgroup $L \subset G$ is \emph{$d$-admissible} if there exists $d_L \in \pi_1(L)$ such that the map  $\pi_1(L) \to \pi_1(G)$ sends $d_L$ to $d$, and for any character $\chi \colon L \to \mathbb{G}_{\mathrm{m}}$ with $\chi |_{\mathrm{Z}(G)^{\circ}} = 1$, the map $\pi_1(L) \xrightarrow[]{\chi} \pi_1(\mathbb{G}_{\mathrm{m}}) \cong \mathbb{Z}$ sends $d_L$  to $0$.
    Such a degree $d_L$ is uniquely determined if it exists: see \Cref{lem-unique-d-admissible}.
    In the case $G = \mathrm{GL}_r$, the admissible degrees correspond to compositions of tuples with the same slopes.

    With these preparations, we can state the result:
\end{para}

\begin{theorem}[Special case of \Cref{thm-cohint-bundles}]\label{thm-main-intro-bundle}
    Let $C$ be a smooth projective curve  with $g(C) \geq 2$. For a connected reductive group $G$ and $d \in \pi_1(G)$, there exists an isomorphism of graded vector spaces 
    \begin{multline}\label{eq-intro-cohint-g-bundles}
        \mathrm{H}^*(\mathcal{B}\mathrm{un}_{G}(C)_d^{\mathrm{ss}})[n_G] \\
        \cong \bigoplus_{\substack{L \subset G: \\ \textnormal{$d$-admissible}}} {} \Biggl(  \mathrm{IH}^*(\mathrm{Bun}_L(C)_{d_L}^{\mathrm{ss}}) \otimes \mathrm{H}^*(\mathrm{B} \mathrm{Z}(L)^{\circ})[- c_L] \otimes \mathrm{sgn}_{L}  \Biggr)^{W_G(L)} ,
    \end{multline}
    where the first direct sum runs over all conjugacy classes of $d$-admissible Levi subgroups of $G$. Here we set $n_G \coloneqq \dim \mathcal{B}\mathrm{un}_{G}(C)_d^{\mathrm{ss}}$ and $c_L \coloneqq \dim \mathrm{Z}(L)$, and $\mathrm{sgn}_{L}$ is a certain sign representation of $W_G(L)$ which will be introduced in \Cref{para-cotangent-distance}. 
\end{theorem}

We note that the theorem holds true at the level of graded Hodge structures, by replacing the shift with the corresponding Lefschetz twist.
This theorem combined with \Cref{item1-intro-cohint-goodmoduli} and \Cref{item2-intro-cohint-qcgr} gives a  recursive formula for the intersection cohomology of $\mathrm{Bun}_G(C)^{\mathrm{ss}}_d$.
We note that a recent preprint due to \textcite{felisetti2025parabolic} provides a recursive formula for $\mathrm{Bun}_{\mathrm{GL}_n}(C)^{\mathrm{ss}}_0$ based on a different approach using the moduli space of parabolic bundles.

\begin{para}
     One can consider a variant of \Cref{question-cohomology-g-bundle} for other smooth moduli stacks, such as the moduli stack of twisted $G$-Higgs bundles on a Riemann surface, or more generally for arbitrary smooth stacks $\mathcal{U}$ with a good moduli space $p \colon \mathcal{U} \to U$ in the sense of \textcite[Definition 4.1]{_Alper_GoodmodulispacesforArtinstacks}.
     Recently, \textcite[Theorem 1.1]{Kinjo_Decomp} proved that the morphism $p$ satisfies the decomposition theorem \`a la \textcite[Th\'eor\`em 6.2.5]{BBDG} provided that $\mathcal{U}$ has affine diagonal.
      Motivated by this result, we pose the following question, generalizations of which, for possibly singular stacks, form the main theme of this paper:
\end{para}

\begin{question}\label{question-decomp-general}
    Can we describe the decomposition for $p_* \mathbb{Q}_{\mathcal{U}}$ explicitly?
\end{question}

\begin{para}\label{para-kirwan-blowup}
In the landmark papers of Kirwan \cite{kirwan1985partial,kirwan1986rational} published in 1985 and 1986, she addressed a similar problem in the setting of projective GIT at the level of the Betti numbers.
Assuming the existence of a stable point, she constructed a partial resolution $\hat{U} \to U$ through a procedure now known as the Kirwan blow-up, and related the cohomology $\mathrm{H}^*(\mathcal{U})$ and the intersection cohomology $\mathrm{IH}^*(U)$ to $\mathrm{H}^*(\hat{U})$.
This method, in principle, gives an answer to \Cref{question-decomp-general}. However, because of the inductive nature of the resolution procedure, it is quite difficult to express the formula explicitly, and it remains unclear why we should have a simple formula as in \Cref{thm-main-intro-bundle}.
\end{para}

\begin{para}
    Although the best possible answer to \Cref{question-decomp-general} for general smooth stacks might be provided by the Kirwan blow-up method, we will see in the next subsection that a certain symmetry assumption on the tangent space of $\mathcal{U}$, which is satisfied by many moduli stacks of interest, leads to a significantly simpler answer to the question.

    Our results, as well as the idea of the proof, are new even for affine GIT quotients, which serve as the local model for stacks with good moduli spaces.
    We employ the formalism of component lattices, recently introduced by C.B., A.I.N. and T.K. in collaboration with Halpern-Leistner, to deduce the globalized statement by reducing it to the local models.
\end{para}

\subsection{Main results}\label{ssec:main_results}

\begin{para}
    Before stating our main results in full generality, it is necessary to introduce some key concepts, starting with the above-mentioned symmetry property:
\end{para}

\begin{definition}[\Cref{def-symmetry}]\label{def-symmetry-intro}
    An algebraic stack $\mathcal{X}$, with a good moduli space morphism $p \colon \mathcal{X} \to X$, is called \emph{orthogonal} (resp.~\emph{almost orthogonal}), if for any closed point $x \in \mathcal{X}$,
    the tangent space $\mathrm{T}_{\mathcal{X}, x}$ admits a non-degenerate symmetric bilinear form, invariant under the action of the stabilizer group $G_x$ (resp.~the neutral component $G_x^{\circ}$).
\end{definition}

\begin{para}[Examples of almost orthogonal stacks]\label{para-intro-almost-orthogonal-stacks}
    Many stacks appearing in moduli theory are almost orthogonal. The following is a list of almost orthogonal stacks studied in the paper:
    \begin{itemize}
        \item Let $C$ be a smooth projective curve and $G$ a reductive group. Then the stack $\mathcal{B}\mathrm{un}_G(C)^{\mathrm{ss}}$ of semistable $G$-bundles on $C$ is almost orthogonal (\Cref{cor-bung-is-orthogonally-symmetric}).
        \item Let $C$ and $G$ be as above, and $\mathcal{L}$ be a line bundle on $C$, either the canonical bundle $\omega_C$ or any line bundle satisfying $\deg(\mathcal{L}) > 2g(C) - 2$.
              Then the moduli stack of semistable $\mathcal{L}$-twisted $G$-Higgs bundles on $C$, denoted $\mathcal{H}\mathrm{iggs}_G^{\mathcal{L}}(C)^{\mathrm{ss}}$, is almost orthogonal (\Cref{cor-higgsG-is-orthogonally-symmetric}).
        \item Let $\Sigma$ be a compact oriented $2$-manifold and $G$ a reductive group. Then the character stack $\mathcal{L}\mathrm{oc}_G(\Sigma)$ is almost orthogonal (\Cref{cor-character-stack-symmetric}).
        \item Let $M$ be a compact oriented $3$-manifold and $G$ a reductive group. Assume that either
        \begin{itemize} 
            \item $M$ is of the form $\Sigma \times S^1$ for a compact oriented $2$-manifold $\Sigma$ (or more generally, the mapping torus associated with a finite order automorphism of $\Sigma$), or  
            \item $G$ is either $\mathrm{GL}_n$ or $\mathrm{SL}_n$.
        \end{itemize}
              Then the character stack $\mathcal{L}\mathrm{oc}_G(M)$ is almost orthogonal (\Cref{cor-sigma-s1-orthogonal} and \Cref{cor-GL-character-stack-orthogonally-symmetric}).
        \item Let $\mathcal{M}$ be the moduli stack of objects in an abelian category $\mathcal{A}$ satisfying the equality $\dim \Ext^1(E, F) = \dim \Ext^1(F, E)$ for any $E, F \in \mathcal{A}$.
              Assume further that $\mathcal{M}$ admits a good moduli space. Then $\mathcal{M}$ is orthogonal (see \Cref{para-orthogonal-linear-moduli} for the precise statement).
    \end{itemize}
\end{para}

\begin{para}[The component lattice]
    To formulate an answer to \Cref{question-decomp-general} generalizing \Cref{eq-intro-cohint-g-bundles},
    we need to describe the right-hand side of \Cref{eq-intro-cohint-g-bundles} in a way that is intrinsic to the stack $\mathcal{B}\mathrm{un}_{G}(C)_d^{\mathrm{ss}}$. We do this using the formalism of component lattices for stacks developed by C.B., A.I.N. and T.K. in collaboration with Halpern-Leistner \cite{component-lattice}, 
    which are a globalized version of the cocharacter lattices for algebraic groups.
    We briefly recall some key concepts: we refer to \Cref{subsec-component-lattice,subsec-constancy} for the details. 
    \begin{itemize}
        \item For an algebraic stack $\mathcal{X}$, we define the \emph{component lattice} $\mathrm{CL}(\mathcal{X})$, a presheaf on the category of finite-dimensional $\mathbb{Z}$-lattices, by the assignment
              \[
               \mathrm{CL}(\mathcal{X})(\mathbb{Z}^n) = \pi_0(\mathrm{Map}(\mathrm{B} \mathbb{G}_{\mathrm{m}}^n, \mathcal{X})).  
              \]
              We let $\mathrm{CL}_{\mathbb{Q}}(\mathcal{X})$ be the rationalization of $\mathrm{CL}(\mathcal{X})$. See \Cref{para-component-lattice} for the precise definition.
              When $\mathcal{X} = \mathrm{B} G$, the component lattice $\mathrm{CL}(\mathcal{X})$ recovers the cocharacter lattice modulo the Weyl group action.
        \item We define a category $\mathsf{Face}(\mathcal{X})$ of \emph{faces} of $\mathcal{X}$, whose objects are pairs $(F, \alpha)$ consisting of a finite dimensional $\mathbb{Q}$-vector space $F$ and a map $\alpha \colon F \to \mathrm{CL}_{\mathbb{Q}}(\mathcal{X})$. Morphisms are defined in the natural manner.
              A face is called \emph{non-degenerate} if it does not factor through a lower-dimensional face.
        \item For a face $(F, \alpha) \in \mathsf{Face}(\mathcal{X})$, take an integral lift $\alpha_{\mathbb{Z}} \colon F_{\mathbb{Z}} \to \mathrm{CL}(\mathcal{X})$, and define the stack
              $\mathcal{X}_{\alpha}$ to be a component of $\mathrm{Map}(\mathrm{B} \mathbb{G}_{\mathrm{m}}^{\dim F}, \mathcal{X})$ corresponding to $\alpha_{\mathbb{Z}}$. It does not depend on the choice of an integral lift, and admits a morphism $\mathrm{tot}_{\alpha} \colon \mathcal{X}_{\alpha} \to \mathcal{X}$.
        \item We say that a non-degenerate face $(F, \alpha)$ is \emph{special} if, for any morphism $(F, \alpha) \to (F', \alpha')$ of non-degenerate faces, if the induced morphism $\mathcal{X}_{\alpha'} \to \mathcal{X}_\alpha$ is an isomorphism, then $(F, \alpha) \to (F', \alpha')$ is an isomorphism. We denote by $\mathsf{Face}^{\mathrm{sp}}(\mathcal{X}) \subset \mathsf{Face}(\mathcal{X})$ the full subcategory of special faces.
        \item When $\mathcal{X}$ is connected, the \emph{central rank} of $\mathcal{X}$, denoted by $\crk \mathcal{X} \in \mathbb{Z}_{\geq 0}$, is defined to be the dimension of the (unique) special face $(F_{\mathrm{ce}}, \alpha_{\mathrm{ce}})$ with the property $\mathcal{X}_{\alpha_{\mathrm{ce}}} \cong \mathcal{X}$.
              For a special face $(F, \alpha) \in \mathsf{Face}^{\mathrm{sp}}(\mathcal{X})$, we have $\crk \mathcal{X}_{\alpha} = \dim F$.
        \item We say that $\mathcal{X}$ has \emph{quasi-compact graded points} if for any face $(F, \alpha) \in \mathsf{Face}(\mathcal{X})$, the morphism $\mathrm{tot}_{\alpha} \colon \mathcal{X}_{\alpha} \to \mathcal{X}$ is quasi-compact. 
        \item For an $n$-dimensional face $(F, \alpha) \in \mathsf{Face}(\mathcal{X})$, a \emph{global equivariant parameter} for $\mathcal{X}_{\alpha}$ is a set of line bundles $\mathcal{L}_1, \ldots, \mathcal{L}_n$ on $\mathcal{X}_{\alpha}$ whose first Chern classes, when restricted to $\mathrm{H^2}(\mathrm{B} \mathbb{G}_{\mathrm{m}} ^n)$, form a basis.
    \end{itemize}

\end{para}

\begin{para}[Assumptions]\label{para-assumptions-intro-cohint}
    With these preparations, we can state the conditions for stacks for which we can provide an answer to \Cref{question-decomp-general}.
    We consider the following conditions for an algebraic stack $\mathcal{U}$:
    \begin{enumerate}
        \item $\mathcal{U}$ has affine diagonal and admits a good moduli space $p \colon \mathcal{U} \to U$. \label{item1-intro-cohint-goodmoduli}
        \item $\mathcal{U}$ has quasi-compact connected components and quasi-compact graded points.  \label{item2-intro-cohint-qcgr}
        \item $\mathcal{U}$ is almost orthogonal.  \label{item3-intro-cohint-symmetric}
        \item For each special face $(F, \alpha) \in \mathsf{Face}^{\mathrm{sp}}(\mathcal{U})$,
        there exists a global equivariant parameter for $\mathcal{U}_{\alpha}$. \label{item4-intro-cohint-gloeq}
    \end{enumerate}
    
    We remark that the conditions \Cref{item2-intro-cohint-qcgr} and \Cref{item4-intro-cohint-gloeq} are quite mild: they are satisfied for quotient stacks, as shown in \Cref{eg-grad-quotient-stack} and \Cref{lem-global-quotient-global-equiv-parameter}, and condition \Cref{item4-intro-cohint-gloeq} is satisfied for smooth stacks, see \Cref{cor-global-equiv-parameter-smooth}.
    Additionally, recall that we have seen many examples of almost orthogonal stacks in \Cref{para-intro-almost-orthogonal-stacks}.
\end{para}

\begin{para}[Cohomological integrality theorem: smooth case]
    We start by explaining the main result of this paper for smooth stacks. Let $\mathcal{U}$ be a smooth stack satisfying assumptions \Crefrange{item1-intro-cohint-goodmoduli}{item4-intro-cohint-gloeq} in \Cref{para-assumptions-intro-cohint}.
    For a face $(F, \alpha) \in \mathsf{Face}(\mathcal{U})$, consider the following commutative diagram:
\[\begin{tikzcd}
	{\mathcal{U}_{\alpha}} & {\mathcal{U}} \\
	{U_{\alpha}} & {U}
	\arrow["\mathrm{tot}_{\alpha}", from=1-1, to=1-2]
	\arrow["{p_{\alpha}}"', from=1-1, to=2-1]
	\arrow["p", from=1-2, to=2-2]
	\arrow["{g_{\alpha}}"', from=2-1, to=2-2]
\end{tikzcd}\]
    where the vertical maps are the good moduli space morphisms.
    The main result, which generalizes \Cref{eq-intro-cohint-g-bundles}, is as follows:
\end{para}

\begin{theorem}[\Cref{thm-cohint-smooth} + \Cref{rmk-cohint-source-smooth}]\label{thm-intro-cohint-smooth}
    There exists an isomorphism of monodromic mixed Hodge complexes on $U$:
    \begin{equation}\label{eq-intro-cohint-smooth}
        \bigoplus_{(F, \alpha) \in \mathsf{Face}^{\mathrm{sp}}(\mathcal{U})}  (g_{\alpha, *}  \mathcal{IC}_{{U}_{\alpha}}^{\circ} \otimes \mathrm{H}^*(\mathrm{B} \mathbb{G}_{\mathrm{m}}^{\dim F})_{\mathrm{vir}} \otimes \mathrm{sgn}_{\alpha})^{\mathrm{Aut}(\alpha)}  \cong p_* \mathcal{IC}_{\mathcal{U}}
    \end{equation}
    where the direct sum runs over all isomorphism classes of special faces. Here we use the following notations:
    \begin{itemize}
        \item $\mathcal{IC}_{U_{\alpha}}$ and $\mathcal{IC}_{\mathcal{U}}$ denote the intersection complexes, normalized as monodromic mixed Hodge modules with weight zero:
             see \Cref{MMHM-section}, and in particular \Cref{para-half-Lefshetz-twist}, for background.
        \item We set 
        \[
            \mathcal{IC}_{U_{\alpha}}^{\circ} = 
            \begin{cases}
                \mathcal{IC}_{U_{\alpha}}  & \quad \text{if the map $\mathcal{U}_{\alpha} / \mathrm{B} \mathbb{G}_{\mathrm{m}}^{\dim F} \to U_{\alpha}$ is generically quasi-finite,} \\
                0  & \quad \text{otherwise.} 
            \end{cases}
        \]
        See \Cref{para-rational-torus} for the choice of the $\mathrm{B} \mathbb{G}_{\mathrm{m}}^{\dim F}$-action on $\mathcal{U}_{\alpha}$. Here, a generically quasi-finite morphism is defined as a morphism that is quasi-finite generically on the target.
        \item We set $\mathrm{H}^*(\mathrm{B} \mathbb{G}_{\mathrm{m}}^{\dim F})_{\mathrm{vir}}  \coloneqq \mathbb{L}^{\dim F / 2} \otimes \mathrm{H}^*(\mathrm{B} \mathbb{G}_{\mathrm{m}}^{\dim F})$.
        \item The action of $\mathrm{Aut}(\alpha)$ on $g_{\alpha, *}  \mathcal{IC}_{{U}_{\alpha}}^{\circ} $ is induced by its action on $U_{\alpha}$. Its action on $\mathrm{H}^*(\mathrm{B} \mathbb{G}_{\mathrm{m}}^{\dim F})_{\mathrm{vir}}$ is induced by the action on $F$.
        \item $\mathrm{sgn}_{\alpha}$ denotes the cotangent sign representation of $\mathrm{Aut}(\alpha)$ (see \Cref{para-cotangent-distance}).
    \end{itemize}
\end{theorem}

By taking the global sections of \Cref{eq-intro-cohint-smooth}, we obtain the following:

\begin{corollary}
    There exists an isomorphism of monodromic mixed Hodge complexes on a point:
    \begin{equation}
        \bigoplus_{(F, \alpha) \in \mathsf{Face}^{\mathrm{sp}}(\mathcal{U})}  (\mathrm{IH}^{\circ, *}(U_{\alpha}) \otimes \mathrm{H}^*(\mathrm{B} \mathbb{G}_{\mathrm{m}}^{\dim F})_{\mathrm{vir}} \otimes \mathrm{sgn}_{\alpha})^{\mathrm{Aut}(\alpha)}  \cong \mathbb{L}^{- {\dim \mathcal{U}} / 2} \otimes  \mathrm{H}^*(\mathcal{U})
    \end{equation}
    where we set $\mathrm{IH}^{\circ, *}(U_{\alpha}) \coloneqq \mathrm{H}^*(U_{\alpha}, \mathcal{IC}_{U_{\alpha}}^{\circ})$.
\end{corollary}

\begin{para}[Enumerative meaning of cohomological integrality]\label{para-enumerative-meaning}
    A special case of \Cref{thm-intro-cohint-smooth}, after passing to the Grothendieck group of mixed Hodge modules, first appeared in the work of \textcite[Theorem 1.1]{Meinhardt_Reineke_DT_vs_intersection_cohomology} where $\mathcal{U}$ is the moduli stack of semistable quiver representations with respect to a generic stability condition.
    Their theorem can be paraphrased as stating that the intersection complex categorifies the BPS (Bogomol'nyi--Prasad--Sommerfield)  invariant for quivers.
    In particular, this result implies an integrality property of the generalized Donaldson--Thomas invariant for quivers introduced in \textcite{joyce2008configurationsIV}. Consequently, their theorem is called the \emph{cohomological integrality theorem}.
    We do not delve into the details of this perspective here and advise the reader to consult \cite[\S 6.7]{davison2015donaldson} for background.
    Following this, we refer to  \Cref{thm-intro-cohint-smooth} as the cohomological integrality theorem. The enumerative aspect of this theorem will be explained in a forthcoming paper \cite{wall-crossing} by C.B., A.I.N. and T.K.; see also \Cref{para-BPS-invariant} for further discussions.
\end{para}

\begin{para}[Cohomological Hall induction]
    As noted in the last paragraph, the special case of \Cref{thm-intro-cohint-smooth} for quiver moduli spaces was established by \textcite{Meinhardt_Reineke_DT_vs_intersection_cohomology}.
    Their proof, however, relies on a smooth model of quiver moduli spaces realized as the moduli space of stably framed quiver representations, which has no analogue for general affine GIT quotients.
    Consequently, we are required to introduce new ideas in the proof; indeed, our argument differs substantially from that of \cite{Meinhardt_Reineke_DT_vs_intersection_cohomology} and is also independent of it.

    A key ingredient of our proof is the \emph{cohomological Hall induction},  a generalization of the cohomological Hall algebra multiplication that includes operations such as parabolic induction.
    For each face $(F, \alpha) \in \mathsf{Face}(\mathcal{U})$, one can define a hyperplane arrangement on $F$ using the weights of the cotangent complex of $\mathcal{U}$.
    Consider a chamber $\sigma \subset F$ with respect to the arrangement and let $\mathcal{U}_{\sigma}^+$ denote the connected component of $\mathrm{Map}(\mathbb{A}^1 / \mathbb{G}_{\mathrm{m}}, \mathcal{U})$ corresponding to $\sigma$: see \Cref{para-notation-x-alpha} and \Cref{para-constancy-theorem} for details.
    Then we can form the following diagram:
\[\begin{tikzcd}
	& {\mathcal{U}_{\sigma}^+} \\
	{\mathcal{U}_{\alpha}} && {\mathcal{U}} \\
	{U_{\alpha}} && U.
	\arrow["{\mathrm{gr}_{\sigma}}"', from=1-2, to=2-1]
	\arrow["{\mathrm{ev}_{1, \sigma}}", from=1-2, to=2-3]
	\arrow["{p_{\alpha}}"', from=2-1, to=3-1]
	\arrow["p", from=2-3, to=3-3]
	\arrow["{g_{\alpha}}"', from=3-1, to=3-3]
\end{tikzcd}\]
     The upper roof is a generalization of the correspondence of the moduli space of objects in abelian categories, provided by the stack of filtrations: see \Cref{ssec-linear-moduli} for details.
     One easily sees that $\mathrm{gr}_{\sigma}$ is smooth and $\mathrm{ev}_{1, \sigma}$ is proper. This allows us to define a natural morphism
     \begin{equation}\label{eq-intro-cohi-smooth}
     g_{\alpha, *} p_{\alpha, *} \mathcal{IC}_{\mathcal{U}_{\alpha}} \to p_* \mathcal{IC}_{\mathcal{U}},   
     \end{equation}
     referred to as the relative cohomological Hall induction. This morphism recovers the cohomological Hall algebra multiplication when $\mathcal{U}$ is the moduli space of objects in an abelian category.
    A choice of the global equivariant parameter induces a natural map
    \begin{equation}\label{eq-global-equiv-induce}
     \mathcal{IC}_{U_{\alpha}} \otimes \mathrm{H}^*(\mathrm{B} \mathbb{G}_{\mathrm{m}}^{\dim F})_{\mathrm{vir}} \to p_{\alpha, *} \mathcal{IC}_{\mathcal{U}_{\alpha}}.    
    \end{equation}
    The maps \Cref{eq-intro-cohi-smooth} and \Cref{eq-global-equiv-induce} induce the map \Cref{eq-intro-cohint-smooth} from left to right.
    Finally, this map is proven to be an isomorphism by reducing to the case of the classifying stack, using the vanishing cycle functor with respect to a locally defined non-degenerate quadratic function induced from the almost orthogonal structure.
\end{para}

\begin{para}[Cohomological integrality theorem: \texorpdfstring{$(-1)$}{$(-1)$}-shifted symplectic case]\label{para-intro-cohint--1}
    We now turn our attention to the cohomological integrality theorem for singular spaces.
    This will allow us to prove results for moduli stacks of semistable coherent sheaves on Calabi--Yau threefolds, which are central objects in enumerative geometry, and provide the motivating class of examples for the integrality conjecture in Donaldson--Thomas theory.

    As outlined in \Cref{para-enumerative-meaning}, the cohomological integrality theorem (= \Cref{thm-intro-cohint-smooth}) was first studied in the context of quiver moduli spaces by \textcite{Meinhardt_Reineke_DT_vs_intersection_cohomology}.
    Subsequently, Davison and Meinhardt \cite[Theorem A, Theorem C]{_Davison_Meinhardt_CoDT} generalized this result to the moduli space of representations over Ginzburg dg-algebras associated with quivers with potentials, using the vanishing cycle cohomology instead of the ordinary cohomology.
    Ginzburg dg-algebras are non-commutative 3-Calabi–Yau spaces which are the simplest examples for such studies, and they provide models for certain non-compact Calabi--Yau threefolds such as the resolved conifold.

    Since the appearance of the work of \textcite{_Davison_Meinhardt_CoDT}, it has been an open problem to extend the cohomological integrality theorem for quivers with potentials to general $3$-Calabi--Yau categories, such as the category of coherent sheaves on a smooth Calabi--Yau $3$-fold.
    Our second main theorem addresses this problem by presenting a version of \Cref{thm-intro-cohint-smooth} for stacks with $(-1)$-shifted symplectic structures---a geometric structure introduced by \textcite{pantev2013shifted} that is possessed by moduli stacks of objects in $3$-Calabi--Yau categories, reflecting the trivialization of the Serre duality functor.

    Let $\mathcal{X}$ be a $(-1)$-shifted symplectic stack equipped with an orientation, i.e., a choice of the square root of the canonical bundle $\det(\mathbb{L}_{\mathcal{X}})$.
    For such stacks, \textcite[Theorem 4.8]{ben2015darboux} associated a monodromic mixed Hodge module
    \[
     \varphi_{\mathcal{X}} \in \mathsf{MMHM}(\mathcal{X})    
    \]
    which is a globalization of the vanishing cycle complex. Now, suppose that $\mathcal{X}$ satisfies assumptions \Crefrange{item1-intro-cohint-goodmoduli}{item4-intro-cohint-gloeq} in \Cref{para-assumptions-intro-cohint}.
    The version of \Cref{question-decomp-general} for $(-1)$-shifted symplectic stacks asks whether we can decompose $p_* \varphi_{\mathcal{X}}$ into simpler pieces.

    For the purpose of answering this question, we will introduce the \emph{$n$-th BPS sheaf} on the good moduli space for connected $\mathcal{X}$,
    which is defined as the $0$-th perverse cohomology sheaf
    \[
    \mathcal{BPS}^{(n)}_{X} \coloneqq {}^{\mathrm{p}} \mathcal{H}^0(\mathbb{L}^{- {n /2}} \otimes p_{ *} \varphi_{\mathcal{X}}) \in \mathsf{MMHM}(X).
    \]
    We will prove in \Cref{prop-support-lemma} that for $n < \crk \mathcal{X}$,
    the $n$-th BPS sheaf $\mathcal{BPS}^{(n)}_{X}$ vanishes  (see \Cref{para-central-rank} for the definition of the central rank $\crk \mathcal{X}$ of $\mathcal{X}$).
    We will write $\mathcal{BPS}_{X} = \mathcal{BPS}^{(\crk \mathcal{X})}_{X}$.
    We define the \emph{$n$-th BPS cohomology} to be the global sections
    \[
     \mathrm{H}^*_{\mathrm{BPS}^{(n)}} (X) \coloneqq    \mathrm{H}^* (X, \mathcal{BPS}_{X}^{(n)}). 
    \]
    Again, we will write $\mathrm{H}^*_{\mathrm{BPS}} (X)  \coloneqq \mathrm{H}^*_{\mathrm{BPS}^{(\crk \mathcal{X})}} (X)$.

    For a special face $(F, \alpha) \in \mathsf{Face}^{\mathrm{sp}}(\mathcal{X})$, consider the following diagram:
    \[\begin{tikzcd}
        {\mathcal{X}_{\alpha}} & {\mathcal{X}} \\
        {X_{\alpha}} & {X}
        \arrow["\mathrm{tot}_{\alpha}", from=1-1, to=1-2]
        \arrow["{p_{\alpha}}"', from=1-1, to=2-1]
        \arrow["p", from=1-2, to=2-2]
        \arrow["{g_{\alpha}}"', from=2-1, to=2-2]
    \end{tikzcd}\]
    where the vertical maps are the good moduli space morphisms.
    As we will see in \Cref{para-localize-shifted-symplectic-structure} and \Cref{para-localized-orientation}, $\mathcal{X}_{\alpha}$ is naturally equipped with the structure of an oriented $(-1)$-shifted symplectic stack.
    In particular, the BPS sheaf and the BPS cohomology are defined for $X_{\alpha}$.
    The following theorem, providing a decomposition of the vanishing cycle cohomology (also called critical cohomology) into simpler pieces, is the main result for $(-1)$-shifted symplectic stacks:
\end{para}

\begin{theorem}[\Cref{thm-coh-int--1-shifted-symplectic} + \Cref{rmk-cohint-source--1-shifted-symplectic}]\label{thm-intro-cohint--1-symplectic}
    There exists an isomorphism of monodromic mixed Hodge complexes on $X$:
    \begin{equation}\label{eq-intro-cohint--1-shifted-symplectic}
        \bigoplus_{(F, \alpha) \in \mathsf{Face}^{\mathrm{sp}}(\mathcal{X})}  (g_{\alpha, *}  \mathcal{BPS}_{X_{\alpha}} \otimes \mathrm{H}^*(\mathrm{B} \mathbb{G}_{\mathrm{m}}^{\dim F})_{\mathrm{vir}})^{\mathrm{Aut}(\alpha)}  \cong p_* \varphi_{\mathcal{X}}.
    \end{equation}
    In particular, we have the following isomorphism of monodromic mixed Hodge complexes on a point:
    \begin{equation}
        \bigoplus_{(F, \alpha) \in \mathsf{Face}^{\mathrm{sp}}(\mathcal{X})}  ( \mathrm{H}_{\mathrm{BPS}}^*(X_{\alpha}) \otimes \mathrm{H}^*(\mathrm{B} \mathbb{G}_{\mathrm{m}}^{\dim F})_{\mathrm{vir}})^{\mathrm{Aut}(\alpha)}  \cong \mathrm{H}^*(\mathcal{X}, \varphi_{\mathcal{X}}).
    \end{equation}

\end{theorem}

\begin{para}\label{para-intro-cohint-is-induction}
    Similarly to the smooth setting in \Cref{thm-intro-cohint-smooth}, the isomorphism \Cref{eq-intro-cohint--1-shifted-symplectic} is also induced by the cohomological Hall induction.
    However, in the $(-1)$-shifted symplectic case, the existence of the cohomological Hall induction map is highly nontrivial. It relies on a recent result by Kinjo, Park and Safronov \cite[Corollary 7.19]{Kinjo_Park_Safronov_CoHA} establishing the categorified version of the integral identity \`a la \textcite[\S 7.8]{_Kontsevich_Soibelman_CoHA} within the framework of $(-1)$-shifted symplectic stacks.
\end{para}

\begin{para}[BPS invariants]\label{para-BPS-invariant}
    For a connected oriented $(-1)$-shifted symplectic stack $\mathcal{X}$ satisfying assumptions \Crefrange{item1-intro-cohint-goodmoduli}{item4-intro-cohint-gloeq} in \Cref{para-assumptions-intro-cohint},
    we define the \emph{BPS invariant} to be the Euler characteristic
    \[
     \mathrm{BPS}_{X} \coloneqq \chi(X, \mathcal{BPS}_{X}) \in \mathbb{Z}.    
    \]
    The BPS invariant (in the context of enumerative geometry via sheaves) was first introduced in the work of \textcite[Definition 6.10]{_Joyce_Song_AtheoryofgeneralizedDonaldsonThomasinvariants} for the moduli stack of semistable coherent sheaves on a projective  Calabi--Yau threefold with respect to a generic polarization, motivated by the work of \textcite{kontsevich2008stability}.
    These invariants provide a mathematical formulation of the BPS state counting in type II superstring compactifications on a Calabi--Yau threefold.
    The idea of the definition due to Joyce and Song is to use the multiple cover formula \cite[(6.14)]{_Joyce_Song_AtheoryofgeneralizedDonaldsonThomasinvariants} relating the BPS invariant and the generalized Donaldson--Thomas invariant.
    Though our definition of the BPS invariant a priori does not specialize to theirs, one can prove the multiple cover formula with our definition of the BPS invariant using \Cref{eq-intro-cohint--1-shifted-symplectic} and repeating the discussion in \cite[\S 6.7]{davison2015donaldson}.
    In particular, our definition of the BPS invariants recovers the definition by \textcite{_Joyce_Song_AtheoryofgeneralizedDonaldsonThomasinvariants} as a special case.

    In a closely related work, C.B., A.I.N. and T.K. \cite{invariants} introduce generalized Donaldson--Thomas invariants for $(-1)$-shifted symplectic stacks equipped with extra data similar to a stability condition, generalizing the work of \textcite{_Joyce_Song_AtheoryofgeneralizedDonaldsonThomasinvariants}.
    In its forthcoming sequel \cite{wall-crossing}, they will formulate and prove a generalization of the multiple cover formula relating the BPS invariants and the generalized Donaldson--Thomas invariants using the isomorphism \Cref{eq-intro-cohint--1-shifted-symplectic}.

\end{para}

\begin{para}[Cohomological integrality theorem for \texorpdfstring{$3$}{$3$}-Calabi--Yau categories]\label{para-cohint-intro-3-CY}

    We specialize the discussions in \Crefrange{para-intro-cohint--1}{para-BPS-invariant}  to the case in which the stack is the moduli stack of objects in a $3$-Calabi--Yau category.
    Let $\mathcal{C}$ be a finite type left $3$-Calabi--Yau dg-category, such as the derived category of coherent sheaves on a smooth Calabi--Yau threefold.
    Let $\mathcal{M}_{\mathcal{C}}$ be the moduli stack of objects in $\mathcal{C}$, which is $(-1)$-shifted symplectic, as shown by \textcite[Theorem 5.5]{brav2021relative}.
    Let $\mathcal{M} \subset \mathcal{M}_{\mathcal{C}}$ be a $1$-Artin open substack satisfying the following conditions:
    \begin{enumerate}
        \item $\mathcal{M}$ contains the zero object as an open and closed substack.
        \item $\mathcal{M}$ parametrizes objects in an abelian subcategory $\mathcal{A} \subset \mathcal{C}$, and for each pair of objects $E, F \in \mathcal{A}$, we have $\dim \Hom(E, F[1]) = \dim \Hom(F, E[1])$.
        \item $\mathcal{M}$ admits a good moduli space $p \colon \mathcal{M} \to M$.
        \item $\mathcal{M}$ has quasi-compact connected components.
        \item For each non-zero class $\gamma \in \pi_0(\mathcal{M})$, the associated class in the numerical Grothendieck group $K^{\mathrm{num}}(\mathcal{C})$ is non-zero.
        \item $\mathcal{M}$ admits a commutative orientation data, i.e., a choice of orientation for $\mathcal{M}$ which is compatible with the direct sum map in a certain commutative manner. See \Cref{para-COD} for the details.
    \end{enumerate}
    In this case, we see that $\mathcal{M}$ satisfies the assumptions \Crefrange{item1-intro-cohint-goodmoduli}{item4-intro-cohint-gloeq} in \Cref{para-assumptions-intro-cohint}, hence the cohomological integrality theorem (= \Cref{thm-intro-cohint--1-symplectic}) holds for $\mathcal{M}$.
    We will explicitly describe the statement.
    For each $\gamma \in \pi_0(\mathcal{M})$, we let $\mathcal{M}_{\gamma} \subset \mathcal{M}$ be the corresponding connected component, with $p_{\gamma} \colon \mathcal{M}_{\gamma} \to M_{\gamma}$ the good moduli space, and set
    \[
     \mathcal{BPS}_{M, \gamma} \coloneqq  \mathcal{BPS}_{M_{\gamma}}^{(1)}  \in \mathsf{MMHM}(M_{\gamma}).
    \]
    We let $\oplus \colon M \times M \to M$ be the direct sum map and define a symmetric monoidal product $\boxtimes_{\oplus}$ on $\mathsf{D}_{\mathsf{H}}^{\mathsf{mon}, (+)}(M)$ by
    \[
        (E, F) \mapsto \oplus_{*} (E \boxtimes F).         
    \]
    We let $\mathrm{Sym}_{\boxtimes_{\oplus}} \colon  \mathsf{D}_{\mathsf{H}}^{\mathsf{mon}, (+)}(M) \to \mathsf{D}_{\mathsf{H}}^{\mathsf{mon}, (+)}(M)$ be the associated symmetric product functor.
    Then the cohomological integrality theorem (= \Cref{thm-intro-cohint--1-symplectic}) for $\mathcal{M}$ can be written as follows:

\end{para}

\begin{theorem}[\Cref{thm-cohint-linear--1-shifted}]
    Under the above assumptions, there exists an isomorphism
    \begin{equation}\label{eq-coh-int-intro--1-linear}
        \mathrm{Sym}_{\boxtimes_{\oplus}} \left( \bigotimes_{\gamma \in \pi_0(\mathcal{M}) \setminus 0} (\mathcal{BPS}_{{M}, \gamma} \otimes \mathrm{H}^*(\mathrm{B} \mathbb{G}_\mathrm{m})_{\mathrm{vir}})   \right)  \cong p_* \varphi_{\mathcal{M}}.
    \end{equation}
\end{theorem}

\begin{para}
    As explained in \Cref{para-intro-cohint-is-induction}, the isomorphism \Cref{eq-coh-int-intro--1-linear} is induced by the cohomological Hall algebra multiplication,
    which is constructed by \textcite[Corollary 8.8]{Kinjo_Park_Safronov_CoHA} as a special case of the cohomological Hall induction.
\end{para}

\begin{para}[BPS Lie algebra]\label{para-intro-BPS-Lie}
    We adopt the notation from \Cref{para-cohint-intro-3-CY}.
    Assume further that the chosen orientation data is associative (see \cite[Definition 8.4]{Kinjo_Park_Safronov_CoHA}).
    In this case, the commutator of the cohomological Hall algebra multiplication induces a Lie algebra structure on the global sections of the BPS sheaf
    \[
    [-, -] \colon \mathrm{H}^*(M_{\gamma_1}, \mathcal{BPS}_{M, \gamma_1}) \otimes     \mathrm{H}^*(M_{\gamma_2}, \mathcal{BPS}_{M, \gamma_2}) \to \mathrm{H}^*(M_{\gamma_1 + \gamma_2}, \mathcal{BPS}_{M, \gamma_1 + \gamma_2}).
    \]
    We call it the \emph{BPS Lie algebra}. The isomorphism \Cref{eq-coh-int-intro--1-linear} can be regarded as a PBW-type theorem for the BPS Lie algebra.

    The structure of the BPS Lie algebra for the $3$-Calabi--Yau completion of a $2$-Calabi--Yau dg-category was studied in detail by \textcite{davison2023bps}.
    Roughly, they proved that the BPS Lie algebra can be described as a generalized Kac--Moody Lie algebra in this case, and recovered Nakajima's construction of the Kac--Moody Lie algebra action on the homology of quiver varieties \cite{nakajima1998quiver} as well as the Heisenberg algebra action on the homology of Hilbert schemes of points \cite{nakajima1997heisenberg}. 
    The recent result of Botta and Davison \cite{botta2023okounkov} identifies the BPS Lie algebra, for a special class of 3-Calabi--Yau completions, with the Maulik--Okounkov Lie algebra introduced in \cite{maulik2019quantum}. This Lie algebra forms a powerful bridge between quantum groups and the study of quantum cohomology of Nakajima quiver varieties. We believe that the study of the BPS Lie algebra for general 3-Calabi--Yau categories would be similarly fruitful, and that it could provide a representation-theoretic approach to problems in Donaldson--Thomas theory, such as the $\chi$-independence conjecture \cite[Conjecture 1.2]{toda_gopakumar_wall_crossing}.
\end{para}

\begin{para}[Geometric Langlands conjecture for \texorpdfstring{$3$}{$3$}-manifolds]\label{para-intro-3mfd}
    A  motivation for \Cref{thm-intro-cohint--1-symplectic} arises from the geometric Langlands conjecture for $3$-manifolds.
    As demonstrated by \textcite{kapustin2006electric}, the geometric Langlands conjecture is expected to be interpreted via a duality between $4$-dimensional TQFTs.
    In particular, by considering the state space for $3$-manifolds, we should have a certain duality phenomenon for invariants associated with $3$-manifolds.

    One example of a $3$-manifold invariant expected to satisfy the Langlands duality, which we learned from Pavel Safronov, is the vanishing cycle cohomology of character stacks.
    It is shown by \textcite[Theorem 3.45]{naef2023torsion} that a spin structure on a $3$-manifold $M$ induces an orientation for $\mathcal{L}\mathrm{oc}_G(M)$.
    Then we have the following conjecture, proposed by Safronov:
\end{para}

\begin{conjecture*}[Safronov, {\cite[Conjecture 1]{kaubrys2024cohomological}}]
    There exists a natural isomorphism
    \begin{equation}\label{eq-conj-Langlands-intro}
     \mathrm{H}^*(\mathcal{L}\mathrm{oc}_G(M), \varphi_{\mathcal{L}\mathrm{oc}_G(M)}) \cong   \mathrm{H}^*(\mathcal{L}\mathrm{oc}_{G^{\vee}}(M), \varphi_{\mathcal{L}\mathrm{oc}_{G^{\vee}}(M)}).
    \end{equation}
\end{conjecture*}

    When the character stack $\mathcal{L}\mathrm{oc}_G(M)$ is almost orthogonal, e.g.~when $M = \Sigma \times S^1$ (\Cref{cor-sigma-s1-orthogonal}) or $G = \mathrm{GL}_n, \mathrm{SL}_n$ (\Cref{cor-GL-character-stack-orthogonally-symmetric}),
    one can apply \Cref{thm-intro-cohint--1-symplectic} to $\mathcal{L}\mathrm{oc}_G(M)$, which we will explicitly describe below.
    For each cocharacter $\lambda \colon \mathbb{G}_{\mathrm{m}}^n \to G$,
    we let $L_{\lambda}$ be the corresponding Levi subgroup and consider the following diagram
    \[\begin{tikzcd}
	{\mathcal{L}\mathrm{oc}_{L_{\lambda}}(M)} & {\mathcal{L}\mathrm{oc}_{G}(M)} \\
	{\mathrm{Loc}_{L_{\lambda}}(M)} & {\mathrm{Loc}_{G}(M).}
	\arrow[from=1-1, to=1-2]
	\arrow["{p_{\lambda}}"', from=1-1, to=2-1]
	\arrow["p", from=1-2, to=2-2]
	\arrow["{g_{\lambda}}"', from=2-1, to=2-2]
    \end{tikzcd}\]
    Here, the vertical maps are good moduli space morphisms.
    We let $W_{\lambda}$ denote the relative Weyl group of $L_{\lambda}$, which naturally acts on $\mathrm{Loc}_{L_{\lambda}}(M)$.
    Then, \Cref{thm-intro-cohint--1-symplectic} can be written as
    \begin{equation}\label{eq-intro-cohint-3-mfd}
      \bigoplus_{n \geq 0} \ \bigoplus_{\lambda \colon \mathbb{G}_{\mathrm{m}}^n \to G} \left(  g_{\lambda,  *} (\mathcal{BPS}_{\mathrm{Loc}_{L_{\lambda}}(M)}^{(n)} \otimes \mathrm{H}^*(\mathrm{B} \mathbb{G}_\mathrm{m}^n)_{\mathrm{vir}})\right)^{W_{\lambda}}  \cong p_* \varphi_{\mathcal{L}\mathrm{oc}_G(M)}.
    \end{equation}
    Here, the index set of the left-hand side runs over all cocharacters up to conjugation.
    We expect \Cref{eq-intro-cohint-3-mfd} to hold for arbitrary character stacks of $3$-manifolds without the almost orthogonality hypothesis, and for the conjectural Langlands duality isomorphism \Cref{eq-conj-Langlands-intro} to preserve the decomposition.
    This leads to the following conjecture:

\begin{conjecture*}[\Cref{conj-geometric-langlands-3-mfds}]\label{conj-intro-geometric-langlands-3-mfds}
    For a compact spin $3$-manifold $M$ and a reductive group $G$ with $c_G \coloneqq \dim \mathrm{Z} (G)$, there exists a natural isomorphism
        \[
            \mathrm{H}^*_{\mathrm{BPS}^{(c_G)}}(\mathrm{Loc}_G(M)) \cong \mathrm{H}^*_{\mathrm{BPS}^{(c_{G^{\vee}})}}(\mathrm{Loc}_{G^{\vee}}(M))
        \]
        equivariant with respect to the isomorphism $\mathrm{Out}_{\mathrm{symp}}(G) \cong \mathrm{Out}_{\mathrm{symp}}(G^{\vee})$, where $\mathrm{Out}_{\mathrm{symp}}(G)$ denotes the group of outer automorphisms of $G$ preserving a fixed $G$-invariant metric on $\mathfrak{g}$.
\end{conjecture*}

    Once \Cref{thm-intro-cohint--1-symplectic} is established for character stacks of $3$-manifolds, the above conjecture implies the original Langlands duality for the vanishing cycle cohomology of the character stacks.
    We believe that the above conjecture is easier to verify, since the BPS cohomology is finite dimensional.
    For example, when $M = T^3$ and $G = \mathrm{SL}_p$ for a prime $p$, \textcite{kaubrys2024cohomological} proves the Langlands duality conjecture for the vanishing cycle cohomology of character stacks by reducing to the computation of the BPS cohomology.
    Langlands duality for $M = T^3$ with more general gauge groups (including $\mathrm{SL}_n$, $\mathrm{Sp}_{2n}$ and $\mathrm{SO}_{2n + 1}$ for general $n$) will be proved by Hennecart and T.K. \cite{Kinjo_Springer} building on the results of this paper.

    Another source of examples of $3$-manifold invariants satisfying the Langlands duality is given by skein modules, see the survey by \textcite{jordan2024langlands}. 
    A conjectural relation between the skein modules and the vanishing cycle cohomology is explained by \textcite[Conjecture C]{gunningham2023deformation}.

\begin{para}[Cohomological integrality theorem: \texorpdfstring{$0$}{$0$}-shifted symplectic case]
    We now turn our attention to the cohomological integrality theorem for $0$-shifted symplectic stacks.
    Let $\mathcal{Y}$ be a $0$-shifted symplectic stack satisfying assumptions \Crefrange{item1-intro-cohint-goodmoduli}{item4-intro-cohint-gloeq} in \Cref{para-assumptions-intro-cohint}. Define $\mathcal{X} = \mathrm{T}^*[-1] \mathcal{Y}$ and equip it with the natural $(-1)$-shifted symplectic structure and the natural orientation.
    Let $\pi \colon \mathcal{X} \to \mathcal{Y}$ be the natural projection.
    Then the dimensional reduction theorem, proved by Kinjo \cite[Theorem 4.14]{kinjo2022dimensional} based on the work of Davison \cite[Theorem A.1]{davison2017critical}, provides the following isomorphism
    \begin{equation}\label{eq-intro-dimensional-reduction}
    \pi_* \varphi_{\mathcal{X}} \cong \mathbb{L}^{\vdim \mathcal{Y} / 2} \otimes \mathbb{D} \mathbb{Q}_{\mathcal{Y}},
    \end{equation}
    where $\mathbb{D} \mathbb{Q}_{\mathcal{Y}}$ denotes the dualizing complex.
    This isomorphism enables us to use cohomological Donaldson--Thomas theory to study the Borel--Moore homology of $0$-shifted symplectic stacks.

    For each special face $(F, \alpha) \in \mathsf{Face}^{\mathrm{sp}}(\mathcal{Y})$, consider the following diagram
    \[\begin{tikzcd}
        {\mathcal{Y}_{\alpha}} & {\mathcal{Y}} \\
        {Y_{\alpha}} & {Y,}
        \arrow["\mathrm{tot}_{\alpha}", from=1-1, to=1-2]
        \arrow["{p_{\alpha}}"', from=1-1, to=2-1]
        \arrow["p", from=1-2, to=2-2]
        \arrow["{g_{\alpha}}"', from=2-1, to=2-2]
    \end{tikzcd}\]
    where the vertical maps are the good moduli space morphisms.
    Also, consider the projection map
    \[
     \bar{\pi}_{\alpha} \colon X_{\alpha} \to Y_{\alpha},
    \]
    where $X_{\alpha}$ is the good moduli space of $\mathcal{X}_{\alpha}$.
    We set
    \[
     \mathcal{BPS}_{Y_{\alpha}} \coloneqq \mathbb{L}^{- \dim F / 2} \otimes  \bar{\pi}_{\alpha, *}  \mathcal{BPS}_{X_{\alpha}}.  
    \]
    Then, a generalized version of the support lemma  (= \Cref{prop-support-lemma}) extending the result of Davison \cite[Lemma 4.1]{davison_integrality_preprojective} implies that $\mathcal{BPS}_{Y_{\alpha}}$  is a pure mixed Hodge module.
    Combining \Cref{thm-intro-cohint--1-symplectic} with the dimensional reduction theorem \Cref{eq-intro-dimensional-reduction}, we obtain the following:
\end{para}

\begin{theorem}[\Cref{thm-cohint-0-shifted} + \Cref{rmk-cohint-source-0-shifted-symplectic}]\label{thm-intro-cohint-0-shifted}
    There exists an isomorphism of monodromic mixed Hodge complexes on $Y$
    \[
        \bigoplus_{(F, \alpha) \in \mathsf{Face}^{\mathrm{sp}}(\mathcal{Y})}  (g_{\alpha, *}  \mathcal{BPS}_{Y_{\alpha}} \otimes \mathrm{H}^*(\mathrm{B} \mathbb{G}_{\mathrm{m}}^{\dim F}) \otimes \mathrm{sgn}_{\alpha} )^{\mathrm{Aut}(\alpha)}  \cong  \mathbb{L}^{\vdim \mathcal{Y} / 2} \otimes p_* \mathbb{D} \mathbb{Q}_{\mathcal{Y}}
    \]
    where $\mathrm{sgn}_{\alpha}$ denotes the cotangent sign representation of $\mathrm{Aut}(\alpha)$ (see \Cref{para-cotangent-distance}).
    In particular,  we have the following isomorphism of monodromic mixed Hodge complexes on a point:
    \[
        \bigoplus_{(F, \alpha) \in \mathsf{Face}^{\mathrm{sp}}(\mathcal{X})}  ( \mathrm{H}_{\mathrm{BPS}}^*(Y_{\alpha}) \otimes \mathrm{H}^*(\mathrm{B} \mathbb{G}_{\mathrm{m}}^{\dim F}) \otimes \mathrm{sgn}_{\alpha})^{\mathrm{Aut}(\alpha)}  \cong \mathbb{L}^{\vdim \mathcal{Y} / 2} \otimes \mathrm{H}_{- *} ^{\mathrm{BM}} (\mathcal{Y}_{\alpha})
    \]
    where we set $\mathrm{H}_{\mathrm{BPS}}^*(Y_{\alpha}) \cong \mathrm{H}^*(Y_{\alpha}, \mathcal{BPS}_{Y_{\alpha}}  )$.
\end{theorem}

\begin{para}
    
    A special case of \Cref{thm-intro-cohint-0-shifted} was proved by Davison, Hennecart and Schlegel Mejia \cite[Theorem 1.1]{davison2023bps} for the moduli space of objects in $2$-Calabi--Yau categories.
    We note that \cite[Theorem 1.1]{davison2023bps} further provides an explicit formula for the BPS sheaves. An explicit determination of the BPS sheaves for arbitrary almost orthogonal $0$-shifted symplectic stacks would be an interesting challenge, a special case of which will be studied in the forthcoming paper by Hennecart and T.K. \cite{Kinjo_Springer}.  We expect that intersection multiplicities of IC sheaves appearing in BPS sheaves for conical symplectic singularities are related to symplectic duality; see the forthcoming paper of B.D. and McBreen \cite{DMcB2025}.

\end{para}

\begin{para}[Topological mirror symmetry for \texorpdfstring{$G$}{$G$}-Higgs bundles]\label{para-NAHT}
    In non-abelian Hodge theory, we are interested in the cohomology of the moduli space of semistable $G$-Higgs bundles and $G$-local systems on a Riemann surface.
    When $G$ is a group of type A with coprime choice of degree, numerous interesting phenomena have been found in the last two decades, such as the P = W conjecture \cite[]{de2012topology} proved in \cite{maulik2024p, hausel2022p} and the topological mirror symmetry conjecture \cite[Conjecture 5.1]{hausel2003mirror} proved in \cite{maulik2021endoscopic,groechenig2020mirror}.
    However, for other reductive groups (and even for $G = \mathrm{GL}_n$ with non-coprime choice of degree), the moduli space is singular in general, and many things (including the formulation of the conjecture itself) are not known, since the ordinary cohomology does not behave well for these spaces.
   
    Motivated by the recent proof of Kinjo and Koseki \cite[Corollary 5.15]{kinjo2024cohomological} of the cohomological $\chi$-independence for $\mathrm{GL}_n$-Higgs bundles with non-coprime choice of degree, we expect the cohomology of the BPS sheaf on the moduli space to provide a robust framework to generalize celebrated theorems in non-abelian Hodge theory from type A groups to general reductive groups.
    As an example, we explain our proposal of the topological mirror symmetry conjecture for $G$-Higgs bundles for general connected reductive group $G$.
    Let $C$ be a smooth projective curve and $G$ be a reductive group.
    Let $\mathcal{L} \mathcal{H}\mathrm{iggs}_G(C)^{\mathrm{ss}}$ be the loop stack of the moduli stack of semistable $G$-Higgs bundles on $C$, which is $(-1)$-shifted symplectic by the AKSZ formalism \cite[Theorem 2.5]{pantev2013shifted}.
    We let 
    \[
        \mathcal{L} \mathcal{H}\mathrm{iggs}_G(C)^{\mathrm{ss}} \to \mathrm{LHiggs}_G(C)^{\mathrm{ss}}
    \]
    denote the good moduli space. Then we expect the following:
\end{para}

\begin{conjecture*}[\Cref{conj-top-mirror}]
    Set $c_G \coloneqq \dim \mathrm{Z} (G)$.
    Then there exist isomorphisms
    \begingroup
    \setlength{\jot}{7pt}  
    \begin{align*}
        \mathrm{H}^*(\mathcal{LH}\mathrm{iggs}_G(C)^{\mathrm{ss}} , \varphi_{\mathcal{LH}\mathrm{iggs}_G(C)^{\mathrm{ss}}}) &\cong  \mathrm{H}^*(\mathcal{LH}\mathrm{iggs}_{G^{\vee}}(C)^{\mathrm{ss}} , \varphi_{\mathcal{LH}\mathrm{iggs}_{G^{\vee}}(C)^{\mathrm{ss}}}) \\
        \mathrm{H}^*_{\mathrm{BPS}^{(c_G)}}(\mathrm{LHiggs}_G(C)^{\mathrm{ss}}) &\cong  \mathrm{H}^*_{\mathrm{BPS}^{(c_{G^{\vee}})}}(\mathrm{LHiggs}_{G^{\vee}}(C)^{\mathrm{ss}}),
    \end{align*}
    \endgroup
    equivariant with respect to the isomorphism $\mathrm{Out}_{\mathrm{symp}}(G) \cong \mathrm{Out}_{\mathrm{symp}}(G^{\vee})$.
\end{conjecture*}
    As a consequence of \Cref{thm-intro-cohint--1-symplectic}, one can show that the latter isomorphism implies the former isomorphism.
    We expect that the latter isomorphism is easier to verify, since the BPS cohomology is finite-dimensional.
    When $G$ is semisimple, we expect that the above decomposition swaps the connected component decomposition and the weight decomposition.
    Further, we also propose a twisted version of the topological mirror symmetry conjecture generalizing \textcite[Conjecture 5.1]{hausel2003mirror} for type A groups.
    See \Cref{conj-top-mirror-refined} and \Cref{conj-top-mirror-twisted} for detailed discussions.

\begin{para}[Relation to other works]
    While we were preparing this manuscript, Hennecart posted two papers \cite{hennecart2024cohomological,hennecart2024cohomological_rep} on the arXiv establishing a version of the cohomological integrality theorem using an algebraic method, extending the work of Efimov \cite{efimov2012cohomological} on the structure of the cohomological Hall algebras for symmetric quivers.
    In these papers, he found a definition of the BPS sheaves as bounded mixed Hodge complexes on the affine symmetric GIT quotients of smooth varieties and showed that they satisfy the cohomological integrality theorem.
    An advantage of his method is that it works without the orthogonality hypothesis, and has striking applications such as the purity of the Borel--Moore homology of $0$-shifted symplectic stacks with proper good moduli spaces as conjectured by \textcite[Conjecture 4.4]{halpern2015theta} (and extending the results of \textcite{davison2021purity}). 
    On the other hand, our geometric approach has the advantage that we can prove an isomorphism between the BPS sheaf and the intersection complex in the smooth setting,
    which is crucial to extend the cohomological integrality theorem to $(-1)$-shifted and $0$-shifted symplectic stacks.
    Also, we work with general stacks, which might not be written as smooth affine quotients, by building on the formalism of component lattices \cite{component-lattice} established by C.B., A.I.N., T.K. and Halpern-Leistner.
    This level of generality is crucial for applications to the moduli stack of semistable principal $G$-(Higgs) bundles and the $G$-character stacks on $2$- and $3$-manifolds.

    We also note that a generalization of the cohomological integrality for stacks beyond the moduli stack of objects in an abelian category was considered by \textcite{young_cohm}, 
    where he considers the moduli stack of self-dual representations of a quiver $Q$ with an involution $Q \xrightarrow{\sim} Q^{\mathrm{op}}$.
    He formulates a version of the cohomological integrality conjecture \cite[Conjecture B]{young_cohm},
    which can be paraphrased as the statement that the stack of self-dual quiver representations satisfies the cohomological integrality in the sense of \Cref{para-cohint-smoothstack},
    and he proved it for several quivers including affine Dynkin quivers of type $\tilde{A}_1$.
    The orthogonality condition  in \Cref{def-symmetry-intro} for self-dual quivers corresponds to the numerical condition that the number of edges from each orthogonal vertex to each symplectic vertex is even.
    In particular, \Cref{thm-intro-cohint-smooth} implies \cite[Conjecture B]{young_cohm} for self-dual quivers with the above numerical condition.

\end{para}

\subsection*{Acknowledgements}
    C.B. would like to thank Dominic Joyce for related discussions,
    and the Mathematical Institute of the University of Oxford for its support.
    B.D. was supported by a Royal Society University Research Fellowship. A.I.N. is grateful to the Mathematical Institute of the University of Oxford for its support.
    T.K. started this project during his visit to Oxford in 2022. He thanks Dominic Joyce for the hospitality and dedicating time to the weekly discussions related to this project. He also thanks the environment of Oxford which made this collaboration possible.
    He is also grateful to Yukinobu Toda for stressing the importance of Halpern-Leistner's work on the stack of graded/filtered points in enumerative geometry, Yugo Takanashi for patiently teaching him the theory of algebraic groups and Hyeonjun Park for the discussion on the local model of $(-1)$-shifted symplectic stacks.
    T.K. was supported by JSPS KAKENHI Grant Number 23K19007.
    T.P. thanks Yukinobu Toda for useful discussions, and CNRS (grant PEPS JCJC) and FSI Sorbonne Université (grant Tremplins) for providing financial support to visit IPMU and RIMS between 14 October and 2 November 2024. T.P. thanks IPMU and RIMS for their hospitality during this visit. 
    We thank Pavel Safronov for discussions related to the Langlands duality for $3$-manifolds.
    We are also grateful to Lucien Hennecart for his very helpful comments and sharing his note \cite{Hennecart_note} soon after we posted the draft of this paper on the arXiv.
\subsection*{Notations and conventions}
    \
    \begin{itemize}
        \item We let $\mathsf{S}$ denote the $\infty$-category of spaces.
        \item An \emph{algebraic stack} is a $1$-Artin stack locally of finite type over $\mathbb{C}$.  
        \item A \emph{derived algebraic stack} is a derived $1$-Artin stack locally almost of finite presentation over $\mathbb{C}$.
        \item For a derived algebraic stack $\mathcal{X}$, the notation $x\in \mathcal{X}$ means that $x$ is a $\mathbb{C}$-point of $\mathcal{X}$.
        \item We adopt the notion of a good moduli space for derived algebraic stacks as given in \cite[Definition 2.1]{ahlqvist2023good}. As is shown in \cite[Theorem 2.12]{ahlqvist2023good}, a derived algebraic stack $\mathcal{X}$ admits a good moduli space if and only if its classical truncation $\mathcal{X}_{\mathrm{cl}}$ does.
        \item Representations of reductive groups are assumed to be finite dimensional and algebraic.
    \end{itemize}

\section{The component lattice}

In this section, we provide background material
on the stacks of graded and filtered points of an algebraic stack
following \textcite{_HalpernLeistner_Onthestructureofinstabilityinmodulitheory},
and the component lattice of a stack
following \textcite{component-lattice}.
The reader is referred to these two works for more details.

\subsection{Graded and filtered points}

\begin{para}
    \label{assumption-stack-basic}
    Throughout, let~$\mathcal{X}$ be a derived algebraic stack
    such that its classical truncation is quasi-separated, 
    has separated inertia, and has affine stabilizers.
\end{para}

\begin{para}[Graded and filtered points]
    \label{para-grad-filt}
    For an integer $n \geq 0$,
    the (derived) \emph{stack of $\mathbb{Z}^n$-graded points}
    and the (derived) \emph{stack of $\mathbb{Z}^n$-filtered points} of~$\mathcal{X}$
    are defined as the derived mapping stacks
    \begin{align*}
        \Grad^n (\mathcal{X})
        & = \mathrm{Map} ( \mathrm{B} \mathbb{G}_\mathrm{m}^n, \mathcal{X} ) ,
        \\
        \Filt^n (\mathcal{X})
        & = \mathrm{Map} ( \Theta^n, \mathcal{X}) ,
    \end{align*}
    where $\Theta = \mathbb{A}^1 / \mathbb{G}_\mathrm{m}$
    is the quotient stack of~$\mathbb{A}^1$ by the scaling action of~$\mathbb{G}_\mathrm{m}$.

    The stacks $\Grad^n (\mathcal{X})$ and $\Filt^n (\mathcal{X})$
    are again derived algebraic stacks
    satisfying the conditions in \cref{assumption-stack-basic},
    by \textcite[Proposition~1.1.2]{_HalpernLeistner_Onthestructureofinstabilityinmodulitheory} and the discussion after \cite[Lemma~1.2.1]{_HalpernLeistner_Onthestructureofinstabilityinmodulitheory}.

    Note that even if~$\mathcal{X}$ is a classical stack,
    the stacks $\Grad^n (\mathcal{X})$ and $\Filt^n (\mathcal{X})$
    can still have non-trivial derived structure.

    We write $\Grad (\mathcal{X}) = \Grad^1 (\mathcal{X})$
    and $\Filt (\mathcal{X}) = \Filt^1 (\mathcal{X})$,
    and call them the (derived) \emph{stack of graded points}
    and the (derived) \emph{stack of filtered points} of~$\mathcal{X}$, respectively.
\end{para}

\begin{para}[Induced morphisms]
    \label{para-grad-filt-morphisms}
    Consider the morphisms
    \begin{equation*}
        \begin{tikzcd}
            \mathrm{B} \mathbb{G}_\mathrm{m}^n
            \ar[shift left=0.5ex, r, "0"]
            &
            \Theta^n
            \ar[shift left=0.5ex, l, "\mathrm{pr}"]
            &
            \mathrm{pt} \vphantom{^n} ,
            \ar[shift left=0.5ex, l, "1"]
            \ar[shift right=0.5ex, l, "0"']
            \ar[ll, bend right, start anchor=north west, end anchor=north east, looseness=.8]
        \end{tikzcd}
    \end{equation*}
    where the map~$\mathrm{pr}$ is induced by the projection $\mathbb{A}^n \to \mathrm{pt}$,
    and~$1$ denotes the inclusion as the point $(1, \ldots, 1)$.
    These induce morphisms of stacks
    \begin{equation*}
        \begin{tikzcd}
            \Grad^n (\mathcal{X})
            \ar[rr, bend left, start anchor=north east, end anchor=north west, looseness=.8, "\mathrm{tot}"]
            \ar[r, shift right=0.5ex, "\mathrm{sf}"']
            &
            \Filt^n (\mathcal{X})
            \ar[l, shift right=0.5ex, "\mathrm{gr}"']
            \ar[r, shift left=0.5ex, "\mathrm{ev}_0"]
            \ar[r, shift right=0.5ex, "\mathrm{ev}_1"']
            &
            \mathcal{X} \rlap{ ,}
        \end{tikzcd}
    \end{equation*}
    where the notations `$\mathrm{gr}$', `$\mathrm{sf}$', and `$\mathrm{tot}$' stand for
    the \emph{associated graded point},
    the \emph{split filtration},
    and the \emph{total point}, respectively.
\end{para}

\begin{para}[Example:\ Quotient stacks]
    \label[example]{eg-grad-quotient-stack}
    The stacks of graded and filtered points of a quotient stack
    can be described explicitly, following
    \cite[Theorem~1.4.8]{_HalpernLeistner_Onthestructureofinstabilityinmodulitheory}.

    Let~$\mathcal{X} = U / G$ be a quotient stack,
    where~$U$ is a derived algebraic space over~$\mathbb{C}$, acted on by
    an affine algebraic group~$G$ over~$\mathbb{C}$.

    Let $\lambda \colon \mathbb{G}_\mathrm{m}^n \to G$
    be a morphism of algebraic groups.
    Define the \emph{Levi subgroup}
    and the \emph{parabolic subgroup} of~$G$
    associated to~$\lambda$ by
    \begin{align*}
        L_\lambda & =
        \{ g \in G \mid g = \lambda (t) \, g \, \lambda (t)^{-1} \text{ for all } t \} ,
        \\
        P_\lambda & =
        \{ g \in G \mid \lim_{t \to 0} \lambda (t) \, g \, \lambda (t)^{-1} \text{ exists} \} ,
    \end{align*}
    respectively. Define the
    \emph{derived fixed locus} and the \emph{derived attractor} associated to~$\lambda$ by
    \begin{align*}
        U^{\lambda} & =
        \Map^{\mathbb{G}_\mathrm{m}^n} (\mathrm{pt}, U) ,
        \\
        U^{\lambda, +} & =
        \Map^{\mathbb{G}_\mathrm{m}^n} (\mathbb{A}^n, U) ,
    \end{align*}
    where $\Map^{\mathbb{G}_\mathrm{m}^n} (-, -)$
    denotes the $\mathbb{G}_\mathrm{m}^n$-equivariant derived mapping space,
    and~$\mathbb{G}_\mathrm{m}^n$ acts on~$U$ via~$\lambda$,
    and on $\mathbb{A}^n$ by scaling each coordinate.
    These are derived algebraic spaces.

    The $G$-action on~$U$ induces a $P_\lambda$-action on~$U^{\lambda, +}$
    and an $L_\lambda$-action on~$U^{\lambda}$.
    Moreover, we have
    \begin{align*}
        \Grad^n (\mathcal{X}) & \simeq
        \coprod_{\lambda \colon \mathbb{G}_\mathrm{m}^n \to G} U^{\lambda} / L_\lambda ,
        \\
        \Filt^n (\mathcal{X}) & \simeq
        \coprod_{\lambda \colon \mathbb{G}_\mathrm{m}^n \to G} U^{\lambda, +} / P_\lambda ,
    \end{align*}
    where the disjoint unions are taken over all
    conjugacy classes of maps~$\lambda$.
    Equivalently, they are taken over the set $\Lambda_T^n / W$, where
    $T \subset G$ is a maximal torus of~$G$,
    $\Lambda_T = \mathrm{Hom} (\mathbb{G}_\mathrm{m}, T)$
    is the coweight lattice of~$T$,
    and $W = \mathrm{N}_G (T) / \mathrm{Z}_G (T)$
    is the Weyl group of~$G$.
\end{para}

\begin{para}[Coordinate-free notation]
    \label{para-grad-lambda}
    Following \cite[\S3.1.5]{component-lattice},
    we introduce a coordinate-free notation
    for the stack $\Grad^n (\mathcal{X})$.

    For a free $\mathbb{Z}$-module~$\Lambda$ of finite rank,
    let $T_\Lambda = \mathrm{Spec} \, \mathbb{C} [\Lambda^\vee]
    \simeq \mathbb{G}_\mathrm{m}^{\rk \Lambda}$
    be the torus with coweight lattice~$\Lambda$.
    Define the \emph{stack of $\Lambda^\vee$-graded points} of~$\mathcal{X}$ by
    \begin{equation*}
        \Grad^\Lambda (\mathcal{X}) =
        \mathrm{Map} (\mathrm{B} T_\Lambda, \mathcal{X}) .
    \end{equation*}
    This construction is contravariant in~$\Lambda$.
    In particular, we have an isomorphism
    $\Grad^\Lambda (\mathcal{X}) \simeq \Grad^{\rk \Lambda} (\mathcal{X})$
    upon choosing a basis of~$\Lambda$.
\end{para}

\begin{para}[Rational graded points]
    \label{para-rational-graded-points}
    There are also the stacks of \emph{$\mathbb{Q}^n$-graded points}
    of a derived algebraic stack~$\mathcal{X}$,
    as in \cite[\S3.1.6]{component-lattice},
    denoted by $\Grad^n_\mathbb{Q} (\mathcal{X})$.
    For example, if~$\mathcal{X}$ is the moduli stack of objects
    in an abelian category~$\mathcal{A}$,
    then points in~$\Grad^n_\mathbb{Q} (\mathcal{X})$
    correspond to \emph{$\mathbb{Q}^n$-graded objects} in~$\mathcal{A}$,
    that is, families $(x_v)_{v \in \mathbb{Q}^n}$ of objects in~$\mathcal{A}$
    such that $x_v = 0$ for all but finitely many~$v$.

    Precisely, for a $\mathbb{Q}$-vector space~$F$ of finite dimension,
    define the \emph{stack of $F^\vee$-graded points} of~$\mathcal{X}$ by
    \begin{equation*}
        \Grad^F (\mathcal{X}) =
        \operatornamewithlimits{colim}_{\Lambda \subset F} \Grad^\Lambda (\mathcal{X}) ,
    \end{equation*}
    where the colimit is taken over all free
    $\mathbb{Z}$-submodules~$\Lambda \subset F$
    of full rank.
    This construction is contravariant in~$F$.

    In particular, we write $\Grad^n_\mathbb{Q} (\mathcal{X}) = \Grad^{\mathbb{Q}^n} (\mathcal{X})$
    and $\Grad_\mathbb{Q} (\mathcal{X}) = \Grad^{\mathbb{Q}} (\mathcal{X})$.

    This construction does not produce essentially new stacks,
    since $\Grad^n_\mathbb{Q} (\mathcal{X})$ is in fact
    just $\Grad^n (\mathcal{X})$ with each connected component
    duplicated many times.
    This is because all morphisms in the above colimit diagrams
    are open and closed immersions,
    so they induce isomorphisms on each connected component.

    We have an induced morphism
    $\mathrm{tot} \colon \Grad^F (\mathcal{X}) \to \mathcal{X}$,
    defined as the colimit of the morphisms
    $\mathrm{tot} \colon \Grad^\Lambda (\mathcal{X}) \to \mathcal{X}$.
\end{para}

\begin{para}[Cone filtrations]
    \label{para-cone-filtrations}
    There are also coordinate-free and rational versions of the stacks
    $\Filt^n (\mathcal{X})$,
    which we describe now following \cite[\S5.1]{component-lattice}.

    For a commutative monoid~$\Sigma$ which is an \emph{integral cone},
    that is, a polyhedral cone in a lattice~$\mathbb{Z}^n$ for some~$n$,
    one can define the \emph{stack of $\Sigma$-filtered points} of~$\mathcal{X}$
    as a derived mapping stack
    \begin{equation*}
        \Filt^\Sigma (\mathcal{X}) =
        \mathrm{Map} (\Theta_\Sigma, \mathcal{X}) ,
    \end{equation*}
    where $\Theta_\Sigma = R_\Sigma / T_\Sigma$ is a quotient stack,
    with $R_\Sigma = \Spec \mathbb{C} [\Sigma^\vee]$
    and $T_\Sigma = \Spec \mathbb{C} [\Lambda_\Sigma^\vee]$,
    where $\Sigma^\vee = \mathrm{Hom} (\Sigma, \mathbb{N})$
    is the monoid of monoid homomorphisms $\Sigma \to \mathbb{N}$,
    and $\Lambda_\Sigma$ is the groupification of~$\Sigma$,
    and $\Lambda_\Sigma^\vee = \mathrm{Hom} (\Lambda_\Sigma, \mathbb{Z})$.

    The stack~$\Filt^\Sigma (\mathcal{X})$ is again a derived algebraic stack.
    It generalizes both the stacks $\Grad^n (\mathcal{X})$ and $\Filt^n (\mathcal{X})$,
    which are special cases when
    $\Sigma = \mathbb{Z}^n$ and~$\mathbb{N}^n$, respectively.

    This construction is contravariant in~$\Sigma$,
    and we have the induced morphisms
    \begin{equation*}
        \begin{tikzcd}
            \Grad^{\Lambda_\Sigma} (\mathcal{X})
            \ar[r, shift right=0.5ex, "\mathrm{sf}"']
            \ar[rr, bend left, start anchor=north east, end anchor=north west, looseness=.8, "\mathrm{tot}"]
            &
            \Filt^\Sigma (\mathcal{X})
            \ar[l, shift right=0.5ex, "\mathrm{gr}"']
            \ar[r, shift left=0.5ex, "\mathrm{ev}_0"]
            \ar[r, shift right=0.5ex, "\mathrm{ev}_1"']
            &
            \mathcal{X} ,
        \end{tikzcd}
    \end{equation*}
    where $\Lambda_\Sigma$ is the groupification of~$\Sigma$, as above.

    As shown in \cite[Theorem~5.1.4]{component-lattice},
    this construction for general cones~$\Sigma$
    again does not produce essentially new stacks,
    since $\Filt^\Sigma (\mathcal{X})$ is in fact
    an open and closed substack of $\Filt^n (\mathcal{X})$ for some~$n$.
    However, considering general cones is important
    from the coordinate-free point of view,
    as we will see in \Cref{subsec-component-lattice}.

    Now let~$C$ be a \emph{rational cone},
    that is, a monoid isomorphic to a polyhedral cone in a finite-dimensional
    $\mathbb{Q}$-vector space.
    Following \cite[\S5.1.5]{component-lattice},
    define the \emph{stack of $C$-filtered points} of~$\mathcal{X}$ as a colimit
    \begin{equation*}
        \Filt^C (\mathcal{X}) =
        \operatornamewithlimits{colim}_{\Sigma \subset C} \Filt^\Sigma (\mathcal{X}) ,
    \end{equation*}
    with the colimit taken over all integral cones~$\Sigma \subset C$
    such that $C = \Sigma \otimes_\mathbb{N} \mathbb{Q}_{\geq 0}$.

    This construction is contravariant in~$C$,
    and we have the induced morphisms
    \begin{equation}\label{eq-gr-filt-diagram}
        \begin{tikzcd}
            \Grad^{F_C} (\mathcal{X})
            \ar[r, shift right=0.5ex, "\mathrm{sf}"']
            \ar[rr, bend left, start anchor=north east, end anchor=north west, looseness=.8, "\mathrm{tot}"]
            &
            \Filt^C (\mathcal{X})
            \ar[l, shift right=0.5ex, "\mathrm{gr}"']
            \ar[r, shift left=0.5ex, "\mathrm{ev}_0"]
            \ar[r, shift right=0.5ex, "\mathrm{ev}_1"']
            &
            \mathcal{X} ,
        \end{tikzcd}
    \end{equation}
    where $F_C$ is the groupification of~$C$,
    seen as a $\mathbb{Q}$-vector space.

    We also denote
    $\Filt^n_\mathbb{Q} (\mathcal{X}) = \Filt^{(\mathbb{Q}_{\geq 0})^n} (\mathcal{X})$
    and $\Filt_\mathbb{Q} (\mathcal{X}) = \Filt^{\mathbb{Q}_{\geq 0}} (\mathcal{X})$.

    Again, the colimit defining $\Filt^C (\mathcal{X})$
    only involves open and closed immersions,
    and each connected component of $\Filt^C (\mathcal{X})$
    is isomorphic to one in $\Filt^\Sigma (\mathcal{X})$ for some~$\Sigma$,
    and hence one in $\Filt^n (\mathcal{X})$ for some~$n$.
    In particular, every component of $\Filt_\mathbb{Q}^n (\mathcal{X})$
    is isomorphic to one in $\Filt^n (\mathcal{X})$.
\end{para}

\subsection{The component lattice}
\label{subsec-component-lattice}

\begin{para}
    In this section, we define the \emph{component lattice}
    of a derived algebraic stack, following
    \cite{component-lattice}.
    It is the set of connected components of $\Grad (\mathcal{X})$,
    equipped with extra structure that encodes
    useful information about the enumerative geometry of the stack.
\end{para}

\begin{para}[Formal lattices]
    \label{para-formal-lattice}
    Let~$R$ be a commutative ring, which we will only consider to be
    either~$\mathbb{Z}$ or~$\mathbb{Q}$.

    Following \cite[\S2.1]{component-lattice},
    define a \emph{formal $R$-lattice} to be a functor
    \begin{equation*}
        X \colon \mathsf{Lat} (R)^\mathrm{op}
        \longrightarrow \mathsf{Set} ,
    \end{equation*}
    where $\mathsf{Lat} (R)$ is the category of
    finitely generated free $R$-modules, or \emph{$R$-lattices}.

    The \emph{underlying set} of such a formal $R$-lattice
    is the set $|X| = X (R)$.

    For example, every $R$-module is a formal $R$-lattice,
    by considering its Yoneda embedding.
    Also, we are allowed to take arbitrary limits and colimits
    of formal $R$-lattices.
\end{para}

\begin{para}[Faces and cones]
    Let $X$ be a formal $\mathbb{Q}$-lattice.

    As in \cite[\S2.1]{component-lattice},
    define the category of \emph{faces} of~$X$,
    denoted by $\mathsf{Face} (X)$, as follows:

    \begin{itemize}
        \item
            An object is a pair $(F, \alpha)$,
            where $F$ is a finite-dimensional $\mathbb{Q}$-vector space,
            and $\alpha \in X (F)$, or equivalently,
            $\alpha$ is a morphism of formal $\mathbb{Q}$-lattices
            $\alpha \colon F \to X$.

        \item
            A morphism $(F, \alpha) \to (F', \alpha')$
            is a $\mathbb{Q}$-linear map $f \colon F \to F'$,
            such that $\alpha = \alpha' \circ f$.
    \end{itemize}
    Such a face $\alpha \colon F \to X$ is called \emph{non-degenerate},
    if it does not factor through a lower-dimensional face.
    Denote by $\mathsf{Face}^\mathrm{nd} (X) \subset \mathsf{Face} (X)$
    the full subcategory of non-degenerate faces.

    Define the category of \emph{cones} of~$X$,
    denoted by $\mathsf{Cone} (X)$, as follows:

    \begin{itemize}
        \item
            An object is a triple $(F, \alpha, \sigma)$,
            where $(F, \alpha) \in \mathsf{Face} (X)$,
            and $\sigma \subset F$ is a polyhedral cone of full dimension.

        \item
            A morphism $(F, \alpha, \sigma) \to (F', \alpha', \sigma')$
            is a morphism of faces $f \colon (F, \alpha) \to (F', \alpha')$,
            such that $f (\sigma) \subset \sigma'$.
    \end{itemize}
    We often abbreviate $(F, \alpha, \sigma)$ as~$\sigma$.
    The \emph{span} of~$\sigma$ is the face $(F, \alpha)$.
    Such a cone is \emph{non-degenerate}
    if its span is a non-degenerate face.
    Denote by $\mathsf{Cone}^\mathrm{nd} (X) \subset \mathsf{Cone} (X)$
    the full subcategory of non-degenerate cones.

    Note that in \cite{component-lattice},
    cones are denoted by $(C, \sigma)$ rather than~$(F, \alpha, \sigma)$,
    where~$C$ is the underlying monoid of the cone,
    and~$\sigma$ there denotes the map from~$C$ to~$X$.
    We use the notation $(F, \alpha, \sigma)$
    since it is more convenient for the purposes of this paper,
    and the two notions of cones are equivalent.
\end{para}

\begin{para}[The component lattice]\label{para-component-lattice}
    Now, let~$\mathcal{X}$ be a derived algebraic stack over~$\mathbb{C}$,
    as in \Cref{assumption-stack-basic}.

    Following \cite[\S3.2]{component-lattice},
    define the \emph{component lattice} of~$\mathcal{X}$
    as the formal $\mathbb{Z}$-lattice~$\mathrm{CL} (\mathcal{X})$ given by
    \begin{equation*}
        \mathrm{CL} (\mathcal{X}) (\Lambda) =
        \pi_0 (\Grad^\Lambda (\mathcal{X}))
    \end{equation*}
    for all free $\mathbb{Z}$-modules~$\Lambda$ of finite rank,
    where $\pi_0$ denotes taking the set of connected components.

    The \emph{rational component lattice} of~$\mathcal{X}$
    is the formal $\mathbb{Q}$-lattice~$\mathrm{CL}_\mathbb{Q} (\mathcal{X})$ defined by
    \begin{equation*}
        \mathrm{CL}_\mathbb{Q} (\mathcal{X}) (F) =
        \pi_0 (\Grad^F (\mathcal{X}))
    \end{equation*}
    for all finite-dimensional $\mathbb{Q}$-vector spaces~$F$.
    This is also the rationalization of~$\mathrm{CL} (\mathcal{X})$
    in the sense of \cite[\S2.1.8]{component-lattice}.

    Since we have a natural isomorphism
    $\Grad^\Lambda (\mathcal{X}_\mathrm{cl})_\mathrm{cl} \simeq
    \Grad^\Lambda (\mathcal{X})_\mathrm{cl}$
    of classical truncations for any~$\Lambda$,
    and similarly for $\Grad^F (-)$,
    the component lattice of a derived algebraic stack~$\mathcal{X}$
    only depends on its classical truncation~$\mathcal{X}_\mathrm{cl}$.

    By \cite[Lemma~5.1.11]{component-lattice},
    using the notations above,
    for any integral cone $\Sigma \subset \Lambda$ of full rank,
    or any rational cone $C \subset F$ of full dimension,
    the morphisms
    $\mathrm{gr} \colon \Filt^\Sigma (\mathcal{X}) \to \Grad^\Lambda (\mathcal{X})$
    and
    $\mathrm{gr} \colon \Filt^C (\mathcal{X}) \to \Grad^F (\mathcal{X})$
    induce isomorphisms on connected components,
    so that we also have
    $\mathrm{CL} (\mathcal{X}) (\Lambda) \simeq \pi_0 (\Filt^\Sigma (\mathcal{X}))$
    and
    $\mathrm{CL}_\mathbb{Q} (\mathcal{X}) (F) \simeq \pi_0 (\Filt^C (\mathcal{X}))$.

    We introduce shorthand notations
    \begin{alignat*}{2}
        \mathsf{Face} (\mathcal{X})
        & = \mathsf{Face} (\mathrm{CL}_\mathbb{Q} (\mathcal{X})) ,
        & \qquad
        \mathsf{Face}^\mathrm{nd} (\mathcal{X})
        & = \mathsf{Face}^\mathrm{nd} (\mathrm{CL}_\mathbb{Q} (\mathcal{X})) ,
        \\
        \mathsf{Cone} (\mathcal{X})
        & = \mathsf{Cone} (\mathrm{CL}_\mathbb{Q} (\mathcal{X})) ,
        & \qquad
        \mathsf{Cone}^\mathrm{nd} (\mathcal{X})
        & = \mathsf{Cone}^\mathrm{nd} (\mathrm{CL}_\mathbb{Q} (\mathcal{X})) .
    \end{alignat*}
\end{para}

\begin{para}[The notations $\mathcal{X}_\alpha$ and $\mathcal{X}_\sigma^+$]
    \label{para-notation-x-alpha}
    For a face $(F, \alpha) \in \mathsf{Face} (\mathcal{X})$,
    and a cone $\sigma \subset F$ of full dimension,
    we define stacks
    \begin{equation*}
        \mathcal{X}_\alpha \subset \mathrm{Grad}^F (\mathcal{X}) \ ,
        \qquad
        \mathcal{X}_\sigma^+ \subset \mathrm{Filt}^C (\mathcal{X})
    \end{equation*}
    as connected components corresponding to the element
    $\alpha \in \pi_0 (\mathrm{Grad}^F (\mathcal{X}))
    \simeq \pi_0 (\mathrm{Filt}^C (\mathcal{X}))$,
    where~$C$ is the underlying monoid of~$\sigma$.
    We have the induced morphisms
    \begin{equation*}
        \begin{tikzcd}[column sep=3em]
            \mathcal{X}_\alpha
            \ar[rr, bend left, start anchor=north east, end anchor=north west, looseness=.8, "{\smash[t]{\mathrm{tot}_\alpha}}"]
            \ar[r, shift right=0.5ex, "\mathrm{sf}_\sigma"']
            &
            \mathcal{X}_\sigma^+
            \ar[l, shift right=0.5ex, "\mathrm{gr}_\sigma"']
            \ar[r, shift left=0.5ex, "\mathrm{ev}_{0, \sigma}"]
            \ar[r, shift right=0.5ex, "\mathrm{ev}_{1, \sigma}"']
            &
            \mathcal{X} ,
        \end{tikzcd}
    \end{equation*}
    as restrictions of the morphisms \Cref{eq-gr-filt-diagram}.

    For a stack~$\mathcal{Y}$ defined over~$\mathcal{X}$,
    we also write
    \begin{equation*}
        \mathcal{Y}_\alpha \subset \Grad^F (\mathcal{Y}) ,
        \qquad
        \mathcal{Y}_\sigma^+ \subset \Filt^C (\mathcal{Y}) ,
    \end{equation*}
    for the preimages of $\mathcal{X}_\alpha$ and $\mathcal{X}_\sigma^+$,
    respectively, under the induced morphisms
    $\Grad^F (\mathcal{Y}) \to \Grad^F (\mathcal{X})$
    and $\Filt^C (\mathcal{Y}) \to \Filt^C (\mathcal{X})$.
\end{para}

\begin{para}[The notation $X_\alpha$]
    Assume for now that $\mathcal{X}$ admits a (derived) good moduli space $p \colon \mathcal{X} \to X$ in the sense of \textcite[Definition 2.1]{ahlqvist2023good} and that $\mathcal{X}$ has an affine diagonal.
    Then, by \cite[Lemma 2.6.7]{IbanezNunez_stratificationsgoodmodulistacks}, for a face $(F, \alpha) \in \mathsf{Face}(\mathcal{X})$ with $\mathcal{X}_{\alpha} \to \mathcal{X}$ quasi-compact,
    the stack $\mathcal{X}_{\alpha}$ admits a good moduli space $p_{\alpha} \colon \mathcal{X}_{\alpha} \to X_{\alpha}$.
\end{para}

\begin{para}[Example:\ Quotient stacks]\label{para-quotient-component-arrangement}
    Let~$G$ be an affine algebraic group over~$\mathbb{C}$.
    By the explicit description in \Cref{eg-grad-quotient-stack},
    the integral and rational component lattices of $\mathrm{B} G$ are given by
    \begin{align*}
        \mathrm{CL} (\mathrm{B} G)
        & \simeq \Lambda_T / W ,
        \\
        \mathrm{CL}_\mathbb{Q} (\mathrm{B} G)
        & \simeq (\Lambda_T \otimes \mathbb{Q}) / W ,
    \end{align*}
    where $T \subset G$ is a maximal torus of~$G$,
    and $\Lambda_T$ is the coweight lattice of~$T$,
    and~$W$ is the Weyl group.
    The quotient is taken as a colimit in the category of formal lattices.

    Now let~$U$ be an algebraic space over~$\mathbb{C}$,
    acted on by~$G$, and consider the quotient stack $\mathcal{Y} = U / G$.
    Then, for a face $\alpha \in \mathsf{Face} (\mathrm{B} G)$
    and a cone $\sigma \in \mathsf{Cone} (\mathrm{B} G)$,
    using the notations~$\mathcal{Y}_\alpha$ and~$\mathcal{Y}_\sigma^+$
    as at the end of \Cref{para-notation-x-alpha}
    for the projection $U / G \to \mathrm{B} G$, we have
    \begin{align*}
        \mathcal{Y}_\alpha
        & \simeq U^{\alpha} / L_\alpha ,
        \\
        \mathcal{Y}_\sigma^+
        & \simeq U^{\sigma, +} / P_\sigma ,
    \end{align*}
    where $L_\alpha$ and $U^{\alpha}$ are the fixed loci,
     $P_\sigma$ and $U^{\sigma, +}$ are the attractor loci,
    generalizing the corresponding notions in \Cref{eg-grad-quotient-stack}.
    See \cite[Example~5.1.8]{component-lattice} for more details.
\end{para}

\begin{para}[Finite quotients]\label{para-graded-filtered-finite-quotient}
    For later use, we will give a description of the stack of graded points and filtered points for a finite quotient of a derived algebraic stack.
    Let~$\mathcal{X}$ be a derived algebraic stack over~$\mathbb{C}$ as in \Cref{assumption-stack-basic} acted on by a finite group $\Gamma$.
    Let $F$ be a $\mathbb{Q}$-lattice and $C$ be a rational cone.
    Then we have the isomorphisms
    \begin{equation}\label{eq-grad-filt-finite-quotient}
     \Grad^F(\mathcal{X} / \Gamma ) \cong  \Grad^F(\mathcal{X}) / \Gamma, \quad          \Filt^{C}(\mathcal{X} / \Gamma ) \cong  \Filt^{C}(\mathcal{X}) / \Gamma.
    \end{equation}
    In particular, we have an isomorphism of formal lattices
    \[
    \mathrm{CL}_{\mathbb{Q}}(\mathcal{X} / \Gamma) \cong  \mathrm{CL}_{\mathbb{Q}}(\mathcal{X}) / \Gamma .
    \]

    To see this, consider the Cartesian square
    \begin{equation*}
        \begin{tikzcd}
            \mathcal{X} \ar[r] \ar[d] \ar[dr, phantom, "\lrcorner", very near start] & \mathcal{X}/\Gamma \ar[d] \\
            \mathrm{pt} \ar[r] & \mathrm{pt}/\Gamma,
        \end{tikzcd}
    \end{equation*}
    which witnesses a $\Gamma$-action on $\cX$ with quotient $\cX/\Gamma$. Applying $\Grad^F$, we get a Cartesian square

        \begin{equation*}
            \begin{tikzcd}
                \Grad^F(\mathcal{X}) \ar[r] \ar[d] \ar[dr, phantom, "\lrcorner", very near start] & \Grad^F(\mathcal{X}/\Gamma) \ar[d] & \\
                \mathrm{pt} \ar[r] & \mathrm{pt}/\Gamma ,
            \end{tikzcd}
        \end{equation*}
    by \cite[Corollary~1.3.17]{_HalpernLeistner_Onthestructureofinstabilityinmodulitheory}, witnessing a $\Gamma$-action on $\Grad^F(\cX)$ with quotient $\Grad^F(\cX)/\Gamma\cong \Grad^F(\cX/\Gamma)$.
    By using \cite[Lemma 5.2.7]{component-lattice}, we obtain an isomorphism $\Filt^C(\cX)/\Gamma\cong \Filt^C(\cX/\Gamma)$.

    We now give a component-wise description of $\mathrm{Grad}^F(\mathcal{X} / \Gamma)$.
    Let $(F, \alpha) \in \mathsf{Face}(\mathcal{X})$ be a face and $(F, \bar{\alpha}) \in \mathsf{Face}(\mathcal{X} / \Gamma)$ be its image.
    Let $\sigma = (F, \alpha, \sigma) \in \mathsf{Cone}(\mathcal{X})$ be a cone,
    and let $\bar{\sigma} = (F, \bar{\alpha}, \sigma) \in \mathsf{Cone}(\mathcal{X} / \Gamma)$.
    Let $\Gamma_{\alpha} \subset \Gamma$ denote the subset that fixes the component ${\mathcal{X}}_{\alpha} \subset \Grad^F({\mathcal{X}})$.
    Then the equivalence \Cref{eq-grad-filt-finite-quotient} implies the following isomorphisms
    \[
     (\mathcal{X} /\Gamma)_{\bar{\alpha}}   \cong {\mathcal{X}}_{{\alpha}} / \Gamma_{{\alpha}},  \quad        (\mathcal{X} /\Gamma)_{\bar{\sigma}}^+   \cong {\mathcal{X}}_{{\sigma}}^+ / \Gamma_{{\alpha}}.
    \]
    In particular, we obtain the following commutative diagram:
    \begin{equation}\label{eq-Induction-diagram-for-finite-quotient}
\begin{tikzcd}
	{\mathcal{X} _{\alpha}} & {\mathcal{X}_{\sigma}^+} & {\mathcal{X}} \\
	{\mathcal{X}_{\alpha} / \Gamma_{\alpha}} & {\mathcal{X}_{\sigma}^+ / \Gamma_{{\alpha}}} & {\mathcal{X} / \Gamma} \\
	{(\mathcal{X} / \Gamma)_{\bar{\alpha}}} & {(\mathcal{X} /\Gamma)_{\bar{\sigma}}} & {\mathcal{X} / \Gamma.}
	\arrow[from=1-1, to=2-1]
	\arrow["{{\gr_{{\sigma}}}}"', from=1-2, to=1-1]
	\arrow["{{\ev_{1, \sigma}}}", from=1-2, to=1-3]
	\arrow[from=1-2, to=2-2]
	\arrow[from=1-3, to=2-3]
	\arrow["\cong"', from=2-1, to=3-1]
	\arrow[from=2-2, to=2-1]
	\arrow[from=2-2, to=2-3]
	\arrow["\cong"', from=2-2, to=3-2]
	\arrow[Rightarrow, no head, from=2-3, to=3-3]
	\arrow["{{\gr_{\bar{\sigma}}}}"', from=3-2, to=3-1]
	\arrow["{{\ev_{1, \bar{\sigma}}}}", from=3-2, to=3-3]
\end{tikzcd}
\end{equation}
\end{para}

\begin{para}[Base change over good moduli spaces]\label{para-graded-filtered-etale-cover}
    Let $\mathcal{X}$ be a derived algebraic stack over $\mathbb{C}$ as in \Cref{assumption-stack-basic} admitting a good moduli space $p \colon \mathcal{X} \to X$.
    Let $\eta_{\GIT} \colon Y \to X$ be a quasi-separated and \'etale morphism from a derived algebraic space and set $\mathcal{Y} = Y \times_{X} \mathcal{X}$ and let $\eta \colon \mathcal{Y} \to \mathcal{X}$ be the base change of $\eta_{\GIT}$.
    Then it follows from \cite[Corollary 1.3.17]{_HalpernLeistner_Onthestructureofinstabilityinmodulitheory} and \cite[Lemma 5.2.7]{component-lattice} that the following diagrams are Cartesian for any $\mathbb{Q}$-lattice $F$ and a rational cone $C$:
\begin{equation}\label{eq-grad-filt-etale-cartesian}
\begin{tikzcd}
	{\mathrm{Grad}^F(\mathcal{Y})} & {\mathcal{Y}} & Y \\
	{\mathrm{Grad}^F(\mathcal{X})} & {\mathcal{X}} & X,
	\arrow["{\mathrm{tot}}", from=1-1, to=1-2]
	\arrow[from=1-1, to=2-1]
	\arrow["\lrcorner"{anchor=center, pos=0.125}, draw=none, from=1-1, to=2-2]
	\arrow[from=1-2, to=1-3]
	\arrow[from=1-2, to=2-2]
	\arrow["\lrcorner"{anchor=center, pos=0.125}, draw=none, from=1-2, to=2-3]
	\arrow[from=1-3, to=2-3]
	\arrow["{\mathrm{tot}}", from=2-1, to=2-2]
	\arrow[from=2-2, to=2-3]
\end{tikzcd}
    \quad
\begin{tikzcd}
	{\mathrm{Filt}^C(\mathcal{Y})} & {\mathcal{Y}} & Y \\
	{\mathrm{Filt}^C(\mathcal{X})} & {\mathcal{X}} & {X.}
	\arrow["{{\ev_1}}", from=1-1, to=1-2]
	\arrow[from=1-1, to=2-1]
	\arrow["\lrcorner"{anchor=center, pos=0.125}, draw=none, from=1-1, to=2-2]
	\arrow[from=1-2, to=1-3]
	\arrow[from=1-2, to=2-2]
	\arrow["\lrcorner"{anchor=center, pos=0.125}, draw=none, from=1-2, to=2-3]
	\arrow[from=1-3, to=2-3]
	\arrow["{{\ev_1}}", from=2-1, to=2-2]
	\arrow[from=2-2, to=2-3]
\end{tikzcd}
\end{equation}

\end{para}

\subsection{The constancy and finiteness theorems}\label{subsec-constancy}

\begin{para}
    We summarize the main results of \cite{component-lattice}.
    Though the authors work with classical algebraic stacks in loc.~cit., 
    the main results, as explained below, can be generalized to derived algebraic stacks locally finitely presented over $\mathbb{C}$ using the same proofs.
\end{para}

\begin{para}[Special faces]\label{para-special-faces}
    Let~$\mathcal{X}$ be a derived algebraic stack over~$\mathbb{C}$,
    as in \Cref{assumption-stack-basic}.

    As in \cite[\S4.1]{component-lattice},
    a \emph{special face} of~$\mathcal{X}$ is a non-degenerate face
    $\alpha \colon F \to \mathrm{CL}_\mathbb{Q} (\mathcal{X})$
    which is maximal in preserving the stack~$\mathcal{X}_\alpha$,
    in the sense that for any morphism $\alpha \to \alpha'$
    in $\mathsf{Face}^\mathrm{nd} (\mathcal{X})$,
    if the induced morphism $\mathcal{X}_{\alpha'} \to \mathcal{X}_\alpha$
    is an isomorphism, then $\alpha \to \alpha'$ is an isomorphism.

    Let $\mathsf{Face}^\mathrm{sp} (\mathcal{X})
    \subset \mathsf{Face}^\mathrm{nd} (\mathcal{X})$
    be the full subcategory of special faces.
    By \cite[Theorem~4.1.5]{component-lattice},
    the inclusion
    $\mathsf{Face}^\mathrm{sp} (\mathcal{X}) \hookrightarrow \mathsf{Face} (\mathcal{X})$
    admits a left adjoint
    \begin{equation*}
        (-)^\mathrm{sp} \colon \mathsf{Face} (\mathcal{X})
        \longrightarrow \mathsf{Face}^\mathrm{sp} (\mathcal{X}) ,
    \end{equation*}
    called the \emph{special face closure} functor.
    It sends a face~$\alpha$ to, roughly speaking,
    the minimal special face~$\alpha^\mathrm{sp}$ containing~$\alpha$.
    The adjunction unit gives a canonical morphism
    $\alpha \to \alpha^\mathrm{sp}$, which induces an isomorphism
    $\mathcal{X}_{\alpha^\mathrm{sp}} \simeq \mathcal{X}_\alpha$.
\end{para}

\begin{para}[Central rank]\label{para-central-rank}
    Let~$\mathcal{X}$ be as above, and assume for now that it is connected.
    It is shown in \cite[\S 4.2.4]{component-lattice} that there exists an initial object
    \[
        (F_{\mathrm{ce}}, \alpha_{\mathrm{ce}}) \in \mathsf{Face}^\mathrm{sp}(\mathcal{X})
    \] 
    called the \emph{maximal central face}. 
    Further, it satisfies the property $\mathcal{X}_{\alpha_{\mathrm{ce}}} \cong \mathcal{X}$.
    The central rank of $\mathcal{X}$ is a non-negative integer defined by
    \[
     \crk \mathcal{X} \coloneqq \dim F_{\mathrm{ce}}.    
    \]
    Equivalently, the central rank of~$\mathcal{X}$
    is the maximal dimension of a torus~$T$
    such that $\mathrm{B} T$ admits an action on~$\mathcal{X}$
    that does not factor through a lower-dimensional torus.

    For a special face $(F, \alpha) \in \mathsf{Face}^{\mathrm{sp}}(\mathcal{X})$, we have an equality
    \begin{equation}\label{eq-special-face-central-rank}
        \dim F =  \crk \mathcal{X}_{\alpha}.
    \end{equation}
\end{para}

\begin{para}[The cotangent arrangement]\label{para-cotangent-arrangement}
    Let~$\mathcal{X}$ be as above,
    and let $\alpha \colon F \to \mathrm{CL}_\mathbb{Q} (\mathcal{X})$ be a face.
    Consider the complex
    $\mathbb{L}_\alpha = \mathrm{tot}_\alpha^* (\mathbb{L}_\mathcal{X})$
    on the stack~$\mathcal{X}_\alpha$,
    where $\mathbb{L}_\mathcal{X}$ is the cotangent complex of~$\mathcal{X}$.
    It admits a canonical $F^\vee$-grading.

    Define the set of \emph{cotangent weights} of~$\mathcal{X}$ in~$\alpha$
    to be the subset
    \begin{equation*}
        W^- (\mathcal{X}, \alpha) =
        \bigl\{
            \lambda \in F^\vee \bigm|
            (\mathbb{L}_\alpha)_\lambda \not\simeq 0
        \bigr\}
        \subset F^\vee ,
    \end{equation*}
    where $(-)_\lambda$ denotes taking the $\lambda$-graded piece.

    We say that~$\mathcal{X}$ has \emph{finite cotangent weights}
    if the set $W^- (\mathcal{X}, \alpha)$ is finite for all faces~$\alpha$.
    For example, this is the case if $\mathcal{X}$ has a perfect cotangent complex and all the stacks~$\mathcal{X}_\alpha$
    are quasi-compact, a condition which we discuss
    in \Cref{para-quasi-compact-graded-points} below.
    In this case, the \emph{cotangent arrangement} of~$\mathcal{X}$ at~$\alpha$
    is the hyperplane arrangement in~$F$
    consisting of the hyperplanes dual to the
    non-zero elements of~$W^- (\mathcal{X}, \alpha)$.
\end{para}

\begin{para}[The constancy theorem]\label{para-constancy-thm}
    \label{para-constancy-theorem}
    Let $\mathcal{X}$ be as above and assume further that $\mathcal{X}$ has finite cotangent weights.
    A particular case of the \emph{constancy theorem}
    \cite[Theorem~6.1.2]{component-lattice} states that, for any point
    $\lambda \in |\mathrm{CL}_\mathbb{Q} (\mathcal{X})|$,
    the isomorphism types of the components
    $\mathcal{X}_\lambda \subset \Grad_\mathbb{Q} (\mathcal{X})$
    and $\mathcal{X}_\lambda^+ \subset \Filt_\mathbb{Q} (\mathcal{X})$
    corresponding to~$\lambda$
    only depend on the special face closure $(F, \alpha)$ of the face
    $\mathbb{Q} \cdot \lambda$,
    together with the chamber $\sigma \subset F$ in the cotangent arrangement
    whose interior contains the point~$\lambda$, which always exists.
    Moreover, we have isomorphisms
    \[
        \mathcal{X}_{\alpha} \overset{\simeq}{\longrightarrow} \mathcal{X}_{\lambda} , \qquad
        \mathcal{X}_{\sigma}^+ \overset{\simeq}{\longrightarrow} \mathcal{X}_{\lambda}^+ .
    \]
    The former equivalence is nothing but the definition of the special face closure, 
    and the latter equivalence is a part of the constancy theorem.
\end{para}

\begin{para}[Example:\ Linear quotient stacks]
    \label[example]{para-ex-linear-quotient}
    Consider the stack $\mathcal{X} = V / G$,
    where~$G$ is a connected affine algebraic group over~$\mathbb{C}$,
    and~$V$ is a $G$-representation.
    We describe the special faces of~$\mathcal{X}$
    following \cite[Example~4.1.6]{component-lattice}.

    In this case, we have
    $\mathrm{CL}_\mathbb{Q} (\mathcal{X}) \simeq
    (\Lambda_T \otimes \mathbb{Q}) / W$,
    where~$T \subset G$ is a maximal torus,
    $\Lambda_T$~is the cocharacter lattice of~$T$,
    and~$W$ is the Weyl group of~$G$.
    Write also $\Lambda^T = \Lambda_T^\vee$
    for the character lattice of~$T$.

    Consider the hyperplane arrangement~$\Phi_{V / G}$
    in $\Lambda_T \otimes \mathbb{Q}$
    given by the following two types of hyperplanes:

    \begin{itemize}
        \item
            Hyperplanes dual to non-zero weights~$w \in \Lambda^T$ in~$V$.

        \item
            Hyperplanes dual to the roots~$r \in \Lambda^T$ of~$G$.
    \end{itemize}
    Then, the special faces of~$\mathcal{X}$
    are precisely the images of
    intersections of hyperplanes in~$\Phi_{V / G}$
    under the projection
    $\Lambda_T \otimes \mathbb{Q} \to (\Lambda_T \otimes \mathbb{Q}) / W$,
    and the cotangent arrangement on such a special face
    is given by hyperplanes which are restrictions of hyperplanes in~$\Phi_{V / G}$.
    The maximal central face $(F_{\mathrm{ce}}, \alpha_{\mathrm{ce}})$ from \Cref{para-central-rank} corresponds to the intersection of all hyperplanes in~$\Phi_{V / G}$.
    In particular, the corresponding torus $T_{\alpha_\mathrm{ce}} \subset G$
    is the maximal subtorus of the centre $\mathrm{Z} (G)$ which acts trivially on $V$.

    The constancy theorem in this case
    is the statement that for a rational cocharacter
    $\lambda \in \Lambda_T \otimes \mathbb{Q}$,
    the stacks $V^\lambda / L_\lambda$ and~$V^{\lambda, +} / P_\lambda$
    do not change if we move~$\lambda$
    inside a chamber in an intersection of hyperplanes in~$\Phi_{V / G}$,
    with notations as in \Cref{eg-grad-quotient-stack}.
\end{para}

\begin{para}[Stacks with quasi-compact graded points]
    \label{para-quasi-compact-graded-points}
    We now introduce a very mild finiteness condition for the stack~$\mathcal{X}$,
    following \cite[\S6.2]{component-lattice}.

    We say that~$\mathcal{X}$ has \emph{quasi-compact graded points},
    if for any face $(F, \alpha) \in \mathsf{Face}(\mathcal{X})$,
    the morphism $\mathcal{X}_\alpha \to \mathcal{X}$ is quasi-compact.
    By the discussion in \Cref{eg-grad-quotient-stack}, if the classical truncation $\mathcal{X}_{\mathrm{cl}}$ is of the form $U / G$ for a quasi-separated algebraic space $U$ which is finite type over $\mathbb{C}$ and an affine algebraic group $G$,
    the stack $\mathcal{X}$ has quasi-compact graded points.
\end{para}

\begin{para}[The finiteness theorem]\label{para-finiteness-theorem}
    The \emph{finiteness theorem} in
    \cite[Theorem~6.2.3]{component-lattice}
    states that if~$\mathcal{X}$ is quasi-compact,
    has quasi-compact graded points
    and has a perfect cotangent complex,
    then~$\mathcal{X}$ only has finitely many
    special faces.
    In particular, all the possible stacks
    $\mathcal{X}_\alpha$ 
    can only take finitely many isomorphism classes.
    Further, it is shown that the number of possible isomorphism classes for $\mathcal{X}_{\sigma}^+$ is also finite.
\end{para}

\begin{para}[Local finiteness]
    There is also a weaker finiteness result,
    \cite[Theorem~6.2.5]{component-lattice},
    which states that if~$\mathcal{X}$ has quasi-compact graded points,
    then the component lattice~$\mathrm{CL}_\mathbb{Q} (\mathcal{X})$
    is \emph{locally finite}, meaning that
    for any faces $\alpha, \beta \in \mathsf{Face} (\mathcal{X})$,
    with~$\alpha$ non-degenerate, the set of morphisms
    $\mathrm{Hom} (\beta, \alpha)$ is finite.

    In particular, if~$\mathcal{X}$ has quasi-compact graded points,
    then it has \emph{finite Weyl groups},
    meaning that every non-degenerate face~$\alpha$
    has a finite automorphism group in $\mathsf{Face} (\mathcal{X})$.
\end{para}

\subsection{Linear moduli stacks}\label{ssec-linear-moduli}

\begin{para}
    In this section, we introduce the notion of \emph{linear moduli stacks},
    following \textcite[\S7.1]{component-lattice}.
    These are algebraic stacks that behave like
    moduli stacks of objects in abelian categories,
    such as the moduli stack of representations of a quiver,
    or the stack of coherent sheaves on a projective scheme.
\end{para}

\begin{para}[Linear moduli stacks]\label{para-linear-moduli}
    A \emph{linear moduli stack} is a classical algebraic stack~$\mathcal{M}$
    as in \Cref{assumption-stack-basic},
    together with the following additional structures:

    \begin{itemize}
        \item
            A commutative monoid structure
            $\oplus \colon \mathcal{M} \times \mathcal{M} \to \mathcal{M}$,
            with unit~$0 \colon \Spec \mathbb{C} \hookrightarrow \mathcal{M}$
            an open and closed immersion.

        \item
            A $\mathrm{B} \mathbb{G}_\mathrm{m}$-action
            $\odot \colon \mathrm{B} \mathbb{G}_\mathrm{m} \times \mathcal{M} \to \mathcal{M}$
            compatible with the monoid structure.
    \end{itemize}
    Note that these structures come with extra coherence data.
    We require the following additional property:

    \begin{itemize}
        \item
            There is an isomorphism
            \begin{equation}
                \label{eq-lms-grad}
                \coprod_{\gamma \colon \mathbb{Z} \to \pi_0 (\mathcal{M})} {}
                \prod_{n \in \mathrm{supp} (\gamma)}
                \mathcal{M}_{\gamma (n)}
                \overset{\sim}{\longrightarrow} \mathrm{Grad} (\mathcal{M})_\mathrm{cl} ,
            \end{equation}
            where~$\gamma$ runs through maps of sets
            $\mathbb{Z} \to \pi_0 (\mathcal{M})$ such that
            $\mathrm{supp} (\gamma) = \mathbb{Z} \setminus \gamma^{-1} (0)$
            is finite,
            and the morphism is defined by the composition
            \begin{equation*}
                \mathrm{B} \mathbb{G}_\mathrm{m} \times
                \prod_{n \in \mathrm{supp} (\gamma)} \mathcal{M}_{\gamma (n)}
                \overset{(-)^n}{\longrightarrow}
                \prod_{n \in \mathrm{supp} (\gamma)} {}
                (\mathrm{B} \mathbb{G}_\mathrm{m} \times \mathcal{M}_{\gamma (n)})
                \overset{\odot}{\longrightarrow}
                \prod_{n \in \mathrm{supp} (\gamma)} \mathcal{M}_{\gamma (n)}
                \overset{\oplus}{\longrightarrow}
                \mathcal{M}
            \end{equation*}
            on the component corresponding to~$\gamma$,
            where the first morphism is given by the
            $n$-th power map $(-)^n \colon \mathrm{B} \mathbb{G}_\mathrm{m} \to \mathrm{B} \mathbb{G}_\mathrm{m}$
            on the factor corresponding to~$\mathcal{M}_{\gamma (n)}$.
    \end{itemize}
    One could think of~\cref{eq-lms-grad}
    roughly as an isomorphism $\mathrm{Grad} (\mathcal{M})_{\mathrm{cl}} \simeq \mathcal{M}^\mathbb{Z}$,
    where we only consider components of~$\mathcal{M}^\mathbb{Z}$
    involving finitely many non-zero classes in~$\pi_0 (\mathcal{M})$.

    In this case, the set $\pi_0 (\mathcal{M})$
    has a commutative monoid structure~$+$
    induced by~$\oplus$.

    For a finitely generated free $\mathbb{Z}$-module~$\Lambda$,
    and a finite-dimensional $\mathbb{Q}$-vector space~$F$,
    we also have similar isomorphisms
    \begin{align}
        \label{eq-lms-grad-n}
        \coprod_{\gamma \colon \Lambda^\vee \to \pi_0 (\mathcal{M})} {}
        \prod_{\lambda \in \mathrm{supp} (\gamma)}
        \mathcal{M}_{\gamma (\lambda)}
        & \overset{\sim}{\longrightarrow} \Grad^\Lambda (\mathcal{M})_\mathrm{cl} ,
        \\
        \label{eq-lms-grad-q-n}
        \coprod_{\gamma \colon F^\vee \to \pi_0 (\mathcal{M})} {}
        \prod_{\lambda \in \mathrm{supp} (\gamma)}
        \mathcal{M}_{\gamma (\lambda)}
        & \overset{\sim}{\longrightarrow} \Grad^F (\mathcal{M})_\mathrm{cl} ,
    \end{align}
    where~$\gamma$ is assumed of finite support in both cases.
    In particular, $\mathcal{M}$ has quasi-compact graded points as long as the direct sum map $\oplus$ is quasi-compact.

    For a choice of components $\gamma_1, \ldots, \gamma_n \in \pi_0(\mathcal{M})$, we define a face $\alpha(\gamma_1, \ldots, \gamma_n) \colon \mathbb{Q}^n \to \mathrm{CL}_\mathbb{Q} (\mathcal{M})$ to be the one which, under the equivalence \Cref{eq-lms-grad-q-n}, corresponds to the map $(\mathbb{Q}^n)^\vee \to \pi_0 (\mathcal{M})$ given by
    \[
     v \longmapsto 
     \begin{cases}
         \gamma_i, & \text{if $v = e_i$}, \\
         0, & \text{otherwise},
     \end{cases}    
    \]
    where $e_i \in (\mathbb{Q}^{n})^\vee$ denotes the $i$-th standard basis vector.
\end{para}

\begin{para}[Examples]\label{para-examples-linear}
    Following \cite[Examples~7.1.3]{component-lattice},
    we list here some examples of linear moduli stacks.

    \begin{enumerate}
        \item
            Let~$\mathcal{A}$ be a $\mathbb{C}$-linear abelian category
            which is locally noetherian and cocomplete,
            in the sense of
            \textcite[\S7]{_Alper_Existenceofmodulispacesforalgebraicstacks}.
            Consider the moduli stack~$\mathcal{M}_\mathcal{A}$
            of finitely presented objects in~$\mathcal{A}$
            in the sense of
            \textcite{artin-zhang-2001-hilbert}
            and \cite[\S7]{_Alper_Existenceofmodulispacesforalgebraicstacks}.
            Then, if~$\mathcal{M}_\mathcal{A}$ is an algebraic stack
            locally of finite presentation over~$\mathbb{C}$,
            it is a linear moduli stack over~$\mathbb{C}$.

        \item \label{item-linear-dg}
            Let~$\mathcal{C}$ be a $\mathbb{C}$-linear dg-category
            of finite type, in the sense of
            \textcite[Definition~2.4]{toen2007moduli}.
            Consider the moduli stack~$\mathcal{M}_\mathcal{C}$
            of right proper objects in~$\mathcal{C}$,
            as in \cite[\S3]{toen2007moduli}.
            If we are given an open substack
            $\mathcal{M} \subset \mathcal{M}_\mathcal{C}$,
            closed under direct sums and direct summands,
            such that it contains the zero object as an open and closed substack,
            and its classical truncation~$\mathcal{M}_\mathrm{cl}$
            is a $1$-stack that is quasi-separated and has affine stabilizers,
            then~$\mathcal{M}_\mathrm{cl}$ is a linear moduli stack over~$\mathbb{C}$.

        \item
            Let~$A$ be a finitely generated associative $\mathbb{C}$-algebra.
            Consider the moduli stack~$\mathcal{M}_A$
            of representations of~$A$,
            defined by the moduli functor
            \begin{equation*}
                \mathcal{M}_A (R) = \bigl(
                    \text{left $(A \otimes_{\mathbb{C}} R)$-modules,
                    flat and finitely presented over~$R$}
                \bigr)
            \end{equation*}
            for commutative $\mathbb{C}$-algebras~$R$.
            Then~$\mathcal{M}_A$ is a linear moduli stack over~$\mathbb{C}$.
            In particular, the moduli stack of representations
            of a (finite) quiver, possibly with relations, over~$\mathbb{C}$,
            is a linear moduli stack over~$\mathbb{C}$.
    \end{enumerate}
\end{para}

\begin{para}[Special faces]\label{para-special-linear}
    For a linear moduli stack~$\mathcal{M}$,
    the isomorphism~\cref{eq-lms-grad-q-n}
    implies that the faces $F \to \mathrm{CL}_\mathbb{Q} (\mathcal{M})$
    correspond to maps $F^\vee \to \pi_0 (\mathcal{M})$ of finite support.
    Such a face is non-degenerate if and only if
    the support of the latter map spans~$F^\vee$.

    It follows from~\cref{eq-lms-grad-n}
    that if $\alpha \colon F \to \mathrm{CL}_\mathbb{Q} (\mathcal{M})$
    is a special face, then the corresponding map
    $F^\vee \to \pi_0 (\mathcal{M})$
    is supported precisely on a basis of~$F^\vee$.
    However, the converse need not be true,
    that is, faces of this form need not be special faces.  For example,
    if for some $\gamma_1, \gamma_2 \in \pi_0 (\mathcal{M})$,
    the morphism $\oplus \colon \mathcal{M}_{\gamma_1} \times \mathcal{M}_{\gamma_2}
    \to \mathcal{M}_{\gamma_1 + \gamma_2}$
    is an isomorphism,
    then the one-dimensional face
    corresponding to the map $\mathbb{Q} \to \pi_0 (\mathcal{M})$
    sending $1$ to $\gamma_1 + \gamma_2$
    and everything else to~$0$
    will not be special.

    In particular, as in \cite[\S7.1.4]{component-lattice},
    all linear moduli stacks have finite cotangent weights.
\end{para}

\begin{para}[Stacks of filtrations]\label{para-stack-of-filtrations}
    For a linear moduli stack~$\mathcal{M}$,
    the constancy theorem described in \Cref{para-constancy-theorem}
    implies that for classes $\gamma_1, \dotsc, \gamma_n \in \pi_0 (\mathcal{M})$,
    there is a canonically defined stack
    $\mathcal{M}_{\gamma_1, \dotsc, \gamma_n}^+$
    of filtrations whose graded pieces are of classes
    $\gamma_1, \dotsc, \gamma_n$, in that order.

    This is defined as a connected component in $\Filt (\mathcal{M})_\mathrm{cl}$
    corresponding to a map $\mathbb{Z} \to \pi_0 (\mathcal{M})$
    whose non-zero values are $\gamma_n, \dotsc, \gamma_1$ in that order,
    where the zero terms are omitted.
    The constancy theorem implies that this stack
    is canonically defined, depending on the order of the classes,
    but not on the precise gradings.

\end{para}

\begin{para}[Derived linear moduli stacks]\label{para-derived-linear-moduli-stack}
    One can define a derived version of a linear moduli stack in an analogous manner.
    Namely, a \emph{derived linear moduli stack} is a derived algebraic stack $\mathcal{M}$ equipped with an $E_{\infty}$-monoid structure and a compatible $\mathrm{B} \mathbb{G}_m$-action such that the morphism \Cref{eq-lms-grad}, without taking the classical truncation, is an isomorphism of derived algebraic stacks.
    The linear moduli stacks discussed in \Cref{item-linear-dg} of \Cref{para-examples-linear} naturally upgrade to derived linear moduli stacks.
\end{para}

\begin{para}[Euler pairing]\label{para-euler-pairing}
    Let $\mathcal{M}$ be a derived linear moduli stack which is locally finitely presented over $\mathbb{C}$.
    For each $\gamma_1, \gamma_2 \in \pi_0(\mathcal{M})$, we let $\gr_{\gamma_1, \gamma_2} \colon \mathcal{M}_{\gamma_1, \gamma_2}^{+} \to \mathcal{M}_{\gamma_1} \times \mathcal{M}_{\gamma_2}$ be the restriction of the map $\mathrm{gr} \colon \Filt(\mathcal{M}) \to \Grad(\mathcal{M})$.
    Consider a map
    \[
    \chi_{\mathcal{M}}(-, -) \colon \pi_0(\mathcal{M}) \times  \pi_0(\mathcal{M}) \to \mathbb{Z}, \quad (\gamma_1, \gamma_2) \mapsto - {\vdim \mathrm{gr}_{\gamma_2, \gamma_1}}.
    \]
    We say that $\mathcal{M}$ admits an \emph{Euler pairing} if the map $\chi_{\mathcal{M}}(-, -)$ satisfies the following conditions:
    \begin{itemize}
        \item $\chi_{\mathcal{M}}(-, -)$ is bilinear.
        \item For each $\gamma \in \pi_0(\mathcal{M})$, we have an equality
        \[
         \vdim \mathcal{M}_{\gamma} = - \chi_{\mathcal{M}}(\gamma, \gamma).    
        \]
    \end{itemize}

    If $\mathcal{M}$ is a derived linear moduli stack defined as an open substack of the moduli stack of objects in a finite type dg-category, it is clear that $\mathcal{M}$ admits an Euler pairing.
\end{para}

\section{Shifted symplectic structures}

\stepcounter{subsection}

\begin{para}
    Here we briefly recall the notion of shifted symplectic structures introduced by \textcite{pantev2013shifted}.
\end{para}

\begin{para}
    Let $\mathcal{X}$ be a derived Artin stack locally finitely presented over $\mathbb{C}$. 
    We define the space of $n$-shifted $p$-forms on $\mathcal{X}$ by
    \[
    \mathcal{A}^p(\mathcal{X}, n) \coloneqq |\Gamma(\mathcal{X}, (\wedge^p \mathbb{L}_{\mathcal{X}})[n])| \in \mathsf{S}    
    \]
    where $|-|$ denotes the geometric realization. As explained in \cite[\S 1.2]{pantev2013shifted}, one can also define the space of closed $n$-shifted $p$-forms $ \mathcal{A}^{p, \mathrm{cl}}(\mathcal{X}, n)$ on $\mathcal{X}$ together with a forgetful map
    \[
        \mathcal{A}^{p, \mathrm{cl}}(\mathcal{X}, n) \to \mathcal{A}^p(\mathcal{X}, n),
    \]
    which we do not repeat the definition of here. A morphism between derived Artin stacks $f \colon \mathcal{X} \to \mathcal{Y}$ naturally induces pullback maps
    \[
        f^{\star} \colon \mathcal{A}^p(\mathcal{Y}, n) \to \mathcal{A}^p(\mathcal{X}, n), \quad f^{\star} \colon \mathcal{A}^{p, \mathrm{cl}}(\mathcal{Y}, n) \to \mathcal{A}^{p, \mathrm{cl}}(\mathcal{X}, n).
    \]

    An $n$-shifted symplectic structure on $\mathcal{X}$ is an $n$-shifted closed $2$-form $\omega_{\mathcal{X}} \in \mathcal{A}^{2, \mathrm{cl}}(\mathcal{X}, n)$ whose underlying $n$-shifted $2$-form induces an equivalence
    \[
     \Phi \colon \mathbb{T}_{\mathcal{X}} \simeq \mathbb{L}_{\mathcal{X}}[n].    
    \]
    One can easily check the following equivalence of maps
    \[
        \Phi \simeq  (-1)^{n(n-1)/2 + 1} \cdot \Phi^{\vee}[n].
    \]
    See \cite[Lemma 0.3, Remark 0.4]{symplecticsign} for details of the proof and the sign convention we are using. 
\end{para}

\begin{para}[Examples of shifted symplectic stacks]\label{para-example-shifted-symplectic}
    We list some examples of shifted symplectic stacks which are related with this paper.
    We start with the $(-1)$-shifted case:

    \begin{enumerate}
        \item \label{item-example-symplectic-critical}
            Let $\mathcal{Y}$ be a derived Artin stack locally finitely presented over $\mathbb{C}$ and $f \colon \mathcal{Y} \to \mathbb{A}^1$ be a function.
              We let 
              \[
                \mathrm{Crit}(f) \coloneqq \mathcal{Y} \times_{0, \mathrm{T}^* \mathcal{Y}, df}   \mathcal{Y} 
              \]
              be the derived critical locus of $f$.
              Then it is shown by Calaque \cite[Theorem 2.22]{calaque2019shifted} that $\mathrm{Crit}(f)$ is naturally equipped with a $(-1)$-shifted symplectic structure which we call the \emph{standard $(-1)$-shifted symplectic structure}.
              When the function $f$ is zero, we have
              \[
                \mathrm{Crit}(0) \simeq \mathrm{T}^*[-1]{\mathcal{Y}}\coloneqq \mathrm{Tot}_{\mathcal{Y}}(\mathbb{L}_{\mathcal{Y}}[-1]),
              \]
              hence $\mathrm{T}^*[-1]{\mathcal{Y}}$ is equipped with a standard $(-1)$-shifted symplectic structure.

              For later purposes, we explicitly write down the $(-1)$-shifted symplectic structure on $\mathrm{T}^*[-1]{\mathcal{Y}}$ following \cite[\S 2.1]{calaque2019shifted}.
              Let $\pi \colon \mathrm{T}^*[-1]{\mathcal{Y}} \to \mathcal{Y}$ be the projection and $s_{\mathrm{taut}} \in \Gamma(\mathrm{T}^*[-1]{\mathcal{Y}}, \pi^* \mathbb{L}_{\mathcal{Y}}[-1])$ be the tautological section.
              We let $\lambda_{\mathrm{T}^*[-1]{\mathcal{Y}}} \in \mathcal{A}^1(\mathrm{T}^*[-1]{\mathcal{Y}}, -1)$ be the image of $s_{\mathrm{taut}}$.
              Then the $(-1)$-shifted symplectic structure on $\mathrm{T}^*[-1]{\mathcal{Y}}$ is given by $d_{\mathrm{dR}} \lambda_{\mathrm{T}^*[-1]{\mathcal{Y}}}$.

        \item \label{item-example-symplectic-3CY}
            Let $\mathcal{C}$ be a finite type left $3$-Calabi--Yau dg-category over $\mathbb{C}$; see \cite[Definition 3.5]{brav2019relative} for the definition.
              Then it is shown by \textcite[Theorem 5.5]{brav2021relative} that the moduli stack $\mathcal{M}_{\mathcal{C}}$ of objects in $\mathcal{C}$ is $(-1)$-shifted symplectic.
              In particular, for a smooth Calabi--Yau threefold $X$, the moduli stack $\mathcal{M}_X$ of compactly supported coherent sheaves on $X$ is $(-1)$-shifted symplectic.

        \item \label{item-example-symplectic-3mfd}
            Let $M$ be a compact oriented real $3$-manifold and $G$ be a reductive group over~$\mathbb{C}$.
             Let $M_{\mathrm{B}}$ be the constant derived stack of the homotopy type of $M$.
             The derived character stack $\mathcal{L}\mathrm{oc}_G(M)$, which is the derived moduli stack of $G$-local systems on $M$, is defined as the mapping stack
             \[
                \mathcal{L}\mathrm{oc}_G(M) \coloneqq \mathrm{Map}(M_{\mathrm{B}}, \mathrm{B} G).
             \]
            It is shown by \textcite[Corollary 2.6]{pantev2013shifted} that $\mathcal{L}\mathrm{oc}_G(M)$ is equipped with a $(-1)$-shifted symplectic structure.
    \end{enumerate}

    Now we list examples of $0$-shifted symplectic stacks studied in the literature:

    \begin{enumerate}[resume]
        \item \label{item-example-symplectic-unshiftedcotangent}
        Let $\mathcal{Y}$ be a derived Artin stack locally finitely presented over $\mathbb{C}$.
        Then its cotangent stack $\mathrm{T}^*{\mathcal{Y}} \coloneqq \mathrm{Tot}_{\mathcal{Y}}(\mathbb{L}_\mathcal{Y})$ is $0$-shifted symplectic.
                
        \item \label{item-example-symplectic-G-Higgs}
        Let $C$ be a smooth projective curve and $G$ be a reductive group. We let $\mathcal{B}\mathrm{un}_G(C) \coloneqq \mathrm{Map}(C, \mathrm{B} G)$ be the moduli stack of principal $G$-bundles on $C$.
        We define the moduli stack of $G$-Higgs bundles on $C$ to be the cotangent stack
        \[
        \mathcal{H}\mathrm{iggs}_G(C) \coloneqq   \mathrm{T}^*   \mathcal{B}\mathrm{un}_G(C).
        \]
        A point in $\mathcal{H}\mathrm{iggs}_G(C)$ corresponds to a $G$-Higgs bundle on $C$, i.e., a pair $(E, \phi)$ of a principal $G$-bundle $E$ on $C$ and a section $\phi \in \Gamma(C, \mathrm{Ad}(E) \otimes K_C )$ where $\mathrm{Ad}(E) $ denotes the adjoint bundle. 
        Since $\mathcal{H}\mathrm{iggs}_G(C) $ is defined as the cotangent stack, it is equipped with a $0$-shifted symplectic structure.

        \item \label{item-example-symplectic-2CY}
            Let $\mathcal{C}$ be a finite type left $2$-Calabi--Yau dg-category over $\mathbb{C}$.
         Then it is shown by \textcite[Theorem 5.5]{brav2021relative} that the moduli stack $\mathcal{M}_{\mathcal{C}}$ of objects in $\mathcal{C}$ is $0$-shifted symplectic.
        In particular, for a K3 surface $S$, the moduli stack $\mathcal{M}_S$ of coherent sheaves on $S$ is $0$-shifted symplectic.

        \item \label{item-example-symplectic-2mfd}
            Let $\Sigma$ be a compact oriented real $2$-manifold and $G$ be a reductive group over~$\mathbb{C}$. Then the derived character stack of $G$-local systems on $\Sigma$ 
        \[
            \mathcal{L}\mathrm{oc}_G(\Sigma) \coloneqq \mathrm{Map}(\Sigma_{\mathrm{B}}, \mathrm{B} G)
         \] 
         is equipped with a $0$-shifted symplectic structure by \cite[Corollary 2.6]{pantev2013shifted}.
    \end{enumerate}

\end{para}

\begin{para}[Darboux theorem]
    As we have seen in \Cref{item-example-symplectic-critical} of \Cref{para-example-shifted-symplectic}, the derived critical locus of a function on a derived algebraic stack is $(-1)$-shifted symplectic.
    The Darboux theorem, proved by \textcite{brav2019darboux} for derived schemes and \textcite{ben2015darboux} for derived algebraic stacks, roughly says that the converse is true locally.
    Here we will state a combined version of the local structure theorem for stacks admitting a good moduli space due to \textcite[Theorem 4.12]{_Alper_ALunaetaleslicetheoremforalgebraicstacks} and the Darboux theorem. This combined version follows from the result proved by the Kinjo, Park and Safronov \cite[Proposition 3.13]{Kinjo_Park_Safronov_CoHA} which builds on Park's result \cite[Theorem B]{park2024shifted}, together with Luna's fundamental lemma \cite[Proposition 4.13]{_Alper_ALunaetaleslicetheoremforalgebraicstacks}.
\end{para}

\begin{theorem}\label{thm-Darboux}
 Let $\mathcal{X}$ be a $(-1)$-shifted symplectic stack having affine diagonal admitting a good moduli space $\mathcal{X} \to X$.
 Then for any closed point $x \in \mathcal{X}$ with stabilizer group $G_x$, there exist a smooth affine scheme $U$ acted on by $G_x$ with a fixed point $u$, a regular function $f / G_x$ on $U / G_x$, and a Cartesian diagram
\[\begin{tikzcd}
	{\mathrm{Crit}(f/G_x)} & {\mathcal{X}} \\
	{\mathrm{Crit}(f/G_x)_{\GIT}} & X.
	\arrow["\eta", from=1-1, to=1-2]
	\arrow[from=1-1, to=2-1]
	\arrow["\lrcorner"{anchor=center, pos=0.125}, draw=none, from=1-1, to=2-2]
	\arrow[from=1-2, to=2-2]
	\arrow[from=2-1, to=2-2]
\end{tikzcd}\]
Here the vertical maps are the (derived) good moduli space morphisms in the sense of \textcite[Definition 2.1]{ahlqvist2023good} and the horizontal maps are \'etale.
Further, $\eta$ is a symplectomorphism with $\eta(u) = x$ inducing the identity map of the stabilizer group at $u$.
\end{theorem}

\begin{para}[Localization of shifted symplectic structures]\label{para-localize-shifted-symplectic-structure}
    Let $(\mathcal{X}, \omega_{\mathcal{X}})$ be an $n$-shifted symplectic stack and $(F, \alpha) \in \mathsf{Face}(\mathcal{X})$ be a face.
    Then using that the duality map $\mathrm{tot}_{\alpha}^* \mathbb{T}_{\mathcal{X}} \simeq \mathrm{tot}_{\alpha}^* \mathbb{L}_{\mathcal{X}}[n]$ preserves the $F^{\vee}$-grading, one sees that the pair
    \[
     (\mathcal{X}_{\alpha}, \mathrm{tot}_{\alpha}^{\star} \omega_{\mathcal{X}})    
    \]
    is an $n$-shifted symplectic stack.

    Now let $\mathcal{Y}$ be a derived Artin stack locally finitely presented over $\mathbb{C}$ and $f$ be a function on $\mathcal{Y}$.
    For a face $(F, \alpha) \in \mathsf{Face}(\mathcal{Y})$, set $f_{\alpha} \coloneqq f \circ \mathrm{tot}_{\alpha}$. Then we have an equivalence of $(-1)$-shifted symplectic stacks
    \begin{equation}\label{eq-localize-standard-symplectic}
     (\mathrm{Crit}(f_{\alpha}), \omega_{\mathrm{Crit}(f_{\alpha})}) \simeq  (\mathrm{Crit}(f)_{\alpha}, \mathrm{tot}_{\alpha}^\star \omega_{\mathrm{Crit}(f)})
    \end{equation}
    where $\omega_{\mathrm{Crit}(f_{\alpha})}$ and $\omega_{\mathrm{Crit}(f)}$ are the standard $(-1)$-shifted symplectic structures recalled in \Cref{item-example-symplectic-critical} of \Cref{para-example-shifted-symplectic}.
    This is proved in \cite[Lemma 5.15]{Kinjo_Park_Safronov_CoHA}.
\end{para}

\begin{para}[Shifted symplectic linear moduli stacks]
    We will introduce the notion of shifted symplectic structure on a derived linear moduli stack introduced in \Cref{para-derived-linear-moduli-stack}.
    For an integer $n$, an \emph{$n$-shifted symplectic linear moduli stack} is given by the following data:
    \begin{itemize}
        \item A derived linear moduli stack $\mathcal{M}$.
        \item An $n$-shifted symplectic structure $\omega$ on $\mathcal{M}$ such that there exists a homotopy $\omega \boxplus \omega \sim \oplus^{\star} \omega$.
    \end{itemize}
    We note that this is not a ``correct'' definition from a higher-categorical perspective, since we do not care about the coherence of the chosen homotopy.
    Nevertheless, for our applications which are $1$-categorical, this definition will suffice.

    Let $\mathcal{M}$ be a derived linear moduli stack defined as an open substack of the moduli stack of objects in a finite type left $n$-Calabi--Yau dg-category.
    Then it follows as in \cite[Proposition 8.29]{Kinjo_Park_Safronov_CoHA} that $\mathcal{M}$ is a $(2-n)$-shifted symplectic linear moduli stack.
    
\end{para}

\begin{para}[Moment map for \texorpdfstring{$\mathrm{B} T$}{\mathrm{B}T}-actions]\label{para-moment-map}
    
    Let $T$ be a torus and $(\mathcal{Y}, \omega_{\mathcal{Y}})$ be a $0$-shifted symplectic derived Artin stack.
    In this section, we will prove that any $\mathrm{B} T$ action on $\mathcal{X}$ admits a moment map.
    We thank Hyeonjun Park for his help on the content of this paragraph.

    For a $(-1)$-shifted function $H \in \Gamma (\mathcal{Y}, \mathcal{O}_{\mathcal{Y}}[-1])$, we define the associated Hamiltonian vector field 
    $X_H \in \Gamma(\mathcal{Y}, \mathbb{T}_{\mathcal{Y}}[-1])$ to be the one satisfying the following equation of $(-1)$-shifted $1$-forms:
    \[
        \iota_{X_{H}} \omega_{\mathcal{Y}} = d_{\mathrm{dR}} H
    \]
    where $\iota_{X_{H}}$ denotes the contraction map with respect to $X_{H}$ and $d_{\mathrm{dR}}$ denotes the de Rham differential.
    Assume that we are given a $\mathrm{B}T$-action on $\mathcal{Y}$ and a vector $\xi \in \mathrm{Lie}(T)$.
    We define the $(-1)$-shifted vector field $\rho(\xi) \in \Gamma(\mathcal{Y}, \mathbb{T}_{\mathcal{Y}[-1]})$ to be the image of $\xi$ under the following map:
    \[
        \mathrm{Lie}(T) \cong \Gamma(\mathrm{B} T, \mathbb{T}_{\mathrm{B} T}[-1]) \hookrightarrow
        \Gamma(\mathrm{B} T \times \mathcal{Y}, \mathbb{T}_{\mathrm{B} T \times \mathcal{Y}}[-1])
        \to  \Gamma(\mathcal{Y}, \mathbb{T}_{\mathcal{Y}}[-1])
    \]
    where the inclusion is given by $v \mapsto (v, 0)$ and the final map is induced by the action map.
    A moment map $\mu$ is a map
    \[
     \mu \colon \mathcal{Y} \to \mathrm{Lie}(T)^{\vee}[-1]
    \] 
    such that for any $\xi \in \mathrm{Lie}(T)$, the following identity holds:
    \[
     \rho(\xi) = X_{\mu(\xi, -)}.    
    \]

\end{para}

\begin{proposition}\label{prop-moment}
    We adopt the notation from the last paragraph. 
    Then there exists a moment map for any $\mathrm{B} T$-action on $\mathcal{Y}$.
\end{proposition}

\begin{proof}
    It is enough to show that, for any $\xi \in \mathrm{Lie}(T)$, 
    the $(-1)$-shifted $1$-form $\iota_{\rho(\xi)} \omega_{\mathcal{Y}}$ is the de Rham differential of some $(-1)$-shifted function $H_{\xi}$.
    Indeed, for a chosen basis $\xi_1, \ldots, \xi_n$ of $\mathrm{Lie}(T)$, the map $\sum_i H_{\xi_i} \cdot \xi_i^{\vee}$ satisfies the condition for the moment map.
    To prove the existence of such $H_{\xi}$,
    using \cite[Proposition 6.1.1]{park2024shifted}, it is enough to show that $\iota_{\rho(\xi)} \omega_{\mathcal{Y}}$ admits a closed structure.
    We will prove this by constructing a contraction map at the level of graded mixed complexes.

    Firstly, note that there is an isomorphism of mixed complexes
    \begin{equation}\label{eq-BT-Kunneth}
     \mathrm{DR}(\mathrm{B} T \times \mathcal{Y}) \cong \mathrm{DR}(\mathrm{B} T) \otimes \mathrm{DR}(\mathcal{Y}),
    \end{equation}
    where $\mathrm{DR}(-)$ denotes the de Rham complex as a graded mixed complex (see \cite[\S 1.2]{pantev2013shifted}).
    To see this, since there is a natural map from the right to left, it is enough to prove that this map induces an isomorphism on the underlying graded complexes without a mixed structure.
    Namely, using \cite[Remark 2.4.4]{calaque2017shifted}, we are reduced to proving the following isomorphism for each $p, q \in \mathbb{Z}_{\geq 0}$
    \[
     \Gamma(\mathrm{B}T \times \mathcal{Y}, \wedge^p \mathbb{L}_{\mathrm{B}T} \boxtimes \wedge^q \mathbb{L}_{\mathcal{Y}}) \cong \Gamma(\mathrm{B}T , \wedge^p \mathbb{L}_{\mathrm{B}T} ) \otimes \Gamma(\mathcal{Y}, \wedge^q \mathbb{L}_{\mathcal{Y}}).
    \]
    This formally follows from the projection formula, the base change formula and the finite dimensionality of $\Gamma(\mathrm{B}T , \wedge^p \mathbb{L}_{\mathrm{B}T} )$.
    See \cite[Lemma A.7]{bae2022counting} for an analogous discussion.

    Next, as is shown by \cite[Classifying stacks]{pantev2013shifted}, there exists a natural isomorphism of mixed graded complexes
    \[
        \mathrm{DR}(\mathrm{B} T) \cong \bigoplus_i \mathrm{Sym}^i(\mathrm{Lie}(T)^{\vee})(-i)
    \]
    where the right-hand side is equipped with the trivial mixed structure
    and $(-i)$ denotes the shift of the mixed grading.
    In particular, there exists a projection map of graded mixed complexes
    \begin{equation}\label{eq-BT-proj}
        \mathrm{DR}(\mathrm{B} T) \to \mathrm{Lie}(T)^{\vee}(-1).
    \end{equation}

    Now consider the following composition 
    \[
    \iota^{\mathrm{mixed}} \colon \mathrm{DR}(\mathcal{Y}) \to \mathrm{DR}(\mathrm{B}T \times \mathcal{Y}) \xrightarrow[\cong]{\eqref{eq-BT-Kunneth}} \mathrm{DR}(\mathrm{B} T) \otimes \mathrm{DR}(\mathcal{Y}) \xrightarrow{\eqref{eq-BT-proj}}  \mathrm{Lie}(T)^{\vee} \otimes \mathrm{DR}(\mathcal{Y})(-1)
    \]
    where the first map is induced by the action map.
    For an element $\xi \in \mathrm{Lie}(T)$, we define a map
    \[
     \iota_{\rho(\xi)}^{\mathrm{mixed}} \colon   \mathrm{DR}(\mathcal{Y}) \to \mathrm{DR}(\mathcal{Y})(-1)
    \]
    to be the composite of $\iota^{\mathrm{mixed}}$ and evaluation at $\xi$.
    By passing to the negative cyclic complexes, we obtain a map 
    \[
    \iota_{\rho(\xi)}^{\mathrm{cl}} \colon \mathcal{A}^{p, \mathrm{cl}}(\mathcal{X}, n) \to \mathcal{A}^{p-1, \mathrm{cl}}(\mathcal{X}, n-1)  
    \]
    which is compatible with $\iota_{\rho(\xi)}$ under the forgetful map to $\mathcal{A}^{-}(\mathcal{X}, -)$.
    In particular, $\iota_{\rho(\xi)} \omega_{\mathcal{Y}}$ lifts to a closed form as desired.
\end{proof}

\begin{para}
    We keep the notations from \Cref{para-moment-map} and fix a $\mathrm{B}T$-action on $\mathcal{Y}$.
    By taking the $(-1)$-shifted tangent, we obtain an action of $\mathrm{Lie}(T) / T$ on $\mathrm{T}[-1]\mathcal{Y}$.
    Since we have $\mathrm{T}[-1] \mathcal{Y} \cong \mathrm{T}^*[-1] \mathcal{Y} $, we obtain an action of $\mathrm{Lie}(T) $ on $\mathrm{T}^*[-1] \mathcal{Y}$,
    where $\mathrm{Lie}(T)$ is equipped with the additive group structure.
    For an element $\xi \in \mathrm{Lie}(T)$, we let 
    \[
     a_{\xi} \colon  \mathrm{T}^*[-1] \mathcal{Y} \cong \mathrm{T}^*[-1] \mathcal{Y}
    \]
    be the action map.
\end{para}

\begin{lemma}\label{lem-symplectic-preservation}
    The map $a_{\xi}$ preserves the standard $(-1)$-shifted symplectic structure.
\end{lemma}

\begin{proof}
    Recall from \Cref{item-example-symplectic-critical} in \Cref{para-example-shifted-symplectic} that the $(-1)$-shifted symplectic structure on $\mathrm{T}^*[-1]{\mathcal{Y}}$ is given by
    \[
        \omega_{\mathrm{T}^*[-1]{\mathcal{Y}}} = d_{\mathrm{dR}} \lambda_{\mathrm{T}^*[-1]{\mathcal{Y}}}
    \]
    where $\lambda_{\mathrm{T}^*[-1]{\mathcal{Y}}}$ is the image of the tautological section $s_{\mathrm{taut}}$ for $\mathbb{L}_{\mathcal{Y}}[-1]$.
    By the construction of the map $a_{\xi}$, we have the identity
    \[
    a_{\xi}^* s_{\mathrm{taut}} = s_{\mathrm{taut}} + \pi^* \iota_{\xi} \omega_{\mathcal{Y}}.
    \]
    In particular, it is enough to prove $d_{\mathrm{dR}} (\iota_{\xi} \omega_{\mathcal{Y}}) = 0$ as a $(-1)$-shifted closed $2$-form.
    However, this follows from the fact that $\iota_{\xi} \omega_{\mathcal{Y}}$ upgrades to an exact form as we have seen in \Cref{prop-moment}.
\end{proof}

\section{Symmetric stacks}

In this section, we will introduce the notion of symmetric stacks, which corresponds to the symmetry condition of quivers.
The symmetry condition is necessary for the cohomological integrality theorem.

\subsection{Symmetric representations}

\begin{para}
    Here, we will study some basic properties of symmetric representations.

\end{para}

\begin{definition}
    Let $G$ be a reductive group and $V$ be a representation of $G$.
    \begin{enumerate}
        \item  $V$ is \textit{symmetric} if there exists an isomorphism $V \cong V^{\vee}$ as $G$-representations.
        \item  $V$ is \emph{orthogonal} (resp.~\emph{symplectic}) if there exists a $G$-invariant non-degenerate symmetric bilinear form (resp.~symplectic form) on $V$.
    \end{enumerate}
\end{definition}

\begin{example}\label{ex-adjoint-invariant-metric}
    Let $G$ be a reductive group and equip its Lie algebra $\mathfrak{g}$ with the adjoint representation.
    Then the Killing form defines a non-degenerate symmetric bilinear form on $\mathfrak{g} / \mathfrak{z}$ where $\mathfrak{z} \subset \mathfrak{g}$ denotes the centre of the Lie algebra.
    On the other hand,  the adjoint action of $G$ on $\mathfrak{z}$ factors through $\Gamma = G / G^{\circ}$ where $G^{\circ} \subset G$ denotes the neutral component.
    Further, the $\Gamma$-action on $\mathfrak{z}$ is induced by a $\Gamma$-action on the torus $\mathrm{Z}(G^{\circ})^{\circ}$.
    In particular, $\mathfrak{z}$ admits an integral form as a $\Gamma$-representation, therefore it admits a $G$-invariant non-degenerate symmetric bilinear form.
    Using the isomorphism of $G$-representations $\mathfrak{g} \cong \mathfrak{g} / \mathfrak{z} \oplus \mathfrak{z}$,
    we conclude that $\mathfrak{g}$ is an orthogonal $G$-representation.
\end{example}

\begin{para}
    Let $V$ be a symmetric representation of a reductive group $G$. 
    By considering the irreducible decomposition of $V$, we see that there is a unique decomposition
    \[
    V \cong \left(\bigoplus_{i \in I} U_i^{a_i}\right) \oplus    \left(\bigoplus_{j \in J} V_j^{b_j}\right) \oplus   \left(\bigoplus_{k \in K} (W_k \oplus W_{k}^{\vee})^{c_k}\right)   
    \]
    where $\{ U_i \}_{i \in I}$ (resp.~$\{V_j\}_{j \in J}$, $\{W_k\}_{k \in K}$) are mutually distinct orthogonal (resp.~symplectic, non-symmetric) irreducible representations.
    The representation $V$ is orthogonal if and only if all $b_j$ are even numbers.
    With this observation, we conclude the following two-out-of-three property:
\end{para}

\begin{lemma}\label{lemma-two-out-of-three}
    Let $G$ be a reductive group and $V_1$ and $V_2$ be $G$-representations.
    If $V_1$ and $V_1 \oplus V_2$ are symmetric, then $V_2$ is also symmetric. The same statement holds for orthogonality.
\end{lemma}

\begin{para}
    We say that a virtual representation $a \in \mathrm{K}_0(\mathsf{Rep}(G))$ is symmetric (resp.~orthogonal) if there exists a presentation $a = [V_1] - [V_2]$ for symmetric (resp.~orthogonal) representations $V_1$ and $V_2$.
    Clearly, the symmetric (resp.~orthogonal) elements form a subgroup of $\mathrm{K}_0(\mathsf{Rep}(G))$. 
    For a $G$-representation $V$, an element $[V]$ is symmetric (resp.~orthogonal) if and only if $V$ is symmetric (resp.~orthogonal): this follows from Lemma \ref{lemma-two-out-of-three}.
    We will say that an object $V \in \mathsf{D}^\mathrm{b}(\mathsf{Rep}(G))$ is symmetric (resp.~orthogonal) if the corresponding virtual representation 
    \[
    [V] \coloneqq \sum_i (-1)^i [\mathrm{H}^i(V)] \in \mathrm{K}_0(\mathsf{Rep}(G))    
    \]
    is symmetric (resp.~orthogonal).
\end{para}

\begin{para}\label{para-symmetry-weights}
    Now assume that $G$ is a connected reductive group and $T \subset G$ is a maximal torus. 
    Since the isomorphism class of a representation is determined by its characters, a representation $V$ is symmetric if and only if 
    $\dim V_{\gamma} = \dim V_{- \gamma}$ holds for any character $\gamma \in \mathrm{Hom}(T, \mathbb{G}_\mathrm{m})$.

\end{para}

\begin{para}
    There is a criterion for identifying an irreducible symmetric representation to be orthogonal, which we briefly explain now following \cite[Lemma 79]{steinberg2016lectures}.
    Throughout the paragraph, we assume $G$ to be semisimple and connected.
    Let $\{ h_{\alpha} \}_{\alpha \in \Phi^+}$ be a set of positive coroots and define an element $h \coloneqq \prod_{\alpha \in \Phi^+} h_{\alpha}(-1)$
    where we regard $h_{\alpha}$ as a morphism $\mathbb{G}_{\mathrm{m}} \to G$.
    Then it is shown in loc.cit.\ that $h$ is contained the centre of $G$ and the equality $h^{2} = e$ holds.
    An irreducible symmetric representation $V$ is orthogonal if and only if $h$ acts on $V$ by the identity map.
    As explained in \cite[Observation after Lemma 79]{steinberg2016lectures}, the equality $h = e$ holds in the following cases:
    \begin{itemize}
        \item If $G$ has a trivial centre.
        \item If $G$ is either of type $A_{2n}$, $E_6$, $E_8$, $F_4$ or $G_2$.
    \end{itemize}
    In particular, for these groups, symmetric representations are automatically orthogonal.
\end{para}

\subsection{Symmetric stacks}\label{ssec-symmetric-stacks}

\begin{para}

Here we will introduce  various versions of the notion of symmetric stacks.

\end{para}

\begin{definition}\label{def-symmetry}
    Let $\mathcal{X}$ be a derived algebraic stack with affine diagonal.
    Assume further that $\mathcal{X}$ admits a good moduli space $X$. Take a point $x \in \mathcal{X}$ with reductive stabilizer group $G_x$, and let $\mathrm{T}_{\mathcal{X}, x} \coloneqq \mathrm{H}^0(\mathbb{T}_{\mathcal{X}, x})$ be the tangent space. 
    \begin{itemize}
        \item $\mathcal{X}$ is \emph{symmetric} at $x$ if $\mathrm{T}_{\mathcal{X}, x}$ is symmetric as a $G_x$-representation.
         \item $\mathcal{X}$ is \emph{orthogonal} at $x$ if $\mathrm{T}_{\mathcal{X}, x}$ is orthogonal as a $G_x$-representation.
        \item $\mathcal{X}$ is \emph{almost symmetric} (resp.~\emph{almost orthogonal}) at $x$ if $\mathrm{T}_{\mathcal{X}, x}$ is symmetric (resp.~orthogonal) as a $G_x^{\circ}$-representation, where $G_x^{\circ} \subset G_x$ denotes the neutral component.
        \item $\mathcal{X}$ is called \emph{symmetric} if it is symmetric at all closed points; we define other symmetry properties for stacks analogously. Recall that stabilizer groups at closed points are reductive by \cite[Proposition 12.14]{_Alper_GoodmodulispacesforArtinstacks}.
    \end{itemize}
\end{definition}

\begin{para}\label{para-symmetric-weights-stacks}
    By \Cref{para-symmetry-weights}, an algebraic stack $\mathcal{X}$ with a good moduli space is almost symmetric if and only if, for any graded point $\lambda \colon \mathrm{B} \mathbb{G}_\mathrm{m} \to \mathcal{X}$ underlying a closed point $x$,
    the object $\lambda^* \mathrm{T}_{\mathcal{X}, x}$ regarded as a $\mathbb{G}_\mathrm{m}$-representation is symmetric.

 We now see some examples of symmetric stacks. 
 The following lemma will be useful to prove symmetry properties of stacks:
\end{para}

\begin{lemma}\label{lemma-smooth-symplectic-symmetric}
    Let $\mathcal{X}$ be a derived algebraic stack with affine diagonal.
    Assume further that $\mathcal{X}$ admits a good moduli space $X$, and take a point $x \in \mathcal{X}$ with reductive stabilizer group.
    \begin{enumerate}
        \item Assume that $\mathcal{X}$ is either smooth, $0$-shifted symplectic or $(-1)$-shifted symplectic. Then $\mathcal{X}$ is symmetric or almost symmetric at $x$ if and only if the virtual $G_x$-representation $[\mathbb{T}_{\mathcal{X}, x}]$ is.
        \item Assume that $\mathcal{X}$ is either smooth or $0$-shifted symplectic. Then $\mathcal{X}$ is orthogonal or almost orthogonal at $x$ if and only if the virtual $G_x$-representation $[\mathbb{T}_{\mathcal{X}, x}]$ is.
    \end{enumerate}
\end{lemma}

\begin{proof}
    We first deal with the smooth case.
    Note that the tangent complex $\mathbb{T}_{\mathcal{X}, x}$ has amplitude $[-1, 0]$.
    Further, the isomorphism $\mathrm{H}^{-1}(\mathbb{T}_{\mathcal{X}, x}) \cong \mathrm{Lie}(G_x)$ together with \Cref{ex-adjoint-invariant-metric} implies that $\mathrm{H}^{-1}(\mathbb{T}_{\mathcal{X}, x})$ is orthogonal.
    Hence we conclude.

    Now assume that $\mathcal{X}$ is $0$-shifted symplectic.
    Then $\mathbb{T}_{\mathcal{X}, x}$ has amplitude $[-1, 1]$ and we have isomorphisms $\mathrm{H}^{1}(\mathbb{T}_{\mathcal{X}, x})^{\vee}  \cong \mathrm{H}^{- 1}(\mathbb{T}_{\mathcal{X}, x}) \cong \mathrm{Lie}(G_x) $, which imply the desired statement.

    Finally we consider the case when $\mathcal{X}$ is $(-1)$-shifted symplectic.
    Set $h \coloneqq [\mathrm{T}_{\mathcal{X}, x}] \in \mathrm{K}_0(\mathsf{Rep}(G_x))$. 
    By repeating the discussion above, we conclude that $\mathbb{T}_{\mathcal{X}, x}$ is symmetric if and only if $h - h^{\vee}$ is.
    Further, $h - h^{\vee}$ being symmetric is equivalent to $h$ being symmetric, hence we conclude.
\end{proof}

\begin{lemma}\label{lem-quotient-symmetric}
    Let $V$ be a symmetric representation of a reductive group $G$.
    Then the stack $V / G$ is symmetric at all points with reductive stabilizer. A similar statement holds for orthogonality, almost symmetricity and almost orthogonality.
\end{lemma}

\begin{proof}
    Let $x \in V / G$ be a point with reductive stabilizer and $p \colon V / G \to \mathrm{B} G$ be the projection.
    The fibre sequence $V \otimes \mathcal{O}_{V / G} \to  \mathbb{T}_{V / G } \to p^* \mathbb{T}_{\mathrm{B} G}$ induces a fibre sequence of $G_x$-representations $V \to \mathbb{T}_{V / G, x} \to \mathfrak{g}[1]$.
    Since $\mathfrak{g}$ is an orthogonal $G_x$-representation as we have seen in Example \ref{ex-adjoint-invariant-metric},
    the virtual representation $[\mathbb{T}_{V / G, x}]$ is symmetric.
    Then the statement follows from Lemma \ref{lemma-smooth-symplectic-symmetric}.
    
    The proof for the orthogonality, almost symmetricity and almost orthogonality are iden\-tical.
\end{proof}

\begin{corollary}\label{cor-symmetric-open}
    Let $\mathcal{U}$ be a smooth algebraic stack having affine diagonal admitting a good moduli space  $U$.
    If $\mathcal{U}$ is (almost) symmetric at a point $x\in\mathcal{U}$ with reductive stabilizer, then $\mathcal U$ is (almost) symmetric at all points with reductive stabilizer in some open neighborhood of $x$. In particular, if $\mathcal{U}$ is (almost) symmetric, then it is (almost) symmetric at all points with reductive stabilizer.

    The same statement holds for (almost) orthogonality.
\end{corollary}

\begin{proof}
    By the local structure theorem of stacks due to \textcite[Theorem 1.2]{_Alper_ALunaetaleslicetheoremforalgebraicstacks} and Luna's fundamental lemma \cite[Proposition 4.13]{_Alper_ALunaetaleslicetheoremforalgebraicstacks}, we may assume that $\mathcal{U} = V / G$ for a reductive group $G$ and a $G$-representation $V$ and that the origin $0 \in V / G$ is (almost) symmetric or orthogonal.
    Then the statement follows from Lemma \ref{lem-quotient-symmetric}.
\end{proof}

\begin{para}

    We now prove that the symmetricity property is inherited by the critical locus:

\end{para}

\begin{lemma}
    Let $\mathcal{U}$ be a smooth algebraic stack having affine diagonal admitting a good moduli space  $U$ and $f \colon \mathcal{U} \to \mathbb{A}^1$ be a regular function.
    Assume that $\mathcal{U}$ is (almost) symmetric (resp.~(almost) orthogonal).
    Then the derived critical locus $\mathcal{X} = \mathrm{Crit}(f)$ is also (almost) symmetric (resp.~(almost) orthogonal).
\end{lemma}

\begin{proof}
    We prove this only for the orthogonality property: the proofs for the other properties are similar.
    Let $x \in \mathcal{X}$ be a closed point and $G_x$ be the stabilizer group.
    It is enough to show that the $G_x$-representation $\mathrm{H}^{0}(\mathbb{T}_{\mathcal{X}, x}) $ is orthogonal.
    By using \cite[Theorem 1.2 and Proposition 4.13]{_Alper_ALunaetaleslicetheoremforalgebraicstacks},
    we may assume $\mathcal{U} = W/G_x$ for some smooth affine scheme $W$ with a fixed point $\tilde{x} \in W$ which maps to $x$.
    Further, we may assume that $\mathrm{T}_{W, \tilde{x}}$ is an orthogonal $G_{{x}}$-representation.
    The tangent complex $\mathbb{T}_{\mathcal{X}, x}$ is given by
    \[
        [\mathfrak{g}_x \xrightarrow{0} \mathrm{T}_{W, \tilde{x}} \xrightarrow{\mathrm{Hess}_{\tilde{f}}} \Omega_{W, \tilde{x}} \xrightarrow{0} \mathfrak{g}_x^{\vee}]
    \]
    as a complex of $G_x$-representations.
    Here $\tilde{f}$ is the pullback of $f$ to $W$ and $\mathrm{Hess}_{\tilde{f}}$ denotes the Hessian of $\tilde{f}$.
    Since the Hessian is a $G_x$-invariant symmetric bilinear from on $\mathrm{T}_{W, x}$, the image of the map $\mathrm{Hess}_{\tilde{f}}$ is orthogonal.
    Using the decomposition of $G_x$-representations
    \[
        \mathrm{T}_{W, \tilde{x}} \cong  \mathrm{Im}(\mathrm{Hess}_{\tilde{f}}) \oplus \mathrm{H}^{0}(\mathbb{T}_{\mathcal{X}, x})
    \]
    together with \Cref{lemma-two-out-of-three}, we conclude that $\mathrm{H}^{0}(\mathbb{T}_{\mathcal{X}, x})$ is an orthogonal $G_x$-representation as desired.
\end{proof}

Arguing as in the proof of Corollary \ref{cor-symmetric-open} using \Cref{thm-Darboux}, we obtain the following claim:

\begin{corollary}\label{cor-symmetric-open--1-symplectic}
    Let $\mathcal{X}$ be derived algebraic stack having affine diagonal admitting a good moduli space  $X$. Assume that $\mathcal{X}$ is equipped with a $(-1)$-shifted symplectic structure.
    Then the set of almost orthogonal points in $\mathcal{X}$ forms an open subset of $\mathcal{X}$.
    Similar statements hold for symmetricity, almost symmetricity and orthogonality.
\end{corollary}

\begin{corollary}\label{cor--1^-shifted-cotangent-orthogonal}
Let $\mathcal{Y}$ be a derived algebraic stack having affine diagonal admitting a good moduli space.
Assume that $\mathcal{Y}$ is $0$-shifted symplectic and almost orthogonal.
Then the $(-1)$-shifted cotangent stack $\mathrm{T}^*_{\mathcal{Y}}[-1]$ is almost orthogonal.
Similar statements hold for symmetricity, almost symmetricity and orthogonality.
\end{corollary}

\begin{proof}
    We prove the statement only for the almost orthogonality: the same argument works for other properties.
    First, since $\mathrm{T}^*_{\mathcal{Y}}[-1]$ is affine over $\mathcal{Y}$, it also has affine diagonal and admits a good moduli space by \cite[Lemma 4.14]{_Alper_GoodmodulispacesforArtinstacks}.
    Let $0_{\mathcal{Y}} \colon \mathcal{Y} \hookrightarrow \mathrm{T}^*_{\mathcal{Y}}[-1]$ be the zero section and take a closed point $y \in \mathcal{Y}$.
    Note that we have the decomposition
    \[
     \mathrm{T}_{\mathrm{T}^*_{\mathcal{Y}}[-1], 0_{\mathcal{Y}}(y) } \cong \mathrm{T}_{\mathcal{Y}, y} \oplus \mathrm{H}^{-1}( \mathbb{L}_{\mathcal{Y}, y}) \cong \mathrm{T}_{\mathcal{Y}, y} \oplus \mathrm{H}^{-1}( \mathbb{T}_{\mathcal{Y}, y}),
    \]
    where we used the fact that $\mathcal{Y}$ is $0$-shifted symplectic for the latter isomorphism.
    Since $\mathrm{H}^{-1}( \mathbb{T}_{\mathcal{Y}, y})$ is an orthogonal $G_y$-representation as shown in \Cref{ex-adjoint-invariant-metric}, we see that $\mathrm{T}^*_{\mathcal{Y}}[-1]$ is almost orthogonal at $0_{\mathcal{Y}}(y)$.
    Using Corollary \ref{cor-symmetric-open--1-symplectic}, we can take an open subset $\mathcal{U} \subset \mathrm{T}^*_{\mathcal{Y}}[-1]$ whose closed points are exactly the almost orthogonal closed points in $\mathcal{X}$.
    Note that $\mathcal{U}$ is closed under the scaling $\mathbb{G}_\mathrm{m}$-action and contains the zero section.
    Therefore by using \cite[Lemma 1.3.5]{_HalpernLeistner_Onthestructureofinstabilityinmodulitheory}, we conclude $\mathcal{U} = \mathrm{T}^*_{\mathcal{Y}}[-1]$.
\end{proof}

\subsection{Examples of symmetric stacks in moduli theory}

\begin{para}[Stack of semistable $G$-bundles]
    We will now move on to moduli-theoretic examples of symmetric stacks.
    Let $C$ be a smooth projective curve and $G$ be a connected reductive group.
    Let $\mathcal{B}\mathrm{un}_G(C)$ be the moduli stack of $G$-bundles on $C$ and $\mathcal{B}\mathrm{un}_G(C)^{\mathrm{ss}} \subset \mathcal{B}\mathrm{un}_G(C)$ be the open substack consisting of semistable bundles: see \textcite[Definition 1.1]{ramanathan-stable} for the definition.
    The following statement is a direct consequence of \cite[Lemma 2.1]{ramanathan-stable}: For a semistable $G$-bundle $E$ on $C$, its reduction $E_L$ to a Levi subgroup $L \subset G$ and
    a character $\chi \colon L \to \mathbb{G}_{\mathrm{m}}$ with $\chi |_{Z(G)^{\circ}} =1$, the following equality holds:
    \begin{equation}\label{eq-semistable-reductin-deg-0}
        \deg (E_L(\chi)) = 0
    \end{equation}
    where $E_L(\chi)$ is the line bundle associated with the induction of $E_{L}$ along the map $\chi$.
    
    We will show that $\mathcal{B}\mathrm{un}_G(C)^{\mathrm{ss}} $ is almost orthogonal. 
    To prove this, we need the following lemma:
\end{para}

\begin{lemma}\label{lemma-equivariant-Riemann-Roch}
    Let $E$ be a polystable principal $G$-bundle on $C$ and $\mathrm{Aut}(E)^{\circ} \subset \mathrm{Aut}(E)$ be the neutral component. Then the following equality of virtual $\mathrm{Aut}(E)^{\circ}$-representations holds:
    \[
     [\mathbb{T}_{\mathcal{B}\mathrm{un}_G(C), [E]}] =  (g(C) - 1) [\mathfrak{g}]    
    \]
    where $\mathfrak{g}$ denotes the restriction of the adjoint representation of $G$ to $\mathrm{Aut}(E)^{\circ}$.
\end{lemma}

\begin{proof}
    We will prove this by repeating the argument used to prove the Riemann--Roch theorem.

    First, note that we have an isomorphism
    \begin{equation*}
        \mathbb{T}_{\mathcal{B}\mathrm{un}_G(C), [E]} \cong R \Gamma(C, \mathrm{Ad}(E))[1]
    \end{equation*}
    where $\mathrm{Ad}(E)$ is the adjoint vector bundle.
    Let $\mathcal{L} \in \mathrm{Pic}(C)$ be a line bundle. 
    We will prove an equality of virtual $\mathrm{Aut}(E)^{\circ}$-representations
    \begin{equation}\label{eq:refined-Riemann-Roch}
        [R \Gamma(C, \mathrm{Ad}(E) \otimes \mathcal{L})] \cong  (\deg(\mathcal{L}) + 1 - g(C) ) [\mathfrak{g}]
    \end{equation}
    which specializes to the lemma by taking $\mathcal{L} = \mathcal{O}_C$.

    Firstly, we claim an equality
    \begin{equation}\label{eq:Riemann-Roch-difference}
        [R \Gamma(C, \mathrm{Ad}(E) \otimes \mathcal{L}_1)] - [R \Gamma(C, \mathrm{Ad}(E) \otimes \mathcal{L}_2)] = (\deg (\mathcal{L}_1) - \deg (\mathcal{L}_2 )) [\mathfrak{g}]
    \end{equation}
    for line bundles $\mathcal{L}_1, \mathcal{L}_2 \in \mathrm{Pic}(C)$.
    It is enough to prove the statement in case $\mathcal{L}_1 = \mathcal{L}_2(D)$ for some effective divisor $D$ of degree one.
    In this case, the statement follows from the fibre sequence 
    \[
        R \Gamma(C, \mathrm{Ad}(E) \otimes \mathcal{L}_2) \to R \Gamma(C, \mathrm{Ad}(E) \otimes \mathcal{L}_2(D)) \to \mathfrak{g}.  
    \]

    Next, we prove that the virtual $\mathrm{Aut}(E)^{\circ}$-representation $[R \Gamma(C, \mathrm{Ad}(E) \otimes \mathcal{L})]$ is symmetric, i.e., the following identity holds:
    \begin{equation}\label{eq-Bun_G-tangent-symmetric-bundle}
        [R \Gamma(C, \mathrm{Ad}(E) \otimes \mathcal{L})] = [R \Gamma(C, \mathrm{Ad}(E) \otimes \mathcal{L})^{\vee}].
    \end{equation}
    To prove this, using \Cref{eq:Riemann-Roch-difference}, we may assume $\mathcal{L} = \mathcal{O}_C$.
    By the discussion in \Cref{para-symmetry-weights}, it is enough to prove the identity after restricting to one-parameter subgroups $\bar{\lambda} \colon \mathbb{G}_\mathrm{m} \to \mathrm{Aut}(E)^{\circ}$.
    Consider the composition $\lambda \colon \mathbb{G}_\mathrm{m} \to \mathrm{Aut}(E)^{\circ} \to G$ and let $E_{L_{\lambda}}$ be the  $L_{\lambda}-$reduction of $E$ corresponding to $\bar{\lambda}$, where $L_{\lambda} \subset G$ denotes the Levi subgroup corresponding to $\lambda$.
    Then the $i$-th degree part of the tangent space $\mathbb{T}_{\mathcal{B}\mathrm{un}_G(C), [E]} \simeq R \Gamma(C, E \times^{G} \mathfrak{g})[1]$ is given by 
    \[
        R \Gamma(C, E_{L_{\lambda}} \times^{L_{\lambda}} \mathfrak{g}^{\lambda}_{i})[1]
    \]
    where $\mathfrak{g}^{\lambda}_{i}$ denotes the $i$-th graded piece of $\mathfrak{g}$ with respect to the action given by $\lambda$. 
    Since $E$ is semistable, we have
    \[
        \deg(E_{L_{\lambda}} \times^{L_{\lambda}} \mathfrak{g}^{\lambda}_{i}) = \deg( \det(E_{L_{\lambda}} \times^{L_{\lambda}} \mathfrak{g}^{\lambda}_{i})) = \deg(E_{L_{\lambda}} \times^{L_{\lambda}} \det(\mathfrak{g}^{\lambda}_{i})) = 0    
    \]
    for any $i$, where we used \Cref{eq-semistable-reductin-deg-0} for the last equality. In particular, By the Riemann--Roch theorem and the equality $\dim \mathfrak{g}^{\lambda}_{i} = \dim \mathfrak{g}^{\lambda}_{- i}$, we have
    \[
     \chi(R \Gamma(C, E_{L_{\lambda}} \times^{L_{\lambda}} \mathfrak{g}^{\lambda}_{i})) = \chi( R \Gamma(C, E_{L_{\lambda}} \times^{L_{\lambda}} \mathfrak{g}^{\lambda}_{ - i}))  
    \]
    for each $i$. Hence we conclude the identity \Cref{eq-Bun_G-tangent-symmetric-bundle}.

    Next we claim an equality
    \begin{equation}\label{eq-riemann-roch-duality}
        [R \Gamma(C, \mathrm{Ad}(E) \otimes \mathcal{L})] = - [R \Gamma(C, \mathrm{Ad}(E) \otimes \mathcal{L}^{\vee} \otimes \omega_C)].   
    \end{equation}
    This follows from the $\mathrm{Aut}(E)$-equivariant isomorphism $\mathrm{Ad}(E) \cong \mathrm{Ad}(E)^{\vee}$ induced by the $G$-invariant symmetric bilinear form on $\mathfrak{g}$ together with Serre duality and \Cref{eq-Bun_G-tangent-symmetric-bundle}.
    The identity \Cref{eq:refined-Riemann-Roch} follows immediately from the identities \Cref{eq:Riemann-Roch-difference} and \Cref{eq-riemann-roch-duality}.
\end{proof}

The following statement is an immediate consequence of \Cref{lemma-equivariant-Riemann-Roch} and \Cref{ex-adjoint-invariant-metric}.

\begin{corollary}\label{cor-bung-is-orthogonally-symmetric}
    The stack $\mathcal{B}\mathrm{un}_G(C)^{\mathrm{ss}} $ is almost orthogonal.
\end{corollary}

\begin{para}[Stack of semistable twisted $G$-Higgs bundles]
    Let $C$ be a smooth projective curve, $G$ be a connected reductive group and $\mathcal{L}$ be a line bundle on $C$.
    An \emph{$\mathcal{L}$-twisted $G$-Higgs bundle} is a pair of a principal $G$-bundle $E$ on $C$ and a section $\phi \in \Gamma(C, \mathrm{Ad}(E) \otimes \mathcal{L})$.
    We let 
    \[
        \mathcal{H}\mathrm{iggs}^{\mathcal{L}}_G(C)
    \]
     denote the derived moduli stack of $\mathcal{L}$-twisted $G$-Higgs bundles. 
     When $\mathcal{L} = \omega_C$, one recovers the definition of the moduli stack of Higgs bundles $ \mathcal{H}\mathrm{iggs}_G(C)$ recalled in \Cref{item-example-symplectic-G-Higgs} of \Cref{para-example-shifted-symplectic}.
     We can define the notion of semistability of an $\mathcal{L}$-twisted $G$-Higgs bundle analogously to the semistability of a principal $G$-bundle.
     There exists an open substack 
     \[
        \mathcal{H}\mathrm{iggs}^{\mathcal{L}}_G(C)^{\mathrm{ss}} \subset \mathcal{H}\mathrm{iggs}^{\mathcal{L}}_G(C)
    \]
     consisting of the semistable objects.
     It is shown in \cite[Proposition 2.8.1.2]{schmitt2008geometric} that $\mathcal{H}\mathrm{iggs}^{\mathcal{L}}_G(C)^{\mathrm{ss}} $ admits a good moduli space.
\end{para}

\begin{lemma}\label{lemma-higgs-orthogonally-symmetric}
    Let $[(E, \phi)] \in \mathcal{H}\mathrm{iggs}^{\mathcal{L}}_G(C)^{\mathrm{ss}}$ be a closed point. Then the virtual $\mathrm{Aut}(E, \phi)^{\circ}$-representation
    $[\mathbb{T}_{\mathcal{H}\mathrm{iggs}^{\mathcal{L}}_G(C)^{\mathrm{ss}}, [(E, \phi)]}]$ is orthogonal.
\end{lemma}

\begin{proof}
    Let $p \colon \mathcal{H}\mathrm{iggs}^{\mathcal{L}}_G(C) \to \mathcal{B}\mathrm{un}_{G}(C)$ be the projection.
    Then the natural map induced on the tangent complexes provides the following fibre sequence of $\mathrm{Aut}(E, \phi)$-representations
    \[
        R \Gamma(C, \mathrm{Ad}(E) \otimes \mathcal{L}) \to \mathbb{T}_{\mathcal{H}\mathrm{iggs}^{\mathcal{L}}_G(C), [(E, \phi)]} \to R \Gamma(C, \mathrm{Ad}(E))[1].
    \]
    We see that the equality \Cref{eq:refined-Riemann-Roch} at the level of virtual $\mathrm{Aut}(E, \phi)^{\circ}$-representations holds by repeating the same argument using the semistability of $(E, \phi)$.
    Then using Example \ref{ex-adjoint-invariant-metric}, we obtain the desired claim.
\end{proof}

\begin{para}
    As we have seen in \Cref{item-example-symplectic-G-Higgs} of \Cref{para-example-shifted-symplectic}, the derived  stack $\mathcal{H}\mathrm{iggs}_G(C) = \mathcal{H}\mathrm{iggs}^{\omega_C}_G(C)$ is $0$-shifted symplectic.
    If the inequality $\deg \mathcal{L} > 2g(C) - 2$ holds, one can show that $\mathcal{H}\mathrm{iggs}^{\mathcal{L}}_G(C)^{\mathrm{ss}}$ is smooth.
    See \cite[Proposition 5.5]{herrero2023automorphisms} for the proof when $\mathrm{H}^0(C, \mathcal{L} \otimes \omega_C^{-1}) > 0$: a similar deformation theory argument can be used to prove the general case.
    By combining these facts and Lemma \ref{lemma-smooth-symplectic-symmetric}, we obtain the following corollary:
\end{para}

\begin{corollary}\label{cor-higgsG-is-orthogonally-symmetric}
    Assume $\mathcal{L} = \omega_C$ or $\deg(\mathcal{L}) > 2g(C) - 2$. Then the stack $\mathcal{H}\mathrm{iggs}^{\mathcal{L}}_G(C)^{\mathrm{ss}}$ is almost orthogonal. 
\end{corollary}

\begin{para}[Character stacks of a compact oriented manifold]

    Let $M$ be a compact oriented $n$-manifold and $G$ be a reductive group. 
    Consider the derived character stack $\mathcal{L}\mathrm{oc}_G(M)$ of $G$-local systems on $M$ recalled in \Cref{item-example-symplectic-3mfd} of \Cref{para-example-shifted-symplectic}.
    Since points in $\mathcal{L}\mathrm{oc}_G(M)$ correspond to $G$-local systems on $M$, we have an equivalence
    \[
        \mathcal{L}\mathrm{oc}_G(M)_{\mathrm{cl}} \cong \mathrm{Hom}(\pi_1(M), G)  / G   
    \]
    where the action is given by the adjoint. In particular, $\mathcal{L}\mathrm{oc}_G(M)_{\mathrm{cl}}$ admits a good moduli space.

\end{para}

\begin{lemma}\label{lem-equiv-index-for-manifolds}
    Let $[\mathcal{K}] \in \mathcal{L}\mathrm{oc}_G(M)$ be a closed point. Then there exists an identity of virtual $\mathrm{Aut}(\mathcal{K})^{\circ}$-representations
    \[
    [\mathbb{T}_{\mathcal{L}\mathrm{oc}_G(M), [\mathcal{K}]}] =  - \chi(M) \cdot [\mathfrak{g}].
    \]
\end{lemma}

\begin{proof}
    Since $\mathrm{Aut}(\mathcal{K})^{\circ}$ is a connected reductive group, by the discussion in \Cref{para-symmetry-weights}, it is enough to prove the identity after restricting to one-parameter subgroups $\bar{\lambda} \colon \mathbb{G}_\mathrm{m} \to \mathrm{Aut}(\mathcal{K})^{\circ}$.
    Consider the composition $\lambda \colon \mathbb{G}_\mathrm{m} \to \mathrm{Aut}(\mathcal{K})^{\circ} \to G$ and let $\mathcal{K}_{L_{\lambda}}$ be the  $L_{\lambda}-$reduction of $\mathcal{K}$ corresponding to $\bar{\lambda}$, where $L_{\lambda} \subset G$ denotes the corresponding subgroup.
    Then the $i$-th degree part of the tangent space $\mathbb{T}_{\mathcal{L}\mathrm{oc}_G(M), [\mathcal{K}]} \simeq R \Gamma(M, \mathcal{K} \times^{G} \mathfrak{g})[1]$ is given by 
    \[
        R \Gamma(M, \mathcal{K}_{L_{\lambda}} \times^{L_{\lambda}} \mathfrak{g}^{\lambda}_{i})[1]
  \]
  where $\mathfrak{g}^{\lambda}_{i}$ denotes the $i$-th graded piece of $\mathfrak{g}$ with respect to the action given by $\lambda$. 
  Since the Euler characteristic of a local system on a finite CW-complex depends only on the rank, we have an equality
  \[
  \rank  R \Gamma(M, \mathcal{K}_{L_{\lambda}} \times^{L_{\lambda}} \mathfrak{g}^{\lambda}_{i})[1]  = - \chi(M) \cdot \dim \mathfrak{g}^{\lambda}_{i}
  \]
  as desired.
\end{proof}

The following corollary is a direct consequence of \Cref{lem-equiv-index-for-manifolds} and \Cref{lemma-smooth-symplectic-symmetric}:

\begin{corollary}\label{cor-character-stack-symmetric}
    \begin{enumerate}
        \item     Let $\Sigma$ be a compact oriented $2$-manifold. Then the derived character stack $\mathcal{L}\mathrm{oc}_G(\Sigma)$ is almost orthogonal.
        \item   Let $M$ be a compact oriented $3$-manifold. Then the derived character stack $\mathcal{L}\mathrm{oc}_G(M)$ is almost symmetric.
    \end{enumerate}
\end{corollary}

\begin{para}\label{para-examples-almost-orthogonal-character-stacks}

We do not know whether $\mathcal{L}\mathrm{oc}_G(M)$ for a general $3$-manifold $M$ is almost orthogonal or not.
On the other hand, we can prove the almost orthogonality in the following cases:

\begin{itemize}
    \item $M$ is of the form $\Sigma \times S^1$ for a compact oriented $2$-manifold $\Sigma$, or more generally, the mapping torus $\Sigma_{\phi}$ associated with a finite order automorphism $\phi \colon \Sigma \xrightarrow[]{\cong} \Sigma$. This will be proved in \Cref{cor-sigma-s1-orthogonal}.
    \item $G$ is either $\mathrm{GL}_n$, $\mathrm{SL}_n$.  This will be proved in \Cref{cor-GL-character-stack-orthogonally-symmetric}.
\end{itemize}

\end{para}

\begin{para}[Orthogonality for loop stacks]
    Let $\mathcal{Y}$ be an algebraic stack with affine diagonal admitting a good moduli space $p \colon \mathcal{Y} \to Y$.
    We will discuss the orthogonality of the loop stack
    \[
     \mathcal{L}\mathcal{Y} \coloneqq \mathcal{Y} \times_{\mathcal{Y} \times \mathcal{Y}} \mathcal{Y}.
    \]  
     For this, we need some lemmas:

\end{para}

\begin{lemma}\label{lem-loop-sends-closed-points}
    We adopt the notation from the last paragraph and let $\pi \colon \mathcal{L} \mathcal{Y} \to \mathcal{Y}$ be the projection.
    Then $\mathcal{L}\mathcal{Y}$ admits a good moduli space and the map $\pi$ sends closed points to closed points.
\end{lemma}

\begin{proof}
    Since $\mathcal{L}\mathcal{Y}$ is affine over $\mathcal{Y}$, it admits a good moduli space by \cite[Lemma 4.14]{_Alper_GoodmodulispacesforArtinstacks}.
    To show that $\pi$ sends closed points to closed points, by using the local structure theorem for stacks by \textcite[Theorem 4.12]{_Alper_ALunaetaleslicetheoremforalgebraicstacks},
    we may assume that $\mathcal{Y}$ is of the form $V / G$ for an affine scheme $V$ and a reductive group $G$.
    In this case,  there is a closed embedding
    \[
        \mathcal{L}\mathcal{Y}  \hookrightarrow (V \times G) / G
    \]
    where $G$ acts on itself by conjugation.
    The image of this embedding is identified with pairs $(v, h)$ with $h \in G_v$, where $G_v$ denotes the stabilizer group at $v$.
    Take such a pair $(v, h) \in V \times G$ having a closed $G$-orbit. It is enough to show that $v$ has a closed $G$-orbit.
    Since $h$ is contained in the centre of a reductive group $G_{(v, h)}$, $h$ is a semisimple element.
    Note that $G \cdot (v, h)$ being closed  in $V \times G$ implies that $G_h \cdot v$ is closed in $V$.
    Therefore by using \cite[Corollaire 1, Remarques 1]{luna1975adherences}, we conclude the closedness of $G \cdot v$.
\end{proof}

\begin{lemma}\label{lemma-semisimple-automorphism}
    Let $G$ be an algebraic group and $V$ be a $G$-representation equipped with a $G$-invariant non-degenerate symmetric bilinear form $q$.
    Let $h \in G$ be a semisimple element.
    Then the restriction $q |_{V^{h}}$ is non-degenerate.
\end{lemma}

\begin{proof}
    This is an immediate consequence of the eigenspace decomposition of $V$.
\end{proof}

\begin{proposition}\label{prop-loop-almost-orthogonal}
    We adopt the notation from \Cref{lem-loop-sends-closed-points}.
    Assume further that $\mathcal{Y}$ is almost orthogonal.
    Then the stack $\mathcal{L}\mathcal{Y}$ is also almost orthogonal.
\end{proposition}

\begin{proof}
    Take a closed point $\tilde{y} \in \mathcal{L} \mathcal{Y}$ which maps to a point $y \in \mathcal{Y}$ which is closed by \Cref{lem-loop-sends-closed-points}.
    Let $h \in G_y$ be an element corresponding to $\tilde{y}$ under the identification of $\pi^{-1}(\mathrm{B} G_y) \cong G_{y} / G_{y}$.
    Then the proof of \Cref{lem-loop-sends-closed-points} implies that $h$ is a semisimple element.
    Consider the following long exact sequence
\[\begin{tikzcd}
	0 &[-18pt] {\mathrm{H}^{-1}(\mathbb{T}_{\mathcal{L}\mathcal{Y}, \tilde{y}})} &[-15pt] {\mathrm{H}^{-1}(\mathbb{T}_{\mathcal{Y}, y})} &[-15pt] {\mathrm{H}^{-1}(\mathbb{T}_{\mathcal{Y}, y})} &[-15pt] {\mathrm{H}^{0}(\mathbb{T}_{\mathcal{L}\mathcal{Y}, \tilde{y}})} &[-15pt] {\mathrm{H}^{0}(\mathbb{T}_{\mathcal{Y}, y})} &[-12pt] {\mathrm{H}^{0}(\mathbb{T}_{\mathcal{Y}, y})} \\
	0 & {\mathfrak{g}_{\tilde{y}}} & {\mathfrak{g}_{y}} & {\mathfrak{g}_{y}}
	\arrow[from=1-1, to=1-2]
	\arrow[from=1-2, to=1-3]
	\arrow["\cong", from=1-2, to=2-2]
	\arrow["\delta^{-1}", from=1-3, to=1-4]
	\arrow["\cong", from=1-3, to=2-3]
	\arrow[from=1-4, to=1-5]
	\arrow["\cong", from=1-4, to=2-4]
	\arrow["{\pi_*}", from=1-5, to=1-6]
	\arrow["\delta^{0}", shorten <=4pt, from=1-6, to=1-7]
	\arrow[from=2-1, to=2-2]
	\arrow[from=2-2, to=2-3]
	\arrow[from=2-3, to=2-4]
\end{tikzcd}\]
    where $\delta^{-1}$ and $\delta^{0}$ denote the boundary maps, which are identified with the action of $h$.
    The almost orthogonality of  ${\mathrm{H}^{-1}(\mathbb{T}_{\mathcal{L}\mathcal{Y}, \tilde{y}})}$ is a consequence of \Cref{lemma-semisimple-automorphism}.
    Since ${\mathrm{H}^{-1}(\mathbb{T}_{\mathcal{L}\mathcal{Y}, \tilde{y}})}$ and ${\mathrm{H}^{-1}(\mathbb{T}_{\mathcal{Y}, {y}})}$ are almost orthogonal, we see that $\coker(\delta^{-1})$ is almost orthogonal.
    Since $\ker(\delta^{0})$ is almost orthogonal by \Cref{lemma-semisimple-automorphism}, we conclude the almost orthogonality of ${\mathrm{H}^{0}(\mathbb{T}_{\mathcal{L}\mathcal{Y}, \tilde{y}})}$.
\end{proof}

\begin{corollary}\label{cor-twisted-loop-almost-orthogonal}
    We adopt the notation from \Cref{prop-loop-almost-orthogonal} and let $\phi \colon \mathcal{Y} \to \mathcal{Y}$ be a finite order automorphism.
    Consider the twisted loop stack $\mathcal{L}_{\phi} \mathcal{Y} \coloneqq \mathcal{Y}_{\Delta,\mathcal{Y} \times \mathcal{Y}, \Gamma_{\phi}} \mathcal{Y}$, where $\Delta$ is the diagonal map and $\Gamma_{\phi}$ is the graph map associated with $\phi$.
    Then $\mathcal{L}_{\phi} \mathcal{Y} $ is also almost orthogonal.
\end{corollary}

\begin{proof}
    Let $n$ be the order of $\phi$ and consider the corresponding $\mu_n$-action on $\mathcal{Y}$.
    Then we have a natural closed and open embedding
    \[
        (\mathcal{L}_{\phi} \mathcal{Y}) /\mu_n \hookrightarrow \mathcal{L}(\mathcal{Y} / \mu_n).
    \]
    Since $\mathcal{L}(\mathcal{Y} / \mu_n)$ is almost orthogonal by \Cref{prop-loop-almost-orthogonal},
    we concluded that $\mathcal{L}_{\phi} \mathcal{Y}$ is also almost orthogonal.
\end{proof}

\begin{corollary}\label{cor-sigma-s1-orthogonal}
    Let $\Sigma$ be a compact oriented $2$-manifold, $\phi \colon \Sigma \xrightarrow[]{\cong} \Sigma$ be a finite order automorphism and $G$ be a reductive group.
    Let $\Sigma_{\phi}$ be the mapping torus associated with $\phi$.
    Then the derived character stack $\mathcal{L}\mathrm{oc}_G(\Sigma_{\phi})$ is almost orthogonal.
    In particular, the derived stack $\mathcal{L}\mathrm{oc}_G(\Sigma \times S^1)$ is almost orthogonal.
\end{corollary}

\begin{proof}
    This is an immediate consequence of \Cref{cor-character-stack-symmetric} and \Cref{cor-twisted-loop-almost-orthogonal}.
\end{proof}

\begin{para}[Orthogonality for linear moduli stacks]\label{para-orthogonal-linear-moduli}
    Let $\mathcal{C}$ be a $\mathbb{C}$-linear dg-category of finite type, $\mathcal{M}_{\mathcal{C}}$ be the moduli stack of objects in $\mathcal{C}$ in the sense of \textcite{toen2007moduli} and $\mathcal{M} \subset \mathcal{M}_{\mathcal{C}}$ be an open substack which is $1$-Artin with affine diagonal admitting a good moduli space $p \colon \mathcal{M} \to M$.
    Assume further that $\mathcal{M}$ is closed under direct sums and direct summands,
    and that it contains the zero object as an open and closed substack.
    As we have seen in \Cref{para-examples-linear}, the stack $\mathcal{M}$ is a linear moduli stack: see \Cref{para-linear-moduli} and \Cref{para-derived-linear-moduli-stack} for the definition.
    We will also assume that there exists an abelian subcategory $\mathcal{A} \subset \mathcal{C}$ such that $\mathcal{M}$ parametrizes the objects in $\mathcal{A}$.
    
    We claim that the condition
    \begin{equation}\label{eq-symmetry-ext1}
     \dim \Hom(E, F[1]) = \dim \Hom(F, E[1])
    \end{equation}
    for any $E, F \in \mathcal{A}$ corresponding to closed points implies the orthogonality for $\mathcal{M}$.
    First note that any closed point in $\mathcal{M}$ is represented by an object
    \[
     E = E_1^{ \oplus m_1} \oplus  \cdots  \oplus  E_l^{ \oplus m_l}
    \]
    for some pairwise distinct simple objects $E_i \in \mathcal{A}$. Since $E_i$ is simple, there exists an isomorphism
    \[
    \mathrm{Aut}(E) \cong \prod_i \mathrm{GL}_{m_i}.    
    \]
    The representation of $\mathrm{Aut}(E)$ on the tangent space $\mathrm{H}^0(\mathbb{T}_{[E]} \mathcal{M})$ is modeled by the tangent space at the origin of the quiver moduli space associated with the Ext-quiver $Q_{E_{\bullet}}$, i.e., the set of the vertices is $\{ 1, \ldots, l\}$ and the number of edges from $i$ to $j$ is given by $\dim \mathrm{Hom}(E_i, E_j[1])$, with the dimension vector $(m_1, \ldots, m_l)$.
    The dual representation of $\mathrm{H}^0(\mathbb{T}_{[E]} \mathcal{M})$ is modeled by quiver moduli  associated with the opposite quiver $Q_{E_{\bullet}}^{\mathrm{op}}$ with the same dimension vector.
    Therefore the equality \Cref{eq-symmetry-ext1} implies that $Q_{E_{\bullet}}$ is symmetric.
    It is clear that the quiver moduli space for a symmetric quiver is orthogonal at the origin, hence we obtain the desired claim.

    As a consequence of the above discussion together with \Cref{cor-character-stack-symmetric}, we obtain the following claim:

 \end{para}

 \begin{corollary}\label{cor-GL-character-stack-orthogonally-symmetric}
    For a compact oriented $3$-manifold $M$,
    the character stack $\mathcal{L} \mathrm{oc}_G(M)$ is almost orthogonal for $G = \mathrm{GL}_n$ and $G = \mathrm{SL}_n$.
 \end{corollary}

 \begin{proof}
    For $G = \mathrm{GL}_n$, the statement follows from \Cref{cor-character-stack-symmetric} together with the above discussion.
    To prove the statement for $\mathrm{SL}_n$, take a semisimple $\mathrm{SL}_n$-local system $\mathcal{L}$ on $M$.
    The tangent space at the point $[\mathcal{L}] \in \mathcal{L}\mathrm{oc}_{\mathrm{SL}_n}(M)$ is isomorphic to 
    $\mathrm{H}^1(M, \mathcal{L} \times^{\mathrm{SL}_n} \mathfrak{sl}_n)$.
    Since $\mathcal{L} \times^{\mathrm{SL}_n} \mathrm{GL}_n$ is a semisimple $\mathrm{GL}_n$-local system (see e.g. \cite[Lemma 3.6.7]{arinkin2020stack}),
    the $\mathrm{Aut}(\mathcal{L})$-representation $\mathrm{H}^1(M, \mathcal{L} \times^{\mathrm{SL}_n} \mathfrak{gl}_n)$ is almost orthogonal.
    Note that there exists a decomposition of $\mathrm{Aut}(\mathcal{L})$-representations
    \[
        \mathrm{H}^1(M, \mathcal{L} \times^{\mathrm{SL}_n} \mathfrak{gl}_n) \cong   \mathrm{H}^1(M, \mathcal{L} \times^{\mathrm{SL}_n} \mathfrak{sl}_n) \oplus \mathrm{H}^1(M, \mathcal{L} \times^{\mathrm{SL}_n} \mathfrak{z})
    \]
    where $\mathfrak{z} \subset \mathfrak{gl}_n$ is the centre. Since the latter summand is a trivial $\mathrm{Aut}(\mathcal{L})$-representation, 
    we conclude the almost orthogonality of $\mathrm{H}^1(M, \mathcal{L} \times^{\mathrm{SL}_n} \mathfrak{sl}_n)$ as desired.
 \end{proof}

 \begin{corollary}
    Let $S$ be smooth surface and $H$ be an ample divisor. 
    Assume either of the following conditions hold:
    \begin{itemize}
        \item $K_S$ is trivial.
        \item $H \cdot K_S < 0$ and $H$ is generic, i.e., for any $H$-semistable sheaves $E$ and $F$ with the same reduced Hilbert polynomial, the equality $\chi(E, F) = \chi(F, E)$ holds.
    \end{itemize}
    Then the moduli stack $\mathcal{M}^{H\textnormal{-}\mathrm{ss}}_S$ of $H$-semistable sheaves on $S$ is orthogonal.
 \end{corollary}

 \begin{proof}
    By the discussion in \Cref{para-orthogonal-linear-moduli}, it is enough to prove the equality
    \[
     \dim \Ext^1(E, F) = \dim \Ext^1(F, E)    
    \]
    for $H$-stable sheaves $E$, $F$ on $S$ with the same reduced Hilbert polynomial.
    Since the statement is obvious for $0$-dimensional sheaves, we will assume that $E$ and $F$ have positive-dimensional supports.
    When $K_S$ is trivial, this is a consequence of Serre duality.
    If $H \cdot K_S < 0$ and $H$ is generic, the statement follows from the equality $\chi(E, F) = \chi(F, E)$ together with the equality $\dim \Hom (E, F) = \dim \Hom (F, E)$ and the vanishing 
    $\Ext^2(E, F) \cong 0 \cong \Ext^2(F, E)$, which is a consequence of Serre duality and the condition $H \cdot K_S < 0$. 
 \end{proof}

\subsection{Cotangent distance of chambers}

\begin{para}
    Let $\mathcal{X}$ be a derived algebraic stack locally finitely presented over $\mathbb{C}$ having finite cotangent weights: see \Cref{para-cotangent-arrangement} for the definition.  Take a face $(F, \alpha) \in \mathsf{Face}(\mathcal{X})$ and chambers $\sigma, \sigma' \subset F$ with respect to the cotangent arrangement.
    Under a certain symmetry assumption, we will introduce a cotangent distance
    \[
    d(\sigma, \sigma') \in \mathbb{Z} /2 \mathbb{Z}    
    \]
    which counts the number of walls between $\sigma$ and $\sigma'$ in the cotangent arrangement.
    As we will see in \Cref{para-supercommutative-smooth}, this parity controls the supercommutativity relation of the cohomological Hall induction.
\end{para}

\begin{para}[Numerically symmetric stacks]
Let $\mathcal{X}$ be a derived algebraic stack locally finitely presented over $\mathbb{C}$. We introduce the following numerical version of the symmetricity condition: 
\end{para}

\begin{definition}
    A derived algebraic stack $\mathcal{X}$ is \emph{numerically symmetric} if for any $\alpha \in \pi_0(\mathrm{Filt}(\mathcal{X}))$, the following equality holds:
    \[
     \vdim \mathcal{X}_{\alpha}^+ = \vdim \mathcal{X}_{- \alpha}^+.    
    \]
\end{definition}

\begin{para}\label{para-numerically-symmetric}
    We summarize the properties of the numerical symmetricity condition that will be used later:
    \begin{enumerate}
        \item \label{item-numerical-vs-almost}
              Assume that $\mathcal{X}$ has affine diagonal and admits a good moduli space.
              Assume further that $\mathcal{X}$ is either smooth, $0$-shifted symplectic or $(-1)$-shifted symplectic.
              Then the almost symmetricity condition is equivalent to the numerical symmetricity. This is an immediate consequence of Lemma \ref{lemma-smooth-symplectic-symmetric} together with the discussion in \Cref{para-symmetric-weights-stacks}.

        \item Assume that $\mathcal{X}$ is a numerically symmetric $(-1)$-shifted symplectic stack. Then for a face $(F, \alpha) \in \mathsf{Face}(\mathcal{X})$ and a cone $\sigma \subset F$, we have
        \begin{equation}\label{eq-numerical-symmetric-implies-zero-dimension}
         \vdim \mathcal{X}_{\sigma}^+ = 0.    
        \end{equation}
        To see this, let $(\mathrm{tot}_{\alpha}^* \mathbb{T}_{\mathcal{X}})^{\sigma, +}\subset \mathrm{tot}_{\alpha}^* \mathbb{T}_{\mathcal{X}}$ denote the positive part with respect to the cone $\sigma$.
        Then we have an identity $\vdim \mathcal{X}_{\sigma}^+ = \rank (\mathrm{tot}_{\alpha}^* \mathbb{T}_{\mathcal{X}})^{\sigma, +}$. Since the existence of the $(-1)$-shifted symplectic structure implies an identity 
        \[
            \rank (\mathrm{tot}_{\alpha}^* \mathbb{T}_{\mathcal{X}})^{\sigma, +} = - \rank (\mathrm{tot}_{\alpha}^* \mathbb{T}_{\mathcal{X}})^{\sigma, -},
        \]
             we obtain the desired claim.

        \item Let $\mathcal{Y}$ be a numerically symmetric derived algebraic stack. Let $f \colon \mathcal{Y} \to \mathbb{A}^1$ be a function and set $\mathcal{X} = \mathrm{Crit}(f)$.
        Then $\mathcal{X}$ is also numerically symmetric. To see this, take $\alpha \in  \pi_0(\mathrm{Filt}(\mathcal{X}))$ and let $\iota_{\alpha} \colon \mathcal{X}_{\alpha} \to \mathcal{Y}_{\alpha}$ be the natural map.
        Then we have the following fibre sequence
        \[
         \iota_{\alpha}^* (\mathrm{tot}_{\alpha}^* \mathbb{L}_{\mathcal{Y}})^+  \to (\mathrm{tot}_{\alpha}^* \mathbb{L}_{\mathcal{X}})^+  \to \iota_{\alpha}^* (\mathrm{tot}_{\alpha}^* \mathbb{T}_{\mathcal{Y}})^+[1] 
         \]
         hence an identity 
         \[
         \vdim \mathcal{X}_{\alpha}^+ =    \vdim \mathcal{Y}_{\alpha}^+ - \vdim \mathcal{Y}_{- \alpha}^+  = 0
         \]
         which implies the numerical symmetricity for $\mathcal{X}$.
    \end{enumerate}
\end{para}   

\begin{para}[Cotangent distance]\label{para-cotangent-distance}
    Let $\mathcal{X}$ be a numerically symmetric derived algebraic stack with finite cotangent weights
     and $(F, \alpha) \in \mathsf{Face}(\mathcal{X})$ be an $n$-dimensional face.
     For chambers $\sigma, \sigma' \subset F$ with respect to the cotangent arrangement, we let 
     \[
      (\mathrm{tot}_{\alpha}^* \mathbb{L}_{\mathcal{X}})^{(\sigma, \sigma'), +, -} \subset    \mathrm{tot}_{\alpha}^* \mathbb{L}_{\mathcal{X}}
     \]
    be the direct summand which has positive weights with respect to the cone $\sigma$ and negative weights with respect to $\sigma'$.
    We define the \emph{cotangent distance} $d(\sigma, \sigma')$ between $\sigma$ and $\sigma'$ by
    \[
     d(\sigma, \sigma') \coloneqq \rank    (\mathrm{tot}_{\alpha}^* \mathbb{L}_{\mathcal{X}})^{(\sigma, \sigma'), +, -} \bmod 2 \in \mathbb{Z} / 2 \mathbb{Z}.
    \]
    Roughly, $d(\sigma, \sigma')$ counts the parity of the number of walls between $\sigma$ and $\sigma'$ in the cotangent arrangement.
    By definition, the distance function has the following properties:
    \begin{itemize}
        \item  $d(\sigma, \sigma) = 0$.
        \item  $d(\sigma, \sigma') = d(\sigma', \sigma)$.
        \item  $d(\sigma, \sigma'') = d(\sigma, \sigma') + d(\sigma', \sigma'')$.
        \item For $g \in \mathrm{Aut}(\alpha)$, we have $d(\sigma, \sigma') = d(g(\sigma), g(\sigma'))$.
    \end{itemize}

    Now define a \emph{cotangent sign representation}
    \[
    \mathrm{sgn}_{\alpha} \colon \mathrm{Aut}(\alpha) \to \mathbb{Z} / 2 \mathbb{Z}
    \]
    as follows: we fix a chamber $\sigma_0 \subset F$ with respect to the cotangent arrangement and set
    \begin{equation}\label{eq-sign-def}
        \mathrm{sgn}_{\alpha}(g) \coloneqq d(\sigma_0, g(\sigma_0)).
    \end{equation}
    By using the properties of the cotangent distance functions, we see that $\mathrm{sgn}_{\alpha}$ defines a group morphism and it does not depend on the choice of $\sigma_0$.
\end{para}

\section{Mixed Hodge modules}

\subsection{Mixed Hodge modules}\label{para-mixed-Hodge-modules}

\begin{para}
    We will recall M. Saito's theory of mixed Hodge modules introduced in \cite{saito-1990-mixed-hodge-modules} and its extension to algebraic stacks due to Tubach \cite{tubach-2024-mixed-hodge-modules-on-stacks}.
    For an algebraic stack $\mathcal{X}$, one can associate an $\infty$-category of mixed Hodge complexes denoted by $\Dh(\mathcal{X})$ and its subcategory $\Dhc(\mathcal{X})$ consisting of objects with constructible cohomology.
    We let $\Dhc^{(+)}(\mathcal{X}) \subset \Dhc(\mathcal{X})$ (resp.~$\Dhc^{(-)}(\mathcal{X}) \subset \Dhc(\mathcal{X})$) denote the subcategory consisting of objects whose restriction to any quasi-compact open substack is bounded below (resp.~above).
    We list minimal properties and notation for mixed Hodge modules that we use in this paper: 
    \begin{itemize}
        \item The category $\Dhc(\mathcal{X})$ is equipped with a t-structure called the \emph{perverse t-structure}. The heart with respect to this t-structure is denoted by $\mathsf{MHM}(\mathcal{X})$.
              The $n$-th cohomology with respect to the perverse t-structure is denoted by ${}^{\mathrm{p}} \mathcal{H}^n(-)$.
              
         \item The category $\mathsf{MHM}(\Spec \mathbb{C})$ is the category of polarizable mixed Hodge structures.

        \item There exists a forgetful functor
            \[
                \mathrm{rat} \colon \Dh(\mathcal{X}) \to \mathsf{D}(\mathcal{X})    
            \]
            to the derived category of sheaves of $\mathbb{Q}$-vector spaces with respect to the analytic topology.
            It restricts to the following functors to the constructible derived category and to the category of perverse sheaves:
            \[
                \mathrm{rat} \colon \Dhc(\mathcal{X}) \to \Dc(\mathcal{X}), \quad \quad \mathrm{rat} \colon \mathsf{MHM}(\mathcal{X}) \to \mathsf{Perv}(\mathcal{X}).
            \]
        
        \item The category $\Dhc(\mathcal{X})$ is equipped with another t-structure called the \emph{ordinary t-structure}, whose heart is denoted by $\mathsf{MHM}_{\mathsf{ord}}(\mathcal{X})$.
              The functor $\mathrm{rat} \colon \Dhc(\mathcal{X}) \to \Dc(\mathcal{X})$ restricts to a functor
              \[
                \mathrm{rat} \colon   \mathsf{MHM}_{\mathsf{ord}}(\mathcal{X}) \to \mathsf{Sh}(\mathcal{X}, \mathbb{Q})
              \]
              where $\mathrm{Sh}(\mathcal{X}, \mathbb{Q})$ denotes the abelian category of sheaves of $\mathbb{Q}$-vector spaces over $\mathcal{X}$ with respect to the analytic topology.

        \item The derived categories of mixed Hodge modules on algebraic stacks are equipped with the six-functor formalism in the sense of Mann (see e.g. \cite[\S 2]{scholze-2024-six-functor-formalism} for the definition).
              The forgetful map $\mathrm{rat}$ is compatible with the six-functor formalism.

        \item For a morphism $f \colon \mathcal{X} \to \mathcal{Y}$ between algebraic stacks, the functor $f^*$ restricts to the functor
              \[
               f^* \colon \Dhc(\mathcal{Y}) \to \Dhc(\mathcal{X}).  
              \]
              If $f$ is of finite type, the functor $f_*$ restricts to the functor
              \[
                f_* \colon \Dhc^{(+)}(\mathcal{X}) \to \Dhc^{(+)}(\mathcal{Y}).  
               \]

        \item Any object $M \in \mathsf{MHM}(\mathcal{X})$ is equipped with a bounded increasing filtration
            \[
                0 \subset \cdots \subset W_{i}M \subset W_{i + 1} M \subset \cdots \subset M
            \]
            called the \textit{weight filtration}. We set $\gr_i^{W} M \coloneqq W_{i}M / W_{i -1} M$. For $n \in \mathbb{Z}$, 
            we say that $M$ has weight~$\leq n$ (resp.~$\geq n$) if $\gr_i^{W} M = 0$ for all $i > n$ (resp.~$i < n$).

        \item For an object $M \in \Dhc(\mathcal{X})$, we say that $M$ has weight~$\leq n$ (resp.~$\geq n$) if ${}^{\mathrm{p}} \mathcal{H}^{i}(M)$ has weight~$\leq i + n$ (resp.~$\geq i + n$) for all $i \in \mathbb{Z}$.
              We say that $M$ is pure of weight~$n$ if $M$ has weight $\leq n$ and $\geq n$.

        \item For an algebraic stack $\mathcal{X}$, let $\mathbb{D}\mathbb{Q}_{\mathcal{X}} \coloneqq (\mathcal{X} \to \mathrm{pt})^! \mathbb{Q}_{\mathrm{pt}}$ be the dualizing complex and  $\mathbb{D} \coloneqq \mathcal{H}\mathrm{om}(-, \mathbb{D}\mathbb{Q}_{\mathcal{X}})$ be the Verdier duality functor. Then $\mathbb{D}$ swaps weights, namely, the following equivalence holds for $M \in \Dhc(\mathcal{X})$:
             \[
             \text{$M$ has weight $\leq n$}   \Longleftrightarrow \text{$\mathbb{D}M$ has weight $\geq -n$}.
             \]

             \item Let $f \colon \mathcal{X} \to \mathcal{Y}$ be a morphism between algebraic stacks. Then the functor $f^!$ does not decrease the weight, i.e., for an object $M \in \Dhc(\mathcal{Y})$ of weight $\geq n$, the complex $f^! M$ is of weight $\geq n$.
             If $\mathcal{X}$ has affine stabilizers and $f$ is of finite type, the functor $f_*$ does not decrease the weight.
        
        \item Assume that $\mathcal{X}$ has affine stabilizers. We let 
        \[
        \mathsf{HM}(\mathcal{X}) \subset \mathsf{MHM}(\mathcal{X})  
        \]
        denote the full subcategory consisting of mixed Hodge modules which are pure of weight $0$. Then $\mathsf{HM}(\mathcal{X})$ is a semi-simple category.
    \end{itemize}

\end{para}

\begin{para}[Pushforward along good moduli space morphisms]
    The following result proved recently by Kinjo \cite[Theorem 1.1]{Kinjo_Decomp} will be repeatedly used throughout the paper:
\end{para}

\begin{theorem}\label{thm-decom-good-moduli}
    Let $\mathcal{X}$ be an algebraic stack with affine diagonal admitting a good moduli space $p \colon \mathcal{X} \to X$.
    Then the functor $p_*$ preserves the weights. In other words, for an object $M \in \Dhc(\mathcal{X})$ which is pure of weight $n$, $p_* M$ is also pure of weight $n$.
\end{theorem}

\subsection{Monodromic mixed Hodge modules}\label{MMHM-section}

\begin{para}
    We will introduce a monodromic extension of the category of mixed Hodge modules. 
    An advantage of working with this larger category is that one has a square root of the Lefschetz Hodge structure $\mathbb{L} = \mathrm{H}^*_{\mathrm{c}}(\mathbb{A}^1)$.  Note also that the Thom--Sebastiani theorem \Cref{eq-local-thom-sebastiani} holds only after taking the monodromy action into account.

    Firstly, we will introduce the $\infty$-category of monodromic mixed Hodge complexes on algebraic stacks following \textcite[\S 2.9]{brav2012symmetries}.
    We let 
    \[
        \mathbb{Q}(1) \coloneqq \mathrm{H}_\mathrm{c}^2(\mathbb{A}^1)^{\vee} \in \mathsf{MHM}(\Spec \mathbb{C})
    \] 
    be the Tate Hodge structure.
    Define a commutative algebra object $R \in \mathrm{CAlg}(\Dh(\Spec \mathbb{C}))$ by the free commutative algebra object
    \[
        R \coloneqq \mathrm{Free}_{\mathrm{comm}}( \mathbb{Q}(1)).
    \]
    Let $\mathcal{X}$ be an algebraic stack and $a_{\mathcal{X}} \colon \mathcal{X} \to \Spec \mathbb{C}$ be the constant map.
    We set
    \[
        R_{\mathcal{X}} \coloneqq a_{\mathcal{X}}^* R \in \mathsf{CAlg}(\Dh(\mathcal{X}))
    \]
    and let $\mathsf{Mod}_{R_{\mathcal{X}}}(\Dhc(\mathcal{X}))$ denote the $\infty$-category of $R_{\mathcal{X}}$-modules in $\Dhc(\mathcal{X})$.
    Explicitly, an object in $\mathsf{Mod}_{R_{\mathcal{X}}}(\Dhc(\mathcal{X}))$ is given by a pair $(M, N)$ 
    of an object $M \in \Dhc(\mathcal{X})$ and a morphism $N \colon M \to M(-1)$, where we write $M(-1) \coloneqq M \otimes \mathbb{Q}(1)^{\vee}$.

    We define the $\infty$-category of unipotently monodromic mixed Hodge complexes on $\mathcal{X}$ to be the full subcategory
    \[
    \Dhc^{\mathsf{umon}}(\mathcal{X}) \subset   \mathsf{Mod}_{R_{\mathcal{X}}}(\Dhc(\mathcal{X}))
    \]
    consisting of pairs $(M, N)$ with $N$ locally nilpotent on each perverse cohomology.

    Consider the group $\hat{\mu} \coloneqq \lim \mu_n$ and let it act trivially on $\mathcal{X}$.
    Define the $\infty$-category of monodromic mixed Hodge complexes on $\mathcal{X}$ to be the full subcategory of $\hat{\mu}$-equivariant objects
    \[
        \Dhc^{\mathsf{mon}}(\mathcal{X}) \subset (\Dhc^{\mathsf{umon}}(\mathcal{X}))^{\hat{\mu}}
    \]
    consisting of objects such that the action of $\hat{\mu}$ on each cohomology sheaf locally factors through $\mu_n$ for some $n$.
    We define $\infty$-categories
    \[
        \mathsf{MHM}^{\mathsf{mon}}(\mathcal{X}), \quad \mathsf{MHM}_{\mathsf{ord}}^{\mathsf{mon}}(\mathcal{X}),\quad \Dhc^{\mathsf{b}, \mathsf{mon}}(\mathcal{X})
    \]
    to be the full subcategories of $\Dhc^{\mathsf{mon}}(\mathcal{X})$ consisting of objects whose underlying mixed Hodge complexes belong to $\mathsf{MHM}(\mathcal{X}), \mathsf{MHM}_{\mathsf{ord}}(\mathcal{X}), \Dhc^{\mathsf{b}}(\mathcal{X})$ respectively.
    
    For a quasi-compact stack, the category $\mathsf{MHM}^{\mathsf{mon}}(\mathcal{X})$ has a concrete description:
    Namely, an object in $\mathsf{MHM}^{\mathsf{mon}}(\mathcal{X})$ is given by a tuple $(M, N, T_s)$
    of an object $M \in \mathsf{MHM}(\mathcal{X})$, morphisms $N \colon M \to M(-1)$ and $T_s \colon M \to M$, where $N$ is nilpotent and some power of $T_s$ is the identity map,  such that the equality $N \circ T_s = T_s \circ N$ holds.
    A similar description holds true for $ \mathsf{MHM}_{\mathsf{ord}}^{\mathsf{mon}}(\mathcal{X})$.

    There exist natural functors given by equipping mixed Hodge complexes with the trivial monodromy operators:
    \[
        \Dhc(\mathcal{X}) \to \Dhc^{\mathsf{mon}}(\mathcal{X}), \quad \mathsf{MHM}(\mathcal{X}) \hookrightarrow \mathsf{MHM}^{\mathsf{mon}}(\mathcal{X}), \quad \mathsf{MHM}_{\mathsf{ord}}(\mathcal{X}) \hookrightarrow \mathsf{MHM}_{\mathsf{ord}}^{\mathsf{mon}}(\mathcal{X}).
    \]
    It is obvious from the description that the latter two functors are fully faithful.
    A mixed Hodge complex is regarded as a monodromic mixed Hodge complex using this functor.
    Also, we can consider the monodromy-forgetting functors:
    \begin{equation}\label{eq-forgetting}
         \Dhc^{\mathsf{mon}}(\mathcal{X}) \to \Dhc(\mathcal{X}), \quad \mathsf{MHM}^{\mathsf{mon}}(\mathcal{X}) \to \mathsf{MHM}(\mathcal{X}), \quad \mathsf{MHM}_{\mathsf{ord}}^{\mathsf{mon}}(\mathcal{X}) \to \mathsf{MHM}_{\mathsf{ord}}(\mathcal{X}).
    \end{equation}
    The latter two functors are faithful.

    Assume that we are given a morphism $f \colon \mathcal{X} \to \mathcal{Y}$ of algebraic stacks.
    Then the four functors $f^*$, $f_*$, $f^!$ and $f_!$ between the $\infty$-category of mixed Hodge complexes recalled in \Cref{para-mixed-Hodge-modules} upgrade to functors for monodromic mixed Hodge complexes in an obvious manner.
    The definition of the tensor product for monodromic mixed Hodge complexes is defined by modifying the Hodge structure using the monodromy operator; this will be explained in \Cref{para-monodromic-tensor}.

\end{para}

\begin{para}[Monodromic vanishing cycle functor]
        Here we explain the definition of the monodromic vanishing cycle functor associated with a function $f \colon \mathcal{X} \to \mathbb{A}^1$.

        Firstly we define the unipotent nearby cycle functor.
        Consider the following diagram:
\[\begin{tikzcd}
	{\mathcal{X}_{\neq 0}} & {\mathcal{X}} & {\mathcal{X}_0} \\
	{\mathbb{G}_\mathrm{m}} & {\mathbb{A}^1} & {\{  0\}.}
	\arrow["j", hook, from=1-1, to=1-2]
	\arrow["{f_{\neq 0}}"', from=1-1, to=2-1]
	\arrow["\lrcorner"{anchor=center, pos=0.125}, draw=none, from=1-1, to=2-2]
	\arrow["f"', from=1-2, to=2-2]
	\arrow["i"', hook', from=1-3, to=1-2]
	\arrow["\lrcorner"{anchor=center, pos=0.125, rotate=-90}, draw=none, from=1-3, to=2-2]
	\arrow["{f_0}"', from=1-3, to=2-3]
	\arrow["{\bar{j}}", hook, from=2-1, to=2-2]
	\arrow["{\bar{i}}"', hook', from=2-3, to=2-2]
\end{tikzcd}\]
Consider the functor
\[
i^* j_* \colon \Dh(\mathcal{X}_{\neq 0}) \to \Dh(\mathcal{X}_{ 0}).
\]
Note that this functor is lax-monoidal.
This in particular shows that $i^* j_* \mathbb{Q}_{\mathcal{X}_{\neq 0}}$ carries an algebra structure and the above functor upgrades to the following functor
\[
 \Dh(\mathcal{X}_{\neq 0}) \to \mathsf{Mod}_{i^* j_* \mathbb{Q}_{\mathcal{X}_{\neq 0}}} (\Dh(\mathcal{X}_{ 0})).
\]
The Beck--Chevalley map induces a map
\[
\mathrm{H}^*(\mathbb{G}_\mathrm{m}) \otimes \mathbb{Q}_{\mathcal{X}_0} \cong f_0^* \bar{i}^* \bar{j}_* \mathbb{Q}_{\mathbb{G}_\mathrm{m}} \to i^* j_* \mathbb{Q}_{\mathcal{X}_{\neq 0}}.
\]
Note that both sides are equipped with commutative algebra structures.
Since the Beck--Chevalley map is given by the composition of the unit and counit map, and the unit and counit maps of the symmetric monoidal adjunction are morphisms of commutative algebra objects (see \cite[Remark 2.5.5.2]{sag}),
this map upgrades to a morphism of commutative algebra objects.
In particular, we obtain a functor
\[
 (i^* j_*)^{\mathrm{enh}} \colon \Dh(\mathcal{X}_{\neq 0}) \to \mathsf{Mod}_{\mathrm{H}^*(\mathbb{G}_\mathrm{m}) \otimes \mathbb{Q}_{\mathcal{X}_0}} (\Dh(\mathcal{X}_{ 0})) \simeq \mathsf{Mod}_{\mathbb{Q}[t^1]} (\Dh(\mathcal{X}_{ 0}))
\]
where $t^1$ is a formal variable of degree one and weight two.
The (non-monodromic) unipotent nearby cycle functor
\[
\psi^{\mathrm{uni}, \textnormal{non-mon}}_{f} \colon \Dh(\mathcal{X}_{\neq 0}) \to \Dh(\mathcal{X}_{ 0})
\]
is given by
\[
\mathcal{F} \mapsto (i^* j_*)^{\mathrm{enh}}(\mathcal{F}) \otimes_{\mathbb{Q}[t^1]}  \mathbb{Q}_{\mathcal{X}_{ 0}}
\]
where the $\mathbb{Q}[t^1]$-module structure on  $\mathbb{Q}_{\mathcal{X}_{ 0}}$ is given by the augmentation $\mathbb{Q}[t^1] \to \mathbb{Q}$.
Let $\mathbb{Q}[x^0]$ be the free commutative algebra generated by a formal variable $x^0$ of degree zero and weight $-2$.
The isomorphism 
$\mathrm{Ext}^*_{\mathbb{Q}[x^0]}(\mathbb{Q}, \mathbb{Q}) \cong \mathbb{Q}[t^1]$ equips
$\mathbb{Q}$ with the structure of $\mathbb{Q}[x^0]$-$\mathbb{Q}[t^1]$-bimodule.
In particular, the unipotent nearby cycle functor upgrades to the functor
\[
\psi^{\mathrm{uni}}_{f} \colon \Dh(\mathcal{X}_{\neq 0}) \to \mathsf{Mod}_{\mathbb{Q}[x^0]} ( \Dh(\mathcal{X}_{ 0})).
\]
By abuse of notation, the functor $\psi_f \circ j^*$ will be also denoted as $\psi_f$.
It follows that for any object $\mathcal{F} \in \Dh(\mathcal{X}_{\neq 0})$, the complex $\psi_{f}^{\mathrm{uni}}(\mathcal{F})$ is equipped with a monodromy operator 
\[
N \colon \psi^{\mathrm{uni}, \textnormal{non-mon}}_{f}(\mathcal{F}) \to \psi^{\mathrm{uni}, \textnormal{non-mon}}_{f}(\mathcal{F}(-1))
\]
corresponding to the $\mathbb{Q}[x^0]$-module structure. 
It is shown in \cite[Proposition 3.29]{tubach-2024-mixed-hodge-modules-on-stacks} that this construction coincides with Saito's definition of the unipotent nearby cycle functor \cite{saito-1990-mixed-hodge-modules}.
In particular, for $\mathcal{F} \in \mathsf{MHM}(\mathcal{X})$, the complex $\psi^{\mathrm{uni}, \textnormal{non-mon}}_{f}(\mathcal{F})$ is  contained in $\mathsf{MHM}(\mathcal{X}_0)$ and the operator $N$ is locally nilpotent for weight reasons.

Now we explain the definition of the full nearby cycle functor following \cite[Definition 3.31]{tubach-2024-mixed-hodge-modules-on-stacks}.
For a positive integer $n$, consider the following Cartesian diagram:
\[\begin{tikzcd}
	{\mathcal{X}_n} & {\mathcal{X}} \\
	{\mathbb{A}^1} & {\mathbb{A}^1}
	\arrow["{e_n}", from=1-1, to=1-2]
	\arrow["{f_n}"', from=1-1, to=2-1]
	\arrow["\lrcorner"{anchor=center, pos=0.125}, draw=none, from=1-1, to=2-2]
	\arrow["f", from=1-2, to=2-2]
	\arrow["{\pi_n}", from=2-1, to=2-2]
\end{tikzcd}\]
where $\pi_n$ is given by $z \mapsto z^n$. For a monodromic mixed Hodge complex $\cF \in \Dh(\mathcal{X})$, we define the full nearby cycle by
\[
\psi_{f}^{\textnormal{non-mon}}(\mathcal{F}) \coloneqq \colim_n \psi_{f_n}^{\mathrm{uni}, \textnormal{non-mon}}(e_n^* \mathcal{F}).  
\]
The $\mu_n$-action on $\mathcal{X}_n$ defines a $\mu_n$-action on $\psi_{f_n}^{\mathrm{uni}, \textnormal{non-mon}}(e_n^* \mathcal{F})$ commuting with the nilpotent monodromy operator.
It follows from \cite[Proposition 3.33]{tubach-2024-mixed-hodge-modules-on-stacks} that the functor $\psi_{f}^{\textnormal{non-mon}}$ preserves constructible objects.
Therefore we obtain a functor 
\[
    \psi_{f} \colon \Dhc(\mathcal{X}) \to \Dhc^{\mathsf{mon}}(\mathcal{X}_0)
\]
which we call the full nearby cycle functor.
Let $\iota \colon \mathcal{X}_0 \hookrightarrow \mathcal{X}$ be the closed immersion.
Then by construction, there exists a natural transform
\[
 \iota^* \to \psi_{f}.    
\]
We define the vanishing cycle functor 
\[
    \varphi_f \colon \mathsf{MHM}(\mathcal{X})  \to \mathsf{MHM}^{\mathsf{mon}}(\mathcal{X}_0)   
\]
to be the cofibre of the map $\iota^* \to \psi_{f}$.
This is a refinement of the vanishing cycle functor studied e.g. in \cite[\S 8.6]{kashiwara2013sheaves}.
        
        We list the minimal properties of the vanishing cycle functor that we use later.
        \begin{itemize}
            \item If $\mathcal{X}$ is smooth, the support of the complex $\varphi_f(\mathbb{Q}_{\mathcal{X}}[\dim \mathcal{X}])$ is contained in the critical locus $\mathrm{Crit}(f)$.
            
            \item Let $h \colon \mathcal{X} \to \mathcal{Y}$ be a morphism between algebraic stacks and $f \colon \mathcal{Y} \to \mathbb{A}^1$ be a regular function.
                Let $h_0 \colon \mathcal{X}_0 \to \mathcal{Y}_0$ be the restriction of $h$.
                  Then there exist natural transformations
                  \begin{align}
                    \begin{aligned}\label{eq-van-functorial}
                        h_{0, !} \circ \varphi_{f \circ h} & \to \varphi_f \circ h_!, &
                        h_0^* \circ \varphi_f & \to \varphi_{f \circ h} \circ h^*, \\
                        \varphi_f \circ h_* & \to h_{0, *} \circ \varphi_{f \circ h}, &
                        \varphi_{f \circ h} \circ h^! & \to h_0^! \circ \varphi_f. 
                    \end{aligned}
                  \end{align}
                  The maps on the left are invertible if $h$ is proper and representable. The maps on the right are invertible if $h$ is smooth.
            
            \item The vanishing cycle functor commutes with Verdier dual; namely, there exists an equivalence of functors
                  \begin{equation}\label{eq-van-self-dual}
                   \mathbb{D} \circ \varphi_{f} \simeq \varphi_{f} \circ \mathbb{D}.  
                  \end{equation}

            \item Let $\mathcal{X}$ be an algebraic stack with affine diagonal admitting a good moduli space morphism $p \colon \mathcal{X} \to X$.
                  Let $f$ be a function on $X$. Then the natural map
                  \begin{equation}\label{eq-vanishing-good-moduli}
                    \varphi_f \circ p_* \to  p_{0, *} \circ \varphi_{f \circ p}
                  \end{equation}
                  is invertible. This is proved in \cite[Corollary A.2]{Kinjo_Decomp}.

            \item (Dimensional reduction) 
                  Let $E$ be a vector bundle on an algebraic variety $X$ and $\pi \colon E^{\vee} \to X$ be the projection.
                  Take a section $s \in \Gamma(X, E)$, let $Z \subset X$ be the zero locus of $s$ and $s^{\vee} \colon E^{\vee} \to \mathbb{A}^1$ be the corresponding cosection.
                  Then there exists a natural equivalence of functors
                  \begin{equation}\label{eq-local-dimensional-reduction}
                  \mathbb{L}^{\rank E} \otimes  \pi_{!} \varphi_{s^{\vee}} \pi^*  \simeq (Z \hookrightarrow X)_! (Z \hookrightarrow X)^*.  
        \end{equation}
                  In particular, the monodromy operator on the left-hand side is trivial.
                  This is proved in \cite[Theorem A.1]{davison2017critical}.
        \end{itemize}
    \end{para}

    \begin{para}[Geometric description of monodromic mixed Hodge complexes]\label{geo-MHM-para}
        Here we will give an alternative definition of the $\infty$-category of monodromic mixed Hodge complexes following \cite[\S 2]{_Davison_Meinhardt_CoDT} and sketch an argument to show that the two definitions are equivalent.
        We define two full subcategories $\mathcal{B}_{\mathcal{X}}, \mathcal{C}_{\mathcal{X}} \subset \Dhc(\mathcal{X} \times \mathbb{A}^1)$ as follows:
        \begin{itemize}
            \item Objects in $\mathcal{B}_{\mathcal{X}}$ are those $\mathcal{F} \in \Dhc(\mathcal{X} \times \mathbb{A}^1)$ such that for each $x \in \mathcal{X}$ and an integer $n$, 
                the restriction ${}^\mathrm{p}\mathcal{H}^n(\mathcal{F} |_{\{ x \} \times \mathbb{A}^1})$ is constructible with respect to the stratification $\mathbb{A}^1 = (\mathbb{A}^1 \setminus \{ 0 \}) \amalg \{ 0 \}$.
            \item $\mathcal{C}_{\mathcal{X}}$ is the essential image of the pullback functor $\Dhc(\mathcal{X}) \to \Dhc(\mathcal{X} \times \mathbb{A}^1)$. 
        \end{itemize}
      We define the $\infty$-category of geometrically monodromic mixed Hodge complexes by the quotient
      \[
        \Dhc^{\mathsf{geom\textnormal{-}mon}}(\mathcal{X}) \coloneqq \mathcal{B}_{\mathcal{X}} / \mathcal{C}_{\mathcal{X}}.
      \]
      Let $t \colon \mathcal{X} \times \mathbb{A}^1 \to \mathbb{A}^1$ be the projection. Then the monodromic nearby cycle functor $\psi_t$ induces a functor
      \[
        \Dhc^{\mathsf{geom\textnormal{-}mon}}(\mathcal{X}) \to \Dhc^{\mathsf{mon}}(\mathcal{X}).
      \]
     We claim that this is an equivalence of $\infty$-categories. To see this, since this functor is compatible with smooth pullback by construction,
     we may assume that $\mathcal{X}$ is an affine variety. By choosing an embedding into a smooth variety, we may further assume that $\mathcal{X}$ is smooth.
    Then the claim follows from \cite[Theorem 0.2]{saito-2022-monodromic-mixed-Hodge} and the construction of the equivalence therein.

    \end{para}

    \begin{para}[Monodromic tensor product]\label{para-monodromic-tensor}
        We now introduce a monoidal structure on the category $\Dhc^{\mathsf{mon}}(\mathcal{X})$.
        For this, we will use the geometric definition of the $\infty$-category of monodromic mixed Hodge complexes from \Cref{geo-MHM-para}.
        Let $\mathcal{X}$  and $\mathcal{Y}$ be algebraic stacks and take objects $M \in \Dhc^{\mathsf{mon}}(\mathcal{X}) $ and $N \in \Dhc^{\mathsf{mon}}(\mathcal{Y}) $ and denote by $\tilde{M}$ and $\tilde{N}$ the corresponding mixed Hodge complexes on $\mathcal{X} \times \mathbb{A}^1$ and $\mathcal{Y} \times \mathbb{A}^1$ respectively.
        Let $+ \colon \mathbb{A}^1 \times \mathbb{A}^1 \to \mathbb{A}^1$ be the addition map.
        We define the monodromic exterior tensor product of $\tilde{M}$ and $\tilde{N}$  by
        \[
         \tilde{M} \boxtimes^{\mathrm{mon}} \tilde{N} \coloneqq (\id_{\mathcal{X} \times \mathcal{Y}} \times +)_* (\mathrm{pr}_{1, 3}^* \tilde{M} \boxtimes \mathrm{pr}_{2, 4}^* \tilde{N}) \in \Dhc^{\mathsf{geom\textnormal{-}mon}}(\mathcal{X} \times \mathcal{Y}).    
        \]
        We define the monodromic tensor product of $M$ and $N$ to be the one corresponding to $\tilde{M} \boxtimes^{\mathrm{mon}} \tilde{N}$.
        When $\mathcal{X} = \mathcal{Y}$, we set
        \[
        M \otimes^{\mathrm{mon}} N \coloneqq \Delta^*(M \boxtimes^{\mathrm{mon}} N) \in \Dhc^{\mathsf{mon}}(\mathcal{X}).    
        \]
        If there is no possibility of confusion, we simply denote~$\boxtimes^{\mathsf{mon}}$ by~$\boxtimes$ and~$\otimes^{\mathsf{mon}}$ by~$\otimes$.
    \end{para}

    \begin{para}[Thom--Sebastiani theorem]
        Let $\mathcal{X}$ and $\mathcal{Y}$ be algebraic stacks, $f$ and $g$ be regular functions on $\mathcal{X}$ and $\mathcal{Y}$ respectively and $f \boxplus g$ denote the function on $\mathcal{X} \times \mathcal{Y}$ given by $(x, y) \mapsto f(x) + g(y)$.
        For $M \in \Dhc(\mathcal{X})$ and $N \in \Dhc(\mathcal{Y})$, the Thom--Sebastiani theorem provides an isomorphism of monodromic mixed Hodge complexes:
        \begin{equation}\label{eq-local-thom-sebastiani}
         \varphi_{f}(M) \boxtimes^{\mathrm{mon}} \varphi_g(N) \cong \varphi_{f \boxplus g}(M \boxtimes N)|_{\mathcal{X}_0 \times \mathcal{Y}_0}.
        \end{equation}
        This is proved for schemes in Saito's unpublished manuscript \cite{Saito-2010-thom-sebastiani}. See \cite[Appendix A]{brav2012symmetries} for the comparison of Saito's Thom--Sebastiani isomorphism and the Thom--Sebastiani isomorphism at the constructible level given by Massey \cite{masts}.
        One can extend Saito's Thom--Sebastiani isomorphism to stacks by smooth descent.

    \end{para}

    \begin{para}[Half Lefschetz twist]\label{para-half-Lefshetz-twist}
        We will construct the square root of the Lefschetz Hodge structure using the monodromic vanishing cycle functor.
        Consider the function $z^2$ on $\mathbb{A}^1$.
        We set
        \[
            \mathbb{L}^{1/2} \coloneqq \varphi_{z^2}(\mathbb{Q}_{\mathbb{A}^1}) \in \Dhc(\mathrm{pt}). 
        \]
        By using the Thom--Sebastiani theorem \Cref{eq-local-thom-sebastiani} combined with the dimensional reduction isomorphism \Cref{eq-local-dimensional-reduction}, we obtain an isomorphism
        \[
            (\mathbb{L}^{1/2})^{\otimes 2} \cong \varphi_{z^2 \boxplus w^2}(\mathbb{Q}_{\mathbb{A}^2})|_{(0, 0)} = \varphi_{(z + w \sqrt{-1})(z - w \sqrt{-1})}(\mathbb{Q}_{\mathbb{A}^2})|_{(0, 0)} \cong \mathbb{L}. 
        \]
        This isomorphism in particular shows that $\mathbb{L}^{1/2}$ is an invertible object.
        For an integer $k \in \mathbb{Z}$, we set
        \[
        \mathbb{L}^{k/2} \coloneqq (\mathbb{L}^{1/2})^{\otimes k}, \quad \mathbb{Q}(k/2) \coloneqq \mathrm{H}^{-k}(\mathbb{L}^{-k /2}).   
        \]

        Let $\mathcal{X}$ be an equidimensional algebraic stack of dimension $d$ and $j \colon \mathcal{X}_{\mathrm{sm}} \hookrightarrow \mathcal{X}$ be the inclusion from the smooth locus.
        We let $\mathcal{IC}(\mathbb{Q}_{\mathcal{X}_{\mathrm{sm}} }) \in \Dhc(\mathcal{X})$ be the intermediate extension of $\mathbb{Q}_{\mathcal{X}_{\mathrm{sm}}}$.
        We define the intersection complex of $\mathcal{X}$ by 
        \[
          \mathcal{IC}_{\mathcal{X}} \coloneqq \mathbb{L}^{- d / 2} \otimes \mathcal{IC}(\mathbb{Q}_{\mathcal{X}_{\mathrm{sm}} }) \in \mathsf{MMHM}(\mathcal{X}).
        \]
        If $\mathcal{X}=\coprod_{\mathcal{X}'\in\pi_0(\mathcal{X})}\mathcal{X}'$ is a disjoint union of equidimensional algebraic stacks we define $\mathcal{IC}_{\mathcal{X}}\coloneqq \bigoplus_{\mathcal{X}'\in\pi_0(\mathcal{X})} \mathcal{IC}_{\mathcal{X}'}$.
    \end{para}

\section{Donaldson--Thomas mixed Hodge modules}

For a $(-1)$-shifted symplectic stack $\mathcal{X}$ equipped with an orientation, \textcite{brav2012symmetries,ben2015darboux} introduced a certain monodromic mixed Hodge module on $\mathcal{X}$ categorifying the Donaldson--Thomas invariant.
We will briefly recall its construction and prove some properties  which will be used in the later parts of the paper.

\subsection{Orientations}

\begin{para}
    We will discuss orientations for $(-1)$-shifted symplectic stacks.
    These are topological data on a $(-1)$-shifted symplectic stack $\mathcal{X}$ which form an $\mathrm{H}^1(\mathcal{X}, \mu_2)$-torsor (up to a choice of grading) if they exist, and they are necessary for defining the Donaldson--Thomas mixed Hodge modules.
\end{para}

\begin{para}[Determinant functor]
    We will briefly recall the basic properties of the determinant functor of perfect complexes.
    See \textcite{knudsen1976projectivity} for the original reference and \cite[\S 2]{Kinjo_Park_Safronov_CoHA} for a modern treatment.

    Let $\mathcal{X}$ be a derived algebraic stack. 
    Let $\mathrm{Pic}^{\mathrm{gr}}(\mathcal{X})$ be the groupoid of line bundles on $\mathcal{X}$ equipped with locally constant $\mathbb{Z}/ 2 \mathbb{Z}$-gradings.
    We define a symmetric monoidal structure on $\mathrm{Pic}^{\mathrm{gr}}(\mathcal{X})$  as follows:
    \begin{itemize}
        \item For $(\mathcal{L}_1, \alpha_1), (\mathcal{L}_2, \alpha_2) \in \mathrm{Pic}^{\mathrm{gr}}(\mathcal{X})$, we define the monoidal product by
        \[
         (\mathcal{L}_1, \alpha_1) \otimes (\mathcal{L}_2, \alpha_2) \coloneqq (\mathcal{L}_1 \otimes \mathcal{L}_2, \alpha_1 + \alpha_2).
        \]
        \item For $(\mathcal{L}_1, \alpha_1), (\mathcal{L}_2, \alpha_2) \in \mathrm{Pic}^{\mathrm{gr}}(\mathcal{X})$, we define the symmetrizer
        \[
            (\mathcal{L}_1, \alpha_1) \otimes (\mathcal{L}_2, \alpha_2) \cong (\mathcal{L}_2, \alpha_2) \otimes (\mathcal{L}_1, \alpha_1)
        \]
        to be the natural swapping isomorphism of the ungraded line bundles multiplied by  $(-1)^{\alpha_1 \alpha_2}$.
    \end{itemize}
    We warn the reader that the categorical dimension of an object $(\mathcal{L}, \alpha) \in \mathrm{Pic}^{\mathrm{gr}}(\mathcal{X})$, which is defined as the composition
    \[
     \mathcal{O}_{\mathcal{X}} \to (\mathcal{L}, \alpha) \otimes (\mathcal{L}, \alpha) ^{\vee} \cong (\mathcal{L}, \alpha) ^{\vee} \otimes (\mathcal{L}, \alpha)  \to \mathcal{O}_{\mathcal{X}},    
    \]
    is given by $(-1)^{\alpha}$. As a result, particular care must be taken with sign computations when working with graded line bundles.
    If there is no risk of confusion, a graded line bundle $(\mathcal{L}, \alpha)$ will simply be denoted by $\mathcal{L}$.

    We let $\mathsf{Perf}(\mathcal{X})^{\simeq} \subset \mathsf{Perf}(\mathcal{X})$ be the maximal $\infty$-groupoid inside the $\infty$-category of perfect complexes on $\mathcal{X}$.
    We define a symmetric monoidal structure $\oplus$ on $\mathsf{Perf}(\mathcal{X})^{\simeq}$ induced from the Cartesian symmetric monoidal structure on $\mathsf{Perf}(\mathcal{X})$.
    Then the determinant functor is defined as a symmetric monoidal functor
    \[
     \det \colon  (\mathsf{Perf}(\mathcal{X})^{\simeq}, \oplus)  \to (\mathrm{Pic}^{\mathrm{gr}}(\mathcal{X}), \otimes), \quad E \mapsto (\det(E), \rank E).
    \]
    We record the minimal properties of the determinant functor which will be used later in this paper: see \cite[\S 2]{Kinjo_Park_Safronov_CoHA} for details.
    \begin{itemize}
        \item For a fibre sequence $\Delta \colon E_1 \to E_2 \to E_3$ in $\mathsf{Perf}(\mathcal{X})$, there exists an isomorphism
        \begin{equation}\label{eq-det-fibre-seq}
           i(\Delta) \colon \det(E_1) \otimes \det(E_3) \cong \det(E_2).
        \end{equation}
        In particular, there exist natural isomorphisms
        \begin{equation}\label{eq-det-shift}
        \theta_E \colon \det(E[1]) \cong \det(E)^{\vee}. 
        \end{equation}

        \item For $E \in \mathsf{Perf}(\mathcal{X})$, there exists a natural isomorphism
        \begin{equation}\label{eq-det-dual}
            \iota_E\colon \det(E^{\vee}) \cong \det(E)^{\vee}.
        \end{equation}
        Further, the following diagram commutes:
    \begin{equation}\label{eq-det-double-dual}
        \begin{tikzcd}
	{\det(E)} && {\det(E)} \\
	{\det(E^{\vee \vee})} & {\det(E^{\vee})^{\vee}} & {\det(E)^{\vee \vee}.}
	\arrow["{(-1)^{\rank E} \cdot \id}", from=1-1, to=1-3]
	\arrow[from=1-1, to=2-1]
	\arrow[from=1-3, to=2-3]
	\arrow["{\iota_{E^{\vee}}}"', from=2-1, to=2-2]
	\arrow["{\iota_{E}^{\vee}}", from=2-3, to=2-2]
\end{tikzcd}
\end{equation}
    This is proved in \cite[(2.15)]{Kinjo_Park_Safronov_CoHA}.

        \item For a fibre sequence $\Delta \colon E_1 \to E_2 \to E_3$ in $\mathsf{Perf}(\mathcal{X})$, the following diagram commutes:
        \begin{equation}\label{eq-det-shift-fibre-seq}
            \begin{tikzcd}
            {\det(E_1[1]) \otimes \det(E_3[1])} & {\det(E_2[1])} \\
            {\det(E_1)^{\vee} \otimes \det(E_3)^{\vee}} \\
            {(\det(E_3) \otimes \det (E_1))^{\vee}} & {\det(E_{2})^{\vee}.}
            \arrow["{i(\Delta[1])}", from=1-1, to=1-2]
            \arrow["{\theta_{E_1} \otimes \theta_{E_3}}"', from=1-1, to=2-1]
            \arrow["{\theta_{E_2}}", from=1-2, to=3-2]
            \arrow[from=2-1, to=3-1]
            \arrow["{i(\Delta)^{\vee}}", from=3-2, to=3-1]
        \end{tikzcd}
    \end{equation}
    This is proved in \cite[Corollary 2.5]{Kinjo_Park_Safronov_CoHA}. Also, the following diagram commutes:
\begin{equation}\label{eq-det-dual-fibreseq}
    \begin{tikzcd}
	{\det(E_3^{\vee}) \otimes \det(E_1^{\vee})} & {\det(E_2^{\vee})} \\
	{\det(E_3)^{\vee} \otimes \det(E_1)^{\vee}} \\
	{(\det(E_1) \otimes \det(E_3))^{\vee}} & {\det(E_2)^{\vee}.}
	\arrow["{i(\Delta^\vee{})}", from=1-1, to=1-2]
	\arrow["{\iota_{E_3} \otimes \iota_{E_1}}"', from=1-1, to=2-1]
	\arrow["{\iota_{E_2}}", from=1-2, to=3-2]
	\arrow[from=2-1, to=3-1]
	\arrow["{i(\Delta)^{\vee}}", from=3-2, to=3-1]
\end{tikzcd}
\end{equation}
This is proved in \cite[(2.16)]{Kinjo_Park_Safronov_CoHA}.

\item For $E \in \mathsf{Perf}(\mathcal{X})$, the following diagram commutes:
\begin{equation}\label{eq-det-dual-shift}
    \begin{tikzcd}
	{\det(E^{\vee}[-1])} & {\det(E[1]^{\vee})} & {\det(E[1])^{\vee}} \\
	{\det(E^{\vee})^{\vee}} && {\det(E)^{\vee\vee}.}
	\arrow[from=1-1, to=1-2]
	\arrow["{(\theta_{E^{\vee}[-1]})^{\vee}}"', from=1-1, to=2-1]
	\arrow["{\iota_{E[1]}}", from=1-2, to=1-3]
	\arrow["{(\theta_{E}^{\vee})^{-1}}", from=1-3, to=2-3]
	\arrow["{(\iota_{E}^{\vee})^{-1}}"', from=2-1, to=2-3]
\end{tikzcd}\end{equation}
Here, for the left vertical map, we identify $\det(E^{\vee}[-1])$ with its double dual.
This is proved in \cite[(2.18)]{Kinjo_Park_Safronov_CoHA}.
See \cite[Remark 0.4]{symplecticsign} for the choice of the equivalence $E^{\vee}[-1] \simeq E[1]^{\vee}$.
\end{itemize}
    
\end{para}

\begin{para}[Orientations for \texorpdfstring{$(-1)$}{TEXT}-shifted symplectic stacks]\label{para-orientation-def}
    Let $\mathcal{X}$ be a derived algebraic stack locally finitely presented over $\mathbb{C}$. We define the canonical bundle of $\mathcal{X}$ by
    \[
     K_{\mathcal{X}} \coloneqq \det(\mathbb{L}_{\mathcal{X}}) \in \mathrm{Pic}^{\mathrm{gr}} (\mathcal{X}).  
    \]

    Now assume that $\mathcal{X}$ is a $(-1)$-shifted symplectic stack. Then an \emph{orientation} for $\mathcal{X}$ is a pair of a graded line bundle $\mathcal{L} \in \mathrm{Pic}^{\mathrm{gr}} (\mathcal{X})$ with an isomorphism
    \[
    o \colon \mathcal{L}^{\otimes 2} \cong K_{\mathcal{X}}.    
    \]
    If there is no risk of confusion, we denote a pair $(\mathcal{L}, o)$ simply by $o$.
\end{para}

\begin{para}[Standard orientation for critical loci]\label{para-standard-ori}
    Let $\mathcal{Y}$ be  derived algebraic stack locally finitely presented over $\mathbb{C}$ and $f \colon \mathcal{Y} \to \mathbb{A}^1$ be a regular function.
    Set $\mathcal{X} = \mathrm{Crit}(f)$ and equip it with the standard $(-1)$-shifted symplectic structure $\omega_{\mathcal{X}}$.
    Let $\iota \colon \mathcal{X} \to \mathcal{Y}$ be the canonical map. Then there exists a fibre sequence
    \[
    \Delta \colon \iota^* \mathbb{L}_{\mathcal{Y}}  \to  \mathbb{L}_{\mathcal{X }} \to \iota^* \mathbb{T}_{\mathcal{Y}}[1]
    \]
    where the former map is the canonical one and the latter map is given by the composition
    \[
        \mathbb{L}_{\mathcal{X }} \xrightarrow[\simeq]{(\cdot \omega_{\mathcal{X}}[1])^{-1}} \mathbb{T}_{\mathcal{X}}[1] \to \iota^* \mathbb{T}_{\mathcal{Y}}[1].
    \]
    By using \Cref{eq-det-fibre-seq}, \Cref{eq-det-shift} and \Cref{eq-det-dual}, we obtain an isomorphism
    \[
   \bar{o}^{\mathrm{sta}} \colon \iota^* K_{\mathcal{Y}}^{\otimes 2} \cong K_{\mathcal{X}}. 
    \]
    The \emph{standard orientation} $o^{\mathrm{sta}}_{\mathcal{X}}$ for $\mathcal{X}$ is defined as the pair $(\iota^* K_{\mathcal{Y}}, (-1)^{\vdim \mathcal{Y}} \cdot \bar{o}^{\mathrm{sta}}_{\mathcal{X}})$. Note that the standard orientation is called the canonical orientation in \cite{Kinjo_Park_Safronov_CoHA}.

\end{para}

\begin{para}\label{para-any-orientation-is-locally-standard}
    Let $\mathcal{X}$ be a $(-1)$-shifted symplectic stack which has affine diagonal and admits a good moduli space $\mathcal{X} \to X$.
    Further assume that $\mathcal{X}$ is equipped with an orientation $o \colon \mathcal{L}^{\otimes 2} \cong K_{\mathcal{X}}$.
    We claim that $o$ is equivalent to the standard orientation for some critical chart \'etale locally over $X$.
    
    To see this, using \Cref{thm-Darboux}, we may assume that $\mathcal{X} = \mathrm{Crit}(f \colon V / G \to \mathbb{A}^1)$ for a reductive group $G$, a smooth affine scheme $V$ acted on by $G$ with a fixed point $v \in V$ and a $G$-invariant function $f$ on $V$.
    Using the orientation $o$ and the standard orientation for $\mathcal{X}$, we see that the line bundle
    \[
     \mathcal{L} |_{\{ v \} / G} \otimes   K_{V/G} |_{\{ v \} / G} ^{\vee}
    \]
    squares to a trivial line bundle on $\mathrm{B} G$. Therefore it corresponds to some sign representation $\rho \colon G \to \mu_2$.
    Consider $\mathbb{A}^1$ equipped with a $G$-representation induced by $\rho$, and let $q$ be a non-degenerate $G$-invariant quadratic function on $\mathbb{A}^1$.
    Then by replacing $(V, f)$ with $(V^+, f^+) \coloneqq (V \times \mathbb{A}^1, f \boxplus q)$, we may assume that $\mathcal{L} |_{\{ v \} / G}$ is isomorphic to $K_{V / G} |_{\{ v \} / G}$ as an ungraded line bundle.
    Further, possibly replacing $(V^+, f^+)$ with $(V^+ \times \mathbb{A}^1, f^+ \boxplus x^2)$ with trivial $G$-action on the $\mathbb{A}^1$-factor, we may assume that the isomorphism preserves the grading.
    By using \cite[Theorem 10.3]{_Alper_GoodmodulispacesforArtinstacks} and \cite[Lemma 6.8]{bellamy2016symplectic}, we see that $\mathcal{L}$ is isomorphic to $K_{V / G}$ around $v$ \'etale locally over the good moduli space.
    Since any function on a good moduli stack admits a square root \'etale locally over the good moduli space, we conclude that there is an isomorphism of orientations over an \'etale cover.

\end{para}

\begin{para}[Localized orientation]\label{para-localized-orientation}
    Let $\mathcal{X}$ be a $(-1)$-shifted symplectic stack and $(F, \alpha) \in \mathsf{Face}^{\mathrm{nd}}(\mathcal{X})$ be a non-degenerate face.
    We will show that, as long as $\mathcal{X}$ is numerically symmetric, an orientation  $o \colon \mathcal{L}^{\otimes 2} \cong K_{\mathcal{X}}$ induces an orientation on $\mathcal{X}_{\alpha}$, which will be denoted by $\alpha^{\star} o$.
    Further, we will show that the orientation $\alpha^{\star} o$ is $\mathrm{Aut}(\alpha)$-equivariant.
\end{para}

\begin{para}\label{para-localized-orientation-continued}
    Let $\mathcal{X}$ be a $(-1)$-shifted symplectic stack with finite cotangent weights.
    Take a non-degenerate face $(F, \alpha) \in \mathsf{Face}^{\mathrm{nd}}(\mathcal{X})$ and a cone $\sigma \subset F$ which is a chamber in the cotangent arrangement.
    We let 
    \[
        (\mathrm{tot}_{\alpha}^* \mathbb{L}_{\mathcal{X}})^{\sigma, +}  \subset \mathrm{tot}_{\alpha}^* \mathbb{L}_{\mathcal{X}}, \quad (\mathrm{tot}_{\alpha}^* \mathbb{L}_{\mathcal{X}})^{\sigma, -}  \subset \mathrm{tot}_{\alpha}^* \mathbb{L}_{\mathcal{X}},
    \]
    denote the direct summands which have positive and negative weights with respect to the cone $\sigma$ respectively.
    The $(-1)$-shifted symplectic structure induces an equivalence
    \[
       \Phi_{\sigma} \colon (\mathrm{tot}_{\alpha}^* \mathbb{L}_{\mathcal{X}})^{\sigma, +, \vee}[1] \simeq  (\mathrm{tot}_{\alpha}^* \mathbb{L}_{\mathcal{X}})^{\sigma, -}.
    \]
    By taking the determinant, we obtain an isomorphism
    \[
    \phi_{\sigma} \colon  \det \left((\mathrm{tot}_{\alpha}^* \mathbb{L}_{\mathcal{X}})^{\sigma, +}  \right) \cong  \det \left((\mathrm{tot}_{\alpha}^* \mathbb{L}_{\mathcal{X}})^{\sigma, -}  \right).
    \]
    The following technical lemma will be important later:
    
\end{para}

\begin{lemma}\label{lem-phi-indep-segment}
    Assume that $\mathcal{X}$ is numerically symmetric.
    Let $- \sigma$ be the cone opposite to $\sigma$. Then under the identification $(\mathrm{tot}_{\alpha}^* \mathbb{L}_{\mathcal{X}})^{(-\sigma), +} \simeq (\mathrm{tot}_{\alpha}^* \mathbb{L}_{\mathcal{X}})^{\sigma, -}$ and $(\mathrm{tot}_{\alpha}^* \mathbb{L}_{\mathcal{X}})^{(-\sigma), -} \simeq (\mathrm{tot}_{\alpha}^* \mathbb{L}_{\mathcal{X}})^{\sigma, +}$,
    we have $\phi_{- \sigma} = \phi_{\sigma}^{-1}$.
\end{lemma}

\begin{proof}
    First, by generalities on $(-1)$-shifted symplectic structures, we have $\Phi_{\sigma}^{\vee} = \Phi_{- \sigma}[-1]$: see e.g., \cite[Lemma 0.2, Remark 0.4]{symplecticsign}.
    To simplify the notation, we will write $V_{\sigma} \coloneqq (\mathrm{tot}_{\alpha}^* \mathbb{L}_{\mathcal{X}})^{\sigma, +} $ and $V_{- \sigma} \coloneqq (\mathrm{tot}_{\alpha}^* \mathbb{L}_{\mathcal{X}})^{(-\sigma), +} $.
    Now consider the following diagram:
\[\begin{tikzcd}
	{\det(V_{\sigma})} & {\det(V_{- \sigma}^{\vee}[1])} & {\det(V_{ \sigma}^{\vee}[1]^{\vee}[1])} & {\det(V_{ \sigma}^{\vee\vee})} & {\det(V_{\sigma})} \\
	& {\det(V_{-\sigma})} & {\det(V_{\sigma}^{\vee}[1])} & {\det(V_{ \sigma})^{\vee\vee}} & {\det(V_{\sigma}).}
	\arrow["{\Phi_{\sigma}}"', from=1-1, to=1-2]
	\arrow[""{name=0, anchor=center, inner sep=0}, "{\id }", curve={height=-30pt}, from=1-1, to=1-5]
	\arrow[""{name=1, anchor=center, inner sep=0}, "{{\phi_{\sigma}}}"', curve={height=6pt}, from=1-1, to=2-2]
	\arrow[""{name=2, anchor=center, inner sep=0}, "{\Phi_{-\sigma}}"', from=1-2, to=1-3]
    \arrow["\substack{ \text{\Cref{eq-det-shift}} \\ \text{\Cref{eq-det-dual}} }", from=1-2, to=2-2]
	\arrow[""{name=3, anchor=center, inner sep=0}, from=1-3, to=1-4]
    \arrow["\substack{ \text{\Cref{eq-det-shift}} \\ \text{\Cref{eq-det-dual}} }", from=1-3, to=2-3]
	\arrow[from=1-4, to=2-4]
	\arrow[""{name=4, anchor=center, inner sep=0}, from=1-5, to=1-4]
	\arrow[Rightarrow, no head, from=1-5, to=2-5]
	\arrow[""{name=5, anchor=center, inner sep=0}, "{\Phi_{- \sigma}}"', from=2-2, to=2-3]
	\arrow[""{name=6, anchor=center, inner sep=0}, "{{\phi_{- \sigma}}}"', curve={height=30pt}, from=2-2, to=2-5]
    \arrow[""{name=7, anchor=center, inner sep=0}, "\substack{ \text{\Cref{eq-det-shift}} \\ \text{\Cref{eq-det-dual}} }"', from=2-3, to=2-4]
	\arrow[""{name=8, anchor=center, inner sep=0}, from=2-5, to=2-4]
	\arrow["{{(A)}}"{description}, draw=none, from=0, to=1-3]
	\arrow["{{(B)}}"{description}, draw=none, from=1-2, to=1]
	\arrow["{{(C)}}"{description}, shift left=3, draw=none, from=2, to=5]
	\arrow["{{(D)}}"{description}, draw=none, from=3, to=7]
	\arrow["{{(E)}}"{description}, draw=none, from=4, to=8]
	\arrow["{{(F)}}"{description, pos=0.2}, shift right=2, draw=none, from=2-3, to=6]
\end{tikzcd}\]
The diagram $(A)$ commutes by the identity $\Phi_{\sigma}^{\vee} = \Phi_{- \sigma}[-1]$.
The diagrams $(B)$ and $(F)$ commute by the definition of $\phi_{\sigma}$ and $\phi_{- \sigma}$. The diagram $(C)$ commutes by naturality.
The diagram $(D)$ commutes by using \Cref{eq-det-dual-shift}. The diagram $(E)$ commutes by \Cref{eq-det-double-dual} and the equality ${\rank V_{\sigma}} = 0$ which follows from \Cref{eq-numerical-symmetric-implies-zero-dimension}. Therefore we obtain the desired claim.
\end{proof}

\begin{para}\label{para-localized-orientation-definition}
    We adopt the notation from \Cref{para-localized-orientation-continued}. Assume that we are given an orientation $o \colon \mathcal{L}^{\otimes 2} \cong K_{\mathcal{X}}$.
    Define a localized orientation $\sigma^{\star} o$ for $\mathcal{X}_{\alpha}$ by the composition
    \begin{align*}
       &\left(\mathrm{tot}_{\alpha}^* \mathcal{L} \otimes \det \left((\mathrm{tot}_{\alpha}^* \mathbb{L}_{\mathcal{X}})^{\sigma, +}  \right)^{\vee} \right)^{\otimes 2}  \\
        &\xrightarrow[]{\mathrm{tot}_{\alpha}^*o \otimes \id} \mathrm{tot}_{\alpha}^* K_{\mathcal{X}} \otimes  \det \left((\mathrm{tot}_{\alpha}^* \mathbb{L}_{\mathcal{X}})^{\sigma, +}  \right)^{\vee} \otimes \det \left((\mathrm{tot}_{\alpha}^* \mathbb{L}_{\mathcal{X}})^{\sigma, +}  \right)^{\vee} \\
        &\xrightarrow[]{\id \otimes (\phi_{\sigma}^{\vee})^{-1} \otimes \id} \mathrm{tot}_{\alpha}^* K_{\mathcal{X}} \otimes \det\left((\mathrm{tot}_{\alpha}^* \mathbb{L}_{\mathcal{X}})^{\sigma, -}  \right)^{\vee} \otimes \det \left((\mathrm{tot}_{\alpha}^* \mathbb{L}_{\mathcal{X}})^{\sigma, +}  \right)^{\vee} \\
        & \cong K_{\mathcal{X}_{\alpha}}
    \end{align*}
    where the final isomorphism is constructed using the decomposition
     \[
        \mathrm{tot}_{\alpha}^* \mathbb{L}_{\mathcal{X}} = \mathbb{L}_{\mathcal{X_{\alpha}}} \oplus (\mathrm{tot}_{\alpha}^* \mathbb{L}_{\mathcal{X}})^{\sigma, +} \oplus (\mathrm{tot}_{\alpha}^* \mathbb{L}_{\mathcal{X}})^{\sigma, -}.
    \]
    We now claim that the localized orientation does not depend on the choice of a chamber as long as $\mathcal{X}$ is numerically symmetric:
\end{para}

\begin{lemma}\label{lem-indep-localized-ori-on-segment}
    Let $\mathcal{X}$ be a numerically symmetric $(-1)$-shifted symplectic stack equipped with an orientation $o \colon \mathcal{L}^{\otimes 2} \cong K_{\mathcal{X}}$.
    Let $(F, \alpha) \in \mathsf{Face}^{\mathrm{nd}}(\mathcal{X})$ be a non-degenerate face. Then for cones $\sigma, \sigma' \subset F$ which are chambers in the cotangent arrangement, there exists a natural isomorphism
    \begin{equation}\label{eq-localized-ori-indep-on-segment}
        a_{\sigma, \sigma'} \colon \sigma^{\star} o \cong {\sigma'}^{\star} o
    \end{equation}
     satisfying the following properties:
    \begin{enumerate}
        \item The identity $a_{\sigma, \sigma} = \id$ holds for any $\sigma \subset F$.
        \item The identity $a_{\sigma, \sigma'} = a_{\sigma', \sigma}^{-1}$ holds for any $\sigma, \sigma' \subset F$.
        \item The identity $a_{\sigma', \sigma''} \circ a_{\sigma, \sigma'} = a_{\sigma, \sigma''}$ holds for any $\sigma, \sigma', \sigma'' \subset F$. \label{item-trans-local-ori}
    \end{enumerate}
\end{lemma}

\begin{proof}
    Let $\sigma, \sigma' \subset F$ be chambers in the cotangent arrangement.
    To simplify the notation, we set 
    \[
        V_{\sigma}^{+} \coloneqq (\mathrm{tot}_{\alpha}^* \mathbb{L}_{\mathcal{X}})^{\sigma, +} , \quad V_{\sigma}^{-} \coloneqq (\mathrm{tot}_{\alpha}^* \mathbb{L}_{\mathcal{X}})^{\sigma, -}
    \]
    and similarly for $\sigma'$.
    Define a direct summand
    \begin{equation}\label{eq-notation_V_sigma_sigma'}
      V_{\sigma, \sigma'}^{+, +} \coloneqq  (\mathrm{tot}_{\alpha}^* \mathbb{L}_{\mathcal{X}})^{(\sigma, \sigma'), +, +} \subset    \mathrm{tot}_{\alpha}^* \mathbb{L}_{\mathcal{X}}
    \end{equation}
    which has positive weights with respect to $\sigma$ and $\sigma'$, and other direct summands corresponding to the pair of signs $(+, -)$, $(-, +)$ and $(-, -)$ in a similar manner.
    By definition, we have natural isomorphisms
    \begin{align*}
        V_{\sigma}^+ \cong V_{\sigma, \sigma'}^{+, +} \oplus V_{\sigma, \sigma'}^{+, -}, \\
        V_{\sigma'}^{+} \cong V_{\sigma, \sigma'}^{+, +} \oplus V_{\sigma, \sigma'}^{-, +}.
    \end{align*}
    The $(-1)$-shifted symplectic structure on $\mathcal{X}$ induces a natural equivalence
    \[
       \Phi_{\sigma, \sigma'}^{+, -} \colon  V_{\sigma, \sigma'}^{+, -, \vee} [1] \simeq    V_{\sigma, \sigma'}^{-, +}
    \]
    which induces an isomorphism
    \begin{equation}\label{eq-phi_sigma_sigma'_def}
    \phi_{\sigma, \sigma'}^{+, -} \colon \det \left(  V_{\sigma, \sigma'}^{+, -} \right)  \cong \det\left(  V_{\sigma, \sigma'}^{-, +} \right)
    \end{equation}
    hence an isomorphism
    \begin{equation}\label{eq-a_sigma_sigma'}
     a_{\sigma, \sigma'} \colon  \det ( V_{\sigma}^+ )  \cong \det ( V_{\sigma'}^+).
    \end{equation}
    We claim that it induces an isomorphism $\sigma^{\star} o \cong {\sigma'}^{\star} o$.
    To prove this, it is enough to show that the middle rectangle of the following diagram commutes:
\[\begin{tikzcd}
	\begin{array}{c} \substack{\displaystyle\det(V_{\sigma, \sigma'}^{+, +}) \otimes \det(V_{\sigma, \sigma'}^{+, -})  \\ \displaystyle \otimes \det(V_{\sigma, \sigma'}^{+, +}) \otimes \det(V_{\sigma, \sigma'}^{+, -}) } \end{array} && \begin{array}{c} \substack{\displaystyle\det(V_{\sigma, \sigma'}^{+, +}) \otimes \det(V_{\sigma, \sigma'}^{-, +})  \\ \displaystyle \otimes \det(V_{\sigma, \sigma'}^{+, +}) \otimes \det(V_{\sigma, \sigma'}^{-, +}) } \end{array} \\
	{\det(V_{\sigma}^+)^{\otimes 2}} && {\det(V_{\sigma'}^+)^{\otimes 2}} \\
	{\det(V_{\sigma}^+) \otimes \det(V_{\sigma}^-)} & {\mathrm{tot}_{\alpha}^* K_{\mathcal{X}} \otimes K_{\mathcal{X}_{\alpha}}^{\vee}} & {\det(V_{\sigma'}^+) \otimes \det(V_{\sigma'}^-)} \\
	\begin{array}{c} \substack{\displaystyle\det(V_{\sigma, \sigma'}^{+, +}) \otimes \det(V_{\sigma, \sigma'}^{+, -})  \\ \displaystyle \otimes \det(V_{\sigma, \sigma'}^{-, +}) \otimes \det(V_{\sigma, \sigma'}^{-, -}) } \end{array} && \begin{array}{c} \substack{\displaystyle\det(V_{\sigma, \sigma'}^{+, +}) \otimes \det(V_{\sigma, \sigma'}^{-, +})  \\ \displaystyle \otimes \det(V_{\sigma, \sigma'}^{+, -}) \otimes \det(V_{\sigma, \sigma'}^{-, -}). } \end{array}
	\arrow[""{name=0, anchor=center, inner sep=0}, "{\id \otimes \phi_{\sigma, \sigma'}^{+, -} \otimes \id \otimes \phi_{\sigma, \sigma'}^{+, -} }", from=1-1, to=1-3]
	\arrow[from=2-1, to=1-1]
	\arrow[""{name=1, anchor=center, inner sep=0}, "{a_{\sigma, \sigma'}^{\otimes 2}}"', from=2-1, to=2-3]
	\arrow["{\id \otimes \phi_{\sigma}}", from=2-1, to=3-1]
	\arrow[from=2-3, to=1-3]
	\arrow["{\id \otimes \phi_{\sigma'}}"', from=2-3, to=3-3]
	\arrow[from=3-1, to=3-2]
	\arrow[from=3-1, to=4-1]
	\arrow[from=3-3, to=3-2]
	\arrow[from=3-3, to=4-3]
	\arrow[""{name=2, anchor=center, inner sep=0}, "{\id \otimes \mathrm{swap} \otimes \id}"', from=4-1, to=4-3]
	\arrow["{(A)}"{description}, draw=none, from=0, to=1]
	\arrow["{(B)}"{description}, draw=none, from=3-2, to=2]
\end{tikzcd}\]
Note that the square $(A)$ commutes by the definition of the map $a_{\sigma, \sigma'}$ and the diagram $(B)$ obviously commutes.
Therefore it is enough to prove that the outer square commutes. Equivalently, it is enough to show that the following diagram commutes:
\[\begin{tikzcd}
	{\det(V_{\sigma, \sigma'}^{+, -})   \otimes \det(V_{\sigma, \sigma'}^{+, -})} &[2cm] {\det(V_{\sigma, \sigma'}^{-, +})   \otimes \det(V_{\sigma, \sigma'}^{-, +})} \\
	{\det(V_{\sigma, \sigma'}^{+, -})   \otimes \det(V_{\sigma, \sigma'}^{-, +})} & {\det(V_{\sigma, \sigma'}^{-, +})   \otimes \det(V_{\sigma, \sigma'}^{+, -}).}
	\arrow["{\phi_{\sigma, \sigma'}^{+, -} \otimes \phi_{\sigma, \sigma'}^{+, -} }", from=1-1, to=1-2]
	\arrow["{\id \otimes \phi_{\sigma, \sigma'}^{+, -}}"', from=1-1, to=2-1]
	\arrow["{\id \otimes \phi_{\sigma', \sigma}^{+, -}}", from=1-2, to=2-2]
	\arrow["{\mathrm{swap}}"', from=2-1, to=2-2]
\end{tikzcd}\]
Firstly, arguing as in the proof of \Cref{lem-phi-indep-segment}, one obtains an identity
\begin{equation}\label{eq-Phi_sigma_sigma'equalPhi_sigma'sigma}
    \Phi_{\sigma, \sigma'} = \Phi_{\sigma', \sigma}^{\vee}[1]
\end{equation}
and hence an identity
$\phi_{\sigma', \sigma}^{+, -} = (\phi_{\sigma, \sigma'}^{+, -})^{-1}$.
Therefore it is enough to show that the swapping isomorphism
\[
 \mathrm{swap} \colon \det(V_{\sigma, \sigma'}^{+, -})   \otimes \det(V_{\sigma, \sigma'}^{+, -}) \cong \det(V_{\sigma, \sigma'}^{+, -})   \otimes \det(V_{\sigma, \sigma'}^{+, -})    
\]
is the identity map. This follows from $\rank V_{\sigma, \sigma'}^{+, -} = 0$ which is a consequence of the assumption that $\mathcal{X}$ is numerically symmetric.

The identity $a_{\sigma, \sigma'} = a_{\sigma', \sigma}^{-1}$ is a straightforward consequence of the identity $\phi_{\sigma', \sigma}^{+, -} = (\phi_{\sigma, \sigma'}^{+, -})^{-1}$.
The identity $a_{\sigma', \sigma''} \circ a_{\sigma, \sigma'} = a_{\sigma, \sigma''}$  can be proved analogously to \Cref{lem-phi-indep-segment}.
\end{proof}

\begin{para}[Equivariant structure on localized orientation]\label{para-localize-orientation-is-equivariant}
    Let $\mathcal{X}$ be a numerically symmetric $(-1)$-shifted symplectic stack equipped with an orientation $o$.
    For a non-degenerate face $(F, \alpha) \in \mathsf{Face}^{\mathrm{nd}}(\mathcal{X})$, we define an orientation $\alpha^{\star} o$ for $\mathcal{X}_{\alpha}$ to be $\sigma^{\star} o$ for a cone $\sigma \subset F$ which is a chamber in the cotangent arrangement.
    By Lemma \ref{lem-indep-localized-ori-on-segment}, the orientation $\alpha^{\star} o$ does not depend on the choice of a chamber up to a unique isomorphism.
    
    For an element $g \in \mathrm{Aut}(\alpha)$, let $\rho_g \colon \mathcal{X}_{\alpha} \xrightarrow[]{\cong} \mathcal{X}_{\alpha}$ be the isomorphism induced by $g$.
    Then, for a chamber $\sigma \subset F$, we have a natural isomorphism $\rho_g^{\star} (g(\sigma)^{\star} o) \cong \sigma^{\star} o$, which can be rewritten as $\rho_g^{\star} (\alpha^{\star} o) \cong \alpha^{\star} o$.
    Further, these isomorphisms respect the product structure in $\mathrm{Aut}(\alpha)$ by \Cref{item-trans-local-ori} in \Cref{lem-indep-localized-ori-on-segment}.
    In particular, $\alpha^{\star} o$ is naturally equipped with an $\mathrm{Aut}(\alpha)$-equivariant structure.
\end{para}

\begin{para}[Localizing the standard orientation]\label{para-locallize-can-ori}
    Let $\mathcal{Y}$ be a quasi-smooth derived algebraic stack and $f \colon \mathcal{Y} \to \mathbb{A}^1$ be a function.
    Form the critical locus $\mathcal{X} \coloneqq \mathrm{Crit}(f)$  and equip it with the standard $(-1)$-shifted symplectic structure and orientation $o_{\mathcal{X}}^{\mathrm{sta}} \colon \iota^* K_{\mathcal{Y}}^{\otimes 2} \cong K_{\mathcal{X}}$ introduced in \Cref{para-standard-ori}, where $\iota \colon \mathcal{X} \to \mathcal{Y}$ denotes the canonical map.
    Let $(F, \alpha) \in \mathsf{Face}^{\mathrm{nd}}(\mathcal{Y})$ be a non-degenerate face and $\sigma \subset F$ be a chamber in the cotangent arrangement.
    Set $f_{\alpha} \coloneqq f \circ \mathrm{tot}_{\alpha}$. As we have seen in \Cref{para-localized-orientation}, we have a natural equivalence of $(-1)$-shifted symplectic stacks
    \[
     \mathcal{X}_{\alpha} \cong \mathrm{Crit}(f_{\alpha})    
    \]
    and we define the standard orientation $o_{\mathcal{X}_{\alpha}}^{\mathrm{sta}} \colon \iota_{\alpha}^* K_{\mathcal{Y}_{\alpha}}^{\otimes 2} \cong K_{\mathcal{X}_{\alpha}}$ using this critical chart description.
    Now consider the following fibre sequence
    \begin{equation}\label{eq-critical-fibre-seq}
    (\mathrm{tot}_{\alpha}^* \iota^* \mathbb{L}_{\mathcal{Y}})^{\sigma, +}  \to  (\mathrm{tot}_{\alpha}^* \mathbb{L}_{\mathcal{X}})^{\sigma, +} \to   (\mathrm{tot}_{\alpha}^* \iota^* \mathbb{L}_{\mathcal{Y}})^{\sigma, -, \vee}[1]
    \end{equation}
    where the latter map is given by the composition 
    \[
        (\mathrm{tot}_{\alpha}^* \mathbb{L}_{\mathcal{X}})^{\sigma, +} \simeq (\mathrm{tot}_{\alpha}^* \mathbb{L}_{\mathcal{X}})^{\sigma, -, \vee}[1] \to  (\mathrm{tot}_{\alpha}^* \iota^* \mathbb{L}_{\mathcal{Y}})^{\sigma, -, \vee}[1].    
    \]
    It induces an isomorphism
    \begin{equation}\label{eq-det-crit-pm}
    b_{\sigma} \colon \det( (\mathrm{tot}_{\alpha}^* \mathbb{L}_{\mathcal{X}})^{\sigma, +}) \cong  \iota_{\alpha}^*  \det\left( (\mathrm{tot}_{\alpha}^*  \mathbb{L}_{\mathcal{Y}})^{\sigma, \pm}  \right) 
    \end{equation}
    where we set $(\mathrm{tot}_{\alpha}^*  \mathbb{L}_{\mathcal{Y}})^{\sigma, \pm} \coloneqq (\mathrm{tot}_{\alpha}^*  \mathbb{L}_{\mathcal{Y}})^{\sigma, +} \oplus (\mathrm{tot}_{\alpha}^*  \mathbb{L}_{\mathcal{Y}})^{\sigma, -}$.
    This induces an isomorphism
    \[
     \mathrm{tot}_{\alpha}^* \iota^* K_{\mathcal{Y}} \otimes \det( (\mathrm{tot}_{\alpha}^* \mathbb{L}_{\mathcal{X}})^{\sigma, +})^{\vee} \cong \iota_{\alpha}^* K_{\mathcal{Y}_{\alpha}}.
    \]
    One can show that this isomorphism induces an isomorphism of orientations
    \begin{equation}\label{eq-localize-standard-orientaion}
    c_{\sigma} \colon \sigma^{\star} o^{\mathrm{sta}}_{\mathcal{X}} \cong     o^{\mathrm{sta}}_{\mathcal{X}_{\alpha}},
    \end{equation}
    see \cite[(6.75)]{Kinjo_Park_Safronov_CoHA} for a related statement.

    The following technical lemma will be important for a supercommutativity property of the cohomological Hall induction in \Cref{para-supercommutative-1-shifted-symplectic}:

\end{para}

\begin{lemma}\label{lem-ori-localize-indep-segment-up-to-sign}
    We adopt the notation from \Cref{para-locallize-can-ori} and assume further that $\mathcal{Y}$ is numerically symmetric.
    Then for any choice of chambers $\sigma, \sigma' \subset F$ in the cotangent arrangement, the following diagram commutes:
\begin{equation}\label{eq-ori-localize-indep-segment-up-to-sign}
    \begin{tikzcd}
    {\sigma ^{\star} o_{\mathcal{X}}^{\mathrm{sta}}} & {o^{\mathrm{sta}}_{\mathcal{X}_{\alpha}}} \\
    {{\sigma'} ^{\star} o_{\mathcal{X}}^{\mathrm{sta}}} & {o^{\mathrm{sta}}_{\mathcal{X}_{\alpha}}.}
    \arrow["{c_{\sigma}}", from=1-1, to=1-2]
    \arrow["{(-1)^{d(\sigma, \sigma')} \cdot a_{\sigma, \sigma'}}"', from=1-1, to=2-1]
    \arrow[Rightarrow, no head, from=1-2, to=2-2]
    \arrow["{c_{\sigma'}}"', from=2-1, to=2-2]
\end{tikzcd}
\end{equation}

\end{lemma}

\begin{proof}
    It is enough to prove that the following diagram of the underlying line bundles commute:
\[\begin{tikzcd}
	{\mathrm{tot}_{\alpha}^*  \iota^*K_{\mathcal{Y}} \otimes \det(\mathrm{tot}_{\alpha}^* \mathbb{L}_{\mathcal{X}}^{\sigma+})^{\vee}} & {\mathrm{tot}_{\alpha}^*  \iota^*K_{\mathcal{Y}} \otimes \iota_{\alpha}^* \det(\mathrm{tot}_{\alpha}^* \mathbb{L}_{\mathcal{Y}}^{\sigma \pm}) ^{\vee}} & {\iota_{\alpha}^*K_{\mathcal{Y}_{\alpha}}} \\
	{\mathrm{tot}_{\alpha}^*  \iota^*K_{\mathcal{Y}} \otimes \det(\mathrm{tot}_{\alpha}^* \mathbb{L}_{\mathcal{X}}^{\sigma'+}) ^{\vee}} & {\mathrm{tot}_{\alpha}^*  \iota^*K_{\mathcal{Y}} \otimes \iota_{\alpha}^* \det(\mathrm{tot}_{\alpha}^* \mathbb{L}_{\mathcal{Y}}^{\sigma \pm}) ^{\vee}} & {\iota_{\alpha}^*K_{\mathcal{Y}_{\alpha}}.}
	\arrow["\text{\Cref{eq-det-crit-pm}}", from=1-1, to=1-2]
	\arrow["{(-1)^{d(\sigma, \sigma')}\cdot \text{\Cref{eq-a_sigma_sigma'}}}"', from=1-1, to=2-1]
	\arrow[from=1-2, to=1-3]
	\arrow[Rightarrow, no head, from=1-3, to=2-3]
	\arrow["\text{\Cref{eq-det-crit-pm}}"', from=2-1, to=2-2]
	\arrow[from=2-2, to=2-3]
\end{tikzcd}\]
This is equivalent to proving the commutativity of the following diagram:
\begin{equation}\label{eq-sigma-indep-canonical}
    \begin{tikzcd}
	{\det(\mathrm{tot}_{\alpha}^* \mathbb{L}_{\mathcal{X}}^{\sigma+})} & {\iota_{\alpha}^* \det(\mathrm{tot}_{\alpha}^* \mathbb{L}_{\mathcal{Y}}^{\sigma \pm})} \\
	{\det(\mathrm{tot}_{\alpha}^* \mathbb{L}_{\mathcal{X}}^{\sigma'+})} & {\iota_{\alpha}^* \det(\mathrm{tot}_{\alpha}^* \mathbb{L}_{\mathcal{Y}}^{\sigma \pm}).}
	\arrow["\text{\Cref{eq-det-crit-pm}}", from=1-1, to=1-2]
	\arrow["{(-1)^{d(\sigma, \sigma')}\cdot \text{\Cref{eq-a_sigma_sigma'}}}"', from=1-1, to=2-1]
	\arrow[Rightarrow, no head, from=1-2, to=2-2]
	\arrow["\text{\Cref{eq-det-crit-pm}}"', from=2-1, to=2-2]
\end{tikzcd}
\end{equation}
Now to simplify the notation, as in \Cref{eq-notation_V_sigma_sigma'}, we set
\[
 V_{\sigma, \sigma'}^{+, +} \coloneqq  (\mathrm{tot}_{\alpha}^* \mathbb{L}_{\mathcal{X}})^{(\sigma, \sigma'), +, +}, \quad W_{\sigma, \sigma'}^{+, +} \coloneqq  (\mathrm{tot}_{\alpha}^* \mathbb{L}_{\mathcal{Y}})^{(\sigma, \sigma'), +, +}
\]
and similarly for other choice of signs.
Consider the following fibre sequences defined similarly to \Cref{eq-critical-fibre-seq}:
\begin{align*}
    \Delta &\colon W_{\sigma, \sigma'}^{+, -} \to V_{\sigma, \sigma'}^{+, -} \to W_{\sigma, \sigma'}^{-, +, \vee}[1], \\
    \Delta' &\colon W_{\sigma, \sigma'}^{-, +} \to V_{\sigma, \sigma'}^{-, +} \to W_{\sigma, \sigma'}^{+, -, \vee}[1].
\end{align*}
These fibre sequences induce the following isomorphisms:
\begin{align*}
    b_{\sigma, \sigma'}^{+, -} &\colon \det(V_{\sigma, \sigma'}^{+, -}) \cong \iota_{\alpha}^* \det(W_{\sigma, \sigma'}^{+, -}) \otimes \iota_{\alpha}^* \det(W_{\sigma, \sigma'}^{-, +}), \\
    b_{\sigma, \sigma'}^{-, +} &\colon \det(V_{\sigma, \sigma'}^{-, +}) \cong \iota_{\alpha}^* \det(W_{\sigma, \sigma'}^{-, +}) \otimes \iota_{\alpha}^* \det(W_{\sigma, \sigma'}^{+, -}).
\end{align*}
By the construction of the map \Cref{eq-a_sigma_sigma'}, the commutativity of the diagram \Cref{eq-sigma-indep-canonical} is equivalent to the commutativity of the following diagram:
\begin{equation}\label{eq-localized-can-ori-equiv}
\begin{tikzcd}
	{\det(V_{\sigma, \sigma'}^{+, -})} & {\iota_{\alpha}^*  \det(W_{\sigma, \sigma'}^{+, -}) \otimes \iota_{\alpha}^* \det(W_{\sigma, \sigma'}^{-, +})} \\
	{\det(V_{\sigma, \sigma'}^{-, +})} & {\iota_{\alpha}^*  \det(W_{\sigma, \sigma'}^{-, +}) \otimes \iota_{\alpha}^*  \det(W_{\sigma, \sigma'}^{+, -})}
	\arrow["{b_{\sigma, \sigma'}^{+, -}}", from=1-1, to=1-2]
	\arrow[" \phi_{\sigma, \sigma'}^{+, -}"', from=1-1, to=2-1]
	\arrow["(-1)^{d(\sigma, \sigma')} \cdot{\mathrm{swap}}", from=1-2, to=2-2]
	\arrow["{b_{\sigma, \sigma'}^{-, +}}"', from=2-1, to=2-2]
\end{tikzcd}
\end{equation}
where the map $\phi_{\sigma, \sigma'}^{+, -}$ is defined in \Cref{eq-phi_sigma_sigma'_def}.
To prove this, consider the following diagram:
\begin{equation}\label{eq-localized-can-ori-expanded}
    \begin{tikzcd}
	& {\iota_{\alpha}^* \det(W_{\sigma, \sigma'}^{+, -}) \otimes \iota_{\alpha}^* \det(W_{\sigma, \sigma'}^{-, +})} \\
	{\det(V_{\sigma, \sigma'}^{+, -})} & {\iota_{\alpha}^* \det(W_{\sigma, \sigma'}^{+, -}) \otimes \iota_{\alpha}^* \det(W_{\sigma, \sigma'}^{-, +, \vee}[1])} \\
	{\det(V_{\sigma, \sigma'}^{+, -, \vee}[1])} & {\iota_{\alpha}^* \det(W_{\sigma, \sigma'}^{-, +, \vee}[1]^{\vee}[1]) \otimes \iota_{\alpha}^* \det(W_{\sigma, \sigma'}^{+, -, \vee}[1])  } \\
	{\det(V_{\sigma, \sigma'}^{-, +})} & {\iota_{\alpha}^* \det(W_{\sigma, \sigma'}^{-, +}) \otimes \iota_{\alpha}^* \det(W_{\sigma, \sigma'}^{+, -, \vee}[1])  } \\
	& {\iota_{\alpha}^* \det(W_{\sigma, \sigma'}^{-, +}) \otimes \iota_{\alpha}^* \det(W_{\sigma, \sigma'}^{+, -})}
	\arrow[from=1-2, to=2-2]
	\arrow[""{name=0, anchor=center, inner sep=0}, "{b_{\sigma, \sigma'}^{+, -}}", from=2-1, to=1-2]
	\arrow[""{name=1, anchor=center, inner sep=0}, "{i(\Delta)}", from=2-1, to=2-2]
	\arrow[from=2-1, to=3-1]
	\arrow[""{name=2, anchor=center, inner sep=0}, "{\phi_{\sigma, \sigma'}^{+, -}}"', curve={height=80pt}, from=2-1, to=4-1]
	\arrow["{\mathrm{swap}}", from=2-2, to=3-2]
	\arrow[""{name=3, anchor=center, inner sep=0}, "{i(\Delta^{\vee}[1])}", from=3-1, to=3-2]
	\arrow["\gamma", from=3-1, to=4-1]
	\arrow["\eta \otimes \id", from=3-2, to=4-2]
	\arrow[""{name=4, anchor=center, inner sep=0}, "{i(\Delta')}"', from=4-1, to=4-2]
	\arrow[""{name=5, anchor=center, inner sep=0}, "{b_{\sigma, \sigma'}^{-, +}}"', from=4-1, to=5-2]
	\arrow[from=4-2, to=5-2]
	\arrow["{(A)}"{description}, draw=none, from=2-2, to=0]
	\arrow["{(B)}"{description, pos=0.7}, draw=none, from=3, to=1]
	\arrow["{(C)}"{description}, draw=none, from=3-1, to=2]
	\arrow["{(D)}"{description}, draw=none, from=4, to=3]
	\arrow["{(E)}"{description}, draw=none, from=4-2, to=5]
\end{tikzcd}
\end{equation}
where the map $\gamma$ is induced by the $(-1)$-shifted symplectic structure, the map $\eta$ is induced by the isomorphism
\[
    W_{\sigma, \sigma'}^{-, +} \cong W_{\sigma, \sigma'}^{-, +, \vee \vee} \cong  W_{\sigma, \sigma'}^{-, +, \vee}[1]^{\vee}[1]
\]
and other morphisms are induced by \Cref{eq-det-shift} and \Cref{eq-det-dual}.
The diagrams $(A)$, $(C)$ and $(E)$ commute by definition. The diagram $(B)$ commutes by \Cref{eq-det-shift-fibre-seq} and \Cref{eq-det-dual-fibreseq}.
The diagram $(D)$ commutes since we have the following equivalence of fibre sequences:
\[
    \begin{tikzcd}
	{\Delta^{\vee}[1] \colon} &[-30pt] { W_{\sigma, \sigma'}^{-, +, \vee}[1]^{\vee}[1]} & {V_{\sigma, \sigma'}^{+, -, \vee}[1]} & {W_{\sigma, \sigma'}^{+, -, \vee}[1]} \\
	{\Delta': } & {W_{\sigma, \sigma'}^{-, +}} & {V_{\sigma, \sigma'}^{+, -}} & {W_{\sigma, \sigma'}^{+, -, \vee}[1]}
	\arrow[from=1-2, to=1-3]
	\arrow["\eta", from=1-2, to=2-2]
	\arrow[from=1-3, to=1-4]
	\arrow[from=1-3, to=2-3]
	\arrow[Rightarrow, no head, from=1-4, to=2-4]
	\arrow["\gamma", from=2-2, to=2-3]
	\arrow[from=2-3, to=2-4]
\end{tikzcd}
\]
which exists thanks to the identity $\Phi_{\sigma, \sigma'}^{+, -} = \Phi_{\sigma', \sigma}^{-, +}$ established in \Cref{eq-Phi_sigma_sigma'equalPhi_sigma'sigma}.
Now we claim that the composition of the right vertical maps in \Cref{eq-localized-can-ori-expanded} is given by $(-1)^{d(\sigma, \sigma')}$.
To see this, it is enough to show that the composition
\[
\det(W_{\sigma, \sigma'}^{+, -}) \cong     \det(W_{\sigma, \sigma'}^{+, -, \vee}[1]^{\vee}[1]) \xrightarrow[\cong]{\eta} \det(W_{\sigma, \sigma'}^{+, -}) 
\]
is given by the multiplication by $(-1)^{d(\sigma, \sigma')}$. This follows from \Cref{eq-det-double-dual}, \Cref{eq-det-dual-shift} and the identity $d(\sigma, \sigma') = \rank W^{+, -}_{\sigma, \sigma'}$.
Therefore we conclude that the diagram \Cref{eq-localized-can-ori-equiv} commutes as desired.
\end{proof}

\subsection{Donaldson--Thomas mixed Hodge modules}

\begin{para}
    For an oriented $(-1)$-shifted symplectic stack $(\mathcal{X}, \omega_{\mathcal{X}}, o)$,  \textcite[\S 4]{ben2015darboux} constructed a monodromic mixed Hodge module
    \[
        \varphi_{\mathcal{X}}  = \varphi_{\mathcal{X}, \omega_{\mathcal{X}}, o} \in \mathsf{MMHM}(\mathcal{X})    
    \]
    called the \emph{Donaldson--Thomas mixed Hodge module}, based on the construction for $(-1)$-shifted symplectic schemes by \textcite[\S 6]{brav2012symmetries}.
    We do not repeat the construction here, but we list basic properties that will be used later:
    \begin{itemize}
        \item Let $\mathcal{U}$ be a smooth algebraic stack and $f \colon \mathcal{U} \to \mathbb{A}^1$ be a regular function. Set $\mathcal{X} = \mathrm{Crit}(f)$ and equip it with the standard $(-1)$-shifted symplectic structure $\omega_{\mathcal{X}}$ and the standard orientation $o^{\mathrm{sta}}_{\mathcal{X}}$.
              Then there exists a natural isomorphism
              \begin{equation}\label{eq-DT-Hodge-standard}
                \varphi_{\mathcal{X}, \omega_{\mathcal{X}}, o_{\mathcal{X}}^{\mathrm{sta}}} \cong \varphi_{f} (\mathcal{IC}_{\mathcal{U}}).
              \end{equation}
              For schemes, this is an immediate consequence of the definition. For stacks, this is proved in \cite[Lemma 7.21]{Kinjo_Park_Safronov_CoHA} at the level of perverse sheaves, and the same argument works for monodromic mixed Hodge modules.

        \item Let  $(\mathcal{X}_2, \omega_{\mathcal{X}_2, o_{\mathcal{X}_2}})$ be an oriented $(-1)$-shifted symplectic stack and $\eta \colon \mathcal{X}_1 \to \mathcal{X}_2$ be an \'etale cover.
              Set $\omega_{\mathcal{X}_1} \coloneqq \eta^{\star} \omega_{\mathcal{X}_2}$. The orientation $o_{\mathcal{X}_2}$ naturally induces an orientation $o_{\mathcal{X}_1}$ on $(\mathcal{X}_1, \omega_{\mathcal{X}_1})$.
              Then there exists a natural isomorphism
              \begin{equation}\label{eq-dt-etale-cover}
                \varphi_{\mathcal{X}_1, \omega_{\mathcal{X}_1}, o_{\mathcal{X}_1}} \cong \eta^* \varphi_{\mathcal{X}_2, \omega_{\mathcal{X}_2}, o_{\mathcal{X}_2}}.
              \end{equation}

        \item Let $(\mathcal{X}, \omega_{\mathcal{X}}, o_{\mathcal{X}})$ and  $(\mathcal{Y}, \omega_{\mathcal{Y}}, o_{\mathcal{Y}})$ be $(-1)$-shifted symplectic stacks.
              Then there exists a natural isomorphism
              \begin{equation}\label{eq-Kunneth}
                \varphi_{\mathcal{X} \times \mathcal{Y}, \omega_{\mathcal{X}} \boxplus \omega_{\mathcal{Y}}, o_{\mathcal{X}} \boxtimes o_{\mathcal{Y}}} \cong \varphi_{\mathcal{X}, \omega_{\mathcal{X}}, o_{\mathcal{X}}} \boxtimes \varphi_{\mathcal{Y}, \omega_{\mathcal{Y}}, o_{\mathcal{Y}}}
              \end{equation}
              which we call the Thom--Sebastiani isomorphism.
              This is proved in \cite[Corollary 4.4]{Kinjo_Park_Safronov_CoHA} at the level of perverse sheaves, and the same argument works for monodromic mixed Hodge modules.
              One can easily check that the Thom--Sebastiani isomorphism does not depend on the order of the product decomposition: namely, the following diagram commutes:
\[\begin{tikzcd}
	{\varphi_{\mathcal{X} \times \mathcal{Y}, \omega_{\mathcal{X}} \boxplus \omega_{\mathcal{Y}}, o_{\mathcal{X}} \boxtimes o_{\mathcal{Y}}}  } & {\varphi_{\mathcal{X}, \omega_{\mathcal{X}}, o_{\mathcal{X}}} \boxtimes \varphi_{\mathcal{Y}, \omega_{\mathcal{Y}}, o_{\mathcal{Y}}}} \\
	{\mathrm{sw}^* \varphi_{\mathcal{Y} \times \mathcal{X}, \omega_{\mathcal{Y}} \boxplus \omega_{\mathcal{X}}, o_{\mathcal{Y}} \boxtimes o_{\mathcal{X}}}  } & {\mathrm{sw}^* (\varphi_{\mathcal{Y}, \omega_{\mathcal{Y}}, o_{\mathcal{Y}}} \boxtimes \varphi_{\mathcal{X}, \omega_{\mathcal{X}}, o_{\mathcal{X}}})}
	\arrow["\cong"', from=1-1, to=1-2]
	\arrow["\text{\Cref{eq-Kunneth}}", from=1-1, to=1-2]
	\arrow["\cong"', from=1-1, to=2-1]
	\arrow["\cong", from=1-2, to=2-2]
	\arrow["\cong"', from=2-1, to=2-2]
	\arrow["\text{\Cref{eq-Kunneth}}", from=2-1, to=2-2]
\end{tikzcd}\]
            where $\mathrm{sw} \colon \mathcal{X} \times \mathcal{Y} \cong \mathcal{Y} \times \mathcal{X}$ denotes the swapping isomorphism.

        \item Let $\mathcal{Y}$ be a quasi-smooth derived algebraic stack. Set $\mathcal{X} = \mathrm{T}^*[-1] \mathcal{Y}$  and equip it with the standard $(-1)$-shifted symplectic structure $\omega_{\mathcal{X}}$ and the standard orientation $o^{\mathrm{sta}}_{\mathcal{X}}$.
              Let $\pi \colon \mathcal{X} \to \mathcal{Y}$ be the canonical projection. Then there exists a natural isomorphism
              \begin{equation}\label{eq-dimensional-reduction}
                \pi_* \varphi_{\mathcal{X}} \cong \mathbb{L}^{\vdim \mathcal{Y}/2} \otimes \mathbb{D}\mathbb{Q}_{\mathcal{Y}}
              \end{equation}
              called the \emph{dimensional reduction isomorphism}.
              This is proved in \cite[Theorem 4.14]{kinjo2022dimensional} at the level of constructible complexes and in \cite[Proposition 5.3]{kinjo2024cohomological} at the level of the mixed Hodge complexes.
            The statement for monodromic mixed Hodge complexes follows from the faithfulness of the monodromy-forgetting functor on the heart of the ordinary t-structure \Cref{eq-forgetting} together with the fact that the statement holds true locally \Cref{eq-local-dimensional-reduction}.
    \end{itemize}
\end{para}

\begin{remark}[Independence of the grading of the orientation]
    Let $\mathrm{pt}$ be a point equipped with the trivial $(-1)$-shifted symplectic structure $\omega_{\mathrm{pt}}$.
    Let $o^{\mathrm{triv}}_{\mathrm{odd}}$ be the trivial orientation for the point with the odd grading.
    Then by \Cref{eq-DT-Hodge-standard}, there exists an isomorphism
    \[
    c \colon \varphi_{\mathrm{pt}, \omega_{\mathrm{pt}}, o^{\mathrm{triv}}_{\mathrm{odd}}} \cong \mathbb{L}^{-1/2} \otimes \varphi_{z^2}(\mathbb{Q}_{\mathbb{A}^1}) \cong \mathbb{Q}_{\mathrm{pt}}.    
    \]

    Now let $(\mathcal{X}, \omega_{\mathcal{X}}, o_{\mathrm{even}})$ be an oriented $(-1)$-shifted symplectic stack with an even grading.
    Let $o_{\mathrm{odd}}$ be the orientation for $\mathcal{X}$ whose underlying line bundle and the isomorphism is the same as $o_{\mathrm{even}}$ but the grading is taken to be odd.
    Then we have an isomorphism of orientations $o_{\mathrm{odd}} \cong o_{\mathrm{even}} \boxtimes o^{\mathrm{triv}}_{\mathrm{odd}}$, therefore we have an isomorphism
    \[
     \varphi_{\mathcal{X}, \omega_{\mathcal{X}}, o_{\mathrm{odd}}} \xrightarrow[\cong]{\text{\Cref{eq-Kunneth}}} \varphi_{\mathcal{X}, \omega_{\mathcal{X}}, o_{\mathrm{even}}} \boxtimes   \varphi_{\mathrm{pt}, \omega_{\mathrm{pt}}, o^{\mathrm{triv}}_{\mathrm{odd}}} \xrightarrow[\cong]{ \id \otimes c} \varphi_{\mathcal{X}, \omega_{\mathcal{X}}, o_{\mathrm{even}}}.
    \]
    In particular, the DT mixed Hodge module does not depend on the choice of the grading of the orientation.
    Nevertheless, it is convenient to work with the graded orientation: for example, one can check that the following diagram commutes only up to a sign $(-1)$:
\[\begin{tikzcd}
	{\varphi_{\mathrm{pt}, \omega_{}^{\mathrm{triv}}, o_{\mathrm{odd}}^{\mathrm{triv}}} \boxtimes \varphi_{\mathrm{pt}, \omega_{}^{\mathrm{triv}}, o_{\mathrm{odd}}^{\mathrm{triv}}}  } & {\varphi_{\mathrm{pt}, \omega_{}^{\mathrm{triv}}, o_{\mathrm{even}}^{\mathrm{triv}}} } \\
	{\mathbb{Q}_{\mathrm{pt}} \boxtimes \mathbb{Q}_{\mathrm{pt}}} & {\mathbb{Q}_{\mathrm{pt}}.}
	\arrow[from=1-1, to=1-2]
	\arrow["{c \boxtimes c}"', from=1-1, to=2-1]
	\arrow["\cong", from=1-2, to=2-2]
	\arrow[from=2-1, to=2-2]
\end{tikzcd}\]
    Therefore we will not fix the trivialization of the perverse sheaf $\varphi_{\mathrm{pt}, \omega_{}^{\mathrm{triv}}, o_{\mathrm{odd}}^{\mathrm{triv}}}$ to reduce the number of the signs appearing in the paper.
\end{remark}

\begin{para}[Integral isomorphism]\label{para-integral-isomorphism}
 Here we recall the integral isomorphism of the Donaldson--Thomas mixed Hodge modules established by Kinjo, Park and Safronov \cite[Corollary 7.19]{Kinjo_Park_Safronov_CoHA}.
 This will be used in the construction of the cohomological Hall induction for $(-1)$-shifted symplectic stacks in \Cref{para-cohi--1-shifted-symplectic}.

 Let $(\mathcal{X}, \omega_{\mathcal{X}}, o)$ be an oriented $(-1)$-shifted symplectic stack which is quasi-separated and has affine stabilizers. Assume further that $\mathcal{X}$ has quasi-compact connected components and quasi-compact graded points.
 Take a non-degenerate face $(F, \alpha) \in \mathsf{Face}^{\mathrm{nd}}(\mathcal{X})$ and a chamber $\sigma \subset F$ with respect to the cotangent arrangement.
 Equip $\mathcal{X}_{\alpha}$ with the localized $(-1)$-shifted symplectic structure $\mathrm{tot}_{\alpha}^{\star} \omega_{\mathcal{X}}$ introduced in \Cref{para-localize-shifted-symplectic-structure} and the localized orientation $\sigma^{\star} o$ introduced in \Cref{para-localized-orientation}. 
 Consider the following correspondence:
\[\begin{tikzcd}
	& {\mathcal{X}_{\sigma}^+} \\
	{\mathcal{X}_{\alpha}} && {\mathcal{X}.}
	\arrow["{\gr_{\sigma}}"', from=1-2, to=2-1]
	\arrow["{\ev_{1, \sigma}}", from=1-2, to=2-3]
\end{tikzcd}\]
Then there exists an isomorphism 
\begin{equation}\label{eq-integral-isomorphism}
   \zeta_{\sigma} \colon \mathbb{L}^{- {\vdim \mathcal{X}}_{\sigma}^+/2} \otimes \varphi_{\mathcal{X}_{\alpha}} \cong \gr_{\sigma, *} \ev_{1, \sigma}^! \varphi_{\mathcal{X}}
\end{equation}
which we call the \emph{integral isomorphism} following \textcite[\S 7.8]{_Kontsevich_Soibelman_CoHA}.
The existence of this isomorphism at the level of perverse sheaves is a consequence of \cite[Corollary 7.19]{Kinjo_Park_Safronov_CoHA} together with the constancy theorem \Cref{para-constancy-theorem}.
The same proof works at the level of monodromic mixed Hodge modules.

We list properties of the integral isomorphism that will be used later. All stacks will be assumed to be quasi-separated, have affine stabilizers and have finite cotangent weights.
\begin{itemize}
    \item Let $(\mathcal{X}_1, \omega_{\mathcal{X}_1}, o_1)$ and $(\mathcal{X}_2, \omega_{\mathcal{X}_2}, o_2)$ be two oriented $(-1)$-shifted symplectic stacks. 
          Take non-degenerate faces $(F_i, \alpha_i) \in \mathsf{Face}^{\mathrm{nd}}(\mathcal{X}_i)$ and chambers $\sigma_i \subset F_i$  with respect to the cotangent arrangement for $i= 1,2$.
          Then the following diagram commutes:
\[\begin{tikzcd}
	\mathbb{L}^{(-{\vdim \mathcal{X}_{\sigma_1}^+} - \vdim \mathcal{X}_{\sigma_2}^+) /2} \otimes (\varphi_{\mathcal{X}_{1, \alpha_1}} \boxtimes \varphi_{\mathcal{X}_{2, \alpha_2}})  & {\gr_{\sigma_1, *} \mathrm{ev}_{1, \sigma_1}^! \varphi_{\mathcal{X}_{1}} \boxtimes \gr_{\sigma_2, *} \mathrm{ev}_{2, \sigma_2}^! \varphi_{\mathcal{X}_{2}} } \\
	{\mathbb{L}^{-{\vdim (\mathcal{X}_{1} \times \mathcal{X}_2)_{\sigma_1 \times \sigma_2}^+} / 2} \otimes \varphi_{(\mathcal{X}_{1} \times \mathcal{X}_2)_{\alpha_1 \times \alpha_2}}  } & {\gr_{\sigma_1 \times \sigma_2, *} \mathrm{ev}_{1, \sigma_1 \times \sigma_2}^! \varphi_{\mathcal{X}_1 \times \mathcal{X}_2}.}
	\arrow["{\zeta_{\sigma_1} \boxtimes \zeta_{\sigma_2}}", from=1-1, to=1-2]
	\arrow["\text{\Cref{eq-Kunneth}}"', from=1-1, to=2-1]
	\arrow["\text{\Cref{eq-Kunneth}}", from=1-2, to=2-2]
	\arrow["{\zeta_{\sigma_1 \times \sigma_2}}"', from=2-1, to=2-2]
\end{tikzcd}\]
    This is proved in \cite[Corollary 7.19]{Kinjo_Park_Safronov_CoHA}.

    \item Let $\mathcal{U}$ be a smooth algebraic stack. Let $f \colon \mathcal{U} \to \mathbb{A}^1$ be a regular function and set $\mathcal{X} = \mathrm{Crit}(f)$ and equip it with the standard $(-1)$-shifted symplectic structure and orientation.
          Let $(F, \alpha) \in \mathsf{Face}^{\mathrm{nd}}(\mathcal{U})$ be a non-degenerate face and $\sigma \subset F$ be a chamber with respect to the cotangent arrangement.
          Set $f_{\alpha} = f \circ \mathrm{tot}_{\alpha}$.
          Equip $\mathcal{X}_{\alpha}$ with the localized $(-1)$-shifted symplectic structure and orientation.
          Then the following diagram commutes:
\begin{equation}\label{eq-integral-isom-critical}
    \begin{tikzcd}
	{\mathbb{L}^{- {\dim \mathcal{X}_{\sigma}^+}/2}  \otimes\varphi_{\mathcal{X}_{\alpha}}} &[-65pt] &[-65pt] {\gr_{\sigma, *} \mathrm{ev}_{1, \sigma}^! \varphi_{\mathcal{X}_{}} } \\
	{\mathbb{L}^{ -{\dim \mathcal{X}_{\sigma}^+}/2} \otimes \varphi_{f_{\alpha}}(\mathcal{IC}_{\mathcal{U}_{\alpha}})} && { \gr_{\sigma, *} \mathrm{ev}_{1, \sigma}^! \varphi_{f}(\mathcal{IC}_{\mathcal{U}})} \\
	& {\varphi_{f_{\alpha}}(\gr_{\sigma, *} \mathrm{ev}_{1, \sigma}^!\mathcal{IC}_{\mathcal{U}}).}
	\arrow["{\zeta_{\sigma}}", from=1-1, to=1-3]
	\arrow["\text{\Cref{eq-DT-Hodge-standard}}", from=1-1, to=2-1]
	\arrow["\text{\Cref{eq-DT-Hodge-standard}}"', from=1-3, to=2-3]
	\arrow["\cong"', from=2-1, to=3-2]
	\arrow["\text{\Cref{eq-van-functorial}}"', from=3-2, to=2-3]
\end{tikzcd}\end{equation}
Here for the left vertical map, we used the identification of the $(-1)$-shifted symplectic structures \Cref{eq-localize-standard-symplectic} and the orientations \Cref{eq-localize-standard-orientaion}.
This is proved in \cite[Proposition 7.22]{Kinjo_Park_Safronov_CoHA}.

\item  Let  $(\mathcal{X}_2, \omega_{\mathcal{X}_2}, o_{\mathcal{X}_2})$ be an oriented $(-1)$-shifted symplectic stack and $\eta \colon \mathcal{X}_1 \to \mathcal{X}_2$ be an \'etale cover.
        Set $\omega_{\mathcal{X}_1} \coloneqq \eta^{\star} \omega_{\mathcal{X}_2}$ and equip $\mathcal{X}_1$ with the orientation $o_{\mathcal{X}_1}$ induced from $o_{\mathcal{X}_2}$.
       Let $(F, \alpha_2) \in \mathsf{Face}^{\mathrm{nd}}(\mathcal{X}_2)$ be a non-degenerate face and $(F, \alpha_1) \in \mathsf{Face}^{\mathrm{nd}}(\mathcal{X}_1)$ be its lift.
       Take a chamber $\sigma_2 \subset F$ with respect to the cotangent arrangement of $\mathcal{X}_2$.
       We denote by $\sigma_1$ the same cone regarded as an object in $\mathsf{Cone}(\mathcal{X}_1)$.
       Let $\eta_{\alpha} \colon \mathcal{X}_{1, \alpha_1} \to \mathcal{X}_{2, \alpha_2}$ be the natural morphism.
       Then the following diagram commutes:
\begin{equation}\label{eq-integral-isom-etale}
    \begin{tikzcd}
	{   \mathbb{L}^{- {\vdim \mathcal{X}_{1, \sigma_1}^+/2}} \otimes \varphi_{\mathcal{X}_{1, \alpha_1}} } & {\gr_{\sigma_1, *} \ev_{1, \sigma_1}^! \varphi_{\mathcal{X}_1}} \\
	{   \mathbb{L}^{- {\vdim \mathcal{X}_{2, \sigma_2}^+}/2} \otimes \eta_{\alpha}^* \varphi_{\mathcal{X}_{2, \alpha_2}} } & {\gr_{\sigma_1, *} \ev_{1, \sigma_1}^! \eta^* \varphi_{\mathcal{X}_2}.}
	\arrow["\cong"', from=1-1, to=1-2]
	\arrow["{\zeta_{\sigma_1} }", from=1-1, to=1-2]
	\arrow["\cong", from=1-1, to=2-1]
	\arrow["\cong", from=1-2, to=2-2]
	\arrow["{\zeta_{\sigma_2} }", from=2-1, to=2-2]
	\arrow[from=2-1, to=2-2]
	\arrow["\cong"', from=2-1, to=2-2]
\end{tikzcd}
\end{equation}
    This follows from the construction of the map $\zeta_{\sigma}$: see the first diagram in the proof of \cite[Theorem 7.16]{Kinjo_Park_Safronov_CoHA}.

\end{itemize}

\end{para}

\section{BPS sheaves}

In this section, we will introduce BPS sheaves on good moduli spaces of smooth stacks and oriented $(-1)$-shifted symplectic stacks under the almost symmetricity condition (Definition \ref{def-symmetry}).
Our definition of BPS sheaves is a generalization of the definition appearing in \cite[Definition 2.11]{toda_gopakumar_wall_crossing}.
We will prove that the BPS sheaf on a smooth stack is either zero or the intersection complex, generalizing the result of \textcite[Theorem 1.1]{Meinhardt_Reineke_DT_vs_intersection_cohomology}.
We will also prove the support lemma for the BPS sheaves associated with oriented $(-1)$-shifted symplectic stacks, generalizing the result of Davison \cite[Lemma 4.1]{davison_integrality_preprojective}.

\subsection{Definition of the BPS sheaf}

Let $\mathcal{U}$ be an almost symmetric smooth algebraic stack having affine diagonal and quasi-compact graded points.
Assume also that $\mathcal{U}$ admits a good moduli space $p \colon \mathcal{U} \to U$.
For a non-degenerate face $(F, \alpha) \in  \mathsf{Face}^{\mathrm{nd}}(\mathcal{U})$, 
let $p_{\alpha} \colon \mathcal{U}_{\alpha} \to U_{\alpha}$ be the good moduli space morphism.
Define the BPS sheaf by
\[
\mathcal{BPS}^{\alpha}_{U} \coloneqq {}^{\mathrm{p}} \cH^{0} (\mathbb{L}^{- {\dim F} /2} \otimes p_{\alpha, *} \mathcal{IC}_{\mathcal{U}_{\alpha}}) \in \mathsf{MMHM}(U_{\alpha}).
\]
Similarly, for an oriented almost symmetric $(-1)$-shifted symplectic stack $\mathcal{X}$ having affine diagonal and quasi-compact graded points which admits a good moduli space $p \colon \mathcal{X} \to X$, with a non-degenerate face $(F, \alpha) \in  \mathsf{Face}^{\mathrm{nd}}(\mathcal{X})$, we let $p_{\alpha} \colon \mathcal{X}_{\alpha} \to X_{\alpha}$ denote the good moduli space morphism and define the BPS sheaf by
\[
\mathcal{BPS}^{\alpha}_{X} \coloneqq {}^{\mathrm{p}} \cH^{0} (\mathbb{L}^{- {\dim F} /2} \otimes p_{\alpha, *} \varphi_{\mathcal{X_{\alpha}}}) \in \mathsf{MMHM}(X_{\alpha}).
\]
Here $\mathcal{X_{\alpha}}$ is equipped with the localized orientation introduced in \Cref{para-localized-orientation}.
As we will see later in \Cref{cor:small_decom_thm_smooth_stack} and \Cref{prop-support-lemma}, the BPS sheaf is the lowest possibly non-vanishing perverse cohomology.

We also define the $n$-th BPS sheaf on $\mathcal{U}$ and $\mathcal{X}$ by
\[
    \mathcal{BPS}_{U}^{(n)} \coloneqq {}^{\mathrm{p}} \cH^{0} (\mathbb{L}^{- {n} /2} \otimes p_{*} \mathcal{IC}_{\mathcal{U}}) \in \mathsf{MMHM}(U), \quad      \mathcal{BPS}^{(n)}_{X} \coloneqq {}^{\mathrm{p}} \cH^{0} (\mathbb{L}^{- {n} /2} \otimes p_{*} \mathcal{\varphi}_{\mathcal{X}})  \in \mathsf{MMHM}(X).
\]
When $\mathcal{U}$ and $\mathcal{X}$ are connected, we set
\[
    \mathcal{BPS}_{U} = \mathcal{BPS}_{U}^{(\crk \mathcal{U})}, \quad \mathcal{BPS}_{X} = \mathcal{BPS}_{X}^{(\crk \mathcal{X})}.
\]
Clearly, we have isomorphisms
\[
    \mathcal{BPS}_{U} = \mathcal{BPS}_{U}^{\alpha_{\mathrm{ce}}}, \quad \mathcal{BPS}_{U} = \mathcal{BPS}_{X}^{\alpha_{\mathrm{ce}}}
\]
where $\alpha_{\mathrm{ce}}$ denotes the maximal central face recalled in \Cref{para-central-rank}.

\subsection{Smallness of the good moduli morphism}

    Let $G$ be a reductive algebraic group and $V$ be a representation of $G$.
    Assume that the quotient stack $V / G$ is almost symmetric.
    By \Cref{lem-quotient-symmetric}, this is equivalent to the existence of an isomorphism of $G^{\circ}$-representations $V \cong V^{\vee}$, where $G^{\circ} \subset G$ denotes the neutral component.
    Let $p \colon V / G \to V \GIT G$ be the good moduli space morphism.
    We will prove the following statement:

\begin{proposition}\label{prop:virtual_small}
    There exists a stratification by connected locally closed subsets $\{ S_{\xi}\}_{\xi }$ of $V \GIT G$ with the following property:
    Set $\mathcal{S}_{\xi} \coloneqq p^{-1}(S_{\xi})$ and let $p_{\xi} \coloneqq \mathcal{S}_{\xi} \to S_{\xi}$ denote the restriction of $p$.
    Then the following conditions hold:
    \begin{enumerate}
        \item The map $p_{\xi}$ is an \'etale-locally trivial fibration. \label{item:trivial_stratification}
        \item For each $s_{\xi} \in S_{\xi}$, the inequality
        \begin{equation}\label{eq-virtual-small-symmrep}
            \dim S_{\xi} +  2 \cdot \dim p^{-1}(s_{\xi}) \leq \dim V / G  
        \end{equation}
        holds. Further, equality holds if and only if $S_{\xi}$ is an open stratum and $\mathcal{S}_{\xi}$ has finite stabilizers. \label{item:small_stratification}
    \end{enumerate}
\end{proposition}

\begin{proof}
    The existence of a stratification of $V \GIT G$ with the condition \ref{item:trivial_stratification} is a general statement for an affine GIT quotient of a smooth variety proved in \cite[\S 3]{_Luna_Slicesetales}
    as a consequence of the \'etale slice theorem.

    We now prove \Cref{eq-virtual-small-symmrep}.
    We first prove a weaker inequality
    \begin{equation}\label{eq:dim_estimate_nilp_locus}
    2 \cdot \dim p^{-1}(0) \leq \dim V / G    
    \end{equation}
    where $0 \in V \GIT G$ is the image of the origin of $V$.
    Let $V^{\mathrm{nilp}} \subset V$ denote the inverse image of the map $V \to V \GIT G$ at $0$.
    By the fundamental theorem in GIT \cite[\S 2.1]{mumford1994geometric} (see \cite[Theorem 6.13]{hoskins_lecture} for the version that we will use),
    for each $x \in V^{\mathrm{nilp}}$, there exists a $1$-parameter subgroup $\lambda_x \colon \mathbb{G}_\mathrm{m} \to G$ such that 
    $\lim_{t \to 0} \lambda_{x}(t) \cdot x \to 0$.
    In particular, if we let $T \subset G$ be the maximal torus, the family of action maps
    \begin{equation}\label{eq:contracting_action_maps_jointly_surjective}
    \coprod_{\lambda \colon  \mathbb{G}_\mathrm{m} \to T} \Bigl\{G \times V^{\lambda}_{> 0} \to V^{\mathrm{nilp}} \Bigr\}    
    \end{equation}
    is jointly surjective. Let $P_{\lambda} \subset G$ be the parabolic subgroup associated with $\lambda$.
    Then the action maps \Cref{eq:contracting_action_maps_jointly_surjective} factor through maps
    \begin{equation}\label{eq:contracting_action_maps_jointly_surjective_parabolic}
        \coprod_{\lambda \colon  \mathbb{G}_\mathrm{m} \to T} \Bigl\{G \times^{P_{\lambda}} V^{\lambda}_{> 0} \to V^{\mathrm{nilp}} \Bigr\}.    
    \end{equation}
    Note that there exists a decomposition of the underlying set of the cocharacter lattice 
    \[
        \Hom(\mathbb{G}_\mathrm{m}, T) = K_1 \amalg \cdots \amalg K_l
    \] 
    such that the assignments $\lambda \mapsto V^{\lambda}_{> 0}$ and $\lambda \mapsto P_{\lambda}$ are constant over each factor.
    In particular, the surjectivity of the map \Cref{eq:contracting_action_maps_jointly_surjective_parabolic} implies the inequality
    \[
        \dim V^{\mathrm{nilp}} \leq \dim V^{\lambda}_{> 0} + \dim G - \dim P_{\lambda}
    \]
    for some cocharacter $\lambda$. Since $V$ is a symmetric $G^{\circ}$-representation,
    we obtain the inequality
    \[
        2 \cdot \dim V^{\lambda}_{> 0} \leq \dim V.
        \] 
  In particular, we have
    \[
    2 \cdot \dim p^{-1}(0) =  2 \cdot (\dim V^{\mathrm{nilp}} - \dim G) \leq 2 \cdot (\dim V^{\lambda}_{> 0} - \dim P_{\lambda}) \leq \dim V / G  
    \]
    as desired. By the proof, we see that equality holds if and only if $V$ is the $0$-dimensional vector space and $G$ is a finite group.

    We now prove that the inequality \Cref{eq:dim_estimate_nilp_locus} implies \ref{item:small_stratification}.
    Fix a stratification $\{ S_{\xi }\}_{\xi}$ of $V \GIT G$ satisfying \ref{item:trivial_stratification}, set $d_{\xi} \coloneqq \dim S_{\xi}$ and take a closed point $\tilde{s}_{\xi} \in \mathcal{S}_{\xi}$.
    Let $G_{\xi}$ be the stabilizer group of $\tilde{s}_{\xi}$ and set $V_{\xi} \coloneqq \mathrm{H}^0(\mathbb{T}_{V / G, \tilde{s}_{\xi}} )$.
    By Luna's \'etale slice theorem, we may find the following diagram:
    \[
    \begin{tikzcd}
        {V_{\xi} / G_{\xi}} \arrow{d}
        & {\mathcal{U}} \arrow[l, "\eta"']  \arrow{d} \arrow[r, "\iota"]
        & {V / G} \arrow{d} \\
        {V_{\xi} \GIT G_{\xi}}
        & {U} \arrow[l, "\bar{\eta}"'] \arrow[r, "\bar{\iota}"]
        & {V \GIT G}
    \end{tikzcd}
    \]
    where the two squares are Cartesian, the rightmost horizontal maps are open inclusions and the leftmost horizontal maps are \'etale,
    such that there exists a closed point $u \in \mathcal{U}$ with $\iota(u) \sim \tilde{s}_{\xi}$ and $\eta(u) \sim 0$.
    The condition \ref{item:trivial_stratification} implies a decomposition of $G_{\xi}$-representations
     $V_{\xi} = \mathbb{C}^{d_{\xi}} \oplus V_{\xi}'$ where $\mathbb{C}^{d_{\xi}}$ is the trivial representation.
     Then by applying the inequality \Cref{eq:dim_estimate_nilp_locus} for the stack $V_{\xi}'/G_{\xi}$,
     we obtain the desired statement.
\end{proof}

\begin{remark}
    Hennecart informed us that he has also obtained a similar result in \cite[Theorem 1.1]{Hennecart_note},  under the assumption that the map $p \colon V /G \to V \GIT G$ is generically quasi-finite on the target.
\end{remark}

\begin{para}

The property in Proposition \ref{prop:virtual_small} is nothing but a stacky version of smallness of a morphism (see e.g.~\cite[Remark 4.2.4]{deCataldo_Migliorini_decomp} in the case of morphisms of schemes).
The smallness of morphisms for stacks was first studied by \textcite[\S 5.1]{Meinhardt_Reineke_DT_vs_intersection_cohomology} where they proved a special case of Proposition \ref{prop:virtual_small} when $V / G$ is the moduli stack of representations of a quiver.
Similarly to the case of schemes, the smallness property implies a strong constraint for the summands of the decomposition theorem, as we will see below:
\end{para}

\begin{corollary}\label{cor:small_decom_thm_smooth_stack}
    Let $\mathcal{U}$ be an almost symmetric smooth algebraic stack having affine diagonal admitting a good moduli space $p \colon \mathcal{U} \to U$. 
    For a non-degenerate face $(F, \alpha) \in  \mathsf{Face}^{\mathrm{nd}}(\mathcal{U})$, we have the following properties:
    \begin{enumerate}
        \item For $i < \dim F$, we have ${}^{\mathrm{p}} \mathcal{H}^i(p_{\alpha, *} \mathcal{IC}_{\mathcal{U}_{\alpha}}) = 0$.
        \item  Assume that there exists a non-empty open substack of $\mathcal{U}_{\alpha}$ whose stabilizer groups at closed points contain the torus $\mathrm{B} \mathbb{G}_{\mathrm{m}}^{\dim F}$ as a subgroup of finite index.
              Then the natural map $\mathbb{Q}_{U_{\alpha}} \to p_{\alpha, *} \mathbb{Q}_{\mathcal{U}_{\alpha}}$ induces an isomorphism
              \[
                 \mathcal{IC}_{U_{\alpha}} \cong \mathcal{BPS}^{\alpha}_{U}.
              \]
              If not, $\mathcal{BPS}^{\alpha}_{U}$ is zero.
    \end{enumerate}
\end{corollary}

\begin{proof}
    Since the statement is local, by using \cite[Theorem 4.12]{_Alper_ALunaetaleslicetheoremforalgebraicstacks}, one may assume $\mathcal{X} = V / G$,
    where $G$ is a reductive group and $V$ is a $G$-representation.
    Let $L_F \subset G$ be the Levi subgroup corresponding to $F$, $T_F \subset L_F$ be the torus whose Lie algebra is the image of the natural inclusion of $F \otimes_{\mathbb{Q}} \mathbb{C}$ and set $V^{F} \coloneqq V^{T_F}$.
    Set $\bar{L}_F \coloneqq L_{F} /T_{F}$ and let $\pi \colon V^{F} / L_{F} \to V^{F} / \bar{L}_F$ be the canonical projection.
    We claim that there is an isomorphism
    \[
        \pi_* \mathcal{IC}_{V^{F}/L_F} \cong   \mathcal{IC}_{V^{F}/\bar{L}_F} \otimes \mathrm{H}^*(\mathrm{B} T_F) \otimes \mathbb{L}^{\dim F/2}.
   \]
   To see this, since the composite $T_F \to L_{F} \to L_{F}^{\mathrm{ab}}$ has finite kernel,
   we may take cohomology classes $\alpha_1, \ldots, \alpha_{\dim F} \in \mathrm{H}^2(V^{F}/L_{F})$ such that the restriction to $\mathrm{B} T_F$ generates the rational cohomology $\mathrm{H}^2(\mathrm{B} T_F)$.
   Using these cohomology classes, one can construct a natural morphism
   \[
       \mathcal{IC}_{V^{F}/\bar{L}_{F}} \otimes \mathrm{H}^*(\mathrm{B} T_F) \otimes \mathbb{L}^{\dim F/2} \to \pi_* \mathcal{IC}_{V^{F}/L_{F}}.
   \]
   One sees that this morphism is an isomorphism using the following Cartesian diagram:
   \[
   \begin{tikzcd}
       {V \times \mathrm{B} T_F}
       \arrow{r} \arrow{d}
       & {V}
       \arrow{d} \\
       {V/L_{F}}
       \arrow{r}
       & {V/\bar{L}_{F}.}
   \end{tikzcd}    
   \]
   Now let $\bar{p}_{F} \colon V^{F} / \bar{L}_{F} \to V^{F} \GIT \bar{L}_{F}$ be the good moduli space morphism.
   By the above discussion, it is enough to prove the following isomorphism:
   \[
    {}^{\mathrm{p}} \mathcal{H}^0(\bar{p}_{F, *} \mathcal{IC}_{V^{F} / \bar{L}_{F}} ) \cong 
    \begin{cases} 
        \mathcal{IC}_{V^{F} \GIT L_{F}}  & \text{if $\bar{p}_F$ is generically quasi-finite,}\\ 
         0 & \text{otherwise,}    \end{cases}
   \]
   and the vanishing ${}^{\mathrm{p}} \mathcal{H}^i(\bar{p}_{F, *} \mathcal{IC}_{V^{F} / \bar{L}_{F}} ) = 0 $ for $i < 0$.
   Here, a generically quasi-finite morphism is defined as a morphism that is quasi-finite generically on the target.

   Take a stratification $\{ S_{\xi} \}_{\xi}$ of $V_{F} \GIT \bar{L}_F$ as in Proposition \ref{prop:virtual_small}, set $d_{\xi} \coloneqq \dim S_{\xi}$ and take $s_{\xi} \in S_{\xi}$.
   Let $\iota_{\xi} \colon {\bar{p}_F}^{-1}(s_{\xi}) \hookrightarrow V^{F}/\bar{L}_F$ be the natural inclusion.
   Using \cite[Theorem 3.2.9]{Dimca_sheaves}, we see that the complex $\iota_{\xi}^! \mathcal{IC}_{V^{F}/\bar{L}_F}$ is concentrated in cohomological degrees
   \[
       \left[\dim V^{F}/\bar{L}_{F}  - 2 \cdot \dim {\bar{p}_F}^{-1}(s_{\xi}) , \infty\right)
   \]
   with respect to the ordinary t-structure.
   By Proposition \ref{prop:virtual_small} \ref{item:small_stratification}, we obtain the vanishing
   \begin{equation}\label{eq:vanishing_costalk}
   \mathrm{H}^i ({\bar{p}_F}^{-1}(s_{\xi}), \iota_{\xi}^! \mathcal{IC}_{V^{F}/\bar{L}_F} )  = 0
   \end{equation}
   for $i < d_{\xi}$, and the same vanishing for $i = d_{\xi}$ unless $S_{\xi}$ is open and $\mathcal{S}_{\xi}$ has finite stabilizers.
   In particular, we deduce the statement on the vanishing of the perverse cohomology.
   
   Assume now that the open stratum $\mathcal{S}_{\xi}$ has finite stabilizers.
   Then we clearly have an isomorphism
   \[
       {}^{\mathrm{p}} \cH^0(\bar{p}_{F, *} \mathcal{IC}_{V^{F}/\bar{L}_F}) |_{S_{\xi}} \cong \mathcal{IC}_{V ^{F}\GIT \bar{L}_F} |_{S_{\xi}}.   
   \]
   By \Cref{thm-decom-good-moduli},  we see that the complex ${}^{\mathrm{p}} \cH^0(\bar{p}_{F, *}  \mathcal{IC}_{V^{F}/\bar{L}_F})$ is pure.
   In particular, ${}^{\mathrm{p}} \cH^0(\bar{p}_{F, *}  \mathcal{IC}_{V^{F}/\bar{L}_F})$ is a direct sum of intersection complexes associated with variation of Hodge structures on subvarieties.
   However, the vanishing \Cref{eq:vanishing_costalk} implies that the complex ${}^{\mathrm{p}} \cH^0(\bar{p}_{F, *}  \mathcal{IC}_{V^{F}/\bar{L}_F})$ cannot have a direct summand with smaller support, hence we conclude.
\end{proof}

\begin{corollary}\label{cor-non-special-vanishing}
    We adopt the notation from \Cref{cor:small_decom_thm_smooth_stack}. Then for a non-degenerate face $(F, \alpha) \in  \mathsf{Face}^{\mathrm{nd}}(\mathcal{U})$, if $\alpha$ is not special (see \Cref{para-special-faces}), we have $\mathcal{BPS}^{\alpha}_{U} = 0$.
\end{corollary}

\begin{para}
    We note that \textcite[Theorem 1.1]{Meinhardt_Reineke_DT_vs_intersection_cohomology} proved a statement analogous to Corollary \ref{cor:small_decom_thm_smooth_stack} for certain moduli stacks of quiver representations.
    However, their definition of the BPS sheaf (which is defined in the Grothendieck ring of monodromic mixed Hodge modules, and which they call the Donaldson--Thomas function) is different from ours. Therefore Corollary \ref{cor:small_decom_thm_smooth_stack} itself is not a generalization of their theorem.
    However, we will prove the cohomological integrality theorem with our definition of the BPS sheaf under the orthogonality assumption later in \Cref{thm-cohint-smooth}, which eventually shows that Meinhardt and Reineke's definition of the BPS sheaf is equivalent to ours.  Hence we reprove their result.
\end{para}

\begin{corollary}\label{cor-BPS-split-pure}
    We adopt the notation from Corollary \ref{cor:small_decom_thm_smooth_stack}. Let $f \colon \mathcal{U} \to \mathbb{A}^1$ be a regular function and set $\mathcal{X} = \mathrm{Crit}(f)$ and equip it with the standard $(-1)$-shifted symplectic structure and the standard orientation.
    Let $p_{\mathcal{X}, \alpha} \colon \mathcal{X}_{\alpha} \to X_{\alpha}$ be the good moduli space morphism.
    Set $\mathcal{BPS}_{X}^{\alpha}  \coloneqq \bigoplus_{(F, \tilde{\alpha}) \in \mathsf{Face}(\mathcal{X})} \mathcal{BPS}_{X}^{\tilde{\alpha}} $ where the direct sum runs over all faces lifting the face $(F, \alpha)$ of $\mathcal{U}$.
    Then the following statements hold:
    \begin{enumerate}
        \item The BPS sheaf is Verdier self-dual: $\mathbb{D} \mathcal{BPS}_{X}^{\alpha} \cong \mathcal{BPS}_{X}^{\alpha}$.
        \item The natural map $\mathbb{L}^{\dim F/2} \otimes \mathcal{BPS}_X^{\alpha} \to p_{\mathcal{X}, \alpha, *}\varphi_{\mathcal{X}, \alpha}$ is a split injection.
    \end{enumerate}
\end{corollary}

\begin{proof}
    Set $f_{\alpha} \coloneqq f \circ \mathrm{tot}_{\alpha}$ and let $\bar{f}_{\alpha} \colon U_{\alpha} \to \mathbb{A}^1$ be the induced morphism.
    Then it follows from \Cref{eq-localize-standard-orientaion} and \Cref{eq-DT-Hodge-standard} that $\bigoplus_{(F, \tilde{\alpha})} \varphi_{\mathcal{X}_{\tilde{\alpha}}}$ is isomorphic to $\varphi_{f_{\alpha}}(\mathcal{IC}_{\mathcal{U}_{\alpha}})$, where the direct sum runs over faces lifting $\alpha$.
    Therefore by combining Corollary \ref{cor:small_decom_thm_smooth_stack} and \Cref{eq-vanishing-good-moduli}, we see that $\mathcal{BPS}_X^{\alpha}$ is either zero or isomorphic to $\varphi_{\bar{f}_{\alpha}}(\mathcal{IC}_{U_{\alpha}})$.
    Then the first statement follows from the Verdier self-duality of the vanishing cycle functor \Cref{eq-van-self-dual}.
    The latter statement follows from the fact that the natural map
    \[
        \mathbb{L}^{\dim F/2} \otimes \mathcal{BPS}_U^{\alpha} \to p_{ \alpha, *}\mathcal{IC}_{\mathcal{U}, \alpha} 
    \]
    is a split injection, which is a consequence of the purity of $p_{\alpha, *} \mathcal{IC}_{\mathcal{U}_{\alpha}}$ explained in \Cref{thm-decom-good-moduli}.
\end{proof}

\begin{para}
As a consequence of Corollary \ref{cor:small_decom_thm_smooth_stack}, we can prove a generalization of the support lemma \cite[Lemma 4.1]{davison_integrality_preprojective} established in [loc. cit.] for particular moduli stacks of Jacobi algebra representations.   
For this, we will introduce the notion of the cotangent rank.
Let $\mathcal{X}$ be a  derived algebraic stack locally finitely presented over $\mathbb{C}$ and $x \in \mathcal{X}$ be a closed point.
Let $G_x$ be the stabilizer group at $x$.
Then the cotangent rank 
\[
    \mathrm{cotrk}_{x}(\mathcal{X}) \in \mathbb{Z}_{\geq 0}
\] 
is the dimension of the maximal subtorus $T \subset G_x$ such that $T$ acts trivially on the vector spaces $\mathrm{H}^{-1}(\mathbb{T}_{\mathcal{X}, x})$ and $\mathrm{H}^{0}(\mathbb{T}_{\mathcal{X}, x})$.
\end{para}

\begin{proposition}\label{prop-support-lemma}
Let $\mathcal{X}$ be an almost symmetric oriented $(-1)$-shifted symplectic stack having affine diagonal admitting a good moduli space $p \colon \mathcal{X} \to X$.
Let $(F, \alpha)  \in \mathsf{Face}^\mathrm{nd} (\mathcal{X})$ be a non-degenerate face.
Then for each $i < \dim F$, we have ${}^{\mathrm{p}} \mathcal{H}^{i}(p_{\alpha, *} \varphi_{\mathcal{X}_{\alpha}}) = 0$.
Furthermore, for $x \in \mathcal{X}_{\alpha}$ a closed point satisfying $\mathrm{cotrk}_{x}(\mathcal{X}_{\alpha}) > \dim F$,
we have $\mathcal{BPS}^{\alpha}_{X} |_{p_{\alpha}(x)} = 0$.
\end{proposition}

\begin{proof}
    We will only prove the latter statement, since the former statement can be proved in an analogous manner.
    It is enough to prove the statement for the maximal central face $(F, \alpha) = (F_{\mathrm{ce}}, \alpha_{\mathrm{ce}})$ introduced in \Cref{para-central-rank}.
    By using \Cref{thm-Darboux}, we may assume that there exists a smooth affine scheme $V$ acted on by a reductive group $G$ and  a function $f \colon V/G \to \mathbb{A}^1$ such that  $\mathcal{X} = \mathrm{Crit}(f)$ and $G$ fixes  a point $x \in \mathrm{Crit}(f)$.
    Also, by the discussion in \Cref{para-any-orientation-is-locally-standard}, we may assume that $o$ is the standard orientation.
    In particular, we may assume $\varphi_{\mathcal{X}} \cong \varphi_f(\mathcal{IC}_{V / G})$ by \Cref{eq-DT-Hodge-standard}.
    Let $\tilde{p} \colon V / G \to V \GIT G$ be the good moduli space morphism.
    Using \Cref{eq-vanishing-good-moduli}, we obtain an isomorphism 
    $p_* \varphi_{\mathcal{X}} \cong \varphi_{f} (\tilde{p}_* \mathcal{IC}_{V / G})$.
    Therefore it is enough to prove the vanishing of ${}^{\mathrm{p}} \mathcal{H}^{\dim F}(\tilde{p}_* \mathcal{IC}_{V / G})$.
    For this, using the \'etale slice theorem, we may assume that $V$ is an affine space isomorphic to $\mathrm{T}_{V/G, x} = \mathrm{H}^0(\mathbb{T}_{V/G, x})$ and $x \in V /G$ is the origin. 
    Further, by possibly replacing $G$ by its neutral component $G^{\circ}$, we may assume that $G$ is connected.
    Using the description of the maximal central face of $V / G$ in \Cref{para-ex-linear-quotient}, we obtain the equality
    $\mathrm{cotrk}_{x}(\mathcal{X}) = \crk V / G$ hence an inequality $\dim F < \crk V / G$.
    Then the claim follows from Corollary \ref{cor:small_decom_thm_smooth_stack} applied to the maximal central face.
\end{proof}

Since we have $\mathrm{cotrk}_x(\mathcal X)\geq \mathrm{crk}(\mathcal X)$ for any point $x \in \mathcal{X}$, the above proposition implies:

\begin{corollary}\label{cor-non-special-vanishing--1shifted symplectic}
    We adopt the notation from \Cref{prop-support-lemma}. Then for a non-degenerate face $(F, \alpha) \in  \mathsf{Face}^{\mathrm{nd}}(\mathcal{X})$, if $\alpha$ is not special, we have $\mathcal{BPS}^{\alpha}_{X} = 0$.
\end{corollary}

\begin{para}[Support lemma for \texorpdfstring{$(-1)$}{(-1)}-shifted cotangent stacks]
    We will apply \Cref{prop-support-lemma} for $(-1)$-shifted cotangent stacks of $0$-shifted symplectic stacks.
    For this, we need the following lemma, which can be proved in a similar manner to \Cref{lem-loop-sends-closed-points}:
\end{para}

\begin{lemma}
    Let $\mathcal{Y}$ be a derived algebraic stack with affine diagonal, locally finitely presented over $\mathbb{C}$ and admitting a good moduli space $p \colon \mathcal{Y} \to Y$.
    Set $\mathcal{X} \coloneqq \mathrm{T}[-1]\mathcal{Y}$ and let $\pi \colon \mathcal{X} \to \mathcal{Y}$ be the projection.
    Then the map $\pi$ sends closed points to closed points.
\end{lemma}

The following lemma is crucial for the support lemma for $(-1)$-shifted cotangent stacks:

\begin{lemma}\label{lem-cotrk-inequality}
    Let $\mathcal{Y}$ be a derived algebraic stack locally finitely presented over $\mathbb{C}$ and set $\mathcal{X} \coloneqq \mathrm{T}[-1]\mathcal{Y}$.
    Let $\pi \colon \mathcal{X} \to \mathcal{Y}$ be the natural projection, and $y \in \mathcal{Y}$ and $x \in \pi^{-1}(y)$ be closed points.
    Let $G_y$ be the stabilizer group at $y$ and $\mathfrak{g}_y$ be its Lie algebra.
    Assume that $G_y$ is reductive, and under the equivalence $\pi^{-1}(\mathrm{B} G_y)_{\mathrm{cl}} \cong \mathfrak{g}_y / G_y$, $x$ acts non-trivially on $\mathrm{H}^{-1}(\mathbb{T}_{\mathcal{Y}, y})$ or $\mathrm{H}^{0}(\mathbb{T}_{\mathcal{Y}, y})$ under the Lie algebra action of $\mathfrak{g}_y$ on them.
    Then the following inequality holds:
    \[
        \mathrm{cotrk}_y(\mathcal{Y}) <  \mathrm{cotrk}_x(\mathcal{X})      .
    \]
\end{lemma}

\begin{proof}
    Consider the following fibre sequence
    \begin{equation}\label{eq-fibre-seq--1-shifted-tangent}
     \pi^* \mathbb{T}_{\mathcal{Y}}[-1] \to \mathbb{T}_{\mathcal{X}} \to   \pi^* \mathbb{T}_{\mathcal{Y}}.
    \end{equation}
    Let $\tilde{x} \in \mathfrak{g}_y$ be a lift of the point $x$ under the identification $\pi^{-1}(\mathrm{B} G_y)_{\mathrm{cl}} \cong \mathfrak{g}_y / G_y$.
    The point $x$ being closed implies that $\tilde{x}$ is a semisimple element.
    Let $T' \subset G_{y}$ be the  torus which is maximal among those acting trivially on $\mathrm{H}^{-1}(\mathbb{T}_{\mathcal{Y}, y})$ and $\mathrm{H}^{0}(\mathbb{T}_{\mathcal{Y}, y})$.
    Since $\tilde{x}$ is semisimple, by possibly replacing $\tilde{x}$ with a conjugate, we may assume that $\mathrm{Lie}(T')$ and  $\tilde{x}$ are contained in the same Cartan subalgebra.
    Consider a torus $\tilde{T}'$ which is generated by $T'$ and $\exp(z \cdot \tilde{x})$ for $z \in \mathbb{C}$.
    Then $\tilde{T}'$ is contained in the stabilizer group of $x$, and $\tilde{T}'$ acts trivially on $\mathfrak{g}_{x}$.
    By assumption, we have an inequality
    \[
        \mathrm{cotrk}_y(\mathcal{Y}) = \dim T' < \dim \tilde{T}'.
    \]
    We will show that $\tilde{T}'$ acts trivially on ${\mathrm{H}^{-1}(\mathbb{T}_{\mathcal{X}, x})}$ and ${\mathrm{H}^{0}(\mathbb{T}_{\mathcal{X}, x})}$, which implies the statement, since we have $\dim \tilde{T}' \leq \mathrm{cotrk}_x(\mathcal{X})$.

    Consider the following diagram
    \[\begin{tikzcd}
        0 & {\mathrm{H}^{-1}(\mathbb{T}_{\mathcal{X}, x})} & {\mathrm{H}^{-1}(\mathbb{T}_{\mathcal{Y}, y})} & {\mathrm{H}^{-1}(\mathbb{T}_{\mathcal{Y}, y})} & {\mathrm{H}^{0}(\mathbb{T}_{\mathcal{X}, x})} & {\mathrm{H}^{0}(\mathbb{T}_{\mathcal{Y}, y})} \\
        0 & {\mathfrak{g}_{x}} & {\mathfrak{g}_{y}} & {\mathfrak{g}_{y}}
        \arrow[from=1-1, to=1-2]
        \arrow[from=1-2, to=1-3]
        \arrow["\cong", from=1-2, to=2-2]
        \arrow["\delta", from=1-3, to=1-4]
        \arrow["\cong", from=1-3, to=2-3]
        \arrow[from=1-4, to=1-5]
        \arrow["\cong", from=1-4, to=2-4]
        \arrow["\pi_*", from=1-5, to=1-6]
        \arrow[from=2-1, to=2-2]
        \arrow[from=2-2, to=2-3]
        \arrow[from=2-3, to=2-4]
    \end{tikzcd}\]
    where the upper sequence is the long exact sequence corresponding to the fibre sequence \Cref{eq-fibre-seq--1-shifted-tangent} and $\delta$ is the boundary map.
    This diagram shows that $\coker(\delta)$ is isomorphic to $\mathfrak{g}_{x}$ as a $\tilde{T}'$-representation.
    In particular, it is enough to show that $\tilde{T}'$ acts trivially on the image of the map $\pi_* \colon \mathrm{H}^0(\mathbb{T}_{\mathcal{X}, x}) \to \mathrm{H}^0(\mathbb{T}_{\mathcal{Y}, y})$.
    To see this, take a map
    \[
    t \colon \Spec k[\epsilon] / \epsilon^2 \to \mathcal{X}    
    \]
    representing a section $[t] \in \mathrm{H}^0(\mathbb{T}_{\mathcal{X}, x}) $.
    Using the definition $\mathcal{X} = \mathrm{T}[-1] \mathcal{Y}$ and the universal property of the tangent complex, we see that the map $t$ corresponds to a map
    \[
    \tilde{t} \colon \Spec k[\epsilon, \eta] / \epsilon^2 \to \mathcal{Y}    
    \]
    with $\deg(\eta) = 1$ and the restriction of $\tilde{t}$ to  $\Spec k[\eta]$ corresponds to $\tilde{x}$.
    This implies that the action of $\tilde{x}$ on $\pi_*[t]$ is trivial. Therefore $\tilde{T}'$ acts trivially on the image of $\pi_*$ as desired.
\end{proof}

\begin{para}
    We will apply the above lemma for a $0$-shifted symplectic stack $\mathcal{Y}$ using the identification $\mathrm{T}[-1]\mathcal{Y} \cong \mathrm{T}^*[-1]\mathcal{Y}$.
    Assume that $\mathcal{Y}$ has affine diagonal and admits a good moduli space $p \colon \mathcal{Y} \to Y$.
    Set $\mathcal{X} \coloneqq \mathrm{T}^*[-1]\mathcal{Y}$ and let $\tilde{p} \colon \mathcal{X} \to X$ be the good moduli space.
    Since the map $\mathrm{Grad}^n(\mathcal{X}) \to \mathrm{Grad}^n(\mathcal{Y})$ is an $\mathbb{A}^1$-deformation retract, we have $\mathrm{CL}_{\mathbb{Q}}(\mathcal{X}) \cong \mathrm{CL}_{\mathbb{Q}}(\mathcal{Y}) $, hence $\mathsf{Face}(\mathcal{X}) \cong \mathsf{Face}(\mathcal{Y})$.
    For a non-degenerate face $(F, \alpha) \in \mathsf{Face}(\mathcal{X}) \cong \mathsf{Face}(\mathcal{Y})$, the $\mathrm{B} \mathbb{G}_{\mathrm{m}}^{\dim F}$-action on $\mathcal{Y}_{\alpha}$ induces a map $\mathcal{O}_{{\mathcal{Y}}_{\alpha}}^{\oplus \dim F} \to \mathbb{T}_{\mathcal{Y}_{\alpha}}[-1]$,
    which defines a closed immersion $\mathbb{A}^{\dim F} \times \mathcal{Y}_{\alpha} \hookrightarrow \mathcal{X}_{\alpha}$.
    By passing to the good moduli space, we obtain a closed immersion
    \[
     \iota_{\alpha} \colon \mathbb{A}^{\dim F} \times Y_{\alpha} \hookrightarrow X_{\alpha}.    
    \]
\end{para}
The following is our generalization of the support lemma.
\begin{theorem}\label{thm-BPS-cotangent-product}
    There exist an $\mathrm{Aut}(\alpha)$-equivariant monodromic mixed Hodge module 
    \[
        \mathcal{BPS}_{Y}^{\alpha} \in \mathsf{MMHM}(Y_{\alpha})
    \]
    and an isomorphism
    \begin{equation}\label{eq-BPS-cotangent-product}
     \mathcal{BPS}_{X}^{\alpha} \cong \iota_{\alpha, *}(\mathcal{IC}_{\mathbb{A}^{\dim F}} \boxtimes \mathcal{BPS}_{Y}^{\alpha} ).    
    \end{equation}
    Furthermore, $\mathcal{BPS}_{Y}^{\alpha}$ is pure and monodromy-free, i.e., it is contained in the subcategory $ \mathsf{MHM}(Y_{\alpha}) \subset  \mathsf{MMHM}(Y_{\alpha}) $.
\end{theorem}

\begin{proof}
    We first prove the former statement without $\mathrm{Aut}(\alpha)$-equivariance.
    To prove the $\mathbb{A}^{\dim F}$-equivariance of $\mathcal{BPS}_{X}^{\alpha}$,
    using \Cref{lemma-A^n-equivariance}, it is enough to prove that, for each $\xi \in \mathbb{A}^{\dim F}$, the action map
    $a_{\xi} \colon \mathcal{X}_{\alpha} \cong \mathcal{X}_{\alpha}$ preserves the $(-1)$-shifted symplectic structure.
    This is already proved in Lemma \ref{lem-symplectic-preservation}.
    Therefore it is enough to show that the support of $\mathcal{BPS}_{X}^{\alpha}$ is contained in the image of $\iota_{\alpha}$.
    Let $x \in \mathcal{X}_{\alpha}$ be a closed point such that $\tilde{p}_{\alpha}(x)$ is not contained in the image of $\iota_{\alpha}$ and set $y = \pi_{\alpha}(x) \in \mathcal{Y}_{\alpha}$.
    Then by using \Cref{prop-support-lemma}, it is enough to prove $\mathrm{cotrk}_{x}(\mathcal{X}_{\alpha}) > \dim F$.
    Using \Cref{lem-cotrk-inequality}, we are reduced to proving either $\mathrm{cotrk}_{y}(\mathcal{Y}_{\alpha}) > \dim F$ or $x$ acts non-trivially on $\mathrm{H}^{-1}(\mathbb{T}_{\mathcal{Y}_{\alpha}, y})$ or $\mathrm{H}^{0}(\mathbb{T}_{\mathcal{Y}_{\alpha}, y})$ under the identification  $\pi_{\alpha}^{-1}(y)_{\mathrm{cl}} \cong \mathfrak{g}_y / G_y$.
    If $\mathrm{cotrk}_{y}(\mathcal{Y}_{\alpha}) = \dim F$ holds, then any element in $\mathfrak{g}_{y}$ not contained in $\mathbb{A}^{\dim F}$ acts non-trivially on $\mathrm{H}^{-1}(\mathbb{T}_{\mathcal{Y}_{\alpha}, y})$ or $\mathrm{H}^{0}(\mathbb{T}_{\mathcal{Y}_{\alpha}, y})$. Therefore we obtain the desired claim.

    Now we will prove that there exists a natural way to equip $\mathcal{BPS}_{Y}^{\alpha}$ with an $\mathrm{Aut}(\alpha)$-equivariant structure which upgrades \Cref{eq-BPS-cotangent-product} to an isomorphism of $\mathrm{Aut}(\alpha)$-equivariant monodromic mixed Hodges modules.
    Take an element $\gamma \in \mathrm{Aut(\alpha)}$. Then we have a natural isomorphism
    \[
        \Hom(\gamma^* (\mathcal{IC}_{\mathbb{A}^{\dim F}} \boxtimes \mathcal{BPS}_{Y}^{\alpha}), \mathcal{IC}_{\mathbb{A}^{\dim F}} \boxtimes \mathcal{BPS}_{Y}^{\alpha}) \cong \Hom(\gamma^*\mathcal{BPS}_{Y}^{\alpha}, \mathcal{BPS}_{Y}^{\alpha}).
    \]
    This shows that the sets of $\mathrm{Aut}(\alpha)$-equivariant structures on $\mathcal{IC}_{\mathbb{A}^{\dim F}} \boxtimes \mathcal{BPS}_{Y}^{\alpha}$ and those on $\mathcal{BPS}_{Y}^{\alpha}$ are isomorphic.
    Therefore the $\mathrm{Aut(\alpha)}$-equivariant structure on  $\mathcal{BPS}_{X}^{\alpha}$ induces an $\mathrm{Aut}(\alpha)$-equivariant structure on $\mathcal{BPS}_{Y}^{\alpha}$.

    We now show that $\mathcal{BPS}_{Y}^{\alpha}$ is pure. Since purity can be checked locally, using \cite[Lemma 4.1.8]{park2024shifted}, we may assume that 
    there exists a reductive group $G$ acting on an smooth affine scheme $U$ with a fixed point $u \in U$, a $G$-equivariant vector bundle $E$ on $U$ and a $G$-equivariant section $s \in \Gamma(U, E)$ with $s(u) = 0$ such that 
    $\mathcal{Y} \cong s^{-1}(0) / G $ holds. Let $s^{\vee} \colon E^{\vee} \to \mathbb{A}^1$ be the cosection corresponding to $s$.
    Then we have an equivalence of $(-1)$-shifted symplectic stacks
    \[
     \mathrm{Crit}(s^{\vee} / G) \cong \mathrm{T}^*[-1] \mathcal{Y}.   
    \]
    Let $o_1$ and $o_2$ be the standard orientations defined using the above two critical locus descriptions.
    It is clear that the underlying line bundles restricted to $\{ u \} / G $ are isomorphic.
    Therefore by using \cite[Theorem 10.3]{_Alper_GoodmodulispacesforArtinstacks} and \cite[Lemma 6.8]{bellamy2016symplectic}, we may assume that $o_1$ and $o_2$ are isomorphic over a saturated Zariski neighborhood of $u$ in $\mathcal{X}$.
    Using the fact that orientations are classified by the topological data $\mathrm{H}^1(-, \mu_2)$, by possibly shrinking $\mathcal{Y}$, we may assume that $o_1$ and $o_2$ are isomorphic.
    Therefore by using \Cref{cor-BPS-split-pure} together with the dimensional reduction isomorphism \Cref{eq-dimensional-reduction}, we may assume that $\mathcal{BPS}_{Y}^{\alpha}$ is Verdier self-dual and there exists a split injection
    \[
     \mathcal{BPS}_{Y}^{\alpha} \hookrightarrow \mathbb{L}^{\vdim \mathcal{Y}_{\alpha} / 2} \otimes  p_{\alpha, *} \mathbb{D} \mathbb{Q}_{\mathcal{Y}_{\alpha}}.    
    \]
    Existence of this embedding implies that $\mathcal{BPS}_{Y}^{\alpha} $ has weight $\geq 0$.
    Since  $\mathcal{BPS}_{Y}^{\alpha}$ is Verdier self-dual, it must be pure of weight zero.
    Also, since $\vdim \mathcal{Y}_{\alpha}$ is an even number, as $\mathcal{Y}_{\alpha}$ is a $0$-shifted symplectic stack,
    the above embedding implies that $\mathcal{BPS}_{Y}^{\alpha}$ is monodromy-free.
\end{proof}

\begin{lemma}\label{lemma-A^n-equivariance}
    Let $\mathcal{X}$ be an Artin stack with an $\mathbb{A}^n$-action and $\mathcal{F} \in \Dhc^{(b)}(\mathcal{X})$ be a monodromic mixed Hodge complex.
    For each $\xi \in \mathbb{A}^n$, we let $a_{\xi} \colon \mathcal{X} \cong \mathcal{X}$ be the action map.
    Assume that there exists a local system $\mathcal{L}_{\xi}$ on $\mathcal{X}$ for each $\xi$ such that 
    $a_{\xi}^* \mathcal{F} \cong \mathcal{F} \otimes \mathcal{L}_{\xi}$ holds.
    Then $\mathcal{F}$ is equivariant with respect to the $\mathbb{A}^n$-action.
\end{lemma}

\begin{proof}
    Let $p \colon \mathcal{X} \to \mathcal{X} / \mathbb{A}^n$ be the quotient map.
    It is enough to show that the unit map
    \[
    \mathcal{F} \to p^! p_! \mathcal{F}    
    \]
    is an isomorphism. In particular, it is enough to prove the statement at the level of constructible complexes.
    Also, using the base change theorem, we may assume $\mathcal{X} = \mathbb{A}^n$.
    Further, by presenting $\mathcal{F}$ as an iterated extension of shifted constructible sheaves,
    we may assume that $\mathcal{F}$ is contained in the heart of the standard t-structure.
    In this case, the assumption implies that $\mathcal{F}$ is locally constant.
    Since $\mathbb{A}^n$ is simply connected, we obtain $\mathcal{F} \cong \mathbb{Q}_{\mathbb{A}^n}^d$ for some $d$ as desired.
\end{proof}

\subsection{Wall-crossing formula}

\begin{para}
    In this section, we will prove a version of the wall-crossing formula for the BPS sheaves, which can be regarded as a generalization of the wall-crossing formula for Gopakumar--Vafa invariants proved by \textcite[Theorem 5.7]{toda_gopakumar_wall_crossing}.
\end{para}

\begin{para}[Semistable points]
    Here, we recall the notion of semistable points for stacks following \cite[\S 4.1]{halpern2020derived} which is needed to discuss the wall-crossing formula.
    
    Let $\mathcal{X}$ be a derived algebraic stack with affine diagonal admitting a good moduli space $\mathcal{X} \to X$.
    Take a line bundle $\mathcal{L} \in \mathrm{Pic}(\mathcal{X})$.  We say that a point $x \in \mathcal{X}$ is \emph{unstable} if there exists a map $f \colon \mathbb{A}^1 / \mathbb{G}_\mathrm{m} \to \mathcal{X}$ with $f(1) \simeq x$ such that $f^* \mathcal{L} |_{\mathrm{B} \mathbb{G}_{\mathrm{m}}}$ has a negative weight,
    and $x$ is called \emph{semistable} otherwise.
    When $\mathcal{X}$ is of the form $\Spec A / G$ for a reductive group $G$, this recovers the notion of semistable points in classical geometric invariant theory.
    
    It is shown in \cite[Theorem 4.1.3]{halpern2020derived} that the semistable locus forms an open substack $\mathcal{X}^{\mathrm{ss}} \subset \mathcal{X}$ and that $\mathcal{X}^{\mathrm{ss}}$ admits a good moduli space $p^{\mathrm{ss}} \colon \mathcal{X}^{\mathrm{ss}} \to X^{\mathrm{ss}}$ with $X^{\mathrm{ss}}$ projective over $X$.
    In the next few paragraphs, we will compare the BPS sheaves on $X^{\mathrm{ss}}$ and $X$ when $\mathcal{X}$ is almost symmetric and $\mathcal{X}$ is either smooth, $(-1)$-shifted symplectic or $0$-shifted symplectic.
\end{para}

\begin{para}[Wall-crossing formula: smooth case]\label{para-wall-crossing-smooth}
    Let $\mathcal{U}$ be a connected almost symmetric smooth algebraic stack having affine diagonal and admitting a good moduli space $p \colon \mathcal{U} \to U$.
    Let $T$ be the torus which is maximal among those tori such that $\mathrm{B} T$ admits a non-degenerate action on $\mathcal{U}$.
    Such a $\mathrm{B} T$-action corresponds to the maximal non-degenerate face $(F_{\mathrm{ce}}, \alpha_{\mathrm{ce}})$ of $\mathcal{U}$ recalled in \Cref{para-central-rank}. 
    Let $\hat{\alpha}_{\mathrm{ce}}$ denote the lift of $\alpha_{\mathrm{ce}}$ to $\mathcal{U}^{\mathrm{ss}}$, if $\mathcal{U}^{\mathrm{ss}}$ is non-empty.
    Take a line bundle $\mathcal{L}$ on $\mathcal{U}$ pulled back from $\mathcal{U} / \mathrm{B} T$ and let $\mathcal{U}^{\mathrm{ss}}$ be the semistable locus with respect to $\mathcal{L}$.
    Consider the following commutative diagram
\[\begin{tikzcd}
	{\mathcal{U}^{\mathrm{ss}}} & {\mathcal{U}} \\
	{U^{\mathrm{ss}}} & {U}
	\arrow["\iota", from=1-1, to=1-2]
	\arrow["{p^{\mathrm{ss}}}"', from=1-1, to=2-1]
	\arrow["p", from=1-2, to=2-2]
	\arrow["q"', from=2-1, to=2-2]
\end{tikzcd}\]
    where the vertical maps are good moduli space morphisms, $\iota$ is an open immersion and $q$ is a projective morphism.
    Since numerical symmetricity is inherited by open substacks, $\mathcal{U}^{\mathrm{ss}}$ is numerically symmetric and therefore almost symmetric by the discussion in \Cref{item-numerical-vs-almost} of \Cref{para-numerically-symmetric}.

    We have the following statement which can be regarded as a wall-crossing formula for BPS sheaves, as it shows that the cohomology of the BPS sheaf of the semistable locus does not depend on the choice of a stability condition satisfying the above hypotheses:
\end{para}

\begin{proposition}\label{prop-smooth-wall-crossing}
    We adopt the notations from the last paragraph. Then there exists a natural isomorphism
    \begin{equation}\label{eq-smooth-wall-crossing}
      \mathcal{BPS}_{U} \cong q_{*} \mathcal{BPS}_{U^{\mathrm{ss}}}^{\hat{\alpha}_{\mathrm{ce}}}
    \end{equation}
    if $\mathcal{U}^{\mathrm{ss}}$ is not the empty set. Otherwise, we have $\mathcal{BPS}_{U}  \cong 0$.
\end{proposition}

\begin{proof}
    Assume first that the map $\bar{p} \colon \mathcal{U} / \mathrm{B} T  \to U$ is not generically quasi-finite.
    In this case, there exists an open subset $\mathcal{V} \subset \mathcal{U} / \mathrm{B} T$ of points whose stabilizers are positive dimensional.
    Therefore either the map $\bar{p}^{\mathrm{ss}}$ is not generically quasi-finite or $\mathcal{U}^{\mathrm{ss}}$ is the empty set.
    In particular,  \Cref{cor:small_decom_thm_smooth_stack} implies that the statement holds in this case.

    Now assume that the map $\bar{p}$ is generically quasi-finite.
    We first claim that $\mathcal{U}^{\mathrm{ss}}$ is non-empty.
    To see this, let $x \in \mathcal{U}$ be a point such that the map $\bar{p}$ is finite at the image $\bar{x} \in \mathcal{U} / \mathrm{B} T$.
    Note that by \cite[Lemma 1.3.14]{_HalpernLeistner_Onthestructureofinstabilityinmodulitheory} any filtered point $\bar{f} \colon \mathbb{A}^1 / \mathbb{G}_{\mathrm{m}} \to \mathcal{U} / \mathrm{B} T$ with $\bar{f}(1) \simeq \bar{x}$ factors through the constant map $\mathbb{A}^1 / \mathbb{G}_{\mathrm{m}} \to \mathrm{pt}$.
    Therefore any filtered point $f \colon \mathbb{A}^1 / \mathbb{G}_{\mathrm{m}} \to \mathcal{U}$  with ${f}(1) \simeq {x}$ factors  through $\mathrm{B} T \times \{ x \}$.
    Since the restriction of $\mathcal{L}$ to $\mathrm{B} T \times \{ x \}$ is trivial, we conclude that $x$ is semistable, hence $\mathcal{U}^{\mathrm{ss}}$ is non-empty.
    By \Cref{cor:small_decom_thm_smooth_stack}, we have isomorphisms
    \begin{equation*}
        \mathcal{BPS}_{U} \cong \mathcal{IC}_U, \quad  \mathcal{BPS}_{U^{\mathrm{ss}}}^{\hat{\alpha}_{\mathrm{ce}}} \cong \mathcal{IC}_{U^{\mathrm{ss}}}.
    \end{equation*}
    Set $n \coloneqq \crk \mathcal{X} = \dim T$.
    Consider the following zig-zag diagram:
    \begin{align}\label{eq-zig-zag-smooth}
        \begin{aligned}
        \mathcal{IC}_{U} \cong {}^{\mathrm{p}} \mathcal{H}^0(\mathbb{L}^{- n / 2} \otimes p_* \mathcal{IC}_{\mathcal{U}}) \to \mathcal{H}^0(\mathbb{L}^{- n / 2} \otimes q_* p^{\mathrm{ss}}_* \mathcal{IC}_{\mathcal{U}^{\mathrm{ss}}}) 
        &\leftarrow {}^{\mathrm{p}} \mathcal{H}^0(q_* {}^{\mathrm{p}}  \mathcal{H}^0(\mathbb{L}^{- n / 2} \otimes p^{\mathrm{ss}}_* \mathcal{IC}_{\mathcal{U}^{\mathrm{ss}}})) \\
        & \cong  {}^{\mathrm{p}} \mathcal{H}^0(q_*   \mathcal{IC}_{U^{\mathrm{ss}}}).
        \end{aligned}
    \end{align}
    We claim that the two maps are isomorphisms and $q_*   \mathcal{IC}_{U^{\mathrm{ss}}}$ is concentrated in the heart of the perverse t-structure, which implies the desired isomorphism.
    Since the map $q \circ p^{\mathrm{ss}} = p \circ \iota$ is small, the proof of \Cref{cor:small_decom_thm_smooth_stack} implies
    \[
    {}^{\mathrm{p}} \mathcal{H}^0(\mathbb{L}^{- n /2} \otimes q_* p^{\mathrm{ss}}_*  \mathcal{IC}_{\mathcal{U}^{\mathrm{ss}}}) \cong \mathcal{IC}_U, \quad {}^{\mathrm{p}} \mathcal{H}^i (\mathbb{L}^{- n/2} \otimes q_* p^{\mathrm{ss}}_*  \mathcal{IC}_{\mathcal{U}^{\mathrm{ss}}}) \cong 0 \ (i < 0).
    \]
    The first isomorphism implies that the forward map in \Cref{eq-zig-zag-smooth} is an isomorphism.
    Since the backward map in \Cref{eq-zig-zag-smooth} is a split injection, the latter isomorphism implies that the complex $q_* \mathcal{IC}_{U^{\mathrm{ss}}} $ is concentrated in non-negative perverse degrees.
    The Verdier self-duality of $q_* \mathcal{IC}_{U^{\mathrm{ss}}} $  implies that it is in fact contained in the heart of the perverse t-structure.
    Since $q$ is generically quasi-finite, ${}^{\mathrm{p}} \mathcal{H}^0(q_*   \mathcal{IC}_{U^{\mathrm{ss}}})$ is non-zero, hence the backward map in \Cref{eq-zig-zag-smooth} is also an isomorphism.
\end{proof}

\begin{para}[Wall-crossing formula: \texorpdfstring{($-1)$}{$(-1)$}-shifted symplectic case]
    Let $\mathcal{X}$ be a connected almost symmetric  $(-1)$-shifted symplectic stack having affine diagonal and admitting a good moduli space $p \colon \mathcal{X} \to X$. 
    Take a torus $T$, a line bundle $\mathcal{L}$ on $\mathcal{X}$ and a face $\hat{\alpha}_{\mathrm{ce}}$ for $\mathcal{X}^{\mathrm{ss}}$ as in  \Cref{para-wall-crossing-smooth}, and
    consider the following commutative diagram
\[\begin{tikzcd}
	{\mathcal{X}^{\mathrm{ss}}} & {\mathcal{X}} \\
	{X^{\mathrm{ss}}} & {X.}
	\arrow["\iota", from=1-1, to=1-2]
	\arrow["{p^{\mathrm{ss}}}"', from=1-1, to=2-1]
	\arrow["p", from=1-2, to=2-2]
	\arrow["q"', from=2-1, to=2-2]
\end{tikzcd}\]
    Since numerical symmetricity is inherited by open substacks, $\mathcal{X}^{\mathrm{ss}}$ is numerically symmetric and therefore almost symmetric by the discussion in \Cref{item-numerical-vs-almost} of \Cref{para-numerically-symmetric}.   
    We have the following wall-crossing formula for BPS sheaves:
\end{para}

\begin{proposition}\label{prop--1-symplectic-wall-crossing}
    We adopt the notations from the last paragraph. Then there exists a natural isomorphism
    \begin{equation}\label{eq--1-symplectic-wall-crossing}
      \mathcal{BPS}_{X} \cong q_{*} \mathcal{BPS}^{\hat{\alpha}_{\mathrm{ce}}}_{X^{\mathrm{ss}}}
    \end{equation}
    if $\mathcal{X}^{\mathrm{ss}}$ is not the empty set. Otherwise, we have $\mathcal{BPS}_{X}  \cong 0$.
\end{proposition}

\begin{proof}
    Set $n = \crk \mathcal{X} = \dim T$.
    Assume first that $\mathcal{X}^{\mathrm{ss}}$ is non-empty.
    Consider the following zig-zag diagram:
    \begin{align*}
        \mathcal{BPS}_{X} = {}^{\mathrm{p}} \mathcal{H}^0(\mathbb{L}^{- n / 2} \otimes p_* \varphi_{\mathcal{X}}) \to \mathcal{H}^0(\mathbb{L}^{- n / 2} \otimes q_* p^{\mathrm{ss}}_* \varphi_{\mathcal{X}^{\mathrm{ss}}}) 
        &\leftarrow {}^{\mathrm{p}} \mathcal{H}^0(q_* {}^{\mathrm{p}}  \mathcal{H}^0(\mathbb{L}^{- n / 2} \otimes p^{\mathrm{ss}}_* \varphi_{\mathcal{X}^{\mathrm{ss}}})) \\
        & =  {}^{\mathrm{p}} \mathcal{H}^0(q_*   \mathcal{BPS}^{\hat{\alpha}_{\mathrm{ce}}}_{X^{\mathrm{ss}}}).
    \end{align*}
    We claim that these morphisms are isomorphisms and $q_* \mathcal{BPS}^{\hat{\alpha}_{\mathrm{ce}}}_{X^{\mathrm{ss}}}$ is contained in the heart of the perverse t-structure.
    Note that this claim can be checked \'etale locally on $X$.
    Using \Cref{thm-Darboux} and the discussion in \Cref{para-any-orientation-is-locally-standard}, we may find jointly surjective orientation-preserving strongly \'etale symplectomorphisms
    \[
     \mathrm{Crit}   (f_i / G_i) \to \mathcal{X} 
    \]
    where $G_i$ is a reductive group acting on a connected smooth scheme $V_i$ such that $V_i / G_i$ is almost symmetric and $f_i$ is a $G_i$-invariant regular function on $V_i$.
    By possibly refining the \'etale cover and using \cite[Theorem 10.3]{_Alper_GoodmodulispacesforArtinstacks} and \cite[Lemma 6.8]{bellamy2016symplectic}, we may assume that the restriction $\mathcal{L} |_{\mathrm{Crit}   (f_i / G_i)}$ is pulled back from $\mathrm{B} G_i$.
    Also, we may assume that the $\mathrm{B} T$ action on $\mathrm{Crit}   (f_i / G_i)$ extends to a $\mathrm{B} T$-action on $V_i / G_i$.

    Assume first that there exists a torus $T'$ which contains $T$ as a proper subtorus and $\mathrm{B} T'$ admits a non-degenerate action on $V_i / G_i$. In this case, \Cref{prop-support-lemma} implies the vanishing 
    \[
        \mathcal{BPS}_{X} |_{\mathrm{Crit}(f)_{\GIT}} = 0 , \quad \mathcal{BPS}^{\hat{\alpha}_{\mathrm{ce}}}_{X^{\mathrm{ss}}} |_{\mathrm{Crit}(f)_{\GIT}^{\mathrm{ss}}} = 0
    \]
    hence we obtain the desired claim.
    Therefore we may assume that $T$ is the torus which is maximal among those tori such that $\mathrm{B} T$ admits a non-trivial action on $V_i / G_i$.
    We let $\mathcal{L}_i \in \mathrm{Pic}(V_i / G_i)$ be the line bundle which restricts to $\mathcal{L} |_{\mathrm{Crit}   (f_i / G_i)}$ and consider the following diagram 
    \[\begin{tikzcd}
        {V_i^{\mathrm{ss}} / G_i} & {V_i / G_i} \\
        {V_i^{\mathrm{ss}} \GIT G_i} & {V_i \GIT G_i }
        \arrow["\tilde{\iota}_i", from=1-1, to=1-2]
        \arrow["{\tilde{p}_i^{\mathrm{ss}}}"', from=1-1, to=2-1]
        \arrow["\tilde{p}_i", from=1-2, to=2-2]
        \arrow["\tilde{q}_i"', from=2-1, to=2-2]
    \end{tikzcd}\]
    where the semistable locus is with respect to the line bundle $\mathcal{L}_i$ and the vertical maps are the good moduli space morphisms.
    \Cref{prop-smooth-wall-crossing} implies a natural isomorphism 
    \begin{equation}\label{eq-wall-crossing-Vi}
        \mathcal{BPS}_{V_i \GIT G_i} \cong  \tilde{q}_{i, *} \mathcal{BPS}^{\hat{\alpha}_{\mathrm{ce}}}_{V^{\mathrm{ss}}_i \GIT G_i}.
    \end{equation}
    Let $\bar{f}_i \colon V_i \GIT G_i \to \mathbb{A}^1$ be the function induced from $f_i$.
    By applying the vanishing cycle functor $\varphi_{\bar{f}_i}$ to the isomorphism \Cref{eq-wall-crossing-Vi}, we obtain the desired statement.

    If $\mathcal{X}^{\mathrm{ss}}$ is an empty set, a similar argument implies $\mathcal{BPS}_{X} \cong 0$ as desired.
\end{proof}

\begin{para}[Wall-crossing formula: \texorpdfstring{$0$}{$0$}-shifted symplectic case]
    Let $\mathcal{Y}$ be a connected almost symmetric $0$-shifted symplectic stack having affine diagonal and admitting a good moduli space $p \colon \mathcal{Y} \to Y$. 
    Take a torus $T$, a line bundle $\mathcal{L}$ on $\mathcal{Y}$ and a face $\hat{\alpha}_{\mathrm{ce}}$ as in  \Cref{para-wall-crossing-smooth}, and
    consider the following commutative diagram
    \[\begin{tikzcd}
	{\mathcal{Y}^{\mathrm{ss}}} & {\mathcal{Y}} \\
	{Y^{\mathrm{ss}}} & {Y.}
	\arrow["\iota", from=1-1, to=1-2]
	\arrow["{p^{\mathrm{ss}}}"', from=1-1, to=2-1]
	\arrow["p", from=1-2, to=2-2]
	\arrow["q"', from=2-1, to=2-2]
\end{tikzcd}\]
    Since numerical symmetricity is inherited by open substacks, $\mathcal{Y}^{\mathrm{ss}}$ is numerically symmetric and therefore almost symmetric by the discussion in \Cref{item-numerical-vs-almost} of \Cref{para-numerically-symmetric}.
    The following wall-crossing formula is an immediate consequence of \Cref{prop--1-symplectic-wall-crossing} and the definition of the BPS sheaf for $0$-shifted symplectic stacks:
\end{para}

\begin{proposition}
    We adopt the notations from the last paragraph. Then there exists a natural isomorphism
    \begin{equation}
      \mathcal{BPS}_{Y} \cong q_{*} \mathcal{BPS}_{Y^{\mathrm{ss}}}^{\hat{\alpha}_{\mathrm{ce}}}
    \end{equation}
    if $\mathcal{Y}^{\mathrm{ss}}$ is not the empty set. Otherwise, we have $\mathcal{BPS}_{Y}  \cong 0$.
\end{proposition}

\section{Cohomological Hall induction (CoHI)}

\subsection{Cohomological Hall induction (CoHI)}

\begin{para}
    The aim of this section is to introduce the cohomological Hall induction (CoHI for short) for smooth stacks, $(-1)$-shifted symplectic stacks and quasi-smooth derived algebraic stacks.
    The cohomological Hall induction is a generalization of the cohomological Hall algebra multiplication studied in \cite{_Kontsevich_Soibelman_CoHA,_Davison_Meinhardt_CoDT,Kapranov_Vasserot_CoHA,Kinjo_Park_Safronov_CoHA, PortaSala},
    and can be applied for stacks which do not necessarily arise as a moduli stack of objects in an abelian category.
\end{para}

\begin{para}[Assumptions]\label{para-Cohi-assumptions}
    Throughout this section, we work with derived algebraic stacks with the following assumptions:
    \begin{enumerate}
        \item $\mathcal{X}$ has affine diagonal and admits a good moduli space $p \colon \mathcal{X} \to X$. \label{item-cohI-goodmoduli}
        \item $\mathcal{X}$ has quasi-compact connected components and quasi-compact graded points. In particular, $\mathcal{X}$ has finite cotangent weights. \label{item-cohI-qcgr}
    \end{enumerate}
\end{para}

\begin{para}[Cohomological Hall induction for smooth stacks]\label{para-Cohi-for-smooth stacks}
    Let $\mathcal{U}$ be a smooth algebraic stack satisfying assumptions \Cref{item-cohI-goodmoduli} and \Cref{item-cohI-qcgr} in \Cref{para-Cohi-assumptions}.
    Take a non-degenerate face $(F, \alpha) \in \mathsf{Face}^{\mathrm{nd}}(\mathcal{U})$  and  a chamber $\sigma \subset F$ with respect to the cotangent arrangement.
    Since $\mathcal{U}$ is $\Theta$-reductive by \cite[Proposition 3.21]{_Alper_Existenceofmodulispacesforalgebraicstacks}, the constancy theorem \Cref{para-constancy-theorem} implies that the evaluation map $\mathrm{ev}_{1, \sigma} \colon \mathcal{U}_{\sigma}^+ \to \mathcal{U}$ is a proper morphism.
    Consider the following diagram:
    \begin{equation}\label{eq-smooth-cohi-diagram}
        \begin{tikzcd}
        & {\mathcal{U}_{\sigma}^+} \\
        {\mathcal{U}_{\alpha}} && {\mathcal{U}} \\
        {U_{\alpha}} && U.
        \arrow["{\mathrm{gr}_{\sigma}}"', from=1-2, to=2-1]
        \arrow["{\mathrm{ev}_{1, \sigma}}", from=1-2, to=2-3]
        \arrow["{p_{\alpha}}", from=2-1, to=3-1]
        \arrow["p"', from=2-3, to=3-3]
        \arrow["{g_{\alpha}}", from=3-1, to=3-3]
    \end{tikzcd}
\end{equation}
        Here the morphism $g_{\alpha}$ is induced by the natural map $\mathcal{U}_{\alpha} \to \mathcal{U}$.
        This diagram commutes by \cite[Lemma 2.2.2]{IbanezNunez_stratificationsgoodmodulistacks}.

        One can show that the map $g_{\alpha}$ in \Cref{eq-smooth-cohi-diagram} is a finite morphism. This is proved in \cite[Th\'eor\`eme]{luna1975adherences} when $\mathcal{U}$ is of the form $V/ G$ for a reductive group $G$ and an affine scheme $V$.
        The general case follows from the local structure theorem for stacks \cite[Theorem 4.12]{_Alper_ALunaetaleslicetheoremforalgebraicstacks} together with the assumption on the quasi-compact graded points.
        See \cite{IbanezNunez2024FinitenessGMS} for a more conceptual and direct proof of the finiteness of $g_{\alpha}$.

        Set $d_{\sigma} \coloneqq \dim \mathcal{U}_{\sigma}^+ - \dim \mathcal{U}$.
        We define the \emph{relative cohomological Hall induction} map by the composition
        \begin{equation}\label{eq-cohi-relative-smooth}
        *^{\mathcal{H}\mathrm{all}}_{\sigma} \colon g_{\alpha, *} p_{\alpha, *} \mathbb{Q}_{\mathcal{U}_{\alpha}}  \cong     g_{\alpha, *} p_{\alpha, *} \gr_{\sigma, *}  \mathbb{Q}_{\mathcal{U}_{\sigma}^+}  \cong p_* \ev_{1, \sigma, *} \mathbb{Q}_{\mathcal{U}_{\sigma}^+} \to \mathbb{L}^{d_{\sigma}} \otimes p_* \mathbb{Q}_{\mathcal{U}}
        \end{equation}
        where the first morphism is constructed via adjunction, and is an isomorphism due to the fact that $\gr_{\sigma}$ is an $\mathbb{A}^1$-deformation retract (as proved in \cite[Lemma 1.3.8]{_HalpernLeistner_Onthestructureofinstabilityinmodulitheory}) and the last map is the integration map.
        By taking global sections, we obtain the \emph{absolute cohomological Hall induction} map
        \[
            *^{\mathrm{Hall}}_{\sigma} \colon \mathrm{H}^*(\mathcal{U}_{\alpha}) \to \mathbb{L}^{d_{\sigma}} \otimes \mathrm{H}^{*}( \mathcal{U}).    
        \] 
        If $\mathcal{U}$ is almost symmetric, we have an equality 
        $\dim \mathcal{U}_{\alpha} - \dim \mathcal{U} = 2 \cdot (\dim \mathcal{U}_{\sigma}^+ - \dim \mathcal{U})$.
        In particular, the relative cohomological Hall induction induces a morphism
        \begin{equation}\label{eq-cohi-smooth-symmetric}
            *^{\mathcal{H}\mathrm{all}}_{\sigma} \colon g_{\alpha, *} p_{\alpha, *}  \mathcal{IC}_{\mathcal{U}_{\alpha}} \to   p_* \mathcal{IC}_{\mathcal{U}}.  
        \end{equation}
\end{para}

\begin{para}[Cohomological Hall induction for a finite quotient]\label{para-cohi-finite-quotient}
   Let $\tilde{\mathcal{U}}$ be a smooth stack satisfying assumptions \Cref{item-cohI-goodmoduli} and \Cref{item-cohI-qcgr} in \Cref{para-Cohi-assumptions}.
   Assume that a finite group $\Gamma$ acts on $\mathcal{U}$  and set $\mathcal{U} = \tilde{\mathcal{U}} / \Gamma$.
   We will relate the cohomological Hall induction for $\tilde{\mathcal{U}}$ and $\mathcal{U}$.
   Let $\tilde{p} \colon \tilde{\mathcal{U}} \to \tilde{U}$ be the good moduli space morphism. Then we have $U \cong \tilde{U} \GIT \Gamma$.
   Take a non-degenerate face $(F, \alpha) \in \mathsf{Face}^{\mathrm{nd}}(\mathcal{U})$ and its lift $(F, \tilde{\alpha}) \in \mathsf{Face}^{\mathrm{nd}}(\tilde{\mathcal{U}})$.
   We also take a chamber $\sigma \subset F$ with respect to the cotangent arrangement for $\mathcal{U}$ and denote by $\tilde{\sigma}$ the same cone regarded as an object in $\mathsf{Cone}(\tilde{\mathcal{U}})$.
   Consider the following diagram:
\[
    \begin{tikzcd}
	{\tilde{U}_{\tilde{\alpha}}} & {\tilde{U}} \\
	{U_{\alpha}} & U.
	\arrow["{g_{\tilde{\alpha}}}", from=1-1, to=1-2]
	\arrow["{r_{\alpha}}"', from=1-1, to=2-1]
	\arrow["r", from=1-2, to=2-2]
	\arrow["{g_{\alpha}}", from=2-1, to=2-2]
\end{tikzcd}
\]
Using the commutativity of the diagram \Cref{eq-Induction-diagram-for-finite-quotient} and the isomorphism 
\[
  \tilde{\mathcal{U}} \times_{\mathcal{U}} \mathcal{U}_{\sigma}^+  \cong \coprod_{\gamma \in \Gamma / \Gamma_{\alpha}} \tilde{\mathcal{U}}_{\gamma(\tilde{\sigma})}^+,   
\]
with $\Gamma_{\alpha}$ as defined in \Cref{para-graded-filtered-finite-quotient} we obtain the following commutative diagram
\begin{equation}\label{eq-cohi-finite-quotient-compatible}
    \begin{tikzcd}
	{g_{\alpha, *} p_{\alpha, *} \mathbb{Q_{\mathcal{U}_{\alpha}}}} &&[30pt] {\mathbb{L}^{d_{\sigma}} \otimes p_* \mathbb{Q}_{\mathcal{U}}} \\
	{\bigoplus_{\gamma \in \Gamma / \Gamma_{\alpha}} r_* g_{\gamma(\tilde{\alpha}), *}  p_{\gamma(\tilde{\alpha}), *} \mathbb{Q_{\tilde{\mathcal{U}}_{\gamma(\tilde{\alpha})}}}} && {\mathbb{L}^{d_{\sigma}} \otimes  r_* \tilde{p}_* \mathbb{Q}_{\tilde{\mathcal{U}}}.}
	\arrow["{*^{\mathcal{H}\mathrm{all}}_{\sigma}}", from=1-1, to=1-3]
	\arrow[from=1-1, to=2-1]
	\arrow[from=1-3, to=2-3]
	\arrow["{\bigoplus_{\gamma \in \Gamma / \Gamma_{\alpha}} \left( *^{\mathcal{H}\mathrm{all}}_{\gamma(\tilde{\sigma})} \right)}"', from=2-1, to=2-3]
\end{tikzcd}
\end{equation}
By taking the cohomology, we obtain an identity
\begin{equation}\label{eq-CoHI-finite-quotient}
    *^{\mathrm{Hall}}_{\sigma}(a) = \frac{1}{|\Gamma_{\alpha}|} \cdot \sum_{\gamma \in \Gamma} *^{\mathrm{Hall}}_{\gamma(\tilde{\sigma})}(\rho_{\gamma}(a))
\end{equation}
for an element $a \in \mathrm{H}^*(\mathcal{U}_{\alpha}) \cong \mathrm{H}^*( \tilde{\mathcal{U}}_{\tilde{\alpha}})^{\Gamma_{\alpha}} \subset \mathrm{H}^*( \tilde{\mathcal{U}}_{\tilde{\alpha}})$, where $\rho_{\gamma}$ denotes the natural isomorphism $\mathrm{H}^*( \tilde{\mathcal{U}}_{\tilde{\alpha}}) \cong \mathrm{H}^*( \tilde{\mathcal{U}}_{\gamma(\tilde{\alpha})})$ induced by $\gamma \in \Gamma$.
\end{para}

\begin{para}[Cohomological Hall induction for an \'etale cover]\label{para-cohi-etale-cover}
    Let ${\mathcal{U}}$ be a smooth stack satisfying assumptions \Cref{item-cohI-goodmoduli} and \Cref{item-cohI-qcgr} in \Cref{para-Cohi-assumptions}.
    Take an \'etale cover $\eta_{\GIT} \colon V \to U$ from an algebraic space and set $\mathcal{V} \coloneqq V \times_U \mathcal{U}$ and let $\eta \colon \mathcal{V} \to \mathcal{U}$ be the base change of $\eta_{\GIT}$.
    We will show that the cohomological Hall induction for $\mathcal{V}$ is compatible with the cohomological Hall induction for $\mathcal{U}$.

    Let $(F, \alpha) \in \mathsf{Face}^{\mathrm{nd}}(\mathcal{U})$ be a non-degenerate $n$-dimensional face and $\sigma \subset F$ be a chamber with respect to the cotangent arrangement.
    Let $\{ \tilde{\alpha}_1, \ldots, \tilde{\alpha}_l \}$ be the set of faces for $\mathcal{V}$ that lift $\alpha$, and $\tilde{\sigma}_i$ denote the element $(F, \tilde{\alpha}_i, \sigma) \in \mathsf{Cone}(\mathcal{V})$.
    Then we have isomorphisms
    \[
    \mathcal{V}_{\alpha} \cong \coprod_{i} \mathcal{V}_{\tilde{\alpha}_i}, \quad  \mathcal{V}_{\sigma}^+ \cong  \coprod_i  \mathcal{V}_{\tilde{\sigma}_i}^+.    
    \]

    By using the Cartesian diagram \Cref{eq-grad-filt-etale-cartesian}, we see that the following diagrams are Cartesian:
\begin{equation}\label{eq-component-Cartesian}
    \begin{tikzcd}
	{\mathcal{V}_{\alpha}} & {\mathcal{V}} & V \\
	{\mathcal{U}_{\alpha}} & {\mathcal{U}} & U,
	\arrow[ from=1-1, to=1-2]
	\arrow["{\eta_{\alpha}}", from=1-1, to=2-1]
	\arrow["\tilde{p}",from=1-2, to=1-3]
	\arrow["\eta", from=1-2, to=2-2]
	\arrow["{\eta_{\GIT}}", from=1-3, to=2-3]
	\arrow[from=2-1, to=2-2]
	\arrow["p",from=2-2, to=2-3]
\end{tikzcd}    \quad
\begin{tikzcd}
	{\mathcal{V}_{\sigma}^+} & {\mathcal{V}} & V \\
	{\mathcal{U}_{\sigma}^+} & {\mathcal{U}} & U.
	\arrow[from=1-1, to=1-2]
	\arrow["{\eta_{\sigma}}", from=1-1, to=2-1]
	\arrow["\tilde{p}", from=1-2, to=1-3]
	\arrow["\eta", from=1-2, to=2-2]
	\arrow["{\eta_{\GIT}}", from=1-3, to=2-3]
	\arrow[from=2-1, to=2-2]
	\arrow["p",from=2-2, to=2-3]
\end{tikzcd}
\end{equation}
Using the Cartesian property of the left diagram, we see that the following diagram is also Cartesian:
\[\begin{tikzcd}
	{\mathcal{V}_{\alpha}} & {U_{\alpha} \times_{U} V} \\
	{\mathcal{U}_{\alpha}} & {U_{\alpha}.}
	\arrow["{\tilde{p}_{\alpha}}", from=1-1, to=1-2]
	\arrow["{\eta_{\alpha}}", from=1-1, to=2-1]
	\arrow[from=1-2, to=2-2]
	\arrow["{p_{\alpha}}", from=2-1, to=2-2]
\end{tikzcd}\]
In particular, using \cite[Proposition 4.7]{_Alper_GoodmodulispacesforArtinstacks}, we see that $V_{\alpha} \coloneqq U_{\alpha} \times_{U} V$ is the good moduli space for $\mathcal{V}_{\alpha}$.
Now consider the following diagram:
\[\begin{tikzcd}
	& {\mathcal{V}_{\sigma}^+} \\
	{\mathcal{V}_{\alpha}} && {\mathcal{V}} &[40pt] & {\mathcal{U}_{\sigma}^+} \\
	{V_{\alpha}} && V & {\mathcal{U}_{\alpha}} && {\mathcal{U}} \\
	&&& {U_{\alpha}} && {U.}
	\arrow["{\tilde{\gr}_{\sigma}}"', from=1-2, to=2-1]
	\arrow["{\tilde{\ev}_{1, \sigma}}"{pos=0.8}, from=1-2, to=2-3]
	\arrow["{\eta_{\sigma}^+}", from=1-2, to=2-5]
	\arrow["{\tilde{p}_{\alpha}}",  from=2-1, to=3-1]
	\arrow["{\tilde{p}}"{pos=0.25}, from=2-3, to=3-3]
    \arrow["\eta"{pos=0.3}, from=2-3, to=3-6]
	\arrow["{\gr_{\sigma}}"{pos=0.9}, crossing over,  from=2-5, to=3-4]
	\arrow["{\ev_{1, \sigma}}", from=2-5, to=3-6]
	\arrow["{\tilde{g}_{\alpha}}", from=3-1, to=3-3]
	\arrow["{\eta_{\GIT, \alpha}}"', from=3-1, to=4-4]
	\arrow["{\eta_{\GIT}}"{pos=0.7}, from=3-3, to=4-6]
	\arrow["{p_{\alpha}}"'{pos=0.75}, crossing over, from=3-4, to=4-4]
	\arrow["p", from=3-6, to=4-6]
	\arrow["{g_{\alpha}}"', from=4-4, to=4-6]
    \arrow["{\eta_{\alpha}}"{pos=0.3}, crossing over, from=2-1, to=3-4]
\end{tikzcd}\]
By the construction of the cohomological Hall induction map together with the Cartesian properties of the diagrams \Cref{eq-component-Cartesian},
we see that the following diagram commutes:
\begin{equation}\label{eq-cohi-etale-cover}
    \begin{tikzcd}
        {\eta_{\GIT}^* g_{\alpha, *} p_{\alpha, *} \mathbb{Q}_{\mathcal{U}_{\alpha}}} && {\mathbb{L}^{d_{\sigma}} \otimes \eta_{\GIT} ^* p_{*} \mathbb{Q}_{\mathcal{U}}.} \\
        {\tilde{g}_{\alpha, *} \tilde{p}_{\alpha, *} \mathbb{Q}_{\mathcal{V}_{\alpha}}} \\
        {\bigoplus_{i} \tilde{g}_{\tilde{\alpha}_i, *} \tilde{p}_{\tilde{\alpha}_i, *} \mathbb{Q}_{\mathcal{V}_{\tilde{\alpha}_i}}} && {\mathbb{L}^{d_{\sigma}} \otimes  \tilde{p}_* \mathbb{Q}_{\mathcal{V}}.}
        \arrow["{\eta_{\GIT}^* \left(*^{\mathcal{H}\mathrm{all}}_{\sigma} \right)}", from=1-1, to=1-3]
        \arrow["\cong"', from=1-1, to=2-1]
        \arrow["\cong", from=1-3, to=3-3]
        \arrow["\cong"', from=2-1, to=3-1]
        \arrow["{\sum_{i} \left( *^{\mathcal{H}\mathrm{all}}_{\tilde{\sigma}_i} \right)}"', from=3-1, to=3-3]
    \end{tikzcd}
\end{equation}
\end{para}

\begin{para}[Cohomological Hall induction for \texorpdfstring{$(-1)$}{TEXT}-shifted symplectic stacks]\label{para-cohi--1-shifted-symplectic}
    The cohomological Hall algebra for $3$-Calabi--Yau dg-categories was introduced recently by \textcite{Kinjo_Park_Safronov_CoHA} based on the integral isomorphism \Cref{eq-integral-isomorphism}.
    We will explain that the same argument can be used to construct the cohomological Hall induction for $(-1)$-shifted symplectic stacks.

    Let $(\mathcal{X}, \omega_{\mathcal{X}}, o)$ be an oriented $(-1)$-shifted symplectic stack satisfying assumptions \Cref{item-cohI-goodmoduli} and \Cref{item-cohI-qcgr} in \Cref{para-Cohi-assumptions}.
    Take a non-degenerate face $(F, \alpha) \in \mathsf{Face}^{\mathrm{nd}}(\mathcal{X})$  and  a chamber $\sigma \subset F$ with respect to the cotangent arrangement.
    Consider the following diagram:
    \begin{equation}
        \begin{tikzcd}
        & {\mathcal{X}_{\sigma}^+} \\
        {\mathcal{X}_{\alpha}} && {\mathcal{X}} \\
        {X_{\alpha}} && X. 
        \arrow["{\mathrm{gr}_{\sigma}}"', from=1-2, to=2-1]
        \arrow["{\mathrm{ev}_{1, \sigma}}", from=1-2, to=2-3]
        \arrow["{p_{\alpha}}", from=2-1, to=3-1]
        \arrow["p"', from=2-3, to=3-3]
        \arrow["{g_{\alpha}}", from=3-1, to=3-3]
    \end{tikzcd}
\end{equation}
    We equip $\mathcal{X}_{\alpha}$ with the localized orientation $\sigma^{\star} o$ introduced in \Cref{para-localized-orientation-definition}.
    By taking the adjoint of the integral isomorphism \Cref{eq-integral-isomorphism} and pushing down to the good moduli space $X$, we obtain the relative cohomological Hall induction map
    \[
        *^{\mathcal{H}\mathrm{all}}_{\sigma}  \colon   \mathbb{L}^{- \vdim \mathcal{X}_{\sigma}^+ / 2} \otimes g_{\alpha, *} p_{\alpha, *}   \varphi_{\mathcal{X}_{\alpha}} \to p_* \varphi_{\mathcal{X}}.
    \]
    If $\mathcal{X}$ is further assumed to be almost symmetric, we have an identity $\vdim \mathcal{X}_{\sigma}^+ = 0$ by \Cref{eq-numerical-symmetric-implies-zero-dimension}, hence the cohomological Hall induction can be written as
    \begin{equation}\label{eq-critical-CoHI}
        *^{\mathcal{H}\mathrm{all}}_{\sigma}  \colon  g_{\alpha, *} p_{\alpha, *}   \varphi_{\mathcal{X}_{\alpha}} \to p_* \varphi_{\mathcal{X}}.    
    \end{equation}

    Now take an \'etale map $\eta_{\GIT} \colon Y \to X$ from an algebraic space and set $\mathcal{Y} \coloneqq \mathcal{X} \times_{X} Y$ and let $\eta \colon \mathcal{Y} \to \mathcal{X}$ be the base change.
    Then $\mathcal{Y}$ is naturally equipped with an oriented $(-1)$-shifted symplectic structure.
    As in \Cref{para-cohi-etale-cover}, consider the following diagram:
\[\begin{tikzcd}
	& {\mathcal{Y}_{\sigma}^+} \\
	{\mathcal{Y}_{\alpha}} && {\mathcal{Y}} &[40pt] & {\mathcal{X}_{\sigma}^+} \\
	{Y_{\alpha}} && Y & {\mathcal{X}_{\alpha}} && {\mathcal{X}} \\
	&&& {X_{\alpha}} && {X.}
	\arrow["{\tilde{\gr}_{\sigma}}"', from=1-2, to=2-1]
	\arrow["{\tilde{\ev}_{1, \sigma}}"{pos=0.8}, from=1-2, to=2-3]
	\arrow["{\eta_{\sigma}^+}", from=1-2, to=2-5]
	\arrow["{\tilde{p}_{\alpha}}",  from=2-1, to=3-1]
	\arrow["{\tilde{p}}"{pos=0.25}, from=2-3, to=3-3]
    \arrow["\eta"{pos=0.3}, from=2-3, to=3-6]
	\arrow["{\gr_{\sigma}}"{pos=0.9}, crossing over,  from=2-5, to=3-4]
	\arrow["{\ev_{1, \sigma}}", from=2-5, to=3-6]
	\arrow["{\tilde{g}_{\alpha}}", from=3-1, to=3-3]
	\arrow["{\eta_{\GIT, \alpha}}"', from=3-1, to=4-4]
	\arrow["{\eta_{\GIT}}"{pos=0.7}, from=3-3, to=4-6]
	\arrow["{p_{\alpha}}"'{pos=0.75}, crossing over, from=3-4, to=4-4]
	\arrow["p", from=3-6, to=4-6]
	\arrow["{g_{\alpha}}"', from=4-4, to=4-6]
    \arrow["{\eta_{\alpha}}"{pos=0.3}, crossing over, from=2-1, to=3-4]
\end{tikzcd}\]
Let $\{ \tilde{\alpha}_1, \ldots, \tilde{\alpha}_l \}$ be the set of faces for $\mathcal{Y}$ that lift $\alpha$ and $\tilde{\sigma}_i \subset F$ be the lift of $\sigma$ corresponding to $\tilde{\alpha}_i$.
Then we have isomorphisms
\[
\mathcal{Y}_{\alpha} \cong \coprod_{i} \mathcal{Y}_{\tilde{\alpha}_i}, \quad   \mathcal{Y}_{\sigma}^+ \cong \coprod_i \mathcal{Y}_{\tilde{\sigma}_i}^+.    
\]
It follows from the commutativity of the diagram \Cref{eq-integral-isom-etale} that the following diagram commutes:
\begin{equation}\label{eq-cohi-etale-cover-critical}
    \begin{tikzcd}
	{\eta_{\GIT}^* g_{\alpha, *} p_{\alpha, *} \varphi_{\mathcal{X}_{\alpha}}} && {\mathbb{L}^{\vdim \mathcal{X}_{\sigma}^+ / 2} \otimes \eta_{\GIT} ^* p_{*} \varphi_{\mathcal{X}}.} \\
	{\tilde{g}_{\alpha, *} \tilde{p}_{\alpha, *} \varphi_{\mathcal{Y}_{\alpha}}} \\
	{\bigoplus_{i} \tilde{g}_{\tilde{\alpha}_i, *} \tilde{p}_{\tilde{\alpha}_i, *} \varphi_{\mathcal{Y}_{\tilde{\alpha}_i}}} && {\mathbb{L}^{\vdim \mathcal{X}_{\sigma}^+ / 2} \otimes  \tilde{p}_* \varphi_{\mathcal{Y}}.}
	\arrow["{{\eta_{\GIT}^* \left(*^{\mathcal{H}\mathrm{all}}_{\sigma} \right)}}", from=1-1, to=1-3]
	\arrow["\text{\Cref{eq-dt-etale-cover}}"', "\cong", from=1-1, to=2-1]
	\arrow["\cong", "\text{\Cref{eq-dt-etale-cover}}"', from=1-3, to=3-3]
	\arrow["\cong"', from=2-1, to=3-1]
	\arrow["{{\sum_{i} \left( *^{\mathcal{H}\mathrm{all}}_{\tilde{\sigma}_i} \right)}}"', from=3-1, to=3-3]
\end{tikzcd}
\end{equation}

\end{para}

\begin{para}[Cohomological Hall induction for critical loci]\label{para-cohi-critical}
   Let $\mathcal{U}$ be a smooth stack  satisfying assumptions \Cref{item-cohI-goodmoduli} and \Cref{item-cohI-qcgr} in \Cref{para-Cohi-assumptions}.
   Take a non-degenerate face $(F, \alpha) \in \mathsf{Face}^{\mathrm{nd}}(\mathcal{U})$  and  a chamber $\sigma \subset F$ with respect to the cotangent arrangement.
   Let $f \colon \mathcal{U} \to \mathbb{A}^1$ be a regular function and set $f_{\alpha} \coloneqq f \circ \mathrm{tot}_{\alpha}$.
   By the universal property of the good moduli space, the maps $f$ and $f_{\alpha}$ descend to functions $\bar{f} \colon U \to \mathbb{A}^1$ and $\bar{f}_{\alpha} \colon U_{\alpha} \to \mathbb{A}^1$.
   Consider the following map
   \[
    \varphi_{\bar{f}}(*^{\mathcal{H}\mathrm{all}}_{\sigma}) \colon \varphi_{\bar{f}}(g_{\alpha, *} p_{\alpha, *} \mathbb{Q}_{\mathcal{U}_{\alpha}})   \to \mathbb{L}^{d_{\sigma}} \otimes \varphi_{\bar{f}}(p_* \mathbb{Q}_{\mathcal{U}}).
   \]
   Since $g_{\alpha}$ is a finite morphism and the pushforward maps along good moduli space morphisms commute with the vanishing cycle functor as shown in \Cref{eq-vanishing-good-moduli},
   this map is identified with the map
   \begin{equation}\label{eq-CoHI-global-critical-locus}
    g_{\alpha, *} p_{\alpha, *} \varphi_{f_{\alpha}}( \mathbb{Q}_{\mathcal{U}_{\alpha}}) \to \mathbb{L}^{d_{\sigma}} \otimes p_* \varphi_{f}(\mathbb{Q}_{\mathcal{U}}).
   \end{equation}
   Now set $\mathcal{X} \coloneqq \mathrm{Crit}(f)$ and $\mathcal{X}_{\alpha} \coloneqq \mathrm{Crit}(f_{\alpha})$ and equip them with the standard orientations.
   Then it follows from the commutativity of the diagram \Cref{eq-integral-isom-critical} that the following diagram commutes:
\begin{equation}\label{eq-cohi-critical}
    \begin{tikzcd}
	{\mathbb{L}^{(- \vdim \mathcal{X}_{\sigma}^+ - \dim \mathcal{U}_{\alpha} )/ 2} \otimes g_{\alpha, *} p_{\alpha, *} \varphi_{f_{\alpha}}( \mathbb{Q}_{\mathcal{U}_{\alpha}})} & {\mathbb{L}^{- \dim \mathcal{U} /2} \otimes p_* \varphi_{f}(\mathbb{Q}_{\mathcal{U}})} \\
	{\mathbb{L}^{- \vdim \mathcal{X}_{\sigma}^+ / 2} \otimes g_{\alpha, *} p_{\alpha, *}\varphi_{\mathcal{X}_{\alpha}}} & {p_{*} \varphi_{\mathcal{X}}.}
	\arrow["\text{\Cref{eq-CoHI-global-critical-locus}}", from=1-1, to=1-2]
	\arrow["\cong", "\text{\Cref{eq-DT-Hodge-standard}}"', from=1-1, to=2-1]
	\arrow["\cong", "\text{\Cref{eq-DT-Hodge-standard}}"', from=1-2, to=2-2]
	\arrow["\text{\Cref{eq-critical-CoHI}}"', from=2-1, to=2-2]
\end{tikzcd}\end{equation}

\end{para}

\begin{para}[Cohomological Hall induction for quasi-smooth derived algebraic stacks]\label{para-cohi-for-quasi-smooth}
    We now introduce the cohomological Hall induction for quasi-smooth derived algebraic stacks.
    Let $\mathcal{Y}$ be a quasi-smooth derived algebraic stack satisfying assumptions \Cref{item-cohI-goodmoduli} and \Cref{item-cohI-qcgr} in \Cref{para-Cohi-assumptions}.
    Set $\mathcal{X} \coloneqq \mathrm{T}^*[-1]\mathcal{Y}$ and equip it with the standard orientation.
    It is clear that $\mathcal{X}$ also satisfies the assumptions \Cref{item-cohI-goodmoduli} and \Cref{item-cohI-qcgr} in \Cref{para-Cohi-assumptions}.
    Take a non-degenerate face $(F, \alpha) \in \mathsf{Face}^{\mathrm{nd}}(\mathcal{Y})$  and  a chamber $\sigma \subset F$ with respect to the cotangent arrangement.
    As is shown in \Cref{eq-localize-standard-symplectic}, there exists a natural equivalence of $(-1)$-shifted symplectic stacks
    \[
     \mathcal{X}_{\alpha} \simeq \mathrm{T}^*[-1]\mathcal{Y}_{\alpha}.    
    \]
    Further, arguing as the proof of \Cref{eq-localize-standard-orientaion}, we see that this equivalence preserves the orientations.
    Let $\pi \colon \mathcal{X}\to \mathcal{Y}$ and $\pi_{\alpha} \colon \mathcal{X}_{\alpha} \to \mathcal{Y}_{\alpha}$ be the projections.
    Then the dimensional reduction theorem \Cref{eq-dimensional-reduction} gives isomorphisms
    \begin{equation}\label{eq-dimensional-reduction-for-grads}
        \pi_* \varphi_{\mathcal{X}} \cong \mathbb{L}^{\vdim \mathcal{Y} / 2} \otimes \mathbb{D}\mathbb{Q}_{\mathcal{Y}}, \quad \pi_{\alpha, *} \varphi_{\mathcal{X}_{\alpha}} \cong \mathbb{L}^{\vdim \mathcal{Y}_{\alpha} / 2} \otimes \mathbb{D}\mathbb{Q}_{\mathcal{Y}_{\alpha}}.
    \end{equation}
    Now consider the following commutative diagram: 
    \begin{equation}
        \begin{tikzcd}
        & {\mathcal{Y}_{\sigma}^+} \\
        {\mathcal{Y}_{\alpha}} && {\mathcal{Y}} \\
        {Y_{\alpha}} && Y.
        \arrow["{\mathrm{gr}_{\sigma}}"', from=1-2, to=2-1]
        \arrow["{\mathrm{ev}_{1, \sigma}}", from=1-2, to=2-3]
        \arrow["{p_{\alpha}}", from=2-1, to=3-1]
        \arrow["p"', from=2-3, to=3-3]
        \arrow["{g_{\alpha}}", from=3-1, to=3-3]
    \end{tikzcd}
\end{equation}
By pushing down the cohomological Hall induction map \Cref{eq-critical-CoHI} on $X$ to $Y$ and using isomorphisms \Cref{eq-dimensional-reduction-for-grads}, we obtain the following relative cohomological Hall induction map
\[
    *^{\mathcal{H}\mathrm{all}}_{\sigma}  \colon   \mathbb{L}^{- \vdim \gr_{\sigma}/ 2} \otimes g_{\alpha, *} p_{\alpha, *} \mathbb{D}\mathbb{Q}_{\mathcal{Y}_{\alpha}}  \to p_* \mathbb{D}\mathbb{Q}_{\mathcal{Y}}.    
\]
By taking the global sections, we obtain the absolute cohomological Hall induction  map
\[
    *^{\mathrm{Hall}}_{\sigma}  \colon   \mathbb{L}^{- \vdim \gr_{\sigma}/ 2} \otimes \mathrm{H}^{\mathrm{BM}}_{-*} (\mathcal{Y}_{\alpha}) \to \mathrm{H}^{\mathrm{BM}}_{-*} (\mathcal{Y}).
\]
\end{para}

\begin{remark}
    There is an alternative way to construct the cohomological Hall induction generalizing the Kapranov--Vasserot cohomological Hall algebra for smooth surfaces, which we briefly explain below.
    By using the quasi-smoothness of $\gr_{\sigma}$ and the purity transform \cite[Remark 3.8]{khan2019virtual}, we obtain a natural morphism
    \[
        \mathbb{L}^{- \vdim \gr_{\sigma}/ 2} \otimes  \gr_{\sigma}^* \mathbb{D} \mathbb{Q}_{\mathcal{Y}_{\alpha}} \to    \mathbb{D} \mathbb{Q}_{\mathcal{Y}_{\sigma}^+}. 
    \]
    By composing this map with the counit map $\ev_{1, \sigma, *} \mathbb{D} \mathbb{Q}_{\mathcal{Y}_{\sigma}^+} \to \mathbb{D} \mathbb{Q}_{\mathcal{Y}}$  and pushing down to $Y$, we obtain a natural map
    \[
        *^{\mathcal{H}\mathrm{all}, \mathrm{KV}}_{\sigma}  \colon   \mathbb{L}^{- \vdim \gr_{\sigma}/ 2} \otimes g_{\alpha, *} p_{\alpha, *} \mathbb{D}\mathbb{Q}_{\mathcal{Y}_{\alpha}}  \to p_* \mathbb{D}\mathbb{Q}_{\mathcal{Y}}.    
    \]
    We expect that $*^{\mathcal{H}\mathrm{all}, \mathrm{KV}}_{\sigma}$ and $*^{\mathcal{H}\mathrm{all}}_{\sigma}$ differ only by some sign.
    This should be an easy consequence of yet another construction of the cohomological Hall induction given by Khan and Kinjo \cite{khan20233d} based on derived microlocal geometry, which is already shown to be compatible with $*^{\mathcal{H}\mathrm{all}, \mathrm{KV}}_{\sigma}$ in \cite[Theorem 4.30]{khan20233d}.
    The compatibility of the cohomological Hall inductions in \cite{khan20233d} and $*^{\mathcal{H}\mathrm{all}}_{\sigma}$ should not be difficult to verify since both of them are constructed using morphisms of monodromic mixed Hodge modules (not complexes) hence can be checked locally: this will be discussed by T.K. in detail elsewhere.
\end{remark}

\subsection{Perversely degenerate CoHI on good moduli spaces}

\begin{para}[Supercommutativity]
    It was observed in \cite[Corollary 6.10]{_Davison_Meinhardt_CoDT} that the cohomological Hall algebras associated with quivers with potentials 
    are equipped with a canonical filtration called the \emph{perverse filtration} as long as the quiver is symmetric,
     and    become supercommutative
    after taking the associated graded with respect to this filtration and modifying the sign of the multiplication. We will prove a generalization of this result for CoHI, which is stated as an independence of the CoHI on the choice of the chamber, up to some sign.
    The sign is controlled by the cotangent distance introduced in \Cref{para-cotangent-distance}.

\end{para}

\begin{para}[Supercommutativity of CoHI: almost symmetric representation]\label{para-supercommutative-absolute-quotient}
    Let $G$ be a reductive group and $V$ be an almost symmetric representation of $V$, i.e., there exists an isomorphism $V \cong V^{\vee}$ as $G^{\circ}$-representations.
    Set $\mathcal{U} = V /G$.
    Let $(F, \alpha) \in \mathsf{Face}^{\mathrm{nd}}(\mathcal{U})$ be a non-degenerate face and $\sigma \subset F$ be a chamber with respect to the cotangent arrangement.
    We will show that the absolute cohomological Hall induction does not depend on the choice of the chamber up to the sign given by the cotangent distance function; namely, for chambers $\sigma, \sigma' \subset F$ the following identity holds:
    \begin{equation}\label{eq-absolute-supercommutative}
        *^{\mathrm{Hall}}_{\sigma} = (-1)^{d(\sigma, \sigma')} \cdot *^{\mathrm{Hall}}_{\sigma'}.
    \end{equation}
    We first deal with the case when $G$ is connected.
    In this case, we can describe the map $*^{\mathrm{Hall}}_{\sigma} $ quite explicitly as follows:
    Let $T \subset G$ be the maximal torus, $T_{\alpha} \subset T$ be the subtorus corresponding to $\alpha$ and $W$ and $W_{\alpha}$ be the Weyl groups for $G$ and $L_{\alpha}$ respectively.
    Then we have natural isomorphisms
    \[
    \mathrm{H}^*(V / G) \cong \mathrm{H}^*(\mathrm{B} T)^{W} \cong \mathbb{Q}[\mathfrak{h}_{\mathbb{Z}}^{\vee}]^{W}, \quad \mathrm{H}^*(V^{T_{\alpha}} / L_{\alpha}) \cong \mathrm{H}^*(\mathrm{B} T)^{W_{\alpha}} \cong \mathbb{Q}[\mathfrak{h}_{\mathbb{Z}}^{\vee}]^{W_{\alpha}}.
    \]
    For each character $\gamma \colon T \to \mathbb{G}_\mathrm{m}$, let $t_{\gamma} \in \mathbb{Q}[\mathfrak{h}_{\mathbb{Z}}^{\vee}]$ be the corresponding monomial.
    Let $\Phi$ and $S_{V}$  be the set (possibly with multiplicity) of roots and $T$-weights of $V$ respectively and $\Phi_{\sigma}^{-} \subset \Phi$ and $S_{V, \sigma}^{-} \subset S$ be the subset which are negative with respect to a cocharacter in $\sigma$.
    Then the cohomological Hall induction $*^{\mathrm{Hall}}_{\sigma} $ is given as follows:
    \begin{equation}\label{eq-CoHI-explicit}
     f \mapsto  \frac{1}{|W_{\alpha}|} \cdot \sum_{w \in W}  w \left( f \cdot \frac{\prod_{\gamma \in S_{V, \sigma}^{-}} t_{\gamma} }{\prod_{\beta \in \Phi_{\sigma}^{-}} t_{\beta}} \right).   
    \end{equation}
    This is a straightforward consequence of the localization formula applied to the partial flag variety $G / P_{\sigma}$: see the proof of \cite[Theorem 2]{_Kontsevich_Soibelman_CoHA} for the argument when $\mathcal{U}$ is the moduli stack of quiver representations.
    Note that $S_V \coprod \Phi$ is the set of weights of the cotangent complex for $\mathcal{X}$.
    This observation together with the formula \Cref{eq-CoHI-explicit} implies the equality \Cref{eq-absolute-supercommutative}.

    Now we discuss the case when $G$ is not necessarily connected.
    Set $\Gamma \coloneqq G /G^{\circ} $. Then $\Gamma$ acts on $\tilde{\mathcal{U}} \coloneqq V /G^{\circ} \cong V / G \times_{\mathrm{B} G} \mathrm{B} G^{\circ}$ and we have an equivalence
    \[
     \mathcal{U} = \tilde{ \mathcal{U} } / \Gamma.    
    \]
     Take a non-degenerate face $(F, \alpha) \in \mathsf{Face}^{\mathrm{nd}}(\mathcal{U})$ and a chamber $\sigma \subset F$ with respect to the cotangent arrangement.
     Take a lift $(F, \tilde{\alpha}) \in \mathsf{Face}(\tilde{\mathcal{U}})$ and set $\tilde{\sigma} \coloneqq (F, \tilde{\alpha}, \sigma) \in \mathsf{Cone}(\tilde{\mathcal{U}})$.
     Using the equalities \Cref{eq-CoHI-finite-quotient} and \Cref{eq-absolute-supercommutative} for $\tilde{\sigma}$ which we have already proven, 
     we obtain the equality \Cref{eq-absolute-supercommutative} for $\sigma$.
\end{para}

\begin{para}[Perverse degeneration of CoHI]\label{para-perv-deg-cohi}
    For a monodromic mixed Hodge complex $M$ on an algebraic stack $\mathcal{X}$, we define its perverse degeneration by
    \[
    {}^{\mathrm{p}} \mathcal{H}(M) \coloneqq  \bigoplus_{i \in \mathbb{Z}}  {}^{\mathrm{p}} \mathcal{H}^i(M) [-i]
    \]
    regarded as an object in the category of $\mathbb{Z}$-graded monodromic mixed Hodge modules.
    A morphism of monodromic mixed Hodge complexes induces a morphism on their perverse degenerations.
    A $\mathbb{Z}$-graded monodromic mixed Hodge module is said to be pure if its $i$-th graded piece is pure of weight $i$.
    A mixed Hodge complex $M$ is pure if and only if its perverse degeneration is.

    Let $\mathcal{U}$ be a smooth stack  satisfying assumptions \Cref{item-cohI-goodmoduli} and \Cref{item-cohI-qcgr} in \Cref{para-Cohi-assumptions}.
    We further assume that $\mathcal{U}$ is almost symmetric.
     For $(F, \alpha) \in \mathsf{Face}^{\mathrm{nd}}(\mathcal{X})$ and a cone $\sigma \subset F$ with respect to the cotangent arrangement, we denote the perverse degeneration of the cohomological Hall induction by
    \[
        *^{\mathrm{p}\mathcal{H}\mathrm{all}}_{\sigma} \colon g_{\alpha, *} {}^{\mathrm{p}} \mathcal{H} ( p_{\alpha, *}  \mathcal{IC}_{\mathcal{U}_{\alpha}}) \to   {}^{\mathrm{p}} \mathcal{H} (p_* \mathcal{IC}_{\mathcal{U}}).  
    \]
    Here, we used the natural commutation of the functors ${}^p \mathcal{H}(-)$ and $g_{\alpha, *}$ which follows from the finiteness of $g_{\alpha}$ proved in \Cref{para-Cohi-for-smooth stacks}.

    Similarly, for an oriented almost symmetric $(-1)$-shifted symplectic stack $\mathcal{X}$ satisfying assumptions \Cref{item-cohI-goodmoduli} and \Cref{item-cohI-qcgr} in \Cref{para-Cohi-assumptions},
    we can define the perversely degenerated CoHI 
    \[
        *^{\mathrm{p}\mathcal{H}\mathrm{all}}_{\sigma} \colon g_{\alpha, *} {}^{\mathrm{p}} \mathcal{H} ( p_{\alpha, *}  \varphi_{\mathcal{X}_{\alpha}}) \to   {}^{\mathrm{p}} \mathcal{H} (p_* \varphi_{\mathcal{X}}).  
    \]
\end{para}

\begin{para}[Supercommutativity of CoHI: smooth case]\label{para-supercommutative-smooth}
    We will show that the perversely degenerated cohomological Hall induction introduced in the last paragraph does not depend on the choice of a chamber up to a sign defined by the cotangent distance function in \Cref{para-cotangent-distance}.

    We adopt the notation from the last paragraph \Cref{para-perv-deg-cohi}.
    Let $\sigma, \sigma' \subset F$ be chambers of a non-degenerate face with respect to the cotangent arrangement.
    Then we claim the identity 
    \begin{equation}\label{eq-supercommutative-deg-relcohi-smooth}
        *^{\mathrm{p}\mathcal{H}\mathrm{all}}_{\sigma} = (-1)^{d(\sigma, \sigma')} \cdot *^{\mathrm{p}\mathcal{H}\mathrm{all}}_{\sigma'}.
    \end{equation}
    By \Cref{thm-decom-good-moduli}, the graded monodromic mixed Hodge modules $g_{\alpha, *} {}^{\mathrm{p}} \mathcal{H} ( p_{\alpha, *}  \mathcal{IC}_{\mathcal{U}_{\alpha}})$ and ${}^{\mathrm{p}} \mathcal{H} (p_* \mathcal{IC}_{\mathcal{U}})  $ are pure.
    Therefore by using \cite[Lemma 6.6]{_Davison_Meinhardt_CoDT}, it is enough to prove the identity at the stalk of each point $x \in U$.
    Using the commutativity of the diagram \Cref{eq-cohi-etale-cover} and the \'etale local structure theorem for smooth good moduli stacks \cite[Theorem 1.2]{_Alper_ALunaetaleslicetheoremforalgebraicstacks},
    we may assume that $\mathcal{U} = V / G$ where $G$ is a reductive group and $V$ is an almost symmetric $G$-representation and $x$ is the origin $0 \in V \GIT G$.
    To prove the identity \Cref{eq-supercommutative-deg-relcohi-smooth}, it is enough to prove the identity before taking the perverse degeneration: namely, it is enough to prove the following identity
    \[
        *^{\mathcal{H}\mathrm{all}}_{\sigma} |_{0} = (-1)^{d(\sigma, \sigma')} \cdot *^{\mathcal{H}\mathrm{all}}_{\sigma'}|_{0}.
    \]
    Note that the graded monodromic mixed Hodge modules $g_{\alpha, *} {}^{\mathrm{p}} \mathcal{H} ( p_{\alpha, *}  \mathcal{IC}_{\mathcal{U}_{\alpha}})$ and ${}^{\mathrm{p}} \mathcal{H} (p_* \mathcal{IC}_{\mathcal{U}})  $  are $\mathbb{G}_\mathrm{m}$-equivariant with respect to the scaling $\mathbb{G}_\mathrm{m}$-action on $V \GIT G$.  Therefore by using the contraction lemma \cite[Proposition 3.7.5]{kashiwara2013sheaves},
    we see that the map $*^{\mathcal{H}\mathrm{all}}_{\sigma} |_{0}$ is identified with the absolute cohomological Hall induction $*^{\mathrm{Hall}}_{\sigma}$.
    Then the desired statement follows from the identity \Cref{eq-absolute-supercommutative}.

\end{para}

\begin{para}[Symmetric CoHI: smooth case]\label{para-symmetrized-cohi}

    We adopt the notation from the paragraph \Cref{para-perv-deg-cohi}.
    By the equality \Cref{eq-supercommutative-deg-relcohi-smooth}, we see that the restriction of the perversely degenerate cohomological Hall induction to the invariant part
    \begin{equation}\label{eq-symmetrized-cohi-component}
        *^{\mathrm{s}\mathcal{H}\mathrm{all}}_{\alpha} \colon  \left( g_{\alpha, *} {}^{\mathrm{p}} \mathcal{H} ( p_{\alpha, *}  \mathcal{IC}_{\mathcal{U}_{\alpha}}) \otimes \mathrm{sgn}_{\alpha} \right)^{\mathrm{Aut}(\alpha)}  \to   {}^{\mathrm{p}} \mathcal{H} (p_* \mathcal{IC}_{\mathcal{U}}),
    \end{equation}
    which we call the symmetric cohomological Hall induction, does not depend on the choice of the chamber $\sigma \subset F$. 
    Here $\mathrm{sgn}_{\alpha}$ is the cotangent sign representation introduced in \Cref{para-cotangent-distance}.

    We will give a component-free description of the symmetric cohomological induction which is convenient for later applications.
    Firstly, let $\Grad_{\mathbb{Q}, \mathrm{nd}}^n(\mathcal{U}) \subset \Grad_{\mathbb{Q}}^n(\mathcal{U}) $ be the open and closed substack consisting of non-degenerate graded points
    and let 
    \[
       p_{n} \colon  \Grad_{\mathbb{Q}, \mathrm{nd}}^n(\mathcal{U}) \to \Grad_{\mathbb{Q}, \mathrm{nd}}^n(\mathcal{U})_{\GIT}    
    \]
    denote the good moduli space morphism. The $\mathrm{GL}_n(\mathbb{Z})$-action on $\mathrm{B} \mathbb{G}_\mathrm{m}^n$ induces a
    $\mathrm{GL}_n(\mathbb{Z})$-action on $\mathrm{Grad}^n(\mathcal{U})$  hence a $\mathrm{GL}_n(\mathbb{Q})$-action on $\Grad_{\mathbb{Q}, \mathrm{nd}}^n(\mathcal{U})$ and $\Grad_{\mathbb{Q}, \mathrm{nd}}^n(\mathcal{U})_{\GIT}$.
    For an $n$-dimensional face $(\mathbb{Q}^n, \alpha) \in \mathsf{Face}^{\mathrm{nd}}(\mathcal{U})$, note that $\mathrm{Aut}(\alpha)$ is the stabilizer group of the component $\alpha \in \pi_0(\Grad_{\mathbb{Q}, \mathrm{nd}}^n(\mathcal{U}))$.
    Therefore the cotangent sign representation for each $\alpha$ corresponds to a $\mathrm{GL}_n(\mathbb{Q})$-equivariant structure on the trivial $\mathbb{Z} / 2 \mathbb{Z}$-local system $\mathrm{sgn}_n$ on $\Grad^n_{\mathbb{Q}, \mathrm{nd}}(\mathcal{U})_{\GIT}$ which we call the cotangent sign local system.
    Now consider the natural map
    \[
        g_n \colon \Grad^n_{\mathbb{Q}, \mathrm{nd}}(\mathcal{U})_{\GIT} \to U
    \]
    and factor it by the maps
    \[
        \Grad^n_{\mathbb{Q}, \mathrm{nd}}(\mathcal{U})_{\GIT} \xrightarrow[]{q_n} \Grad^n_{\mathbb{Q}, \mathrm{nd}}(\mathcal{U})_{\GIT} / \mathrm{GL}_n(\mathbb{Q}) \xrightarrow[]{\bar{g}_n}  U.
    \]
    Note that the map $\bar{g}_n$ is a disjoint union of finite morphisms.
    The $\mathrm{GL}_n(\mathbb{Q})$-equivariant graded monodromic mixed Hodge module ${}^{\mathrm{p}} \mathcal{H}(p_{n, *} \mathcal{IC}_{\Grad^n_{\mathbb{Q}, \mathrm{nd}} (\mathcal{U}) } ) \otimes \mathrm{sgn}_n$ descends to a graded monodromic mixed Hodge module $\mathcal{H}_n$ on  $\Grad^n_{\mathbb{Q}, \mathrm{nd}}(\mathcal{U})_{\GIT} / \mathrm{GL}_n(\mathbb{Q})$.
    We set
    \begin{align*}
       ( g_{n, *} ({}^{\mathrm{p}} \mathcal{H}(p_{n, *} \mathcal{IC}_{\Grad^n_{\mathbb{Q}, \mathrm{nd}} (\mathcal{U}) } ) \otimes \mathrm{sgn}_n))^{\mathrm{GL}_n(\mathbb{Q}), \mathrm{lfin}} \coloneqq \bar{g}_{n, !} \mathcal{H}_n.
    \end{align*}
    The natural map $g_{n, !} \to g_{n, *}$ induces an inclusion
    \[
        ( g_{n, *} ({}^{\mathrm{p}} \mathcal{H}(p_{n, *} \mathcal{IC}_{\Grad^n_{\mathbb{Q}, \mathrm{nd}} (\mathcal{U})} ) \otimes \mathrm{sgn}_n))^{\mathrm{GL}_n(\mathbb{Q}), \mathrm{lfin}} \hookrightarrow ( g_{n, *} ({}^{\mathrm{p}} \mathcal{H}(p_{n, *} \mathcal{IC}_{\Grad^n_{\mathbb{Q}, \mathrm{nd}} (\mathcal{U}) } ) \otimes \mathrm{sgn}_n))^{\mathrm{GL}_n(\mathbb{Q})}.
    \]
    Explicitly, both sides can be written as follows:
    \begin{align*}
        ( g_{n, *} ({}^{\mathrm{p}} \mathcal{H}(p_{n, *} \mathcal{IC}_{\Grad^n_{\mathbb{Q}, \mathrm{nd}} (\mathcal{U}) } ) \otimes \mathrm{sgn}_n))^{\mathrm{GL}_n(\mathbb{Q}), \mathrm{lfin}} &\cong \bigoplus_{\substack{\alpha \colon \mathbb{Q}^n \to \mathrm{CL}_{\mathbb{Q}}(\mathcal{U}) \\ \text{$\alpha$: non-degenerate} }} \left( g_{\alpha, *} {}^{\mathrm{p}} \mathcal{H} ( p_{\alpha, *}  \mathcal{IC}_{\mathcal{U}_{\alpha}}) \otimes \mathrm{sgn}_{\alpha} \right)^{\mathrm{Aut}(\alpha)}, \\
        ( g_{n, *} ({}^{\mathrm{p}} \mathcal{H}(p_{n, *} \mathcal{IC}_{\Grad^n_{\mathbb{Q}, \mathrm{nd}} (\mathcal{U})} ) \otimes \mathrm{sgn}_n))^{\mathrm{GL}_n(\mathbb{Q})} &\cong \prod_{\substack{\alpha \colon \mathbb{Q}^n \to \mathrm{CL}_{\mathbb{Q}}(\mathcal{U}) \\ \text{$\alpha$: non-degenerate} }} \left( g_{\alpha, *} {}^{\mathrm{p}} \mathcal{H} ( p_{\alpha, *}  \mathcal{IC}_{\mathcal{U}_{\alpha}}) \otimes \mathrm{sgn}_{\alpha} \right)^{\mathrm{Aut}(\alpha)}.
    \end{align*}
    By taking the direct sum of the map \Cref{eq-symmetrized-cohi-component}, we obtain a map
    \begin{equation}\label{eq-symmetrized-cohi-without-component}
      *^{\mathrm{s}\mathcal{H}\mathrm{all}}_{n, \mathcal{U}} \colon   ( g_{n, *} ({}^{\mathrm{p}} \mathcal{H}(p_{n, *} \mathcal{IC}_{\Grad^n_{\mathbb{Q}, \mathrm{nd}} (\mathcal{U})} ) \otimes \mathrm{sgn}_n))^{\mathrm{GL}_n(\mathbb{Q}), \mathrm{lfin}} \to  {}^{\mathrm{p}} \mathcal{H} (p_* \mathcal{IC}_{\mathcal{U}})
    \end{equation}
    which we call the symmetric cohomological Hall induction map.    

\end{para}

\begin{para}[Symmetric CoHI for finite quotient]
    We adopt the notation from \Cref{para-cohi-finite-quotient}.
    We will compare the symmetric cohomological Hall induction for $\tilde{\mathcal{U}}$ and $\mathcal{U} = \tilde{\mathcal{U}} / \Gamma$ for a finite group $\Gamma$.
    Consider the following diagram:
\[\begin{tikzcd}
	{\mathrm{Grad}_{\mathbb{Q}, \mathrm{nd}}^{n}(\tilde{\mathcal{U}})_{\GIT} } & {\tilde{U} } \\
	{\mathrm{Grad}_{\mathbb{Q}, \mathrm{nd}}^{n}(\mathcal{U})_{\GIT} } & U.
	\arrow["{\tilde{g}_n}", from=1-1, to=1-2]
	\arrow["{r_n}", from=1-1, to=2-1]
	\arrow["r", from=1-2, to=2-2]
	\arrow["{g_n}"', from=2-1, to=2-2]
\end{tikzcd}\]
The lower horizontal arrow is identified with the affine quotient of the $\Gamma$-action on the upper horizontal arrow.
Now let $\tilde{\mathrm{sgn}}_n$ be the cotangent sign local system on $\mathrm{Grad}_{\mathbb{Q}, \mathrm{nd}}^{n}(\tilde{\mathcal{U}})_{\GIT}$.
Since the map $\tilde{\mathcal{U}} \to \mathcal{U}$ is \'etale, we have a $\mathrm{GL}_n(\mathbb{Q})$-equivariant isomorphism $r_n^* \mathrm{sgn}_n \cong \tilde{\mathrm{sgn}}_n$.
Now using \Cref{eq-cohi-finite-quotient-compatible}, we see that the following diagram commutes:

\begin{equation}\label{eq-finite-quotient-cohi-compatible-symmetrized}
    \begin{tikzcd}
	{   ( g_{n, *} ({}^{\mathrm{p}} \mathcal{H}(p_{n, *} \mathcal{IC}_{\Grad^n_{\mathbb{Q}, \mathrm{nd}} (\mathcal{U})_{\GIT} } ) \otimes \mathrm{sgn}_n))^{\mathrm{GL}_n(\mathbb{Q}), \mathrm{lfin}} } & {{}^{\mathrm{p}} \mathcal{H} (p_* \mathcal{IC}_{\mathcal{U}})} \\
	{   r_* ( \tilde{g}_{n, *} ({}^{\mathrm{p}} \mathcal{H}(\tilde{p}_{n, *} \mathcal{IC}_{\Grad^n_{\mathbb{Q}, \mathrm{nd}} (\tilde{\mathcal{U}})_{\GIT} } ) \otimes \tilde{\mathrm{sgn}}_n))^{\mathrm{GL}_n(\mathbb{Q}), \mathrm{lfin}} } & { r_* {}^{\mathrm{p}} \mathcal{H} (\tilde{p}_* \mathcal{IC}_{\tilde{\mathcal{U}}}).}
	\arrow["{ *^{\mathrm{s}\mathcal{H}\mathrm{all}}_{n, \mathcal{U}}}", from=1-1, to=1-2]
	\arrow[from=1-1, to=2-1]
	\arrow[from=1-2, to=2-2]
	\arrow["{ r_* \left( *^{\mathrm{s}\mathcal{H}\mathrm{all}}_{n, \tilde{\mathcal{U}}} \right)}"', from=2-1, to=2-2]
\end{tikzcd}
\end{equation}
Further, the lower horizontal map is $\Gamma$-equivariant, and the $\Gamma$-invariant part recovers the upper horizontal map.

\end{para}

\begin{para}[Symmetric CoHI for an \'etale cover]\label{para-sym-cohi-etale-cover}

    We adopt the notation from \Cref{para-cohi-etale-cover}.
    We will compare the symmetric cohomological Hall induction for an almost symmetric smooth stack $\mathcal{U}$ and its \'etale cover $\mathcal{V} = \mathcal{U} \times_{U} V$.
    Consider the following diagram:
\[
    \begin{tikzcd}
	{\mathrm{Grad}_{\mathbb{Q}, \mathrm{nd}}^{n}(\mathcal{V})_{\GIT} } & {V} \\
	{\mathrm{Grad}_{\mathbb{Q}, \mathrm{nd}}^{n}(\mathcal{U})_{\GIT} } & {U.}
	\arrow["{\tilde{g}_n}", from=1-1, to=1-2]
	\arrow["{\eta_{n, \GIT}}", from=1-1, to=2-1]
	\arrow["{\eta_{\GIT}}", from=1-2, to=2-2]
	\arrow["{g_n}"', from=2-1, to=2-2]
\end{tikzcd}\]    
By using the Cartesian property of the left diagram in \Cref{eq-grad-filt-etale-cartesian}, we see that this diagram is also Cartesian.
Now let $\tilde{\mathrm{sgn}}_n$ be the cotangent sign local system on $\mathrm{Grad}_{\mathbb{Q}, \mathrm{nd}}^{n}(\mathcal{V})_{\GIT}$.
Since the map $\mathcal{V} \to \mathcal{U}$ is \'etale, we have a $\mathrm{GL}_n(\mathbb{Q})$-equivariant isomorphism $\eta_{n, \GIT}^* \mathrm{sgn}_n \cong \tilde{\mathrm{sgn}}_n$.
Now using \Cref{eq-cohi-etale-cover}, we see that the following diagram commutes:
\begin{equation}\label{eq-etale-cohi-compatible-symmetrized}
    \begin{tikzcd}
	{   \eta_{\GIT}^*  (g_{n, *} ({}^{\mathrm{p}} \mathcal{H}(p_{n, *} \mathcal{IC}_{\Grad^n_{\mathbb{Q}, \mathrm{nd}} (\mathcal{U})_{\GIT} } ) \otimes \mathrm{sgn}_n))^{\mathrm{GL}_n(\mathbb{Q}), \mathrm{lfin}} } &[30pt] { \eta_{\GIT}^* {}^{\mathrm{p}} \mathcal{H} (p_* \mathcal{IC}_{\mathcal{U}})} \\
	{ ( \tilde{g}_{n, *} ({}^{\mathrm{p}} \mathcal{H}(\tilde{p}_{n, *} \mathcal{IC}_{\Grad^n_{\mathbb{Q}, \mathrm{nd}} (\mathcal{V})_{\GIT} } ) \otimes \tilde{\mathrm{sgn}}_n))^{\mathrm{GL}_n(\mathbb{Q}), \mathrm{lfin}} } & {  {}^{\mathrm{p}} \mathcal{H} (\tilde{p}_* \mathcal{IC}_{\mathcal{V}}).}
	\arrow["{{ \eta_{\GIT}^* \left( *^{\mathrm{s}\mathcal{H}\mathrm{all}}_{n, \mathcal{U}} \right)}}", from=1-1, to=1-2]
	\arrow["\cong"', from=1-1, to=2-1]
	\arrow["\cong", from=1-2, to=2-2]
	\arrow["{{  *^{\mathrm{s}\mathcal{H}\mathrm{all}}_{n, \mathcal{V}}}}"', from=2-1, to=2-2]
\end{tikzcd}
\end{equation}

\end{para}

\begin{para}[Supercommutativity of CoHI:  \texorpdfstring{$(-1)$}{TEXT}-shifted symplectic case]\label{para-supercommutative-1-shifted-symplectic}
    Let $(\mathcal{X}, \omega_{\mathcal{X}}, o)$ be an oriented $(-1)$-shifted symplectic stack satisfying assumptions \Cref{item-cohI-goodmoduli} and \Cref{item-cohI-qcgr} in \Cref{para-Cohi-assumptions}.
    Assume further that $\mathcal{X}$ is almost symmetric.
    Let $(F, \alpha) \in \mathsf{Face}^{\mathrm{nd}}(\mathcal{X})$ be a non-degenerate face and $\sigma, \sigma' \subset F$ be chambers with respect to the cotangent arrangement.
    Then we claim the identity  of the perversely degenerate cohomological Hall induction
    \begin{equation}\label{eq-supercommutative-deg-relcohi--1-symplectic}
        *^{\mathrm{p}\mathcal{H}\mathrm{all}}_{\sigma} =  *^{\mathrm{p}\mathcal{H}\mathrm{all}}_{\sigma'}.
    \end{equation}
    Using the commutativity of the diagram \Cref{eq-cohi-etale-cover-critical}, we may replace $\mathcal{X}$ with an \'etale cover over the good moduli space.
    Therefore by \Cref{thm-Darboux}, we may assume that there exists an almost symmetric smooth stack $\mathcal{U}$  satisfying assumptions \Cref{item-cohI-goodmoduli} and \Cref{item-cohI-qcgr} in \Cref{para-Cohi-assumptions} and a regular function $f \colon \mathcal{U} \to \mathbb{A}^1$,
    such that $\mathcal{X} = \mathrm{Crit}(f)$ holds.    
    Further, using the discussion in \Cref{para-any-orientation-is-locally-standard}, we may assume that $o$ is the standard orientation $o^{\mathrm{sta}}$.
    Now let $(F, \tilde{\alpha}) \in \mathsf{Face}^{\mathrm{nd}}(\mathcal{U})$ be the image of $\alpha$ and $\tilde{\sigma}$ and $\tilde{\sigma}'$ be the chambers containing the image of $\sigma$ and $\sigma'$.
    Consider the following diagram:
\[\begin{tikzcd}
	{g_{\alpha, *}\varphi_{f_{\tilde{\alpha}}}(\mathcal{IC}_{\mathcal{U}_{\tilde{\alpha}}})_{\alpha}} &[-15pt] {g_{\alpha, *}\varphi_{\mathcal{X}_{\alpha}, \mathrm{tot_{\alpha}^*\omega_{\mathcal{X}}}, o_{\mathcal{X}_{\alpha}}^{\mathrm{sta}}}} &[-15pt] {g_{\alpha, *}\varphi_{\mathcal{X}_{\alpha}, \mathrm{tot_{\alpha}^*\omega_{\mathcal{X}}}, \sigma^{\star}o_{\mathcal{X}}^{\mathrm{sta}}}} &[-5pt] {\varphi_{\mathcal{X}, \omega_{\mathcal{X}}, o_{\mathcal{X}}^{\mathrm{sta}}}} & {\varphi_{f}(\mathcal{IC}_{\mathcal{U}})} \\
	{g_{\alpha, *}\varphi_{f_{\tilde{\alpha}}}(\mathcal{IC}_{\mathcal{U}_{\tilde{\alpha}}})_{\alpha}} & {g_{\alpha, *}\varphi_{\mathcal{X}_{\alpha}, \mathrm{tot_{\alpha}^*\omega_{\mathcal{X}}}, o_{\mathcal{X}_{\alpha}}^{\mathrm{sta}}}} & {g_{\alpha, *}\varphi_{\mathcal{X}_{\alpha}, \mathrm{tot_{\alpha}^*\omega_{\mathcal{X}}}, {\sigma'}^{\star}o_{\mathcal{X}}^{\mathrm{sta}}}} & {\varphi_{\mathcal{X}, \omega_{\mathcal{X}}, o_{\mathcal{X}}^{\mathrm{sta}}}} & {\varphi_{f}(\mathcal{IC}_{\mathcal{U}}).}
	\arrow["\text{\Cref{eq-DT-Hodge-standard}}", "\cong"', no head, from=1-1, to=1-2]
	\arrow[""{name=0, anchor=center, inner sep=0}, "{\varphi_{f_{\tilde{\alpha}}}(*^{\mathcal{H}\mathrm{all}}_{\tilde{\sigma}})_{\alpha}}", curve={height=-54pt}, from=1-1, to=1-5]
	\arrow[Rightarrow, no head, from=1-1, to=2-1]
	\arrow["\text{\Cref{eq-localize-standard-orientaion}}", ""{name=1, anchor=center, inner sep=0}, "\cong"', from=1-2, to=1-3]
	\arrow[Rightarrow, no head, from=1-2, to=2-2]
	\arrow[""{name=2, anchor=center, inner sep=0}, "{*^{\mathcal{H}\mathrm{all}}_{\sigma}}", from=1-3, to=1-4]
	\arrow["\text{\Cref{eq-localized-ori-indep-on-segment}}", "\cong"', from=1-3, to=2-3]
	\arrow["\text{\Cref{eq-DT-Hodge-standard}}", "\cong"', from=1-4, to=1-5]
	\arrow[Rightarrow, no head, from=1-4, to=2-4]
	\arrow[Rightarrow, no head, from=1-5, to=2-5]
	\arrow["\text{\Cref{eq-DT-Hodge-standard}}", "\cong"', from=2-1, to=2-2]
	\arrow[""{name=3, anchor=center, inner sep=0}, "{\varphi_{f_{\tilde{\alpha}}}(*^{\mathcal{H}\mathrm{all}}_{\tilde{\sigma}'})_{\alpha}}"', curve={height=54pt}, from=2-1, to=2-5]
	\arrow["\text{\Cref{eq-localize-standard-orientaion}}", ""{name=4, anchor=center, inner sep=0}, "\cong"', no head, from=2-2, to=2-3]
	\arrow[""{name=5, anchor=center, inner sep=0}, "{*^{\mathcal{H}\mathrm{all}}_{\sigma'}}", from=2-3, to=2-4]
	\arrow["\text{\Cref{eq-DT-Hodge-standard}}", "\cong"', from=2-4, to=2-5]
	\arrow["{(A)}"{description}, draw=none, from=0, to=1-3]
	\arrow["{(B)}"{description}, draw=none, from=1, to=4]
	\arrow["{(C)}"{description}, draw=none, from=2, to=5]
	\arrow["{(D)}"{description}, draw=none, from=2-3, to=3]
\end{tikzcd}\]
The diagrams $(A)$ and $(D)$ commute by the commutativity of the diagram \Cref{eq-cohi-critical}.
The diagram $(B)$ commutes up to the sign $(-1)^{d(\tilde{\sigma}, \tilde{\sigma}')}$ by Lemma \ref{lem-ori-localize-indep-segment-up-to-sign}.
The outer square of this diagram commutes up to the sign $(-1)^{d(\tilde{\sigma}, \tilde{\sigma}')}$ after taking the perverse degeneration by \Cref{eq-supercommutative-deg-relcohi-smooth}.
Therefore we conclude that the diagram $(C)$ commutes after taking perverse degeneration, hence we obtain the identity \Cref{eq-supercommutative-deg-relcohi--1-symplectic}.
\end{para}

\begin{para}[Symmetric CoHI: \texorpdfstring{$(-1)$}{TEXT}-shifted symplectic case]\label{para-symmetrized-cohi--1-symplectic}
    We adopt the notation from \Cref{para-supercommutative-1-shifted-symplectic}.
    As we have seen in \Cref{para-localize-orientation-is-equivariant}, we can define an $\mathrm{Aut}(\alpha)$-equivariant orientation $\alpha^{\star} o$ on $\mathcal{X}_{\alpha}$.
    In particular, there exists an $\mathrm{Aut}(\alpha)$-equivariant structure on the monodromic mixed Hodge module $\varphi_{\mathcal{X}_{\alpha}, \mathrm{tot}_{\alpha}^* \omega_{\mathcal{X}}, \alpha^{\star}o}$.
    It follows from the equality \Cref{eq-supercommutative-deg-relcohi--1-symplectic} that the restriction of the perversely degenerated cohomological Hall induction map
    \begin{equation}\label{eq-symmetrized-cohi--1-symplectic-component}
        *^{\mathrm{s}\mathcal{H}\mathrm{all}}_{\alpha} \colon  \left( g_{\alpha, *} {}^{\mathrm{p}} \mathcal{H} ( p_{\alpha, *}  \varphi_{\mathcal{X}_{\alpha}}) \right)^{\mathrm{Aut}(\alpha)}  \to   {}^{\mathrm{p}} \mathcal{H} (p_* \varphi_{\mathcal{X}})
    \end{equation}
    does not depend on the choice of the chamber $\sigma \subset F$.

    We will now give a component-free description as in \Cref{para-symmetrized-cohi}.
    Let $n$ be a positive integer.
    It is clear that the $(-1)$-shifted symplectic structure $\mathrm{tot}_n^{\star} \omega_{\mathcal{X}}$ on $\mathrm{Grad}_{\mathbb{Q}}^n(\mathcal{X})$ is $\mathrm{GL}_n(\mathbb{Q})$-equivariant.
    The $\mathrm{Aut}(\alpha)$-equivariant orientations on $\mathcal{X}_{\alpha}$ for each face $\alpha$ assemble to form a $\mathrm{GL}_n(\mathbb{Q})$-equivariant orientation $\mathrm{tot}_n^{\star}o$.
    Let $p_{n} \colon \mathrm{Grad}_{\mathbb{Q}, \mathrm{nd}} ^n(\mathcal{X}) \to \mathrm{Grad}_{\mathbb{Q}, \mathrm{nd}} ^n(\mathcal{X})_{\GIT}$ be the good moduli space morphism
    and $g_{n} \colon \mathrm{Grad}_{\mathbb{Q}, \mathrm{nd}} ^n(\mathcal{X})_{\GIT} \to X$ be the map induced on the good moduli spaces.
    Then the direct sum of the maps \Cref{eq-symmetrized-cohi--1-symplectic-component} gives a map
    \begin{equation}\label{eq-symmetrized-cohi--1-symplectic-without-component}
        *^{\mathrm{s}\mathcal{H}\mathrm{all}}_{n, \mathcal{X}} \colon   ( g_{n, *} {}^{\mathrm{p}} \mathcal{H}(p_{n, *} \varphi_{\Grad^n_{\mathbb{Q}, \mathrm{nd}} (\mathcal{X})}))^{\mathrm{GL}_n(\mathbb{Q}), \mathrm{lfin}} \to  {}^{\mathrm{p}} \mathcal{H} (p_* \varphi_{\mathcal{X}})
      \end{equation}
      which we call the symmetric cohomological Hall induction.

      Since the cohomological Hall induction map is compatible with \'etale pullback, as we have seen in \Cref{eq-cohi-etale-cover-critical},
      the symmetric cohomological Hall induction is also compatible with \'etale pullback.
      Namely, if we are given an \'etale morphism of $(-1)$-shifted symplectic stacks $\eta \colon \mathcal{Y} \to \mathcal{X}$, the following diagram commutes:
\begin{equation}\label{eq-etale-cohi-compatible-symmetrized--1-shifted-symplectic}
    \begin{tikzcd}
	{   \eta_{\GIT}^*  (g_{n, *} ({}^{\mathrm{p}} \mathcal{H}(p_{n, *} \varphi_{\Grad^n_{\mathbb{Q}, \mathrm{nd}} (\mathcal{X})_{\GIT} } )))^{\mathrm{GL}_n(\mathbb{Q}), \mathrm{lfin}} } &[30pt] { \eta_{\GIT}^* {}^{\mathrm{p}} \mathcal{H} (p_* \varphi_{\mathcal{X}})} \\
	{ ( \tilde{g}_{n, *} ({}^{\mathrm{p}} \mathcal{H}(\tilde{p}_{n, *} \varphi_{\Grad^n_{\mathbb{Q}, \mathrm{nd}} (\mathcal{Y})_{\GIT} } )))^{\mathrm{GL}_n(\mathbb{Q}), \mathrm{lfin}} } & {  {}^{\mathrm{p}} \mathcal{H} (\tilde{p}_* \varphi_{\mathcal{Y}}).}
	\arrow["{{ \eta_{\GIT}^* \left( *^{\mathrm{s}\mathcal{H}\mathrm{all}}_{n, \mathcal{X}} \right)}}", from=1-1, to=1-2]
	\arrow["\cong"', from=1-1, to=2-1]
	\arrow["\cong", from=1-2, to=2-2]
	\arrow["{{  *^{\mathrm{s}\mathcal{H}\mathrm{all}}_{n, \mathcal{Y}}}}"', from=2-1, to=2-2]
\end{tikzcd}
\end{equation}
    Here we adopt the notation from \Cref{para-sym-cohi-etale-cover}.
\end{para}

\begin{para}[Symmetric CoHI for critical loci]\label{para-symmetrized-cohi-critical-loci}
    We adopt the notation from \Cref{para-symmetrized-cohi--1-symplectic}.
    Assume further that there exists an almost symmetric smooth stack $\mathcal{U}$ satisfying assumptions \Cref{item-cohI-goodmoduli} and \Cref{item-cohI-qcgr} in \Cref{para-Cohi-assumptions} and a function $f$ on $\mathcal{U}$ such that $\mathcal{X} = \mathrm{Crit}(f)$ as an oriented $(-1)$-shifted symplectic stack.
    Let $\iota \colon \mathcal{X} \to \mathcal{U}$ be the canonical map and $\iota_n \colon \mathrm{Grad}^n_{\mathbb{Q}, \mathrm{nd}}(\mathcal{X}) \to \mathrm{Grad}^n_{\mathbb{Q}, \mathrm{nd}}(\mathcal{U})$ be the induced map.
    Set $f_n \coloneqq f \circ \mathrm{tot}_n$ and let $\bar{f} \colon U \to \mathbb{A}^1$ denote the induced function.
    Then it follows from \Cref{lem-ori-localize-indep-segment-up-to-sign} that there exists an isomorphism of $\mathrm{GL}_n(\mathbb{Q})$-equivariant orientations
    \[
     \mathrm{tot}_{n}^{\star} o^{\mathrm{sta}}_{\mathcal{X}} \cong   o^{\mathrm{sta}}_{\mathrm{Grad}^n_{\mathbb{Q}, \mathrm{nd}}(\mathcal{X}) } \otimes \iota_n^* \mathrm{sgn}_n,
    \]
    which induces a natural isomorphism of  $\mathrm{GL}_n(\mathbb{Q})$-equivariant perverse sheaves
    \[
        \varphi_{\Grad^n(\mathcal{X})} \cong \varphi_{f_n}(\mathcal{IC}_{\mathrm{Grad}^n_{\mathbb{Q}, \mathrm{nd}}(\mathcal{U})}) \otimes \iota_n^* \mathrm{sgn}_n.
    \]
    It follows from the commutativity of the diagram \Cref{eq-cohi-critical} that the following diagram commutes:
\begin{equation}\label{eq-symmetrized-cohi-critical}
    \begin{tikzcd}
	{         \varphi_{\bar{f}} \left(  ( g_{n, *} {}^{\mathrm{p}} \mathcal{H}(p_{n, *} \mathcal{IC}_{\Grad^n_{\mathbb{Q}, \mathrm{nd}} (\mathcal{U})}) \otimes \mathrm{sgn}_n)^{\mathrm{GL}_n(\mathbb{Q}), \mathrm{lfin}} \right) } & {\varphi_{\bar{f}} \left({}^{\mathrm{p}} \mathcal{H} (p_* \mathcal{IC}_{\mathcal{U}}) \right)} \\
	{           ( g_{n, *} {}^{\mathrm{p}} \mathcal{H}(p_{n, *} \varphi_{\Grad^n_{\mathbb{Q}, \mathrm{nd}} (\mathcal{X})}))^{\mathrm{GL}_n(\mathbb{Q}), \mathrm{lfin}} } & {{}^{\mathrm{p}} \mathcal{H} (p_* \varphi_{\mathcal{X}}).}
	\arrow["{\varphi_{\bar{f}}(*^{\mathrm{s}\mathcal{H}\mathrm{all}}_{n, \mathcal{U}}) }", from=1-1, to=1-2]
	\arrow["\cong"', from=1-1, to=2-1]
	\arrow["\cong", from=1-2, to=2-2]
	\arrow["{*^{\mathrm{s}\mathcal{H}\mathrm{all}}_{n, \mathcal{X}} }"', from=2-1, to=2-2]
\end{tikzcd}
\end{equation}
\end{para}

\section{Cohomological integrality}\label{sect-cohint}

In this section, we will state and prove the cohomological integrality theorem for almost orthogonal stacks.
This statement can be regarded as a non-linear and global generalization of the cohomological integrality theorem of \textcite[Theorem A, C]{_Davison_Meinhardt_CoDT} for quivers with potentials and generic stability conditions.

\subsection{Cohomological integrality for smooth stacks: Statement}

\begin{para}[A rational torus]\label{para-rational-torus}
    For a $\mathbb{Z}$-lattice $\Lambda$, we let $T_{\Lambda} \coloneqq \Spec \mathbb{C}[\Lambda^{\vee}]$ be the corresponding torus.
    For an algebraic stack $\mathcal{X}$, giving a $\mathrm{B} T_{\Lambda}$-action on $\mathcal{X}$ is equivalent to defining a section of the map $\mathrm{Grad}^{\Lambda}(\mathcal{X}) \to \mathcal{X}$: see \cite[\S 3.4]{IbanezNunez_stratificationsgoodmodulistacks} for the details.
   
    Now let $F$ be a finite dimensional $\mathbb{Q}$-vector space. Motivated by the above discussion, we define a ``$\mathrm{B} T_F$-action'' on $\mathcal{X}$ to be a section of the map  $\mathrm{Grad}^{F}(\mathcal{X}) \to \mathcal{X}$, though we do not define the object $\mathrm{B} T_F$ itself.
    We say that a $\mathrm{B} T_F$-action is non-degenerate if the corresponding components of $\mathrm{Grad}^{F}(\mathcal{X})$ are non-degenerate.

    If one is given a submodule $\Lambda \subset F$ which is a free $\mathbb{Z}$-module of full rank satisfying $\Lambda \otimes \mathbb{Q} \cong F$, a $\mathrm{B} T_{\Lambda}$-action naturally induces a $\mathrm{B} T_{F}$-action.
    Conversely, if $\mathcal{X}$ is connected and we are given a $\mathrm{B} T_F$-action $\mu$ on $\mathcal{X}$, there exists a free $\mathbb{Z}$-module of full rank satisfying $\Lambda \otimes \mathbb{Q} \cong F$ such that $\mu$ is induced from a $\mathrm{B}T_{\Lambda}$-action on $\mathcal{X}$.
    Further, using \cite[Proposition 1.3.9]{_HalpernLeistner_Onthestructureofinstabilityinmodulitheory}, we may choose such a lift canonically, if one requires the action to be faithful, i.e., there is no non-trivial subgroup $Z \subset T_{\Lambda}$ with the property that $BZ$ acts trivially on $\mathcal{X}$.
    
    For a non-degenerate face $(F, \alpha) \in \mathsf{Face}^{\mathrm{nd}}(\mathcal{X})$, the stack $\mathcal{X}_{\alpha}$ is naturally equipped with a non-degenerate $\mathrm{B} T_F$-action.
    By the discussion above, we may find a canonical $\mathrm{B} \mathbb{G}_{\mathrm{m}}^{\dim F}$-action on $\mathcal{X}_{\alpha}$ lifting it.
    The faithfulness implies that the quotient stack $\mathcal{X}_{\alpha} / \mathrm{B} \mathbb{G}_{\mathrm{m}}^{\dim F}$ is a $1$-Artin stack.

    For a $\mathbb{Z}$-lattice $\Lambda$, there exists a natural isomorphism $\mathrm{H}^*(\mathrm{B} T_{\Lambda}) \cong \mathrm{Sym}(\Lambda^{\vee} \otimes_{\mathbb{Z}} \mathbb{L})$.
    Motivated by this, for a finite dimensional $\mathbb{Q}$-vector space $F$, we set
    \begin{equation*}
     \mathrm{H}^*(\mathrm{B} T_F) \coloneqq     \mathrm{Sym}(F^{\vee} \otimes_{\mathbb{Q}} \mathbb{L}).
    \end{equation*}

    Assume that we are given a $\mathrm{B} T_F$-action $\mu$ on a connected algebraic stack $\mathcal{X}$.
    Take a submodule $\Lambda \subset F$ which is a free $\mathbb{Z}$-module satisfying $\Lambda \otimes \mathbb{Q} \cong F$ such that $\mu$ corresponds to a $\mathrm{B} T_{\Lambda}$-action $\mu_{\Lambda}$ on $\mathcal{X}$,
    and a point $x \in \mathcal{X}$.
    Then we define a map 
    \begin{equation}\label{eq-restrict-rational-torus-action}
     \mathrm{H}^*(\mathcal{X}) \to     \mathrm{H}^*(\mathrm{B} T_F) 
    \end{equation}
    by the composition
    \[
        \mathrm{H}^*(\mathcal{X}) \xrightarrow[]{\mu_{\Lambda}^*} \mathrm{H}^*(\mathrm{B} T_{\Lambda} \times \mathcal{X}) \to \mathrm{H}^*(\mathrm{B} T_{\Lambda}) \cong \mathrm{H}^*(\mathrm{B} T_F)
    \]
    where the second map is given by the restriction to the point $x \in \mathcal{X}$.
    One easily sees that this definition does not depend on the choice of $\Lambda$ and $x \in \mathcal{X}$.

\end{para}

\begin{para}[Global equivariant parameter]
    Let $\mathcal{X}$ be a connected algebraic stack, $F$ be a $d$-dimensional $\mathbb{Q}$-vector space and $\mu$ be a $\mathrm{B} T_F$-action on $\mathcal{X}$.
    A \textit{global equivariant parameter} of $\mathcal{X}$ with respect to the action $\mu$ is defined to be a tuple of line bundles $\boldsymbol{\mathcal{L}}=(\mathcal{L}_1, \cdots, \mathcal{L}_d)$ such that the image of $\{c_1(\mathcal{L}_1), \ldots, c_1(\mathcal{L}_d) \}$ under the map \Cref{eq-restrict-rational-torus-action} forms a basis of $\mathrm{H}^2(\mathrm{B} T_F)$.
    The following lemma ensures the existence of a global equivariant parameter under reasonable assumptions:

\end{para}

\begin{lemma}\label{lem-global-quotient-global-equiv-parameter}
    Let $\mathcal{X}$ be a connected algebraic stack which admits a presentation $\mathcal{X} = X / G$ where $X$ is an algebraic space and $G$ is an affine algebraic group. Let $F$ be a $d$-dimensional $\mathbb{Q}$-vector space.
    Then for a non-degenerate $\mathrm{B} T_F$-action $\mu$  on $\mathcal{X}$, there exists a global equivariant parameter.
\end{lemma}

\begin{proof}
    Without loss of generality, we may assume $G = \GL_n$.
    Take a submodule $\Lambda \subset F$ which is a free $\mathbb{Z}$-module satisfying $\Lambda \otimes_{\mathbb{Z}} \mathbb{Q} \cong F$ such that $\mu$ lifts to a $\mathrm{B} T_{\Lambda}$-action $\mu_{\Lambda}$ on $\mathcal{X}$.
    By the description of the stacks of graded points for quotient stacks explained in  \Cref{eg-grad-quotient-stack}, we see that $\mu_{\Lambda}$ corresponds to a map $\lambda \colon T_{\Lambda} \to \GL_n$ with finite kernel such that 
    $\mathcal{X}$ is isomorphic to the connected component $X^{\lambda} / L_{\lambda}$, where $L_{\lambda}$ is the Levi subgroup of $\GL_n$ associated with $\lambda$.
    Note that the image of $\lambda$ is contained in the center of $L_{\lambda}$.
    Therefore the composition $T_{\Lambda} \xrightarrow[]{\lambda} L_{\lambda} \to L_{\lambda}^{\mathrm{ab}}$ has finite kernel.
    In particular, for each point $x \in \mathcal{X}$, the composite
    \[
    \mathrm{H}^2(\mathrm{B} L_{\lambda}^{\mathrm{ab}}) \to  \mathrm{H}^2(\mathrm{B} L_{\lambda}) \to \mathrm{H}^2( X^{\lambda} / L_{\lambda}) \to  \mathrm{H}^2(\mathcal{X}) \to   \mathrm{H}^2(\mathrm{B} G_x) \to \mathrm{H}^2(\mathrm{B} T_{\Lambda}) 
    \]
    is surjective (note that we are working with the rational cohomology), which clearly implies the existence of the global equivariant parameter.
\end{proof}

\begin{corollary}\label{cor-global-equiv-parameter-smooth}
    Let $\mathcal{U}$ be a connected smooth algebraic stack with affine stabilizers.
    Then, for a non-degenerate $\mathrm{B} T_F$-action $\mu$  on $\mathcal{U}$, there exists a global equivariant parameter.
\end{corollary}

\begin{proof}
    Using \cite[Proposition 3.5.9]{1999kreschChow}, we may find an open embedding
    \[
     \emptyset \neq V / \mathrm{GL}_n \hookrightarrow \mathcal{U}
    \]
    where $V$ is a quasi-projective smooth variety.
    Using \Cref{lem-global-quotient-global-equiv-parameter}, we may take a global equivariant parameter  $\mathcal{L}_1, \ldots, \mathcal{L}_d \in \mathrm{Pic}(V / G)$ with respect to the restriction of $\mu$ to $V / G$.
    Using \cite[Corollaire 15.5]{_Laumon_Champsalgebriques}, one can extend $\mathcal{L}_i$ to a coherent sheaf $F_i$ on $\mathcal{U}$.
    Since $\mathcal{U}$ is smooth, $F_i$ is a perfect complex.
    Therefore one can define the tuple of line bundles $(\det(F_1), \ldots, \det(F_d))$, which is clearly a global equivariant parameter for $\mathcal{U}$.
\end{proof}

\begin{para}[Global equivariant parameter for stack of objects in a dg-category]\label{para-equiv-parameter-linear}
    Let $\mathcal{C}$ be a finite type dg-category and $\mathcal{M}_{\mathcal{C}}$ be the moduli stack of objects in the sense of \textcite{toen2007moduli}.
    Here we discuss the existence of a global equivariant parameter for $\mathcal{M}_{\mathcal{C}}$ with respect to the natural $\mathrm{B} \mathbb{G}_{\mathrm{m}}$-action.
    Note that the global equivariant parameter was introduced only for $1$-stacks, but the same definition makes sense for higher Artin stacks.

    By definition, a point in $\mathcal{M}_{\mathcal{C}}$ corresponds to a functor $\mathcal{C} \to \mathrm{Perf}_k$.
    We say that a connected substack $\mathcal{M}_{\mathcal{C}}^{\circ} \subset \mathcal{M}_{\mathcal{C}}$ is numerically non-trivial if there exists $[F] \in \mathcal{M}_{\mathcal{C}}^{\circ} $ and $c \in \mathcal{C}$ such that $\rank F(c) \neq 0$.
    For such a $c$, we define a perfect complex $E_c$ on $\mathcal{M}_{\mathcal{C}}^{\circ} $ with the following property
    \[
     E_c |_{[F']} \cong F'(c)
    \]
    for any $[F'] \in  \mathcal{M}_{\mathcal{C}}^{\circ}$.
    Then it is clear from the definition that $\det(E_c)$ defines a global equivariant parameter for $\mathcal{M}_{\mathcal{C}}^{\circ}$.

    We note that the assumption of being numerically non-trivial is a quite mild one.
    For example, it is satisfied for non-zero components of the moduli stack of compactly supported coherent sheaves on a smooth quasi-projective variety.
    Also, it is satisfied for the non-zero component of the moduli stack of Bridgeland semistable objects, whenever we consider a stability condition which factors through the numerical Grothendieck group. 
\end{para}

\begin{para}
The following technical lemma will be useful to control the global equivariant parameters for good moduli stacks locally over good moduli spaces:

\end{para}

\begin{lemma}\label{lemma-global-equiv-parameter-BG}
   Let $G$ be a reductive group acting on an affine scheme $Y$. Assume that there exists a point $y \in Y$ fixed by $G$ whose image in $Y/G$ is denoted by $\bar{y}$.
   Take a line bundle $\mathcal{L}$ on $Y/G$.
   Let $p \colon Y/G \to Y \GIT G$ be the good moduli space morphism and $r \colon Y / G \to \mathrm{B} G$ be the projection map.
   Then there exists a Zariski open neighborhood $p(\bar{y}) \in U \subset Y \GIT G$ and a line bundle $\mathcal{M}$ on $\mathrm{B} G$ such that there exists an isomorphism
   $\mathcal{L} |_{p^{-1}(U)} \cong r^{*} \mathcal{M} |_{p^{-1}(U)}$.

\end{lemma}

\begin{proof}
    Set $\mathcal{M} \coloneqq \mathcal{L} |_{\mathrm{B} G_y}$. By definition, we have a natural isomorphism $\phi \colon \mathcal{L}|_{\mathrm{B} G_y} \cong (r^* \mathcal{M})|_{\mathrm{B} G_y}$.
    Then the statement follows from \cite[Lemma 6.8]{bellamy2016symplectic} and \cite[Theorem 10.3]{_Alper_GoodmodulispacesforArtinstacks}.
\end{proof}

\begin{para}[Assumptions]\label{para-cohint-assumptions}
    For the rest of \Cref{sect-cohint}, we will work with derived algebraic stacks satisfying the following assumptions:
    \begin{enumerate}
        \item $\mathcal{X}$ has affine diagonal and admits a good moduli space $p \colon \mathcal{X} \to X$. \label{item1-cohint-goodmoduli}
        \item $\mathcal{X}$ has quasi-compact connected components and quasi-compact graded points. In particular, $\mathcal{X}$ has finite cotangent weights. \label{item2-cohint-qcgr}
        \item $\mathcal{X}$ is almost symmetric.  \label{item3-cohint-symmetric}
        \item For each non-degenerate face $(F, \alpha) \in \mathsf{Face}^{\mathrm{nd}}(\mathcal{X})$,
        there exists a global equivariant parameter for $\mathcal{X}_{\alpha}$ with respect to the natural $\mathrm{B} T_F$-action.  \label{item4-cohint-gloeq}
    \end{enumerate}
\end{para}

\begin{para}[Cohomological integrality for smooth stacks]\label{para-cohint-smoothstack}
    We now state the cohomological integrality theorem for smooth stacks.
    Let $\mathcal{U}$ be a smooth stack satisfying \Crefrange{item1-cohint-goodmoduli}{item4-cohint-gloeq} in \Cref{para-cohint-assumptions}.
    For each non-degenerate face $(F, \alpha) \in \mathsf{Face}^{\mathrm{nd}}(\mathcal{U})$ and a choice of a chamber $\sigma \subset F$ with respect to the cotangent arrangement,
    consider the following diagram
\[\begin{tikzcd}
	& {\mathcal{U}_{\sigma}^+} \\
	{\mathcal{U}_{\alpha}} && {\mathcal{U}} \\
	{U_{\alpha}} && U.
	\arrow["{\mathrm{gr}_{\sigma}}"', from=1-2, to=2-1]
	\arrow["{\mathrm{ev}_{1, \sigma}}", from=1-2, to=2-3]
	\arrow["{p_{\alpha}}", from=2-1, to=3-1]
	\arrow["p"', from=2-3, to=3-3]
	\arrow["{g_{\alpha}}", from=3-1, to=3-3]
\end{tikzcd}\]
    By the definition of the BPS sheaf, there exists a natural map
    \begin{equation}\label{eq-bps-map-smooth}
     \mathcal{BPS}^{\alpha}_{U} \otimes \mathbb{L}^{\dim F / 2} \to p_{\alpha, *} \mathcal{IC}_{\mathcal{U}_{\alpha}}.    
    \end{equation}
    Take a global equivariant parameter $\boldsymbol{\mathcal{L}}$ for $\mathcal{U}_{\alpha}$. It induces an $\mathrm{H}^*(\mathrm{B} T_{F})$-action on $\mathcal{IC}_{\mathcal{U}_{\alpha}}$.
    In particular, we can extend the map \Cref{eq-bps-map-smooth} to the following map
    \begin{equation}\label{eq-bps-map-extended-smooth}
        \mathcal{BPS}^{\alpha}_{{U}} \otimes \mathrm{H}^*(\mathrm{B} T_{F})_{\mathrm{vir}}  \to p_{\alpha, *} \mathcal{IC}_{\mathcal{U}_{\alpha}}
     \end{equation}
     where we set $\mathrm{H}^*(\mathrm{B} T_{F})_{\mathrm{vir}}  \coloneqq \mathrm{H}^*(\mathrm{B} T_{F}) \otimes \mathbb{L}^{\dim F / 2}$.
     Using the cohomological Hall induction map introduced in \Cref{eq-cohi-smooth-symmetric}, we obtain a natural map
     \begin{equation}\label{eq-relative-COHI-cohint}
       g_{\alpha, *} p_{\alpha, *} \mathcal{IC}_{\mathcal{U}_{\alpha}} \to p_* \mathcal{IC}_{\mathcal{U}}.
     \end{equation}
     By composing maps \Cref{eq-bps-map-extended-smooth} and \Cref{eq-relative-COHI-cohint}, we obtain the following map
     \begin{equation}\label{eq-CoHI-bps-piece-smooth}
      \mathrm{Ind}_{\alpha, \sigma, \boldsymbol{\mathcal{L}} } \colon g_{\alpha, *} \mathcal{BPS}^{\alpha}_{{U}} \otimes \mathrm{H}^*(\mathrm{B} T_{F})_{\mathrm{vir}} \to p_* \mathcal{IC}_{\mathcal{U}}.
     \end{equation}
     Note that the $\mathrm{Aut}(\alpha)$-action on $F$ induces an $\mathrm{Aut}(\alpha)$-action on $\mathrm{H}^*(\mathrm{B} T_{F})_{\mathrm{vir}}$.
     Consider the following restriction of $\mathrm{Ind}_{\alpha, \sigma, \boldsymbol{\mathcal{L}} } $:
     \begin{equation}\label{eq-CoHI-inv-bps-piece-smooth}
        \mathrm{Ind}_{\alpha, \sigma, \boldsymbol{\mathcal{L}} }^{\mathrm{sym}} \colon (g_{\alpha, *} \mathcal{BPS}^{\alpha}_{{U}} \otimes \mathrm{H}^*(\mathrm{B} T_{F})_{\mathrm{vir}} \otimes \mathrm{sgn}_{\alpha})^{\mathrm{Aut}(\alpha)}  \to p_* \mathcal{IC}_{\mathcal{U}}.
     \end{equation}
     We can now define the notion of a smooth stack satisfying the cohomological integrality theorem:

\end{para}

\begin{definition*}
    Let $\mathcal{U}$ be a smooth stack satisfying \Crefrange{item1-cohint-goodmoduli}{item4-cohint-gloeq} in \Cref{para-cohint-assumptions}.
    We say that $\mathcal{U}$ satisfies the \emph{cohomological integrality theorem} if the map
    \begin{equation}\label{eq-cohint-map}
      \bigoplus_{(F, \alpha) \in \mathsf{Face}^{\mathrm{nd}}(\mathcal{U})}  (g_{\alpha, *}  \mathcal{BPS}^{\alpha}_{{U}} \otimes \mathrm{H}^*(\mathrm{B} T_F)_{\mathrm{vir}} \otimes \mathrm{sgn}_{\alpha})^{\mathrm{Aut}(\alpha)}  \xrightarrow[]{\bigoplus_{\alpha} \mathrm{Ind}_{\alpha, \sigma, \boldsymbol{\mathcal{L}} }^{\mathrm{sym}} } p_* \mathcal{IC}_{\mathcal{U}}
    \end{equation}
    is an isomorphism for any choice of the global equivariant parameter for $\mathcal{U}_{\alpha}$ and choice of a chamber $\sigma \subset F$ with respect to the cotangent arrangement.
    Here the direct sum is over all isomorphism classes of non-degenerate faces and $\mathrm{sgn}_{\alpha}$ is the sign representation of $\mathrm{Aut}(\alpha)$ defined in \Cref{eq-sign-def}.
\end{definition*}

\begin{remark}\label{rmk-cohint-source-smooth}
    By \Cref{cor:small_decom_thm_smooth_stack} and \Cref{cor-non-special-vanishing}, the source of the map \Cref{eq-cohint-map} is identified with the following object:
    \[
        \bigoplus_{(F, \alpha) \in \mathsf{Face}^{\mathrm{sp}}(\mathcal{U})}  (g_{\alpha, *}  \mathcal{IC}_{{U}_{\alpha}}^{\circ} \otimes \mathrm{H}^*(\mathrm{B} T_F)_{\mathrm{vir}} \otimes \mathrm{sgn}_{\alpha})^{\mathrm{Aut}(\alpha)}.
    \]
    Here we set 
    \[
        \mathcal{IC}_{U_{\alpha}}^{\circ} = 
        \begin{cases}
            \mathcal{IC}_{U_{\alpha}}  & \quad \text{if the map $\mathcal{U}_{\alpha} / \mathrm{B} \mathbb{G}_{\mathrm{m}}^{\dim F} \to U_{\alpha}$ is generically quasi-finite.} \\
            0  & \quad \text{otherwise.} 
        \end{cases}
    \]
    See \Cref{para-rational-torus} for the choice of the $\mathrm{B} \mathbb{G}_{\mathrm{m}}^{\dim F}$-action on $\mathcal{U}_{\alpha}$.
    In particular, the source of the map \Cref{eq-cohint-map} is a finite sum over each connected component of $U$, thanks to the finiteness of special faces for stacks satisfying \Cref{item1-cohint-goodmoduli} and \Cref{item2-cohint-qcgr} discussed in \Cref{para-finiteness-theorem}.
\end{remark}

\begin{remark}
    The term ``cohomological integrality'' first appeared in the work of \textcite{_Davison_Meinhardt_CoDT} and it comes from the observation that when $\mathcal{U}$ is the moduli stack of representations of a symmetric quiver, the map \Cref{eq-cohint-map} being an isomorphism implies a certain integrality property of the generalized Donaldson--Thomas invariant of quiver representations introduced by Joyce \cite{joyce2008configurationsIV}: see \cite[\S 6.7]{davison2015donaldson} for more details.
    For general smooth stacks, the map \Cref{eq-cohint-map} being an isomorphism implies a certain integrality property of the motivic invariant which was introduced by C.B., A.I.N. and T.K. in \cite{invariants} generalizing the work of Joyce \cite{joyce2008configurationsIV}. This aspect will be discussed in \cite{wall-crossing}. 
\end{remark}

The following is the main theorem for smooth stacks in this paper:

\begin{theorem}\label{thm-cohint-smooth}
    Let $\mathcal{U}$ be a smooth stack satisfying \Crefrange{item1-cohint-goodmoduli}{item4-cohint-gloeq} in \Cref{para-cohint-assumptions}.
    Assume further that $\mathcal{U}$ is almost orthogonal (see \Cref{def-symmetry}).
    Then $\mathcal{U}$ satisfies the cohomological integrality theorem.
\end{theorem}

\begin{para}\label{para-non-orthogonal-smooth-case}
    As we have already seen in \Cref{ssec-symmetric-stacks}, most of the almost symmetric stacks which arise as moduli stacks are at least almost orthogonal.
    For example, this theorem can be applied to the moduli stacks of semistable principal $G$-bundles and twisted $G$-Higgs bundles on Riemann surfaces. We will describe the cohomological integrality isomorphism explicitly for these examples in \Cref{ssec-G-Higgs}.
    On the other hand, we expect that cohomological integrality holds for any smooth stack satisfying \Crefrange{item1-cohint-goodmoduli}{item4-cohint-gloeq} in \Cref{para-cohint-assumptions}.

    We note that the cohomological integrality conjecture was also found by Hennecart in \cite[Conjecture 8.14]{hennecart2024cohomological} for affine quotient stacks while we were preparing this manuscript, 
    and he \cite[Theorem 8.7]{hennecart2024cohomological} proved the existence of $\mathrm{Aut}(\alpha)$-equivariant bounded mixed Hodge complexes $\mathcal{BPS}^{\alpha}_{{U}, \mathrm{Hen}}$ on $U_{\alpha}$ such that there exists an isomorphism as in \Cref{eq-cohint-map} for almost symmetric affine quotient stacks.
\end{para}

\begin{para}[Perversely degenerated version]

    We will state an equivalent reformulation of the cohomological integrality theorem using the component-free description of the symmetric cohomological Hall induction \Cref{eq-symmetrized-cohi-without-component}.
    For a non-negative integer $n$, set
    \[
    \mathcal{BPS}^n_{{U}} \coloneqq     \bigoplus_{\substack{ (\mathbb{Q}^n, \alpha) \in \mathsf{Face}^{\mathrm{nd}}(\mathcal{U}) }}  \mathcal{BPS}^{\alpha}_{{U}} 
    \]
    where the direct sum runs over all $n$-dimensional faces with the fixed source $\mathbb{Q}^n$.
    A choice of a global equivariant parameter $\boldsymbol{\mathcal{L}}$ for $\mathrm{Grad}^n_{\mathbb{Q}, \mathrm{nd}}(\mathcal{U})$ determines a natural map
    \[
       \mathcal{BPS}^n_{{U}} \otimes \mathrm{H}^*(\mathrm{B} \mathbb{G}_\mathrm{m}^n)_{\mathrm{vir}} \to {}^{\mathrm{p}} \mathcal{H}( p_{n, *} \mathcal{IC}_{\mathrm{Grad}^n_{\mathbb{Q}, \mathrm{nd}}(\mathcal{U})}). 
    \]
    Then the cohomological integrality for $\mathcal{U}$ with respect to a global equivariant parameter $\boldsymbol{\mathcal{L}}$ is equivalent to the following restriction of the map \Cref{eq-symmetrized-cohi-without-component} being an isomorphism:
    \begin{equation}\label{eq-coh-int-smooth-perversely-degenerate}
       \bigoplus_{n} \left( g_{n, *} ( \mathcal{BPS}^n_{{U}} \otimes \mathrm{H}^*(\mathrm{B} \mathbb{G}_\mathrm{m}^n)_{\mathrm{vir}} \otimes \mathrm{sgn}_n) \right)^{\mathrm{GL}_n(\mathbb{Q}), \mathrm{lfin}} \xrightarrow{\sum_{n} \left( *^{\mathrm{s}\mathcal{H}\mathrm{all}}_{n, \mathcal{U}} \right) } {}^{\mathrm{p}} \mathcal{H}(p_* \mathcal{IC}_{\mathcal{U}}).
    \end{equation}
    In particular, the validity of the cohomological integrality does not depend on the choice of a chamber.
\end{para}

\subsection{Cohomological integrality for smooth stacks: Proof}

\begin{para}
    The aim of this section is to prove \Cref{thm-cohint-smooth}.
    The main idea is to use the vanishing cycle functor with respect to non-degenerate quadratic functions to reduce the statement to the case of $\mathrm{B} G$.
\end{para}

\begin{para}[Cohomological integrality for products]\label{para-coh-int-product}
    Let $\mathcal{U}$ be a smooth stack satisfying assumptions \Crefrange{item1-cohint-goodmoduli}{item4-cohint-gloeq} in \Cref{para-cohint-assumptions} and $S$ be a smooth and separated connected algebraic space.
    It is obvious that the stack $\mathcal{U} \times S$ also satisfies \Crefrange{item1-cohint-goodmoduli}{item4-cohint-gloeq} in \Cref{para-cohint-assumptions}.
    We will show that the cohomological integrality for $\mathcal{U}$ is equivalent to the cohomological integrality for $\mathcal{U} \times S$.

    Firstly note that there are isomorphisms $\Grad^n(\mathcal{U} \times S) \cong \Grad^n(\mathcal{U}) \times S$ and \\$\Filt^n(\mathcal{U} \times S) \cong \Filt^n(\mathcal{U}) \times S$.
    These isomorphisms imply an isomorphism of formal lattices $\mathrm{CL}_{\mathbb{Q}}(\mathcal{U}) \cong \mathrm{CL}_{\mathbb{Q}}(\mathcal{U} \times S)$.
    Therefore we will identify faces and cones for $\mathcal{U}$ and $\mathcal{U} \times S$.
    For a non-degenerate face $(F, \alpha) \in \mathsf{Face}^{\mathrm{nd}}(\mathcal{U})$, it is obvious by definition that the following isomorphism holds:
    \[
     \mathcal{BPS}_{U \times S}^{\alpha} \cong \mathcal{BPS}_{U}^{\alpha}  \boxtimes \mathcal{IC}_{S}.
    \]
    Also, the cotangent sign representations $\mathrm{sgn}_{\alpha}$ for $\mathcal{U}$ and $\mathcal{U} \times S$ are identical.

    Assume first that cohomological integrality holds for $\mathcal{U} \times S$. For a non-degenerate face $(F, \alpha) \in \mathsf{Face}^{\mathrm{nd}}(\mathcal{U})$, take a global equivariant parameter $\boldsymbol{\mathcal{L}}^{\alpha} = (\mathcal{L}_1^{\alpha}, \ldots, \mathcal{L}_n^{\alpha})$ on $\mathcal{U}_{\alpha}$.
    Then $\boldsymbol{\mathcal{L}}_S^{\alpha} \coloneqq (\mathcal{L}_1^{\alpha} \boxtimes \mathcal{O}_S, \ldots, \mathcal{L}_n^{\alpha} \boxtimes \mathcal{O}_S)$ defines a global equivariant parameter for $\mathcal{U}_{\alpha} \times S$.
    The cohomological integrality isomorphism for $\mathcal{U} \times S$ restricted to any point $s \in S$ implies the cohomological integrality for $\mathcal{U}$.

    Conversely, assume that cohomological integrality holds for $\mathcal{U}$. For a non-degenerate face $(F, \alpha) \in \mathsf{Face}^{\mathrm{nd}}(\mathcal{U} \times S)$, take a global equivariant parameter $\tilde{\boldsymbol{\mathcal{L}}}^{\alpha} = (\tilde{\mathcal{L}}_1^{\alpha}, \ldots, \tilde{\mathcal{L}}_n^{\alpha})$ on $\mathcal{U}_{\alpha} \times S$.
    Then for any point $s \in S$, the restriction $\tilde{\boldsymbol{\mathcal{L}}}^{\alpha} |_{s} = (\tilde{\mathcal{L}}_1^{\alpha} |_{s}, \ldots, \tilde{\mathcal{L}}_n^{\alpha}|_{s})$ defines a global equivariant parameter for $\mathcal{U}_{\alpha}$.
    Then the cohomological integrality for $\mathcal{U} \times \{ s \}$ implies that the cohomological integrality map \Cref{eq-cohint-map} for $\mathcal{U} \times S$ is an isomorphism over $U \times \{ s \}$.
    In particular, the cohomological integrality map for $\mathcal{U} \times S$ is an isomorphism.

\end{para}

\begin{para}[Cohomological integrality for finite group quotients]\label{para-coh-int-finite-quotient}
    Let $\tilde{\mathcal{U}}$ be a smooth stack acted on by a finite group $\Gamma$ and  set $\mathcal{U} \coloneqq \tilde{\mathcal{U}} / \Gamma$. 
    Assume that $\tilde{\mathcal{U}}$ and $\mathcal{U}$ satisfy \Crefrange{item1-cohint-goodmoduli}{item4-cohint-gloeq} in \Cref{para-cohint-assumptions}.
    Then cohomological integrality for $\tilde{\mathcal{U}}$ implies cohomological integrality for $\mathcal{U}$.
    To see this, consider the perversely degenerated cohomological integrality map  \Cref{eq-coh-int-smooth-perversely-degenerate} for the stack $\tilde{\mathcal{U}}$:
    \[
        \bigoplus_{n} \left( \tilde{g}_{n, *} ( \mathcal{BPS}^n_{\tilde{{U}}} \otimes \mathrm{H}^*(\mathrm{B} \mathbb{G}_\mathrm{m}^n)_{\mathrm{vir}} \otimes \tilde{\mathrm{sgn}}_n ) \right)^{\mathrm{GL}_n(\mathbb{Q}), \mathrm{lfin}} \xrightarrow{\sum_{n} \left( *^{\mathrm{s}\mathcal{H}\mathrm{all}}_{n, \tilde{ \mathcal{U}}} \right) } {}^{\mathrm{p}} \mathcal{H}(p_* \mathcal{IC}_{\tilde{\mathcal{U}}}).   
    \]
    Here $\tilde{\mathrm{sgn}}_n $ denotes the cotangent sign local system for $\tilde{\mathcal{X}}$ and $\tilde{g}_n \colon \mathrm{Grad}_{\mathbb{Q}, \mathrm{nd}}^{n}(\tilde{\mathcal{U}})_{\GIT} \to \tilde{U}$ is the natural map.
    Then by using the commutative diagram \Cref{eq-cohi-finite-quotient-compatible}, we see that the $\Gamma$-invariant part of the above map after pushing down to $U$ recovers the perversely degenerated cohomological integrality map for $\mathcal{U}$. Hence we conclude.
\end{para}

\begin{para}[Cohomological integrality for \'etale covers]\label{para-coh-int-etale-cover}
    Let $\mathcal{U}$ be a smooth stack satisfying \Crefrange{item1-cohint-goodmoduli}{item4-cohint-gloeq} in \Cref{para-cohint-assumptions}, with good moduli space $p\colon \mathcal{U}\rightarrow U$.
    Take an \'etale cover $V \to U$ from an algebraic space and set $\mathcal{V} \coloneqq V \times_{U} \mathcal{U}$.
    Assume that $\mathcal{V}$ also satisfies \Crefrange{item1-cohint-goodmoduli}{item4-cohint-gloeq} in \Cref{para-cohint-assumptions}.
    Then it follows from the commutativity of the diagram \Cref{eq-etale-cohi-compatible-symmetrized} that cohomological integrality for $\mathcal{V}$ implies cohomological integrality for $\mathcal{U}$.
    Also cohomological integrality for $\mathcal{V}$ with respect to a choice of global equivariant parameters $\boldsymbol{\mathcal{L}}_{\alpha}$ on $\mathcal{V}_{\alpha}$ pulled back from $\mathcal{U}_{\alpha}$ for special faces $\alpha$ follows from cohomological integrality for $\mathcal{U}$.
\end{para}

\begin{para}[Proof of Theorem \ref{thm-cohint-smooth}]

    We now prove Theorem \ref{thm-cohint-smooth}.
    This will be done by induction on the maximal dimension of stabilizer groups of closed points.
    If the stabilizer groups are zero dimensional, then the statement is obvious.

    From now on, we will assume that we have proved the statement for stacks with maximal stabilizer dimension smaller than $d$ and $\mathcal{U}$ has the maximal stabilizer dimension $d$.
    Using the discussion in \Cref{para-coh-int-etale-cover} and the local structure theorem for smooth stacks \cite[Theorem 4.12]{_Alper_ALunaetaleslicetheoremforalgebraicstacks}, we may assume $\mathcal{U} = Y / G$ for a reductive group $G$ and an affine smooth scheme $Y$.
    Further, by using Lemma \ref{lemma-global-equiv-parameter-BG}, we may assume that there exists a point $y \in Y$ fixed by $G$ and the global equivariant parameter for $(Y / G)_{\alpha}$ for a special face $\alpha$ is pulled back from $\mathrm{B} L_{\alpha}$ where $L_{\alpha}$ denotes the Levi subgroup of $G$ corresponding to $\alpha$.
    Then using the discussion in  \Cref{para-coh-int-etale-cover} again together with \cite[Theorem 1.2]{_Alper_ALunaetaleslicetheoremforalgebraicstacks}, we may further reduce to the case $\mathcal{U} = V / G$ for an almost orthogonal representation $V$ of $G$.
    Let $G^{\circ} \subset G$ denote the neutral component and set $\Gamma \coloneqq G / G^{\circ}$.
    Then the natural map $V / G \to \mathrm{B} G \to \mathrm{B} \Gamma$ determines a $\Gamma$-action on $V / G^{\circ}$ with $(V / G^{\circ}) / \Gamma \cong V/G$.
    Using the discussion in \Cref{para-coh-int-finite-quotient}, we may reduce to the case $G = G^{\circ}$, i.e., when $G$ is connected.
    Also by using the discussion in \Cref{para-coh-int-product}, we may assume that $V$ does not contain a trivial representation as a direct summand.
    Therefore the origin $0 \in V / G$ is the unique point with the maximal stabilizer dimension.
    In particular, by the induction hypothesis, we may assume that the cofibre $\mathcal{G}$ of the cohomological integrality map \Cref{eq-cohint-map} is supported on the origin.
    It is enough to prove $\mathcal{G} = 0$.

    Let $q \colon V / G \to \mathbb{A}^1$ be the  quadratic function induced by the orthogonal structure on $V$ and $\bar{q} \colon V \GIT G \to \mathbb{A}^1$ be the induced function. 
    Since $\mathcal{G}$ is supported on the origin, to prove $\mathcal{G} = 0$, it is enough to prove $\varphi_{\bar{q}}(\mathcal{G}) = 0$.
    In other words, it is enough to show that the functor $\varphi_{\bar{q}}$ applied to the map \Cref{eq-cohint-map} is an isomorphism.
    Let $(F, \alpha) \in \mathsf{Face}^{\mathrm{nd}}(\mathcal{U})$ be a non-degenerate face with $\dim F< \rank G$. Set $q_{\alpha} \coloneqq q |_{\mathcal{U_{\alpha}}}$ and $\bar{q}_{\alpha} \coloneqq \bar{q} |_{U_{\alpha}}$. We claim the following vanishing
    \[
        \varphi_{\bar{q}_{\alpha} }(\mathcal{BPS}_{U}^{\alpha}) = 0.
    \]
    By using \Cref{eq-vanishing-good-moduli}, it is enough to prove ${}^{\mathrm{p}} \cH^{\dim \alpha}(p_{\alpha, *} \varphi_{q_{\alpha}}(\mathcal{IC}_{\mathcal{U}_{\alpha}}) ) = 0$.
    Since the critical locus of $q_{\alpha}$ is $\mathrm{B} L_{\alpha}$ where $L_{\alpha}$ is the Levi subgroup corresponding to $\alpha$, we have $\varphi_{q_{\alpha}}(\mathcal{IC}_{\mathcal{U}_{\alpha}}) \cong \mathcal{IC}_{\mathrm{B} L_{\alpha}}$.
    By the inequality $\dim L_{\alpha} > \dim \alpha$, we obtain the desired vanishing.
    Therefore if we let $T \subset G$ be the maximal torus and $W$ be the Weyl group, we are reduced to proving that the following map is an isomorphism:
    \[
       \varphi_{\bar{q}} \left( ( \mathbb{Q}_{\{ 0 \}} \otimes \mathrm{H}^*(\mathrm{B} T)_{\mathrm{vir}} \otimes \mathrm{sgn}_{\alpha_{\mathrm{max}}})^{W} \right)  \xrightarrow[]{ \varphi_{\bar{q}} (\mathrm{Ind}_{\alpha_{\mathrm{max}}, \sigma, \boldsymbol{\mathcal{L}} }^{\mathrm{sym}} )} \varphi_{\bar{q}} \left( p_* \mathcal{IC}_{U}  \right)
    \]
    where $\alpha_{\mathrm{max}}$ denotes the maximal non-degenerate face and $\sigma \subset \mathbb{Q}^{\dim T}$ is a chamber.
    Note that there exists an equivalence of $(-1)$-shifted symplectic stacks
    \[
     \mathrm{Crit}(q) \cong  \mathrm{Crit}(0 \colon \mathrm{B} G \to \mathbb{A}^1). 
    \]
    These critical charts define isomorphic standard orientations, up to the choice of the grading.
    Therefore by using the commutativity of the diagram \Cref{eq-symmetrized-cohi-critical} twice,
    it is enough to prove that the following restriction of the cohomological Hall induction map for $\mathrm{B} G$ is an isomorphism
    \[
        (\mathrm{H}^*(\mathrm{B} T)_{\mathrm{vir}} \otimes \mathrm{sgn}_{\alpha_{\mathrm{max}}'})^{W} \xrightarrow[]{} \mathrm{H}^*(\mathrm{B} G)_{\mathrm{vir}}
    \]
    where $\alpha_{\mathrm{max}}'$ denotes the maximal non-degenerate face for $\mathrm{B} G$. 
    In other words, it is enough to prove the cohomological integrality for $\mathrm{B} G$.
    By a similar vanishing cycle argument, the cohomological integrality for $\mathrm{B} G$ follows from the cohomological integrality for $\mathfrak{g}/G$ which we prove now.
    By \Cref{cor:small_decom_thm_smooth_stack}, we have an isomorphism
    \[
        \mathcal{BPS}_{\mathfrak{g} \GIT G}^{\alpha} \cong \begin{cases}
           \mathcal{IC}_{\mathfrak{h}} & \text{if $\alpha = \alpha''_{\mathrm{max}}$,} \\
           0                        & \text{otherwise.} 
        \end{cases}   
   \]
    where $\alpha''_{\mathrm{max}}$ denotes the maximal non-degenerate face for $\mathfrak{g} / G$. 
    By the induction hypothesis, we may assume that the cohomological integrality map for $\mathfrak{g}/G$ is an isomorphism outside $\mathfrak{z}/G$ where $\mathfrak{z} \subset \mathfrak{g}$ is the centre.
    Note that the cohomological integrality map \Cref{eq-cohint-map} is equivariant with respect to the natural action of the centre $\mathfrak{z}$.
    Since a $\mathfrak{z}$-equivariant perverse sheaf on $\mathfrak{z}$ is constant, it is therefore enough to prove that the cohomological integrality map  for $\mathfrak{g}/G$ induces an isomorphism on global sections.
    In other words, it is enough to show that the cohomological Hall induction map
    \[
    \mathrm{H}^*(\mathfrak{h}/T)^{W} \to    \mathrm{H}^*(\mathfrak{g}/G) 
    \]
    is an isomorphism. By the explicit description of the cohomological Hall induction \Cref{eq-CoHI-explicit} together with the observation that the set of roots for $G$ coincides with the set of non-zero weights for the adjoint representation $\mathfrak{g}$,
    we see that this map is identified with the natural isomorphism $\mathrm{H}^*(\mathrm{B} T)^{W} \cong \mathrm{H}^*(\mathrm{B} G)$, hence the conclusion follows.

\end{para}

\subsection{Cohomological integrality for \texorpdfstring{$(-1)$}{(−1)}-shifted symplectic stacks}

\begin{para}
    Let $\mathcal{X}$ be a $(-1)$-shifted symplectic stack.
    Assume that $\mathcal{X}$ satisfies the conditions \Crefrange{item1-cohint-goodmoduli}{item4-cohint-gloeq} in \Cref{para-cohint-assumptions} and let $p \colon \mathcal{X} \to X$ be the good moduli space morphism.
    We fix an orientation $o$ for $\mathcal{X}$.
    As we have seen in \Cref{para-localize-orientation-is-equivariant}, for each non-degenerate face $(F, \alpha) \in \mathsf{Face}^{\mathrm{nd}}(\mathcal{X})$, 
    the orientation $o$ induces an $\mathrm{Aut}(\alpha)$-equivariant orientation $\alpha^{\star} o$ on $\mathcal{X}_{\alpha}$.
    We will always equip $\mathcal{X}_{\alpha}$ with this orientation.
    The $\mathrm{Aut}(\alpha)$-equivariance of the orientation implies that the perverse sheaves $\varphi_{\mathcal{X}_{\alpha}}$ and $\mathcal{BPS}_{{X}}^{\alpha}$ are also $\mathrm{Aut}(\alpha)$-equivariant.

    Let $\sigma  \subset F$ be a chamber with respect to the cotangent arrangement and $\boldsymbol{\mathcal{L}}$ be a global equivariant parameter for $\mathcal{X}_{\alpha}$.
    By repeating the construction of the map \Cref{eq-CoHI-inv-bps-piece-smooth} using the cohomological Hall induction for $(-1)$-shifted symplectic stacks introduced in \Cref{para-cohi--1-shifted-symplectic}, we obtain a natural map
    \begin{equation}\label{eq-CoHI-inv-bps-piece--1-symplectic}
        \mathrm{Ind}_{\alpha, \sigma, \boldsymbol{\mathcal{L}} }^{\mathrm{sym}} \colon (g_{\alpha, *} \mathcal{BPS}^{\alpha}_{{X}} \otimes \mathrm{H}^*(\mathrm{B} T_{F})_{\mathrm{vir}})^{\mathrm{Aut}(\alpha)}  \to p_* \mathcal{\varphi}_{\mathcal{X}}.
     \end{equation}
\end{para}

\begin{definition*}
    Let $\mathcal{X}$ be an oriented $(-1)$-shifted symplectic stack satisfying \Crefrange{item1-cohint-goodmoduli}{item4-cohint-gloeq} in \Cref{para-cohint-assumptions}.
    We say that $\mathcal{X}$ satisfies the \emph{cohomological integrality theorem} if the map
    \begin{equation}\label{eq-cohiint-map--1-symplectic}
    \bigoplus_{(F, \alpha) \in \mathsf{Face}^{\mathrm{nd}}(\mathcal{X})}  (g_{\alpha, *} \mathcal{BPS}^{\alpha}_{{X}} \otimes \mathrm{H}^*(\mathrm{B} T_{F})_{\mathrm{vir}})^{\mathrm{Aut}(\alpha)}  \xrightarrow[]{\bigoplus_{\alpha} \mathrm{Ind}_{\alpha, \sigma, \boldsymbol{\mathcal{L}} }^{\mathrm{sym}} } p_* \varphi_{\mathcal{X}}
    \end{equation}
    is an isomorphism for any choice of the global equivariant parameter for $\mathcal{X}_{\alpha}$ and choice of a chamber $\sigma \subset F$.
\end{definition*}

\begin{remark}\label{rmk-cohint-source--1-shifted-symplectic}
    By \Cref{cor-non-special-vanishing--1shifted symplectic} and the equality \Cref{eq-special-face-central-rank}, the source of the map \Cref{eq-cohiint-map--1-symplectic} is identified with the following object:
    \[
        \bigoplus_{(F, \alpha) \in \mathsf{Face}^{\mathrm{sp}}(\mathcal{X})}  (g_{\alpha, *} \mathcal{BPS}_{{X}_{\alpha}} \otimes \mathrm{H}^*(\mathrm{B} T_{F})_{\mathrm{vir}})^{\mathrm{Aut}(\alpha)}.   
    \]
    In particular, the source of the map \Cref{eq-cohiint-map--1-symplectic}  is a finite sum over each connected component of $X$, thanks to the finiteness of special faces for stacks satisfying \Cref{item1-cohint-goodmoduli} and \Cref{item2-cohint-qcgr} discussed in \Cref{para-finiteness-theorem}.
\end{remark}

The following is the main theorem for $(-1)$-shifted symplectic stacks in this paper:

\begin{theorem}\label{thm-coh-int--1-shifted-symplectic}
    Let $\mathcal{X}$ be an oriented $(-1)$-shifted symplectic stack satisfying \Crefrange{item1-cohint-goodmoduli}{item4-cohint-gloeq} in \Cref{para-cohint-assumptions}.
    Assume further that $\mathcal{X}$ is almost orthogonal.
    Then $\mathcal{X}$ satisfies the cohomological integrality theorem.
\end{theorem}

\begin{para}
    As in the smooth case, we believe that the cohomological integrality theorem holds for any oriented $(-1)$-shifted symplectic stack satisfying \Cref{item1-cohint-goodmoduli} -- \Cref{item4-cohint-gloeq} in \Cref{para-cohint-assumptions}.
    By the proof of the above theorem, this follows from the cohomological integrality conjecture for smooth stacks satisfying \Cref{item1-cohint-goodmoduli} -- \Cref{item4-cohint-gloeq} in \Cref{para-cohint-assumptions}.
\end{para}

\begin{para}
    We now prove \Cref{thm-coh-int--1-shifted-symplectic}.
    By repeating the discussion in \Cref{para-coh-int-etale-cover} and using the commutativity of the diagram \Cref{eq-etale-cohi-compatible-symmetrized--1-shifted-symplectic}, it is enough to prove cohomological integrality after taking an \'etale cover.
    In particular, by using Theorem \ref{thm-Darboux} and the discussion in \Cref{para-any-orientation-is-locally-standard}, we may assume that there exists a reductive group $G$, a smooth affine scheme $V$ acted on by $G$ such that $V/G$ is almost orthogonal, and a function $f \colon V / G \to \mathbb{A}^1$ such that there exists an equivalence of oriented $(-1)$-shifted symplectic stacks
    \[
    \mathcal{X} \cong \mathrm{Crit}(f/G).
    \]
    Now by using Lemma \ref{lemma-global-equiv-parameter-BG}, we may further assume that the global equivariant parameters for $\mathcal{X}_{\alpha}$ are obtained from the restriction of a global equivariant parameter for $(V/G)_{\tilde{\alpha}}$, where $\tilde{\alpha}$ is the image of $\alpha$ in $\mathsf{Face}(V/G)$.
    Therefore the theorem follows from the cohomological integrality for $V/G$ proved in \Cref{thm-cohint-smooth} together with the commutativity of the diagram \Cref{eq-symmetrized-cohi-critical}.

\end{para}

\subsection{Cohomological integrality for \texorpdfstring{$0$}{0}-shifted symplectic stacks}

\begin{para}
    Let $\mathcal{Y}$ be a $0$-shifted symplectic stack.
    Assume that $\mathcal{Y}$ satisfies the conditions \Crefrange{item1-cohint-goodmoduli}{item4-cohint-gloeq} in \Cref{para-cohint-assumptions} and let $p \colon \mathcal{Y} \to Y$ be the good moduli space morphism.
    Recall that for a non-degenerate face $(F, \alpha) \in \mathsf{Face}^{\mathrm{nd}}(\mathcal{Y})$, we have constructed in \Cref{thm-BPS-cotangent-product} an $\mathrm{Aut}(\alpha)$-equivariant pure monodromic mixed Hodge module
    \[
    \mathcal{BPS}_{{Y}}^{\alpha} \in \mathsf{MMHM}(Y_{\alpha}).    
    \]
    By the dimensional reduction isomorphism \Cref{eq-dimensional-reduction}, we obtain an $\mathrm{Aut}(\alpha)$-equivariant map
    \[
        \mathcal{BPS}_{{Y}}^{\alpha} \to \mathbb{L}^{ \vdim \mathcal{Y}_{\alpha} / 2} \otimes p_{\alpha, *} \mathbb{D}\mathbb{Q}_{\mathcal{Y}_{\alpha}} \otimes \mathrm{sgn}_{\alpha}.
    \]
    The $\mathrm{Aut}(\alpha)$-equivariance follows from \Cref{lem-ori-localize-indep-segment-up-to-sign}.

    Let $\sigma  \subset F$ be a chamber with respect to the cotangent arrangement and $\boldsymbol{\mathcal{L}}$ be a global equivariant parameter for $\mathcal{Y}_{\alpha}$.
    By repeating the construction of the map \Cref{eq-CoHI-inv-bps-piece-smooth} using the cohomological Hall induction for quasi-smooth derived algebraic stacks introduced in \Cref{para-cohi-for-quasi-smooth}, we obtain a natural map
    \begin{equation}\label{eq-CoHI-inv-bps-piece-0-symplectic}
        \mathrm{Ind}_{\alpha, \sigma, \boldsymbol{\mathcal{L}} }^{\mathrm{sym}} \colon (g_{\alpha, *} \mathcal{BPS}^{\alpha}_{Y} \otimes \mathrm{H}^*(\mathrm{B} T_{F}))^{\mathrm{Aut}(\alpha)}  \to \mathbb{L}^{\vdim \mathcal{Y} / 2} \otimes p_* \mathbb{D} \mathbb{Q}_{\mathcal{Y}}
     \end{equation}
     which factors through a complex $(\mathbb{L}^{ \vdim \mathcal{Y}_{\alpha} / 2} \otimes g_{\alpha, *} p_{\alpha, *} \mathbb{D}\mathbb{Q}_{\mathcal{Y}_{\alpha}} \otimes \mathrm{sgn}_{\alpha})^{\mathrm{Aut(\alpha)}}$.

    \end{para}

\begin{definition*}
    Let $\mathcal{Y}$ be a $0$-shifted symplectic stack satisfying \Crefrange{item1-cohint-goodmoduli}{item4-cohint-gloeq} in \Cref{para-cohint-assumptions}.
    We say that $\mathcal{Y}$ satisfies the \emph{cohomological integrality theorem} if the map
    \begin{equation}\label{eq-cohint-map-0-shifted-symplectic}
    \bigoplus_{(F, \alpha) \in \mathsf{Face}^{\mathrm{nd}}(\mathcal{Y})}  (g_{\alpha, *} \mathcal{BPS}^{\alpha}_{Y} \otimes \mathrm{H}^*(\mathrm{B} T_{F}) )^{\mathrm{Aut}(\alpha)}  \xrightarrow[]{\bigoplus_{\alpha} \mathrm{Ind}_{\alpha, \sigma, \boldsymbol{\mathcal{L}} }^{\mathrm{sym}} } \mathbb{L}^{\vdim \mathcal{Y} / 2} \otimes p_* \mathbb{D} \mathbb{Q}_{\mathcal{Y}}
    \end{equation}
    is an isomorphism for any choice of the global equivariant parameter for $\mathcal{Y}_{\alpha}$ and choice of a chamber $\sigma \subset F$ with respect to the cotangent arrangement.
\end{definition*}

\begin{remark}\label{rmk-cohint-source-0-shifted-symplectic}
    By \Cref{rmk-cohint-source--1-shifted-symplectic} and \Cref{lem-ori-localize-indep-segment-up-to-sign}, the source of the map \Cref{eq-cohint-map-0-shifted-symplectic} is identified with the following object
    \[
        \bigoplus_{(F, \alpha) \in \mathsf{Face}^{\mathrm{sp}}(\mathcal{Y})}  (g_{\alpha, *} \mathcal{BPS}_{{Y}_{\alpha}} \otimes \mathrm{H}^*(\mathrm{B} T_{F}) \otimes \mathrm{sgn}_{\alpha})^{\mathrm{Aut}(\alpha)}
    \]
    and it is a finite sum over each connected component of $Y$.

\end{remark}

The following statement is an immediate consequence of \Cref{thm-coh-int--1-shifted-symplectic} together with \Cref{cor--1^-shifted-cotangent-orthogonal}:

\begin{theorem}\label{thm-cohint-0-shifted}
    Let $\mathcal{Y}$ be a $0$-shifted symplectic stack satisfying \Crefrange{item1-cohint-goodmoduli}{item4-cohint-gloeq} in \Cref{para-cohint-assumptions}.
    Assume further that $\mathcal{Y}$ is almost orthogonal.
    Then $\mathcal{Y}$ satisfies the cohomological integrality theorem.
\end{theorem}

\section{Applications}

In this section, we will describe the cohomological integrality theorems (\Cref{thm-cohint-smooth}, \Cref{thm-coh-int--1-shifted-symplectic} and \Cref{thm-cohint-0-shifted}) explicitly for several stacks of interest.

\subsection{Cohomological integrality for quotient stacks}

\begin{para}[Cohomological integrality for orthogonal representations]
    Let $G$ be a connected reductive algebraic group with the Weyl group $W$ and $V$ be an orthogonal $G$-representation.
    We let $p \colon V / G \to V \GIT G$ be the good moduli space morphism.
    For each $\lambda \colon \mathbb{G}_\mathrm{m}^n \hookrightarrow G$, we let $L_{\lambda}$ be the corresponding Levi subgroup, and denote by $\alpha(\lambda)$ the face for $V / G$ corresponding to $\lambda$.
    Then by using the description of the stack of graded points for quotient stacks explained in \Cref{para-quotient-component-arrangement}, the cohomological integrality theorem for the quotient stack (\Cref{thm-cohint-smooth}) can be written as
    \[
    \bigoplus_{n \in \mathbb{Z}_{\geq 0}} \bigoplus_{\lambda \colon \mathbb{G}_\mathrm{m}^n \hookrightarrow G} g_{\lambda, *} (\mathcal{BPS}_{V \GIT G}^{\alpha(\lambda)} \otimes \mathrm{H}^*(\mathrm{B} \mathbb{G}_\mathrm{m}^n)_{\mathrm{vir}} \otimes \mathrm{sgn}_{\lambda})^{\mathrm{Aut}(\alpha(\lambda))}  \cong p_* \mathcal{IC}_{V / G}
    \]
    where $\lambda$ runs over all inclusions up to conjugation. Here $\mathrm{sgn}_{\lambda}$ is the cotangent sign representation of $\mathrm{Aut}(\alpha(\lambda))$. Further, we have
    \[
         \mathcal{BPS}_{V \GIT G }^{\alpha(\lambda)} \coloneqq 
         \begin{cases}
            \mathcal{IC}_{V^{\lambda} \GIT L_{\lambda} }, & \text{if $(V^{\lambda} / L_{\lambda}) /\mathrm{B} \mathbb{G}_\mathrm{m}^n  \to V^{\lambda} \GIT L_{\lambda}$ is generically quasi-finite,}  \\
            0, & \mathrm{otherwise.}
         \end{cases}
    \]
    We note that the automorphism group $\mathrm{Aut}(\alpha(\lambda))$ is realized as the subgroup of the relative Weyl group $W_G(L_{\lambda})$
    preserving the subspace of the cocharacter lattice corresponding to $\lambda$. The same group was considered in a similar context by \textcite[\S 2]{hennecart2024cohomological}, which is denoted by $\overline{W}_{\lambda}$ there.
\end{para}

\begin{para}[Cohomological integrality for cotangent representations]
    Let $G$ be a connected reductive algebraic group and $V$ be a representation of $G$.
    Let $\mu \colon V \times V^{\vee}  \to \mathfrak{g}^{\vee}$ be the moment map.
    Then we have $\mu^{-1}(0) / G \cong \mathrm{T}^*(V / G)$ and it is equipped with a $0$-shifted symplectic structure.
    Since $V \oplus V^{\vee}$ is an orthogonal $G$-representation, the cohomological integrality theorem (= \Cref{thm-cohint-0-shifted}) holds for the stack $\mu^{-1}(0) / G $.
    For each $\lambda \colon \mathbb{G}_\mathrm{m}^n \hookrightarrow G$, we let $\mu_{\lambda} \colon V^{\lambda} \times V^{\lambda, \vee} \to \mathfrak{l}_{\lambda}^{\vee}$ be the moment map for the $L_{\lambda}$-action on $V^{\lambda}$.
    Let $p \colon \mu^{-1}(0) / G \to \mu^{-1}(0) \GIT G$ be the good moduli space morphism and $g_{\lambda} \colon \mu_{\lambda}^{-1}(0) \GIT L_{\lambda} \to \mu^{-1}(0) \GIT G$ be the natural map.
    We denote by $\alpha(\lambda)$ the face for $\mu^{-1}(0) / G$ corresponding to $\lambda$.
    Then by using the description of the stack of graded points for quotient stacks explained in \Cref{para-quotient-component-arrangement}, the cohomological integrality theorem  (=\Cref{thm-cohint-0-shifted} + \Cref{rmk-cohint-source-0-shifted-symplectic}) can be written as
    \[
        \bigoplus_{n \in \mathbb{Z}_{\geq 0}} \bigoplus_{\substack{\lambda \colon \mathbb{G}_\mathrm{m}^n \hookrightarrow G \\ \text{$\alpha(\lambda)$: special}}} g_{\lambda, *} (\mathcal{BPS}_{\mu_{\lambda}^{-1}(0) \GIT L_{\lambda}} \otimes \mathrm{H}^*(\mathrm{B} \mathbb{G}_\mathrm{m}^n) \otimes \mathrm{sgn_{\alpha(\lambda)}})^{\mathrm{Aut}(\alpha(\lambda))}  \cong \mathbb{L}^{\dim V / G} \otimes p_* \mathbb{D}\mathbb{Q}_{\mu^{-1}(0)/G}
     \]
    where $\lambda$ runs over inclusions up to conjugation.

    It is a very interesting problem to explicitly describe the BPS sheaf $\mathcal{BPS}_{\mu^{-1}(0) \GIT G}$.
    When the stack $V / G$ can be written as a moduli stack of a quiver representations, this is done by Davison, Hennecart and Schlegel Mejia in \cite[Theorem 1.7]{davison2023bps}.
    For $V = \mathfrak{g}$ the adjoint representation of $G$, Hennecart \cite{Hennecart_note} computed the BPS sheaf $\mathcal{BPS}_{\mu^{-1}(0) \GIT G}$ using the result of \textcite{premet2003nilpotent}.
\end{para}

\subsection{Cohomological integrality for linear moduli stacks}

\begin{para}
    We will specialize the discussion on the cohomological integrality theorem in \Cref{sect-cohint} to the case of linear moduli stacks introduced in \Cref{para-linear-moduli}.
    As a consequence, we will prove a generalization of the cohomological integrality theorem for quivers with potentials \cite[Theorem A, Theorem C]{_Davison_Meinhardt_CoDT} to general $3$-Calabi--Yau dg-categories admitting commutative orientation data, an enhancement of usual orientation data that we introduce in \Cref{para-COD}.
\end{para}

\begin{para}[Assumptions]\label{para-assumptions-linear}
    We will work with a linear moduli stack $\mathcal{M}$ with the following assumptions:
    \begin{enumerate}
        \item $\mathcal{M}$ has affine diagonal and admits a good moduli space $p \colon \mathcal{M} \to M$.\label{item1-linear-goodmoduli}
        \item Connected components of $\mathcal{M}$ are quasi-compact. Also, the direct sum map $\oplus \colon  \mathcal{M} \times \mathcal{M} \to \mathcal{M}$ is quasi-compact. In particular, $\mathcal{M}$ has quasi-compact graded points.  \label{item2-linear-qcgr}
        \item $\mathcal{M}$ is almost orthogonal. \label{item3-linear-symmetric}
        \item $\mathcal{M}$  admits a global equivariant parameter $\mathcal{L} \in \mathrm{Pic}(\mathcal{M})$ with respect to the $\mathrm{B} \mathbb{G}_\mathrm{m}$-action. \label{item4-linear-gloeq}
        \item $\mathcal{M}$ admits an Euler pairing: see \Cref{para-euler-pairing} for the definition. \label{item5-euler-pairing}
    \end{enumerate}
\end{para}

\begin{para}[Examples of linear moduli stacks satisfying assumptions]
    We now explain that most almost symmetric linear moduli stacks of interest satisfy these assumptions.
    Let $\mathcal{C}$ be a dg-category of finite type over $\mathbb{C}$, $\mathcal{M}_{\mathcal{C}}$ be the moduli stack of objects in $\mathcal{C}$ and $\mathcal{M} \subset \mathcal{M}_{\mathcal{C}}$ be a $1$-Artin open substack which is closed under direct sums and summands and contains the zero object as an open and closed substack, having quasi-compact connected components and admitting a good moduli space.
    Assume further that there exists an abelian subcategory $\mathcal{A} \subset \mathcal{C}$ such that $\mathcal{M}$ parametrizes the objects in $\mathcal{A}$ and that the equality
    \[
    \dim \mathrm{Hom}(E, F[1]) =  \dim \mathrm{Hom}(F, E[1])    
    \]
    holds for any $E, F \in \mathcal{A}$ corresponding to closed points.
    As we have seen in \Cref{para-examples-linear}, $\mathcal{M}$ is a linear moduli stack as long as $\mathcal{M}$ is quasi-separated and has affine stabilizers.
    We claim that $\mathcal{M}$ automatically satisfies the assumptions \Cref{item1-linear-goodmoduli}, \Cref{item2-linear-qcgr}, \Cref{item3-linear-symmetric} and \Cref{item5-euler-pairing}.
    The condition \Cref{item1-linear-goodmoduli} is proved in \cite[Lemma 5.9]{davison2021purity}.
    The condition \Cref{item2-linear-qcgr} can be proved as follows: take a map $t \colon \Spec A \to  \mathcal{M}$ corresponding to a functor $F_t \colon \mathcal{C} \to \mathsf{Perf}(A)$ and consider the following Cartesian square:
\[\begin{tikzcd}
	{\mathcal{Z}_t} & {\Spec A} \\
	{\mathcal{M} \times \mathcal{M}} & {\mathcal{M}.}
	\arrow[from=1-1, to=1-2]
	\arrow[from=1-1, to=2-1]
	\arrow["\lrcorner"{anchor=center, pos=0.125}, draw=none, from=1-1, to=2-2]
	\arrow["t", from=1-2, to=2-2]
	\arrow["\oplus", from=2-1, to=2-2]
\end{tikzcd}\]
    Since giving a direct summand of $F_t$ is equivalent to giving an idempotent on $F_t$, we can find a closed embedding
    \[
     \mathcal{Z}_t \hookrightarrow \mathrm{Tot}_{\mathcal{M}}(t^* \mathbb{T}_{\mathcal{M}}[-1])
    \]
    which in particular shows that $\mathcal{Z}_t$ is quasi-compact.
    The condition \Cref{item3-linear-symmetric} follows from the discussion in \Cref{para-orthogonal-linear-moduli}.
    The condition \Cref{item5-euler-pairing} is verified in \Cref{para-euler-pairing}.

    By combining the above discussion with the existence result of a global equivariant parameter for the moduli stack of objects in a dg-category which we have discussed in \Cref{para-equiv-parameter-linear},
    we obtain the following claim:
\end{para}

\begin{lemma}
    Let $X$ be a smooth quasi-projective variety and $H$ be an ample divisor and fix a monic polynomial $p$. Assume that for all $H$-semistable sheaves $E$ and $F$ with reduced Hilbert polynomial $p$, the identity 
        $\dim \Ext^1(E, F) = \dim \Ext^1(F, E)$ holds.  Then the moduli stack $\mathcal{M}_{X, p}^{H\textnormal{-}\mathrm{ss}}$ of $H$-semistable sheaves with reduced Hilbert polynomial $p$ satisfies the assumptions \Crefrange{item1-linear-goodmoduli}{item5-euler-pairing} in \Cref{para-assumptions-linear}.

\end{lemma}

\begin{remark}
    A similar statement holds true for the moduli stack of Bridgeland semistable objects on a smooth projective variety $X$ under some assumptions on the stability condition.
    See \cite[Example 7.29]{_Alper_Existenceofmodulispacesforalgebraicstacks} for the assumptions on the stability conditions to guarantee a well-behaved moduli theory.
\end{remark}

\begin{para}[Smooth case]\label{para-cohint-linear-smooth}
    We will first discuss the cohomological integrality theorem for smooth linear moduli stacks satisfying assumptions \Crefrange{item1-linear-goodmoduli}{item3-linear-symmetric} and \Cref{item5-euler-pairing} in \Cref{para-assumptions-linear}.
    We note that the assumption \Cref{item4-linear-gloeq} is automatic by \Cref{cor-global-equiv-parameter-smooth}.
    Set $\Gamma \coloneqq \pi_0(\mathcal{M})$ and equip it with the monoid structure using the direct sum map $\oplus$.
    As we have seen in \Cref{para-special-linear}, an $n$-dimensional special face $(F, \alpha) \in  \mathsf{Face}^{\mathrm{sp}}(\mathcal{X})$ gives rise to a choice of non-zero components $\gamma_1, \ldots, \gamma_n \in \Gamma$ such that 
    \[
        \mathcal{M}_{\alpha} \cong \mathcal{M}_{\gamma_1} \times \cdots \times \mathcal{M}_{\gamma_n}
    \]
    holds. Conversely, components $\gamma_1, \ldots, \gamma_n \in \Gamma$
    define an $n$-dimensional face $\alpha(\gamma_1, \ldots, \gamma_n)$ with the above property, and it is special if and only if $\crk(\mathcal{M}_{\gamma_i}) = 1$.
    The automorphism of a face $\alpha(\gamma_1, \ldots, \gamma_n)$ corresponds to an automorphism of the multiset $(\gamma_1, \ldots, \gamma_n)$.
    We claim that there exists an $\mathrm{Aut}(\gamma_1, \ldots, \gamma_n)$-equivariant isomorphism
    \begin{align}\label{eq-BPS-decomp}
        \begin{aligned}
        &\mathcal{BPS}_{{M}}^{\alpha(\gamma_1, \ldots, \gamma_n)} \otimes \mathrm{H}^*(\mathrm{B} \mathbb{G}_\mathrm{m}^n)_{\mathrm{vir}} \otimes \mathrm{sgn}_{\gamma_1, \ldots, \gamma_n}  \\
        \cong {} &(\mathcal{BPS}_{{M}_{\gamma_1}}^{(1)} \otimes \mathrm{H}^*(\mathrm{B} \mathbb{G}_\mathrm{m})_{\mathrm{vir}}) \boxtimes \cdots \boxtimes (\mathcal{BPS}_{{M}_{\gamma_n}}^{(1)} \otimes \mathrm{H}^*(\mathrm{B} \mathbb{G}_\mathrm{m})_{\mathrm{vir}}) 
        \end{aligned}
    \end{align}
    where $\mathrm{sgn}_{\gamma_1, \ldots, \gamma_n}$ denotes the cotangent sign representation of $\mathrm{Aut}(\gamma_1, \ldots, \gamma_n)$.
    Existence of an isomorphism without $\mathrm{Aut}(\gamma_1, \ldots, \gamma_n)$-equivariance is obvious from the definition.
    We now explain that the natural isomorphism is  $\mathrm{Aut}(\gamma_1, \ldots, \gamma_n)$-equivariant.
    We only explain the case $n=2$ and when $\gamma_1 = \gamma_2$  holds:  the general case can be deduced in an analogous manner.
    Note that we have $\mathrm{Aut}(\gamma_1, \gamma_1) = \mathbb{Z} / 2 \mathbb{Z}$ and the sign representation $\mathrm{sgn}_{\gamma_1, \gamma_1}$ is given by multiplying by $(-1)^{\chi_{\mathcal{M}}(\gamma_1, \gamma_1)}$ for the non-trivial element.
    Therefore it is enough to show that the swapping isomorphism for $(\mathcal{BPS}_{{M}_{\gamma_1}}^{(1)} \otimes \mathrm{H}^*(\mathrm{B} \mathbb{G}_\mathrm{m})_{\mathrm{vir}}) \boxtimes (\mathcal{BPS}_{{M}_{\gamma_1}}^{(1)} \otimes \mathrm{H}^*(\mathrm{B} \mathbb{G}_\mathrm{m})_{\mathrm{vir}}) $ 
    differs from the natural involution on $\mathcal{BPS}_{{M}}^{\alpha(\gamma_1, \gamma_1)} \otimes \mathrm{H}^*(\mathrm{B} \mathbb{G}_\mathrm{m}^2)_{\mathrm{vir}} $ by the sign $(-1)^{\chi_{\mathcal{M}}(\gamma_1, \gamma_1)}$.
    This is a consequence of the equality $\dim \mathcal{M}_{\gamma} = - \chi_{\mathcal{M}}(\gamma, \gamma)$ together with the following two general statements about signs in homological algebra.
    Firstly, the following diagram commutes only up to sign $(-1)^{nm}$:
\[\begin{tikzcd}
	{\mathbb{L}^{n/2} \otimes \mathbb{L}^{m/2}} & {\mathbb{L}^{m/2} \otimes \mathbb{L}^{n/2}} \\
	{\mathbb{L}^{(n + m) /2}} & {\mathbb{L}^{(n + m) /2}.}
	\arrow["{\mathrm{swap}}", from=1-1, to=1-2]
	\arrow["\cong"', from=1-1, to=1-2]
	\arrow["\cong"', from=1-1, to=2-1]
	\arrow["\cong"', from=1-2, to=2-2]
	\arrow[Rightarrow, no head, from=2-1, to=2-2]
\end{tikzcd}\]
    This can be checked at the level of the underlying dg-vector spaces, where the statement is an immediate consequence of the Koszul sign rule.
    Secondly, for monodromic mixed Hodge complexes $M_i$ on a stack $\mathcal{X}_i$, the following diagram commutes only up to sign $-1$:
\[\begin{tikzcd}
	{{}^{\mathrm{p}}\mathcal{H}^1(M_1) \boxtimes {}^{\mathrm{p}}\mathcal{H}^1(M_2) } & {{}^{\mathrm{p}}\mathcal{H}^1(M_2) \boxtimes {}^{\mathrm{p}}\mathcal{H}^1(M_1) } \\
	{{}^{\mathrm{p}} \mathcal{H}^2(M_1 \boxtimes M_2)} & {{}^{\mathrm{p}} \mathcal{H}^2(M_2 \boxtimes M_1).}
	\arrow["{\mathrm{swap}}", from=1-1, to=1-2]
	\arrow[from=1-1, to=1-2]
	\arrow["\cong"', from=1-1, to=1-2]
	\arrow["\cong"', from=1-1, to=2-1]
	\arrow["\cong"', from=1-2, to=2-2]
	\arrow["\cong"', from=2-1, to=2-2]
	\arrow["{\mathrm{swap}}", from=2-1, to=2-2]
\end{tikzcd}\]
    This follows immediately from the Koszul sign rule.

We will use \Cref{eq-BPS-decomp} to rewrite the cohomological integrality isomorphism.
To do this, we will introduce several notations.
Firstly, define a symmetric monoidal structure $\boxtimes_{\oplus}$ on $\mathsf{D}_{\mathsf{H}}^{\mathsf{mon}, (+)}(M)$ as follows:
    For objects $E, F \in \mathsf{D}_{\mathsf{H}}^{\mathsf{mon}, (+)}(M)$, the monoidal product is defined by
    \[
     E \boxtimes_{\oplus} F \coloneqq \oplus_* (E \boxtimes F).    
    \]
The symmetrizer is the usual unmodified symmetrizer, for which the sign rule is the usual Koszul sign rule.
Define the symmetric product operation with respect to the symmetric monoidal structure by
\[
 \mathrm{Sym}_{\boxplus_{\oplus}} \colon \mathsf{D}_{\mathsf{H}}^{\mathsf{mon}, (+)}(M) \to \mathsf{D}_{\mathsf{H}}^{\mathsf{mon}, (+)}(M),\quad  E \mapsto \bigoplus_n (\underbrace{E \boxtimes_{\oplus} \cdots \boxtimes_{\oplus} E}_{n \text{ copies}})^{S_n}.  
\]
Then by combining the isomorphism \Cref{eq-BPS-decomp} together with the cohomological integrality theorem (= \Cref{thm-cohint-smooth}) and the description of the BPS sheaf (= \Cref{cor:small_decom_thm_smooth_stack}), we obtain the following statement, first proved by \textcite[Theorem 1.1]{meinhardt2015donaldson} with a slightly different formulation:

\end{para}

\begin{theorem}\label{thm-cohint-linear-smooth}
    Let $\mathcal{M}$ be a smooth linear moduli stack satisfying assumptions \Crefrange{item1-linear-goodmoduli}{item3-linear-symmetric} and \Cref{item5-euler-pairing} in \Cref{para-assumptions-linear}.
    Then we have an isomorphism
    \[
     \mathrm{Sym}_{\boxtimes_{\oplus}} \left( \bigotimes_{\gamma \in \mathcal{M} \setminus 0} (\mathcal{BPS}_{{M}_{\gamma}}^{(1)} \otimes \mathrm{H}^*(\mathrm{B} \mathbb{G}_\mathrm{m})_{\mathrm{vir}})   \right)  \cong p_* \mathcal{IC}_{\mathcal{M}}  
    \]
    induced by the cohomological Hall induction. Further, we have a natural isomorphism
    \[
        \mathcal{BPS}_{{M}_{\gamma}}^{(1)} \cong 
     \begin{cases}
        \mathcal{IC}_{M_{\gamma}} & \quad \text{if $\mathcal{M}_{\gamma} / \mathrm{B} \mathbb{G}_\mathrm{m} \to M_{\gamma}$ is generically quasi-finite. } \\
        0 & \quad \text{otherwise}.
     \end{cases}    
    \]
\end{theorem}

\begin{para}[Commutative orientation data]\label{para-COD}
    Here we will introduce the notion of \emph{commutative orientation data} for $(-1)$-shifted symplectic linear moduli stacks, which is needed for a formulation of the cohomological integrality theorem close to the one in \cite[Theorem A, Theorem C]{_Davison_Meinhardt_CoDT}.
    This is an orientation for $\mathcal{M}$ in the sense of \Cref{para-orientation-def}, together with an extra structure on compatibility with the direct sum map, itself satisfying a commutativity property.  
    Though an orientation with the same compatibility structure satisfying the associativity property has been studied in the literature e.g. in \cite[Definition 4.6]{joyce2021orientation} where the notion is called strong orientation data, it seems that such orientation data satisfying the commutativity property has not been studied so far.

\end{para}

\begin{definition*}
    Let $\mathcal{M}$ be a numerically symmetric $(-1)$-shifted symplectic linear moduli stack.
    \emph{Commutative orientation data} is given by the following data:
    \begin{itemize}
        \item A choice of orientation $o_{\gamma}$ on $\mathcal{M}_{\gamma}$ for each $\gamma \in \pi_0(\mathcal{M})$.
        \item For each $\gamma_1, \ldots, \gamma_n \in \pi_0(\mathcal{M}_{\gamma}) \setminus 0$ with $\gamma = \gamma_1 + \cdots + \gamma_n$, a choice of an $\mathrm{Aut}(\gamma_1, \ldots, \gamma_n)$-equivariant isomorphism
               \begin{equation}\label{eq-orientation-data-decomp}
                 \alpha(\gamma_1, \ldots, \gamma_n)^{\star} o_{\gamma} \cong o_{\gamma_1} \boxtimes \cdots \boxtimes  o_{\gamma_n}.  
               \end{equation}
               See \Cref{para-cohint-linear-smooth} for the definition of $\alpha(\gamma_1, \ldots, \gamma_n)$ and \Cref{para-localize-orientation-is-equivariant} for the notation $\alpha(\gamma_1, \ldots, \gamma_n)^{\star} o_{\gamma}$.
    \end{itemize}

\end{definition*}

In \cite[Theorem 3.6]{joyce2021orientation}, Joyce and Upmeier constructed an orientation for the moduli stack of coherent sheaves on a smooth projective Calabi--Yau threefold such that isomorphisms \Cref{eq-orientation-data-decomp} exist without $\mathrm{Aut}(\gamma_1, \ldots, \gamma_n)$-equivariance.
We are not sure whether their orientation data upgrades to commutative orientation data.
An example of commutative orientation data comes from the standard orientation for the $(-1)$-shifted cotangent stack:

\begin{lemma}\label{lem-standard-orientation-is-commutative}
    Let $\mathcal{N}$ be a numerically symmetric quasi-smooth linear moduli stack admitting an Euler pairing.
     Set $\mathcal{M} \coloneqq \mathrm{T}^*[-1]\mathcal{N}$.
    Then the standard orientation $o^{\mathrm{sta}}$ on $\mathcal{M}$ naturally upgrades to commutative orientation data.
\end{lemma}

\begin{proof}
    For each $\gamma \in \pi_0(\mathcal{M}) = \pi_0(\mathcal{N})$, we let $o^{\mathrm{sta}}_{\gamma}$ denote the standard orientation on $\mathcal{M}_{\gamma}$.
    We will only prove that there exists a natural $\mathbb{Z} / 2 \mathbb{Z}$-equivariant isomorphism
    \[
        \alpha(\gamma, \gamma)^{\star} o^{\mathrm{sta}}_{2\gamma} \cong o^{\mathrm{sta}}_{\gamma} \boxtimes o^{\mathrm{sta}}_{\gamma}.
    \]
    The existence of an isomorphism \Cref{eq-orientation-data-decomp} in general can be proved in an analogous manner.
    Firstly, note that there are (at most) two chambers $\sigma_1, \sigma_2 \subset \alpha(\gamma, \gamma)$ with respect to the cotangent arrangement.
    By the definition of the cotangent distance \Cref{para-cotangent-distance} and the Euler pairing, we have an identity
    \[
     d(\sigma_1, \sigma_2) = \chi_{\mathcal{N}}(\gamma, \gamma) \bmod 2.    
    \]
    Now consider the following diagram:
\[\begin{tikzcd}
	{\sigma_1^{\star}o^{\mathrm{sta}}_{\mathcal{M}_{2\gamma}}} & {o^{\mathrm{sta}}_{\mathcal{M}_{\gamma} \times \mathcal{M}_{\gamma}}} & {o^{\mathrm{sta}}_{\mathcal{M}_{\gamma}} \boxtimes o^{\mathrm{sta}}_{\mathcal{M}_{\gamma}}} \\
	{\sigma_2^{\star}o^{\mathrm{sta}}_{\mathcal{M}_{2\gamma}}} & {o^{\mathrm{sta}}_{\mathcal{M}_{\gamma} \times \mathcal{M}_{\gamma}}} & {o^{\mathrm{sta}}_{\mathcal{M}_{\gamma}} \boxtimes o^{\mathrm{sta}}_{\mathcal{M}_{\gamma}}.}
	\arrow["\text{\Cref{eq-localize-standard-orientaion}}", from=1-1, to=1-2]
	\arrow["\text{\Cref{eq-localized-ori-indep-on-segment}}"', from=1-1, to=2-1]
	\arrow[from=1-2, to=1-3]
	\arrow["{(-1)^{\chi_{\mathcal{N}}(\gamma, \gamma)} \cdot \id}", from=1-2, to=2-2]
	\arrow["{\mathrm{swap}}", from=1-3, to=2-3]
	\arrow["\text{\Cref{eq-localize-standard-orientaion}}"', from=2-1, to=2-2]
	\arrow[from=2-2, to=2-3]
\end{tikzcd}\]
    The left square commutes by \Cref{lem-ori-localize-indep-segment-up-to-sign} and the right square commutes since the grading of the underlying line bundle of $o^{\mathrm{sta}}_{\mathcal{M}_{\gamma}}$ is given by $(-1)^{\chi_{\mathcal{N}}(\gamma, \gamma)}$.
    The commutativity of the outer square is exactly what we wanted to prove.
\end{proof}

\begin{para}[Cohomological integrality for \texorpdfstring{$(-1)$}{(-1)}-shifted symplectic linear moduli stacks]
    Let $\mathcal{M}$ be a $(-1)$-shifted symplectic linear moduli stack satisfying the assumptions \Crefrange{item1-linear-goodmoduli}{item5-euler-pairing} in \Cref{para-assumptions-linear}.
    Assume further that  $\mathcal{M}$ is equipped with commutative orientation data $\{ o_{\gamma} \}_{\gamma \in \pi_0(\mathcal{M})}$.
    By the Thom--Sebastiani theorem \Cref{eq-Kunneth}, for each $\gamma_1, \ldots, \gamma_n \in \pi_0(\mathcal{M})$, there exists a natural $\mathrm{Aut}(\gamma_1, \ldots, \gamma_n)$-equivariant isomorphism
    \begin{equation}\label{eq-varphi-decomp-equivariantly}
         \varphi_{\mathcal{M}_{\alpha(\gamma_1, \ldots, \gamma_n)} } \cong \varphi_{\mathcal{M}_{\gamma_1}} \boxtimes  \cdots \boxtimes \varphi_{\mathcal{M}_{\gamma_n}} 
    \end{equation}
    under the identification $\mathcal{M}_{\alpha(\gamma_1, \ldots, \gamma_n)} \cong \mathcal{M}_{\gamma_1} \times \cdots \times \mathcal{M}_{\gamma_n}$.
    Combining \Cref{eq-varphi-decomp-equivariantly} with the argument of the proof of the isomorphism \Cref{eq-BPS-decomp}, we see that there exists an $\mathrm{Aut}(\gamma_1, \ldots, \gamma_n)$-equivariant isomorphism
    \[
        \mathcal{BPS}_{{M}}^{\alpha(\gamma_1, \ldots, \gamma_n)} \otimes \mathrm{H}^*(\mathrm{B} \mathbb{G}_\mathrm{m}^n)_{\mathrm{vir}}
        \cong  (\mathcal{BPS}_{{M}_{\gamma_1}}^{(1)} \otimes \mathrm{H}^*(\mathrm{B} \mathbb{G}_\mathrm{m})_{\mathrm{vir}}) \boxtimes \cdots \boxtimes (\mathcal{BPS}_{{M}_{\gamma_n}}^{(1)} \otimes \mathrm{H}^*(\mathrm{B} \mathbb{G}_\mathrm{m})_{\mathrm{vir}}).     
    \]
    
    We will equip the category $\mathsf{D}_{\mathsf{H}}^{\mathsf{mon}, (+)}(M)$ with a symmetric monoidal structure in the same manner as \Cref{para-cohint-linear-smooth}.
    Then the cohomological integrality theorem for $\mathcal{M}$ (= \Cref{thm-coh-int--1-shifted-symplectic}) implies the following theorem:
\end{para}

\begin{theorem}\label{thm-cohint-linear--1-shifted}
    Let $\mathcal{M}$ be a $(-1)$-shifted symplectic linear moduli stack satisfying assumptions \Crefrange{item1-linear-goodmoduli}{item5-euler-pairing} in \Cref{para-assumptions-linear}.
    Assume further that $\mathcal{M}$ is equipped with commutative orientation data.
    Then we have an isomorphism
    \[
     \mathrm{Sym}_{\boxtimes_{\oplus}} \left( \bigotimes_{\gamma \in \pi_0(\mathcal{M}) \setminus 0} (\mathcal{BPS}_{{M}_{\gamma}}^{(1)} \otimes \mathrm{H}^*(\mathrm{B} \mathbb{G}_\mathrm{m})_{\mathrm{vir}})   \right)  \cong p_* \varphi_{\mathcal{M}}  
    \]
    induced by the cohomological Hall induction. 
\end{theorem}

\begin{para}[BPS Lie algebra]
    In this paragraph, we will briefly explain that the above theorem can be regarded as a PBW-type statement for a certain Lie algebra, called the BPS Lie algebra after taking the global sections.
    
    \Cref{thm-cohint-linear--1-shifted} implies that the global sections $\mathrm{H}^*(\mathcal{M}_{\gamma}, \varphi_{\mathcal{M}_{\gamma}})$ are equipped with a perverse filtration induced from the perverse t-structure on $M_{\gamma}$;
    \[
     P_{i}  \mathrm{H}^*(\mathcal{M}_{\gamma}, \varphi_{\mathcal{M}_{\gamma}}) \subset \mathrm{H}^*(\mathcal{M}_{\gamma}, \varphi_{\mathcal{M}_{\gamma}}).
    \]
    Using \Cref{thm-cohint-linear--1-shifted} again, one sees that the perverse filtration satisfies the following property
    \begin{align*}
        &P_{0} \mathrm{H}^*(\mathcal{M}_{\gamma}, \varphi_{\mathcal{M}_{\gamma}}) =0, \\
        &P_{1} \mathrm{H}^*(\mathcal{M}_{\gamma}, \varphi_{\mathcal{M}_{\gamma}}) = \mathrm{H}^*_{\mathrm{BPS}^{(1)}}(M_{\gamma}) \coloneqq \mathrm{H}^*(M_{\gamma}, \mathcal{BPS}_{M_{\gamma}}^{(1)}).
    \end{align*}
    By definition, the cohomological Hall algebra multiplication
    \[
      *^{\mathrm{Hall}}_{\gamma_1, \gamma_2} \colon  \mathrm{H}^*(\mathcal{M}_{\gamma_1}, \varphi_{\mathcal{M}_{\gamma_1}}) \otimes \mathrm{H}^*(\mathcal{M}_{\gamma_2}, \varphi_{\mathcal{M}_{\gamma_2}}) \to \mathrm{H}^*(\mathcal{M}_{\gamma_1 + \gamma_2}, \varphi_{\mathcal{M}_{\gamma_1 + \gamma_2}}) 
    \]
    preserves the perverse filtration.
    Further, the supercommutativity of the perversely degenerate cohomological Hall algebra \eqref{eq-supercommutative-deg-relcohi--1-symplectic} implies that the following diagram commutes:
\[\begin{tikzcd}
	{\mathrm{gr}_{P}\mathrm{H}^*(\mathcal{M}_{\gamma_1}, \varphi_{\mathcal{M}_{\gamma_1}}) \otimes \mathrm{gr}_{P}\mathrm{H}^*(\mathcal{M}_{\gamma_2}, \varphi_{\mathcal{M}_{\gamma_2}}) } &[30pt] {\mathrm{gr}_{P}\mathrm{H}^*(\mathcal{M}_{\gamma_1 + \gamma_2}, \varphi_{\mathcal{M}_{\gamma_1 + \gamma_2}})} \\
	{\mathrm{gr}_{P}\mathrm{H}^*(\mathcal{M}_{\gamma_2}, \varphi_{\mathcal{M}_{\gamma_2}}) \otimes \mathrm{gr}_{P}\mathrm{H}^*(\mathcal{M}_{\gamma_1}, \varphi_{\mathcal{M}_{\gamma_1}})} & {\mathrm{gr}_{P}\mathrm{H}^*(\mathcal{M}_{\gamma_1 + \gamma_2}, \varphi_{\mathcal{M}_{\gamma_1 + \gamma_2}}).}
	\arrow["{\mathrm{gr}_P(*^{\mathrm{Hall}}_{\gamma_1, \gamma_2})}", from=1-1, to=1-2]
	\arrow["{\mathrm{swap}}"', from=1-1, to=2-1]
	\arrow[equals, from=1-2, to=2-2]
	\arrow["{\mathrm{gr}_P(*^{\mathrm{Hall}}_{\gamma_2, \gamma_1})}"', from=2-1, to=2-2]
\end{tikzcd}\]
    In particular, the commutator of the cohomological Hall multiplication preserves the first perverse piece, and we obtain a morphism
    \[
        [-, -] \colon \mathrm{H}^*_{\mathrm{BPS}^{(1)}}(M_{\gamma_1}) \otimes \mathrm{H}^*_{\mathrm{BPS}^{(1)}}(M_{\gamma_2}) \to \mathrm{H}^*_{\mathrm{BPS}^{(1)}}(M_{\gamma_1 + \gamma_2}).
    \]
    If the chosen orientation data is assumed to be associative, the cohomological Hall algebra multiplication is associative (see \cite[Corollary 8.8]{Kinjo_Park_Safronov_CoHA}).
    In particular, the above Lie bracket determines a structure of a Lie algebra on $\bigoplus_{\gamma} \mathrm{H}^*_{\mathrm{BPS}^{(1)}}(M_{\gamma})$, which we call the \emph{BPS Lie algebra}.
    \Cref{thm-cohint-linear--1-shifted} can be seen as a PBW-type theorem for the BPS Lie algebra, or, given the $\mathrm{H}^*(\mathrm{B} \mathbb{G}_\mathrm{m})_{\mathrm{vir}}$ factor appearing in the isomorphism, as a PBW-type theorem presenting $p_* \varphi_{\mathcal{M}}$ as half of a Yangian-type quantum group built from the BPS Lie algebra.
\end{para}

\begin{para}[Cohomological integrality for \texorpdfstring{$0$}{0}-shifted symplectic linear moduli stacks]
    Let $\mathcal{M}$ be a $0$-shifted symplectic linear moduli stack satisfying the assumptions \Crefrange{item1-linear-goodmoduli}{item5-euler-pairing} in \Cref{para-assumptions-linear}.
    We equip the category $\mathsf{D}_{\mathsf{H}}^{\mathsf{mon}, (+)}(M)$ with a symmetric monoidal structure in the same manner as \Cref{para-cohint-linear-smooth}.
    Then the cohomological integrality theorem for the $(-1)$-shifted cotangent stack $\mathrm{T}^*[-1]\mathcal{M}$ established in \Cref{thm-cohint-linear--1-shifted} together with \Cref{lem-standard-orientation-is-commutative} implies the following:
\end{para}

\begin{theorem}\label{thm-cohint-linear-0-shifted}
    Let $\mathcal{M}$ be a $0$-shifted symplectic linear moduli stack satisfying assumptions \Crefrange{item1-linear-goodmoduli}{item5-euler-pairing} in \Cref{para-assumptions-linear}.
    Then we have an isomorphism
    \begin{equation}\label{eq-g-bundle-grad}
     \mathrm{Sym}_{\boxtimes_{\oplus}} \left( \bigotimes_{\gamma \in \pi_0(\mathcal{M}) \setminus 0} (\mathcal{BPS}_{{M}, \gamma} \otimes \mathrm{H}^*(\mathrm{B} \mathbb{G}_\mathrm{m}))   \right)  \cong \mathbb{L}^{\vdim \mathcal{M} / 2} \otimes p_* \mathbb{D}\mathbb{Q}_{\mathcal{M}}  
    \end{equation}
    induced by the cohomological Hall induction. 
\end{theorem}

\begin{remark}
    In \cite[Theorem 1.7]{davison2023bps}, Davison, Hennecart and Schlegel Mejia proved the cohomological integrality theorem for the moduli stack of objects in $2$-Calabi--Yau dg-categories based on a different definition of the BPS sheaf described as a certain sheaf-theoretic version of a generalized Kac--Moody Lie algebra.
    \Cref{thm-cohint-linear-0-shifted} implies that these two definitions of BPS sheaves coincide for moduli stacks of objects in $2$-Calabi--Yau dg-categories.

    It would be an interesting problem to determine the BPS sheaves for general $0$-shifted symplectic stacks and to study its interaction with the (Lie algebra-like version of the) cohomological Hall induction.
\end{remark}

\subsection{Cohomological integrality for \texorpdfstring{$G$}{G}-(Higgs) bundles and \texorpdfstring{$G$}{G}-character stacks}\label{ssec-G-Higgs}

\begin{para}[Cohomological integrality for \texorpdfstring{$G$}{G}-bundles on a curve]\label{para-cohint-bundles}
    Let $C$ be a smooth projective curve and $G$ be a connected reductive group.
    We let $\mathcal{B}\mathrm{un}_G(C)^{\mathrm{ss}}$ be the moduli stack of semistable $G$-bundles on $C$.
    The stack  $\mathcal{B}\mathrm{un}_G(C)^{\mathrm{ss}}$ is smooth and almost orthogonal by \Cref{cor-bung-is-orthogonally-symmetric}.
    Also, being a quotient stack, it admits a global equivariant parameter by \Cref{lem-global-quotient-global-equiv-parameter}.
    Therefore the cohomological integrality theorem (= Theorem \ref{thm-cohint-smooth}) holds for $\mathcal{B}\mathrm{un}_G(C)^{\mathrm{ss}}$.

    We now explicitly describe the cohomological integrality theorem for $\mathcal{B}\mathrm{un}_G(C)^{\mathrm{ss}}$.
    Using the description of the stack of graded points for $\mathrm{B} G$ explained in \Cref{para-quotient-component-arrangement}, we obtain
    \[
       \mathrm{Grad}^n( \mathcal{B}\mathrm{un}_G(C)) \cong \coprod_{\lambda \colon \mathbb{G}_\mathrm{m}^n \to G}  \mathcal{B}\mathrm{un}_{L_{\lambda}}(C)
    \]
    where the disjoint union is taken over all maps up to conjugation and $L_{\lambda}$ denotes the corresponding Levi subgroup.
    Now we want to describe the open substack of the right-hand side corresponding to the semistable locus of the left-hand side.
    For this, we will introduce some notation. 
    For a Levi subgroup $L \subset G$,
    we say that an element $d \in \pi_1(L) \cong \pi_0(\mathcal{B}\mathrm{un}_{L}(C))$ is \emph{$G$-admissible} if for any 
    character $\chi \colon L \to \mathbb{G}_\mathrm{m}$ with $\chi |_{\mathrm{Z}(G)^{\circ}} = 1$, the map $\pi_1(L) \xrightarrow[]{\chi} \pi_1(\mathbb{G}_\mathrm{m}) \cong \mathbb{Z}$ sends $d$ to $0$.
    The following statement is proved by \textcite[Lemma 2.14, Lemma 2.15]{fratila2016stack}:

\end{para}

\begin{lemma}\label{lem-induction-semistable}
   Let $E_{L}$ be a principal $L$-bundle and $E_G$ be the induced $G$-bundle.
   Then $E_G$ is semistable if and only if $E_{L}$ is semistable and the degree of $E_{L}$ is $G$-admissible.
\end{lemma}

\begin{remark}
    The proof is based on a ``deeper reduction'' argument of \textcite[\S 4.2]{schieder2015harder} which can be thought of as a generalization of a common refinement of two given filtrations of a vector bundle.
    One easily sees that a similar statement holds true for (twisted) $G$-Higgs bundles by the same argument.
\end{remark}

\begin{para}
    For each Levi subgroup $L \subset G$, we let $A_G(L) \subset \pi_1(L)$ denote the set of $G$-admissible degrees.
    Then \Cref{lem-induction-semistable} implies that the restriction of the isomorphism \Cref{eq-g-bundle-grad} is given by
    \[
        \mathrm{Grad}^n( \mathcal{B}\mathrm{un}_G(C)^{\mathrm{ss}} ) \cong \coprod_{\lambda \colon \mathbb{G}_\mathrm{m}^n \to G} \ \ \coprod_{d \in A_G(L_{\lambda})}  \mathcal{B}\mathrm{un}_{L_{\lambda}}(C)_d^{\mathrm{ss}}.
    \]
    We have the following statement, which shows that any degree for $G$ admits at most one lift to a $G$-admissible degree for $L$:

\end{para}

\begin{lemma}\label{lem-unique-d-admissible}
    The natural map $\pi_1(L) \to \pi_1(G)$ induces an injection $A_G(L) \hookrightarrow \pi_1(G)$.
\end{lemma}

\begin{proof}
    The long exact sequence of the homotopy groups associated with the fibration $L \to G \to G / L$ implies that the following sequence is exact:
    \[
    \{ 1 \} \to \pi_2(G / L) \to \pi_1(L) \to \pi_1(G) \to \{1\} .
    \]
    Here we used the fact that Lie groups have trivial second homotopy groups and the partial flag variety $G / P (\simeq G / L)$ is simply connected.
    Therefore it is enough to show that an element in $\pi_2(G / L)$ is zero if and only if it is killed by some character
    $\chi \colon L \to \mathbb{G}_\mathrm{m}$ with $\chi |_{\mathrm{Z}(G)^{\circ}} = 1$.
    In other words, it is enough to show that the natural map
    \[
        \pi_2(G / L) \to \mathrm{Hom}(L / \mathrm{Z}(G)^{\circ}, \mathbb{G}_{\mathrm{m}} )^{\vee}
    \]
    is an injection. Since $\pi_2(G / L) \cong \pi_2(G / P) \cong \mathrm{H}_2(G / P)$ is torsion free of finite rank, it is enough to show this after rationalization.

    Using the natural isomorphism $\pi_1(\mathrm{Z}(G)^{\circ})_{\mathbb{Q}} \cong \pi_1(G)_{\mathbb{Q}}$,
    we obtain an isomorphism $\pi_2(G / L)_{\mathbb{Q}} \cong \pi_1(L / \mathrm{Z}(G)^{\circ})_{\mathbb{Q}}$.
    Since the group $\pi_1(L / \mathrm{Z}(G)^{\circ})_{\mathbb{Q}}$ is naturally identified with the group $\mathrm{Hom}(L / \mathrm{Z}(G)^{\circ}, \mathbb{G}_{\mathrm{m}} )^{\vee}_{\mathbb{Q}}$,
    we obtain the desired claim.

\end{proof}

\begin{para}

    Let $p \colon \mathcal{B}\mathrm{un}_G(C)^{\mathrm{ss}} \to \mathrm{Bun}_G(C)^{\mathrm{ss}}$ be the good moduli space morphism.
    For each Levi subgroup $L \subset G$ and $d \in A_G(L)$, we let $W_{G}(L) = \mathrm{N}_G(L) / L$ be the corresponding relative Weyl group and 
    \[
        g_{L, d} \colon  \mathrm{Bun}_{L}(C)_{d}^{\mathrm{ss}} \to \mathrm{Bun}_{G}(C)^{\mathrm{ss}}
    \]
    be the natural map.
    \Cref{lem-unique-d-admissible} implies that the $W_{G}(L)$-action preserves the component $\mathrm{Bun}_{L}(C)_{d}^{\mathrm{ss}} \subset \mathrm{Bun}_{L}(C)_{}^{\mathrm{ss}}$.
    Then the cohomological integrality theorem (= \Cref{thm-cohint-smooth} + \Cref{rmk-cohint-source-smooth})  specializes to the following statement, which generalizes the result for $\mathrm{GL}_n$ by \textcite{mozgovoy2015intersection} and \textcite{meinhardt2015donaldson}:

\end{para}

\begin{theorem}\label{thm-cohint-bundles}
    We adopt the notation from the last paragraph.
    Then there exists an isomorphism
    \[
        \bigoplus_{\substack{L \subset G: \\ \textnormal{Levi subgroups}}} \bigoplus_{d \in A_G(L)} (g_{L, d, *} (\mathcal{BPS}^{(c_L)}_{\mathrm{Bun}_{L}(C)_d^{\mathrm{ss}}}\otimes \mathrm{H}^*(\mathrm{B} \mathbb{G}_\mathrm{m}^{c_L})_{\mathrm{vir}} \otimes \mathrm{sgn}_{L}))^{W_{G}(L)} \cong p_* \mathcal{IC}_{\mathcal{B}\mathrm{un}_G(C)^{\mathrm{ss}}}
    \]
    induced by the cohomological Hall induction, where the left-hand side runs over all Levi subgroups up to conjugation.
    Here we set $c_L \coloneqq \dim \mathrm{Z}(L)$ and $\mathrm{sgn}_{L}$ denotes the cotangent sign representation of $W_{G}(L)$: see \Cref{para-cotangent-distance}.
    Further, we have an isomorphism
    \begin{align*}
        & \mathcal{BPS}^{(c_L)}_{\mathrm{Bun}_{L}(C)_d^{\mathrm{ss}}}\\
        & \cong \begin{cases}
            \mathcal{IC}_{\mathrm{Bun}_{L}(C)_d^{\mathrm{ss}}} & \quad \text{if $\mathcal{B}\mathrm{un}_{L}(C)_d^{\mathrm{ss}}/ \mathrm{B} \mathbb{G}_\mathrm{m}^{c_L} \to \mathrm{Bun}_{L}(C)_d^{\mathrm{ss}}$ is generically quasi-finite. } \\
            0 & \quad \text{otherwise}.
         \end{cases}     
     \end{align*}

\end{theorem}

\begin{remark}\label{rmk-sign-rep-bung}
    The sign representation can be explicitly computed as follows.
    Let $\bar{\mathrm{sgn}}_{L}$ be the cotangent sign representation of $W_{G}(L)$ associated with the stack $\mathrm{B} G$.
    Then the Riemann--Roch theorem implies an isomorphism 
    \[
        \mathrm{sgn}_{L} \cong \bar{\mathrm{sgn}}_{L}^{1 - g(C)}.
    \]
\end{remark}

\begin{remark}[Disconnected groups]
By the results of Fernandez Herrero and A.I.N. \cite{fernandez-herrero-ibanez-nunez-disconnected}, \Cref{thm-cohint-bundles} also holds for disconnected reductive groups $G$, with modifications that we now explain. If $G$ is disconnected, then we should replace $A_G(L)$ by the set $A_G(L)_\bQ$ of \emph{rational degrees} for $L$, in the sense of \cite[\S 4.2.4]{fernandez-herrero-ibanez-nunez-disconnected}, that are adapted for the inclusion $L\to G$ in the sense of \cite[Definition 4.2.9]{fernandez-herrero-ibanez-nunez-disconnected}. For $L\subset G$ a Levi subgroup and $d\in A_G(L)_\bQ$, now $\Bun_L(C)_d^{\mathrm{ss}}$ denotes the stack of semistable $L$-bundles of rational degree $d$, as in \cite[\S 4.2.5]{fernandez-herrero-ibanez-nunez-disconnected}. Then \cite[Theorem 4.4.1]{fernandez-herrero-ibanez-nunez-disconnected} and \cite[Proposition 4.4.8]{fernandez-herrero-ibanez-nunez-disconnected} imply that the special faces of $\Bun_G(C)^{\mathrm{ss}}$ correspond to connected components of $\Bun_L(C)_d^{\mathrm{ss}}$ for Levi subgroups $L\subset G$ considered up to conjugation and adapted rational degrees $d\in A_G(L)_\bQ$. While technically $\Bun_L(C)_d^{\mathrm{ss}}$ may be disconnected, this does not affect \Cref{thm-cohint-bundles}.

A similar extension to disconnected groups holds in the case of Higgs bundles below.
\end{remark}

\begin{remark}
 If $C = \mathbb{P}^1$, the BPS sheaf on $\mathrm{Bun}_G(C)_d^{\mathrm{ss}}$ is zero unless $G$ is a torus.

 If $C$ is an elliptic curve, $\mathcal{BPS}_{\mathrm{Bun}_G(C)_d^{\mathrm{ss}}}$ is non-zero if and only if $G / [G, G] \cong \prod_i \mathrm{PGL}_{r_i}$ and the components of the degree are coprime to the corresponding ranks. This follows from the classification of stable $G$-bundles on an elliptic curve established in \cite[Corollary 4.3]{fratila2016stack}.

 If $g(C) \geq 2$, we have $\mathcal{BPS}_{\mathrm{Bun}_G(C)_d^{\mathrm{ss}}} \cong \mathcal{IC}_{\mathrm{Bun}_G(C)_d^{\mathrm{ss}}}$. This follows from the existence of a stable $G$-bundle for all topological types: see \cite[Remark 5.3]{ramanathan-stable}.
\end{remark}

\begin{para}[Cohomological integrality for twisted \texorpdfstring{$G$}{G}-Higgs bundles on a curve]\label{para-cohint-twisted-Higgs-bundles}
    We will adopt the notation from \Cref{para-cohint-bundles} except that we will use $p$ and $g_{L, d}$ for different morphisms.
    Fix a line bundle $\mathcal{L}$ on $C$ with $\deg \mathcal{L} > 2g(C) - 2$.
    Let $\mathcal{H}\mathrm{iggs}_G^{\mathcal{L}}(C)^{\mathrm{ss}}$ be the moduli stack of semistable $\mathcal{L}$-twisted $G$-Higgs bundles.
    As we have seen in \Cref{cor-higgsG-is-orthogonally-symmetric}, the stack $\mathcal{H}\mathrm{iggs}_G^{\mathcal{L}}(C)^{\mathrm{ss}}$ is smooth and almost orthogonal.
    Also, as a quotient stack, it admits a global equivariant parameter  by \Cref{lem-global-quotient-global-equiv-parameter}.
    Therefore the cohomological integrality theorem (= Theorem \ref{thm-cohint-smooth}) holds for $\mathcal{H}\mathrm{iggs}_G^{\mathcal{L}}(C)^{\mathrm{ss}}$.

    We let $p \colon \mathcal{H}\mathrm{iggs}^{\mathcal{L}}_G(C)^{\mathrm{ss}} \to \mathrm{Higgs}^{\mathcal{L}}_G(C)^{\mathrm{ss}}$ be the good moduli space morphism.
    For each $L \subset G$ and $d \in A_G(L)$, let 
    \[
    g_{L, d} \colon  \mathrm{Higgs}_{L}^{\mathcal{L}}(C)_d^{\mathrm{ss}} \to \mathrm{Higgs}_{G}^{\mathcal{L}}(C)^{\mathrm{ss}}
    \]
    be the natural map.
    The cohomological integrality theorem (= \Cref{thm-cohint-smooth} + \Cref{rmk-cohint-source-smooth}) specializes to the following statement:
\end{para}

\begin{theorem}
    We adopt the notation from the last paragraph.
    Then there exists an isomorphism
    \[
        \bigoplus_{\substack{L \subset G: \\ \textnormal{Levi subgroups}}} \bigoplus_{d \in A_G(L)}  (g_{L, d, *} (\mathcal{BPS}^{(c_L)}_{\mathrm{Higgs}^{\mathcal{L}}_{L}(C)_d^{\mathrm{ss}}} \otimes \mathrm{H}^*(\mathrm{B} \mathbb{G}_\mathrm{m}^{c_L})_{\mathrm{vir}} \otimes \mathrm{sgn}_{L}))^{W_{G}(L)} \cong p_* \mathcal{IC}_{\mathcal{H}\mathrm{iggs}^{\mathcal{L}}_G(C)^{\mathrm{ss}}}
    \]
    induced by the cohomological Hall induction.
    Further, we have an isomorphism
    \begin{align*}
            & \mathcal{BPS}^{(c_L)}_{\mathrm{Higgs}^{\mathcal{L}}_{L}(C)_d^{\mathrm{ss}}}  \\
            & \cong 
            \begin{cases}
                \mathcal{IC}_{\mathrm{Higgs}_{L}^{\mathcal{L}}(C)_d^{\mathrm{ss}}} & \quad \text{if $\mathcal{H}\mathrm{iggs}_{L}^{\mathcal{L}}(C)_d^{\mathrm{ss}} / \mathrm{B}\mathbb{G}_\mathrm{m}^{c_L} \to \mathrm{Higgs}_{L}^{\mathcal{L}}(C)_d^{\mathrm{ss}}$ is generically quasi-finite. } \\
                0 & \quad \text{otherwise}.
            \end{cases}     
    \end{align*}
    \begin{remark}
        We can compute the sign representation as in \Cref{rmk-sign-rep-bung}: We have an isomorphism
        \[
            \mathrm{sgn}_{L} \cong \bar{\mathrm{sgn}}_{L}^{\deg \mathcal{L} + 1 - g(C)}.
        \]

    \end{remark}
\end{theorem}

\begin{para}[Cohomological integrality for \texorpdfstring{$G$}{G}-Higgs bundles on a curve]\label{para-cohint-Higgs-bundles}
    We will adopt the notation from \Cref{para-cohint-bundles} except that we will use the letter $p$ and $g_{L, d}$ for different morphisms.
    Let $\mathcal{H}\mathrm{iggs}_G^{}(C)^{\mathrm{ss}}$ be the moduli stack of (untwisted) $G$-Higgs bundles.
    Then as we have seen in \Cref{item-example-symplectic-G-Higgs} of \Cref{para-example-shifted-symplectic} and \Cref{cor-higgsG-is-orthogonally-symmetric}, the stack $\mathcal{H}\mathrm{iggs}_G^{}(C)^{\mathrm{ss}}$ is $0$-shifted symplectic and almost orthogonal.
    Also, as a quotient stack, it admits a global equivariant parameter  by \Cref{lem-global-quotient-global-equiv-parameter}.
    Therefore the cohomological integrality theorem (= Theorem \ref{thm-cohint-0-shifted}) holds for $\mathcal{H}\mathrm{iggs}_G(C)^{\mathrm{ss}}$.

    We let $p \colon \mathcal{H}\mathrm{iggs}_G(C)^{\mathrm{ss}} \to \mathrm{Higgs}_G(C)^{\mathrm{ss}}$ be the good moduli space morphism.
    For each Levi subgroup $L \subset G$ and $d \in A_G(L)$, let 
    \[
        g_{L, d} \colon  \mathrm{Higgs}_{L}(C)^{\mathrm{ss}}_d \to  \mathrm{Higgs}_{G}(C)^{\mathrm{ss}}
    \]
    be the natural map.
    Then the cohomological integrality theorem (= \Cref{thm-cohint-0-shifted} + \Cref{rmk-cohint-source-0-shifted-symplectic}) is equivalent to the following:
\end{para}

\begin{theorem}
    We adopt the notation from the last paragraph.
    Then there exists an isomorphism
    \[
        \bigoplus_{\substack{L \subset G: \\ \textnormal{Levi subgroups}}} \bigoplus_{d \in A_G(L)} (g_{L, d, *} (\mathcal{BPS}^{(c_L)}_{\mathrm{Higgs}_{L}(C)_d^{\mathrm{ss}}} \otimes \mathrm{H}^*(\mathrm{B} \mathbb{G}_\mathrm{m}^{c_L}) \otimes \mathrm{sgn}_{L}))^{W_{G}(L)} \cong \mathbb{L}^{   (g - 1) \cdot \dim G } \otimes p_* \mathbb{D}\mathbb{Q}_{\mathcal{H}\mathrm{iggs}_G(C)^{\mathrm{ss}}}
    \]
    induced by the cohomological Hall induction.
\end{theorem}

\begin{remark}
    It is an interesting challenge to determine the BPS sheaf on $\mathrm{Higgs}_{G}(C)_d^{\mathrm{ss}}$.
    For $G = \mathrm{GL}_n$, this was done in \cite[Theorem 1.7]{davison2023bps} by Davison, Hennecart and Schlegel Mejia.
    For other reductive groups, Hennecart and T.K. will compute the BPS sheaf for the moduli space of $G$-Higgs bundles on an elliptic curve in \cite{Kinjo_Springer}.
\end{remark}

\begin{para}[Topological mirror symmetry conjecture]\label{para-topological-mirror-symmetry}
    We adopt the notation from \Cref{para-cohint-Higgs-bundles}. We will discuss a formulation of the topological mirror symmetry conjecture for the moduli space of $G$-Higgs bundles.
    Our discussion is motivated by the conjecture of \textcite[Conjecture 5.1]{hausel2003mirror} on the $\mathrm{SL}_n / \mathrm{PGL}_n$-duality, which was proved in \cite{groechenig2020mirror,maulik2021endoscopic}.
    
    Let $G$ be a reductive group, and denote by $G^{\vee}$ the Langlands dual group of $G$.
    We let 
    \[ 
        \mathcal{L} \mathcal{H}\mathrm{iggs}_G(C)^{\mathrm{ss}} \coloneqq \mathrm{Map}(S^1_{\mathrm{B}},  \mathcal{H}\mathrm{iggs}_G(C)^{\mathrm{ss}})
    \]
    denote the loop stack.
    Since $\mathcal{H}\mathrm{iggs}_G(C)^{\mathrm{ss}}$ is $0$-shifted symplectic, using the AKSZ formalism \cite[Theorem 2.5]{pantev2013shifted}, we see that $\mathcal{L}\mathcal{H}\mathrm{iggs}_G(C)^{\mathrm{ss}}$ is $(-1)$-shifted symplectic.
    Also, one easily sees that $\mathcal{L}\mathcal{H}\mathrm{iggs}_G(C)^{\mathrm{ss}}$ is equipped with a natural orientation, which can be regarded as a multiplicative version of the standard orientation on the $(-1)$-shifted cotangent stack: see \cite[Proposition 5.3]{naef2023torsion}.
    Consider the good moduli space 
    \[
        \mathcal{L} \mathcal{H}\mathrm{iggs}_G(C)^{\mathrm{ss}} \to \mathrm{LHiggs}_G(C)^{\mathrm{ss}}
    \]
    We propose the following formulation of the topological mirror symmetry conjecture:
\end{para}

\begin{conjecture}[Topological mirror symmetry]\label{conj-top-mirror}
    Set $c_G \coloneqq \dim \mathrm{Z}(G)$. Then there exists a natural isomorphism
    \begin{equation}\label{eq-top-mirror}
        \mathrm{H}^*_{\mathrm{BPS}^{(c_G)}} \left(\mathrm{LHiggs}_G(C)^{\mathrm{ss}} \right) \cong \mathrm{H}^*_{\mathrm{BPS}^{(c_{G^{\vee}})}} \left(\mathrm{LHiggs}_{G^{\vee}}(C)^{\mathrm{ss}} \right).
    \end{equation}
   preserving the Hodge structure.
\end{conjecture}

\begin{remark}
    We also expect the following isomorphism between the vanishing cycle cohomology
    \[
        \mathrm{H}^*(\mathcal{L} \mathcal{H}\mathrm{iggs}_G(C)^{\mathrm{ss}}, \varphi_{\mathcal{L} \mathcal{H}\mathrm{iggs}_G(C)^{\mathrm{ss}}}) \cong \mathrm{H}^*(\mathcal{L} \mathcal{H}\mathrm{iggs}_{G^{\vee}}(C)^{\mathrm{ss}}, \varphi_{\mathcal{L} \mathcal{H}\mathrm{iggs}_{G^{\vee}}(C)^{\mathrm{ss}}}).
    \]
    Using the cohomological integrality theorem (= \Cref{thm-coh-int--1-shifted-symplectic}),
    this isomorphism follows from the conjectural isomorphism \Cref{eq-top-mirror} preserving the equivariance with respect to $\mathrm{Out}(G) \cong \mathrm{Out}(G^{\vee})$,
    where $\mathrm{Out}_{\mathrm{symp}}(G)$ denotes the group of outer automorphisms of $G$ preserving a fixed $G$-invariant metric on $\mathfrak{g}$.
    We believe that \Cref{eq-top-mirror} is easier to verify since the BPS cohomology is finite-dimensional.
\end{remark}

\begin{para}[Refinement of the topological mirror symmetry]

    We now explain a refinement of the above conjecture when $G$ is a semisimple group.
    For each $d \in \pi_1(G) \cong \pi_0(\mathcal{H}\mathrm{iggs}_G(C)^{\mathrm{ss}})$, 
    we let $\mathcal{H}\mathrm{iggs}_G(C)_d^{\mathrm{ss}}$ denote the corresponding connected component.
    Then the action of $\mathrm{B}\mathrm{Z}(G)$ on  $\mathcal{H}\mathrm{iggs}_G(C)_d^{\mathrm{ss}}$  induces an action of $\mathcal{L} \mathrm{B}\mathrm{Z}(G)$ on $\mathcal{L}\mathcal{H}\mathrm{iggs}_G(C)_d^{\mathrm{ss}}$.
    In particular,  the group $\mathrm{Z}(G)$ acts on the good moduli space $\mathrm{LHiggs}_G(C)_d^{\mathrm{ss}}$.
    One sees that the $\mathrm{Z}(G)$-action on $\mathrm{LHiggs}_G(C)_d^{\mathrm{ss}}$ induces a $\mathrm{Z}(G)$-action on $ \mathrm{H}^*_{\mathrm{BPS}^{(0)}}\left(\mathrm{LHiggs}_G(C)_d^{\mathrm{ss}} \right) $.
    For each $d' \in \mathrm{Z}(G)^{\vee}$, we let 
    \[
        \mathrm{H}^*_{\mathrm{BPS}^{(0)}}\left(\mathrm{LHiggs}_G(C)_d^{\mathrm{ss}} \right)_{d'}  \subset  \mathrm{H}^*_{\mathrm{BPS}^{(0)}}\left(\mathrm{LHiggs}_G(C)_d^{\mathrm{ss}} \right) 
    \]
    denote the subspace where $\mathrm{Z}(G)$ acts by $d'$. Note that we have a natural isomorphism $\mathrm{Z}(G)^{\vee} \cong \pi_1(G^{\vee})$, where the latter group corresponds to the set of connected components of ${\mathrm{Higgs}}_{G^{\vee}}(C)$.
    Now we can state the refined version of the topological mirror symmetry conjecture:
\end{para}

\begin{conjecture}\label{conj-top-mirror-refined}
    Let $G$ be a semisimple group.
    For $d \in \pi_1(G) \cong \mathrm{Z}(G^{\vee})^{\vee}$ and $d' \in \mathrm{Z}(G)^{\vee} \cong \pi_1(G^{\vee})$, the conjectural isomorphism \eqref{eq-top-mirror} swaps the subspaces
    \[
        \mathrm{H}^*_{\mathrm{BPS}^{(0)}}\left({\mathrm{LHiggs}}_G(C)^{\mathrm{ss}}_{d}\right)_{d'} \cong \mathrm{H}^*_{\mathrm{BPS}^{(0)}}\left({\mathrm{LHiggs}}_{G^{\vee}}(C)^{\mathrm{ss}}_{d'}\right)_{d}.
    \]
\end{conjecture}

\begin{remark}
    The use of the loop stack is motivated by the fact that the topological mirror symmetry conjecture for the $\mathrm{SL}_n / \mathrm{PGL}_n$-dual pair due to \textcite[Conjecture 5.1]{hausel2003mirror} proved in \cite{groechenig2020mirror,maulik2021endoscopic} is formulated using the stringy Hodge numbers.
    They are a variant of Hodge numbers for Deligne--Mumford stacks and defined to be the Hodge number of the inertia stack with some degree shift called fermionic shift.
    We however note that \Cref{conj-top-mirror-refined} does not directly imply \cite[Conjecture 5.1]{hausel2003mirror} since they define the Dolbeault moduli space on the $\mathrm{SL}_n$-side to be inside the moduli space of $\mathrm{GL}_n$-Higgs bundles whose degree is coprime to $n$, in order to avoid the strictly semistable locus.
    We explain a twisted version of \Cref{conj-top-mirror-refined} in \Cref{para-twisted-TMS-conjecture} which generalizes the original conjecture of \textcite[Conjecture 5.1]{hausel2003mirror}.
\end{remark}

\begin{remark}
    It might be natural to ask whether \Cref{conj-top-mirror-refined} holds for the BPS cohomology of the moduli space of semistable $G$-Higgs bundles without passing to the loop stack.
    However, the following numerical computation shows that this is not the case: For an elliptic curve $E$, we have
    \[
     \dim \mathrm{H}^*_{\mathrm{BPS}^{(0)}}({\mathrm{Higgs}}_{\mathrm{SL}_2}(E)^{\mathrm{ss}}) = 4, \quad    
    \]
    and
    \[
        \dim \mathrm{H}^*_{\mathrm{BPS}^{(0)}}({\mathrm{Higgs}}_{\mathrm{PGL}_2}(E)^{\mathrm{ss}}_0) = 1, \quad \dim \mathrm{H}^*_{\mathrm{BPS}^{(0)}}({\mathrm{Higgs}}_{\mathrm{PGL}_2}(E)^{\mathrm{ss}}_1) = 1.      
    \]
    On the other hand, we have
    \[
        \dim \mathrm{H}^*_{\mathrm{BPS}^{(0)}}({\mathrm{LHiggs}}_{\mathrm{SL}_2}(E)^{\mathrm{ss}}) = 8,    
    \]
    and
    \[
           \dim \mathrm{H}^*_{\mathrm{BPS}^{(0)}}({\mathrm{LHiggs}}_{\mathrm{PGL}_2}(E)^{\mathrm{ss}}_0) = 4, \quad \dim \mathrm{H}^*_{\mathrm{BPS}^{(0)}}({\mathrm{LHiggs}}_{\mathrm{PGL}_2}(E)^{\mathrm{ss}}_1) = 4     
    \]
    which provides numerical evidence for \Cref{conj-top-mirror-refined}.
    We will give more heuristic evidence for \Cref{conj-top-mirror-refined} later in \Cref{para-top-mirror-and-langlands}, based on the geometric Langlands duality for $3$-manifolds.

\end{remark}

\begin{para}[Topological mirror symmetry conjecture for the twisted moduli stack]\label{para-twisted-TMS-conjecture}
Here, we explain a twisted generalization of \Cref{conj-top-mirror-refined} which recovers the topological mirror symmetry conjecture of \textcite[Conjecture 5.1]{hausel2003mirror} as a special case.

Let $C$ be a smooth projective curve and $G$ be a semisimple group.
Set $G_{\mathrm{ad}} \coloneqq G / \mathrm{Z}(G)$. Then we have a natural map
\[
 \mathrm{B}  G_{\mathrm{ad}}  \to \mathrm{B}^2 \mathrm{Z}(G).  
\]
By considering the mapping stack from $C$, we obtain a natural map
\[
 \mathcal{B}\mathrm{un}_{G_{\mathrm{ad}} }(C) \to    \mathrm{Map}(C, \mathrm{B}^2 \mathrm{Z}(G)).
\]
In particular, there exists a natural map
\begin{equation}\label{eq-Higgs-boundary}
\mathcal{H}\mathrm{iggs}_{G_{\mathrm{ad}}}(C)^{\mathrm{ss}} \to \mathrm{Map}(C, \mathrm{B}^2 \mathrm{Z}(G)).    
\end{equation}
Note that we have $\pi_0(\mathrm{Map}(C, \mathrm{B}^2 \mathrm{Z}(G))) = \mathrm{H}^2(C, \mathrm{Z}(G))$ and each connected component is isomorphic to every other since it is a group stack.
The neutral component of $\mathrm{Map}(C, \mathrm{B}^2 \mathrm{Z}(G))$ can be computed using the Čech resolution and it is isomorphic to  $\mathrm{B} \mathrm{Map}(C, \mathrm{B} \mathrm{Z}(G))$. 
In particular, the map \Cref{eq-Higgs-boundary} induces the following map
\[
    \mathcal{H}\mathrm{iggs}_{G_{\mathrm{ad}}}(C)^{\mathrm{ss}} \to \mathrm{B} \mathrm{Map}(C, \mathrm{B}\mathrm{Z}(G)).
\]
For each $d \in \pi_1(G_{\mathrm{ad}})$, we define a stack $\mathcal{H}\mathrm{iggs}_{G}(C)^{\mathrm{ss}, \mathrm{twi}}_d$ by the following Cartesian square
\begin{equation}\label{eq-twisted-moduli-stack}
    \begin{tikzcd}
	{\mathcal{H}\mathrm{iggs}_{G}(C)^{\mathrm{ss}, \mathrm{twi}}_d} & {\mathrm{pt}} \\
	{\mathcal{H}\mathrm{iggs}_{G_{\mathrm{ad}}}(C)^{\mathrm{ss}}_d} & {\mathrm{B} \mathrm{Map}(C, \mathrm{B}\mathrm{Z}(G)).}
	\arrow[from=1-1, to=1-2]
	\arrow[from=1-1, to=2-1]
	\arrow["\lrcorner"{anchor=center, pos=0.125}, draw=none, from=1-1, to=2-2]
	\arrow[from=1-2, to=2-2]
	\arrow[from=2-1, to=2-2]
\end{tikzcd}
\end{equation}
The stack $\mathcal{H}\mathrm{iggs}_{G}(C)^{\mathrm{ss}, \mathrm{twi}}_d$ is $0$-shifted symplectic since the vertical maps in the above diagram are \'etale by semisimplicity of $G$.
For $d \in \pi_1(G) \subset \pi_1(G_{\mathrm{ad}})$, we have $\mathcal{H}\mathrm{iggs}_{G}(C)^{\mathrm{ss}, \mathrm{twi}}_d = \mathcal{H}\mathrm{iggs}_{G}(C)^{\mathrm{ss}}_d$. 

We now define a twisted version of the BPS cohomology, generalizing the twisted version of the stringy E-polynomial   \cite[(4.1)]{hausel2003mirror}.
First, by taking the loop stacks in the diagram \Cref{eq-twisted-moduli-stack}, we obtain the following Cartesian square
\begin{equation}\label{eq-twisted-loop-moduli-stack}
    \begin{tikzcd}
	{\mathcal{L}\mathcal{H}\mathrm{iggs}_{G}(C)^{\mathrm{ss}, \mathrm{twi}}_d} & {\mathrm{pt}} \\
	{\mathcal{L}\mathcal{H}\mathrm{iggs}_{G_{\mathrm{ad}}}(C)^{\mathrm{ss}}_d} & {\mathrm{B} \mathrm{Map}(C, \mathrm{B}\mathrm{Z}(G)) \times  \mathrm{B}\mathrm{Z}(G).}
	\arrow[from=1-1, to=1-2]
	\arrow[from=1-1, to=2-1]
	\arrow["\lrcorner"{anchor=center, pos=0.125}, draw=none, from=1-1, to=2-2]
	\arrow[from=1-2, to=2-2]
	\arrow[from=2-1, to=2-2]
\end{tikzcd}
\end{equation}
We set 
\[
    \mathcal{L}_{\mathrm{ad}} \mathcal{H}\mathrm{iggs}_{G}(C)^{\mathrm{ss}, \mathrm{twi}}_d \coloneqq \mathcal{L}\mathcal{H}\mathrm{iggs}_{G_{\mathrm{ad}}}(C)^{\mathrm{ss}}_d \times_{\mathrm{B} \mathrm{Map}(C, \mathrm{B}\mathrm{Z}(G))} \mathrm{pt}.
\]
By repeating the above discussion for the universal cover $\tilde{G} \to G$ instead of $G \to G_{\mathrm{ad}}$, we obtain a natural map
\[
    \mathcal{L} \mathcal{H}\mathrm{iggs}_{G}(C)^{\mathrm{ss}, \mathrm{twi}}_d \to \mathrm{B} \pi_1(G).
\]
We set 
\[
    \mathcal{L}_{\mathrm{sc}} \mathcal{H}\mathrm{iggs}_{G}(C)^{\mathrm{ss}, \mathrm{twi}}_d \coloneqq \mathcal{L}\mathcal{H}\mathrm{iggs}_{G}(C)^{\mathrm{ss}, \mathrm{twi}}_d \times_{ \mathrm{B}\pi_1(G)} \mathrm{pt}.
\]
Now consider the following composition:
\[
   \eta \colon \mathcal{L}_{\mathrm{sc}} \mathcal{H}\mathrm{iggs}_{G}(C)^{\mathrm{ss}, \mathrm{twi}}_d  \to \mathcal{L} \mathcal{H}\mathrm{iggs}_{G}(C)^{\mathrm{ss}, \mathrm{twi}}_d  \to \mathcal{L}_{\mathrm{ad}} \mathcal{H}\mathrm{iggs}_{G}(C)^{\mathrm{ss}, \mathrm{twi}}_d.
\]
The map $\eta$ is an \'etale $\mathrm{Z}(\tilde{G})$-cover.
We claim that this $\mathrm{Z}(\tilde{G})$-cover descends to the good moduli space $\mathrm{L}_{\mathrm{ad}} \mathrm{H}\mathrm{iggs}_{G}(C)^{\mathrm{ss}, \mathrm{twi}}_d$ of $\mathcal{L}_{\mathrm{ad}} \mathcal{H}\mathrm{iggs}_{G}(C)^{\mathrm{ss}, \mathrm{twi}}_d$.
To see this, using \cite[Theorem 10.3]{_Alper_GoodmodulispacesforArtinstacks}, it is enough to show that $\eta_{*} \mathcal{O}_{\mathcal{L}_{\mathrm{sc}} \mathcal{H}\mathrm{iggs}_{G}(C)^{\mathrm{ss}, \mathrm{twi}}_d }$ is acted on trivially by each stabilizer group at a closed point.
Note that the map $\eta$ is classified by the following map
\[
    \mathcal{L}_{\mathrm{ad}} \mathcal{H}\mathrm{iggs}_{G}(C)^{\mathrm{ss}, \mathrm{twi}}_d \to \mathcal{L}\mathcal{H}\mathrm{iggs}_{G_{\mathrm{ad}}}(C)^{\mathrm{ss}}_d \to \mathrm{B} \mathrm{Map}(C, \mathrm{B}\mathrm{Z}(\tilde{G})) \times  \mathrm{B}\mathrm{Z}(\tilde{G}) \to \mathrm{B}\mathrm{Z}(\tilde{G}).
\]
Take a closed point $x \in \mathcal{L}_{\mathrm{ad}} \mathcal{H}\mathrm{iggs}_{G}(C)^{\mathrm{ss}, \mathrm{twi}}_d $ and let $G_x$ be the stabilizer group at $x$.
We need to show that the map
\[
 \mathrm{B} G_x \to     \mathcal{L}_{\mathrm{ad}} \mathcal{H}\mathrm{iggs}_{G}(C)^{\mathrm{ss}, \mathrm{twi}}_d \to \mathcal{L} \mathrm{B} G_{\mathrm{ad}} \to \mathcal{L} \mathrm{B}^2 \mathrm{Z}(\tilde{G}) \cong \mathrm{B}^2 \mathrm{Z}(\tilde{G}) \times \mathrm{B} \mathrm{Z}(\tilde{G}) \to  \mathrm{B} \mathrm{Z}(\tilde{G})
\]
is a trivial map. In particular, it is enough to show that, for each $y \in \mathcal{L} \mathrm{B}G_{\mathrm{ad}}$, the following map is trivial
\[
\mathrm{B} G_{y} \to \mathcal{L} \mathrm{B} G_{\mathrm{ad}}  \to \mathcal{L} \mathrm{B}^2 \mathrm{Z}(\tilde{G}) \cong \mathrm{B}^2 \mathrm{Z}(\tilde{G}) \times \mathrm{B} \mathrm{Z}(\tilde{G}) \to  \mathrm{B} \mathrm{Z}(\tilde{G}).
\]
This follows from the fact that this map classifies the cover $\coprod_{\tilde{y} \in \tilde{G}} \mathrm{B} G_{y} \to \mathrm{B} G_{y}$ where the left-hand side runs over all lifts of $y$.

Consider the following direct sum decomposition
\[
\eta_{*} \mathbb{Q}_{\mathcal{L}_{\mathrm{sc}} \mathcal{H}\mathrm{iggs}_{G}(C)^{\mathrm{ss}, \mathrm{twi}}_d}   = \bigoplus_{d' \in \mathrm{Z}(G)^{\vee}} \tilde{\mathcal{L}}_{d'},
\]
where $\tilde{\mathcal{L}}_{d'}$ is the $d'$-isotypic component, which is a rank one local system.
By the above discussion, $\tilde{\mathcal{L}}_{d'}$ descends to a local system on the good moduli space $\mathrm{L}_{\mathrm{ad}} \mathrm{H}\mathrm{iggs}_{G}(C)^{\mathrm{ss}, \mathrm{twi}}_d$ of $\mathcal{L}_{\mathrm{ad}} \mathcal{H}\mathrm{iggs}_{G}(C)^{\mathrm{ss}, \mathrm{twi}}_d$.
We define the $d'$-twisted BPS cohomology of $\mathrm{L} \mathrm{H}\mathrm{iggs}_{G}(C)^{\mathrm{ss}, \mathrm{twi}}_d$ to be the cohomology
\[
 \mathrm{H}^*_{\mathrm{BPS}^{(0)}}\left(\mathrm{L} \mathrm{H}\mathrm{iggs}_{G}(C)^{\mathrm{ss}, \mathrm{twi}}_d \right)_{d'}   \coloneqq \mathrm{H^*} \left( \mathrm{L}_{\mathrm{ad}} \mathrm{H}\mathrm{iggs}_{G}(C)^{\mathrm{ss}, \mathrm{twi}}_d, \mathcal{BPS}^{(0)}_{\mathrm{L}_{\mathrm{ad}} \mathrm{H}\mathrm{iggs}_{G}(C)^{\mathrm{ss}, \mathrm{twi}}_d} \otimes \mathcal{L}_{d'} \right).
\]
We can now state a twisted generalization of \Cref{conj-top-mirror-refined}, which recovers the conjecture of \textcite[Conjecture 5.1]{hausel2003mirror} (proved in \cite{groechenig2020mirror,maulik2021endoscopic}) when $G = \mathrm{SL}_n$ and $d$ and $d'$ are coprime to $n$:

\end{para}

\begin{conjecture}\label{conj-top-mirror-twisted}
    Let $G$ be a semisimple group.
    For $d \in \pi_1(G_{\mathrm{ad}}) \cong \mathrm{Z}(\widetilde{G^{\vee}})^{\vee}$ and $d' \in \mathrm{Z}(\tilde{G})^{\vee} \cong \pi_1((G^{\vee})_{\mathrm{ad}})$, there exists a natural isomorphism
    \[
        \mathrm{H}^*_{\mathrm{BPS}^{(0)}}\left(\mathrm{L} \mathrm{H}\mathrm{iggs}_{G}(C)^{\mathrm{ss}, \mathrm{twi}}_d \right)_{d'}  \cong  \mathrm{H}^*_{\mathrm{BPS}^{(0)}}\left(\mathrm{L} \mathrm{H}\mathrm{iggs}_{G^{\vee}}(C)^{\mathrm{ss}, \mathrm{twi}}_{d'} \right)_{d}.
    \]
\end{conjecture}

\begin{para}[Cohomological integrality for character stacks of compact oriented \texorpdfstring{$2$}{$2$}-manifolds]
    Let $\Sigma$ be a compact oriented $2$-manifold and $G$ be a connected reductive group.
    We let $\mathcal{L}\mathrm{oc}_G(\Sigma)$ denote the derived character stack of $G$-local systems on  $\Sigma$.
    As we have seen in \Cref{item-example-symplectic-2mfd} of \Cref{para-example-shifted-symplectic} and \Cref{cor-character-stack-symmetric}, the stack $\mathcal{L}\mathrm{oc}_G(\Sigma)$ is $0$-shifted symplectic and almost orthogonal.
    Also, being a quotient stack, it admits a global equivariant parameter  by \Cref{lem-global-quotient-global-equiv-parameter}.
    Therefore the cohomological integrality theorem (= Theorem \ref{thm-cohint-0-shifted}) holds for $\mathcal{L}\mathrm{oc}_G(\Sigma)$.

    We let $p \colon \mathcal{L}\mathrm{oc}_G(\Sigma) \to \mathrm{Loc}_G(\Sigma)$ be the good moduli space morphism.
    For each Levi subgroup $L \subset G$ and a component $d \in \pi_0(\mathcal{L}\mathrm{oc}_L(\Sigma))$, we let     
    $g_{L, d} \colon \mathrm{Loc}_{L}(\Sigma)_d \to  \mathrm{Loc}_{G}(\Sigma)$ be the natural map.
    Then the cohomological integrality theorem (= \Cref{thm-cohint-0-shifted} + \Cref{rmk-cohint-source-0-shifted-symplectic}) is equivalent to the following:
\end{para}

\begin{theorem}
    We adopt the notation from the last paragraph.
    Then there exists an isomorphism
    \begin{align*}
        \bigoplus_{\substack{L \subset G: \\ \textnormal{Levi subgroups}}} \Biggl( \bigoplus_{d \in \pi_0(\mathcal{L}\mathrm{oc}_{L}(\Sigma)) } g_{L, d, *} (\mathcal{BPS}^{(c_L)}_{\mathrm{Loc}_{L}(\Sigma)_d} &\otimes \mathrm{H}^*(\mathrm{B} \mathbb{G}_\mathrm{m}^{c_L}) \otimes \mathrm{sgn}_{L}) \Biggr)^{W_G(L)}  \\
         &\cong \mathbb{L}^{   (g - 1) \cdot \dim G } \otimes p_* \mathbb{D}\mathbb{Q}_{\mathcal{L}\mathrm{oc}_G(\Sigma)}
    \end{align*}
    induced by the cohomological Hall induction.
\end{theorem}

\begin{remark}\label{rmk-P=W}
    Let $\Sigma$ be the compact oriented $2$-manifold underlying a Riemann surface $C$.
    Under the non-abelian Hodge correspondence \cite[Lemma 9.14]{simpson1994moduli}, there exists a homeomorphism
    \[
     \mathrm{Higgs}_G(C)^{\mathrm{ss}}_{\mathrm{vc}} \cong_{\mathrm{Homeo}} \mathrm{Loc}_G(\Sigma) 
    \]
    where the left-hand side is the disjoint union of the components with vanishing Chern classes in the rational cohomology.
    We expect that this homeomorphism preserves the (underlying perverse sheaf of the) BPS sheaves: this is true for $G = \mathrm{GL}_n$ as proved by \textcite{davison2023bps} when $g(C) \geq 2$ and by \textcite[Theorem 14.3]{davison2023nonabelian} when $g(C) \leq 1$.
    If this is the case, we can explore a version of the P=W conjecture using the cohomology of the BPS sheaves.

\end{remark}

\begin{para}[Cohomological integrality for character stacks of a \texorpdfstring{$3$}{3}-manifold]\label{para-cohint-3-manifolds}
    Let $M$ be a compact oriented $3$-manifold and $G$ be a connected reductive group.
    We let $\mathcal{L}\mathrm{oc}_G(M)$ denote the moduli stack of local systems on  $M$.
    As we have seen in \Cref{item-example-symplectic-3mfd} of \Cref{para-example-shifted-symplectic}, the stack $\mathcal{L}\mathrm{oc}_G(M)$ is $(-1)$-shifted symplectic.
    Also, being a quotient stack, it admits a global equivariant parameter  by \Cref{lem-global-quotient-global-equiv-parameter}.
    Further, it is shown in \cite[Proposition 3.41]{naef2023torsion} that $\mathcal{L}\mathrm{oc}_G(M)$ is equipped with a canonical orientation.
    Therefore the cohomological integrality theorem (= Theorem \ref{thm-coh-int--1-shifted-symplectic}) holds for $\mathcal{L}\mathrm{oc}_G(M)$, as long as it is almost orthogonal.

    We let $p \colon \mathcal{L}\mathrm{oc}_G(M) \to \mathrm{Loc}_G(M)$ be the good moduli space morphism.    
    For each Levi subgroup $L \subset G$ and component $d \in \pi_0(\mathcal{L}\mathrm{oc}_L(M))$, we let     
    $g_{L, d} \colon \mathrm{Loc}_{L}(M)_d \to  \mathrm{Loc}_{G}(M)$ be the natural map.
    Then the cohomological integrality theorem (= \Cref{thm-coh-int--1-shifted-symplectic} + \Cref{rmk-cohint-source--1-shifted-symplectic}) is equivalent to the following:
\end{para}

\begin{theorem}
    We adopt the notation from the last paragraph.
    Assume that $\mathcal{L}\mathrm{oc}_G(M)$ is almost orthogonal (see \Cref{para-examples-almost-orthogonal-character-stacks} for examples satisfying this).
    Then there exists an isomorphism
    \[
        \bigoplus_{\substack{L \subset G: \\ \textnormal{Levi subgroups}}} \left( \bigoplus_{ \substack{ d \in \pi_0(\mathcal{L}\mathrm{oc}_{L}(M))}} g_{L, d, *} (\mathcal{BPS}^{(c_L)}_{\mathrm{Loc}_{L}(M)_d} \otimes \mathrm{H}^*(\mathrm{B} \mathbb{G}_\mathrm{m}^{c_L})_{\mathrm{vir}})\right)^{W_{G}(L)}  \cong p_* \varphi_{\mathcal{L}\mathrm{oc}_G(M)}
    \]
    induced by the cohomological Hall induction.
\end{theorem}

\begin{para}
    Cohomological integrality theorem for $3$-manifolds was first studied by \textcite{kaubrys2024cohomological}, 
    where he proved the theorem when $M$ is a $3$-torus and $G$ is either $\mathrm{GL}_n$, or $\mathrm{SL}_n$ or $\mathrm{PGL}_n$ for prime $n$, and also explicitly described the BPS sheaves in these three cases.
    We note that the definition of the BPS sheaves in [loc. cit.] is slightly different from ours, in that it is defined over the component containing the trivial local system.
\end{para}

\begin{para}[Geometric Langlands conjecture for \texorpdfstring{$3$}{$3$}-manifolds]
    We will formulate a version of the geometric Langlands conjecture for $3$-manifolds using the BPS cohomology, which we learned from Pavel Safronov.
    We adopt the notation from \Cref{para-cohint-3-manifolds}. Set $c_G \coloneqq \dim \mathrm{Z}(G)$.
    Then we expect the following:
\end{para}

\begin{conjecture}[Safronov]\label{conj-geometric-langlands-3-mfds}
    There exists a natural isomorphism
    \begin{equation}\label{eq-geometric-Langlands-3mfds}
        \mathrm{H}^*_{\mathrm{BPS}^{(c_G)}}(\mathrm{Loc}_G(M)) \cong \mathrm{H}^*_{\mathrm{BPS}^{(c_{G^{\vee}})}}(\mathrm{Loc}_{G^{\vee}}(M)).
    \end{equation}
\end{conjecture}

\begin{remark}
    \Cref{conj-geometric-langlands-3-mfds} was proved for $M = T^3$ and $G = \mathrm{SL}_p$ for prime $p$ by \cite[Theorem 1.5]{kaubrys2024cohomological}.
    In a forthcoming paper by Hennecart and T.K. \cite{Kinjo_Springer}, the authors prove \Cref{conj-geometric-langlands-3-mfds} for $M = T^3$ and $G = \mathrm{SL}_n$ and $\mathrm{Sp}_{2n}$ for any $n$.   
    We also note that the above conjecture for $\mathrm{S^3} / \Gamma$ for a finite subgroup $\Gamma \subset \mathrm{SU}_2$ recovers the recent conjecture due to \textcite[Conjecture 1.2]{kojima2025homomorphisms} on the equality of numbers of conjugacy classes for maps from $\Gamma$ to $G$ and $G^{\vee}$, 
    where they prove the cases including $\Gamma = \mathbb{Z} / n \mathbb{Z}$ for any $G$ and $G =\mathrm{SL}_n, \mathrm{Sp}_{2n}$ for any $\Gamma$.
\end{remark}

\begin{para}[Refinement of the \texorpdfstring{$3d$}{$3d$}-Langlands duality conjecture]
    As in the case of the topological mirror symmetry \Cref{conj-top-mirror-refined}, we can consider a refinement of the Langlands duality for $3$-manifolds when $G$ is semisimple.

    Consider the exact sequence
    \[
     1 \to \pi_1(G) \to \tilde{G}  \to G \to 1   
    \]
    where $\tilde{G}$ denotes the universal cover of $G$.
    It induces natural maps
    \[
     \mathcal{L}\mathrm{oc}_{G}(M) \to    \mathcal{L}\mathrm{oc}_{\mathrm{B} \pi_1(G)}(M) \to \pi_0(\mathcal{L}\mathrm{oc}_{\mathrm{B} \pi_1(G)}(M)) \to \mathrm{H}^2(M, \pi_1(G)). 
    \]
    For $d \in \mathrm{H}^2(M, \pi_1(G))$, we let $\mathcal{L}\mathrm{oc}_{G}(M)_d \subset \mathcal{L}\mathrm{oc}_{G}(M)$ denote the open and closed substack which maps to $d$ under the above map.

    On the other hand, note that the action of $\mathrm{Z}(G)$ on $G$ induces an action of $\mathcal{L}\mathrm{oc}_{\mathrm{Z}(G)}(M)$ on $\mathcal{L}\mathrm{oc}_{G}(M)$.
    By passing to the good moduli space, we obtain an action of $\mathrm{H}^1(M, \mathrm{Z}(G))$ on $\mathcal{L}\mathrm{oc}_{G}(M)$,
    which provides a $\mathrm{H}^1(M, \mathrm{Z}(G))$-equivariant structure on $\mathrm{H}^*_{\mathrm{BPS}}(\mathrm{Loc}_G(M))$.
    For $d' \in \mathrm{H}^1(M, \mathrm{Z}(G))$, we let $\mathrm{H}^*_{\mathrm{BPS}}(\mathrm{Loc}_G(M))_{d'}$ denote the subspace where $\mathrm{H}^1(M, \mathrm{Z}(G))$ acts by $d'$.
    
    Using the isomorphism $Z(G) \cong \pi_1(G^{\vee})^{\vee}$ and $\pi_1(G) \cong \mathrm{Z}(G^{\vee})^{\vee}$  together with the Poincar\'e duality for $M$,
    we obtain natural isomorphisms
    \[
     \mathrm{H}^2(M, \mathrm{Z}(G)) \cong \mathrm{H}^1(M, \pi_1(G^{\vee}))^{\vee}, \quad \mathrm{H}^1(M, \pi_1(G))^{\vee} \cong \mathrm{H}^2(M, \mathrm{Z}(G^{\vee})).
    \]
    Then we have the following refinement of \Cref{conj-geometric-langlands-3-mfds}:
    
\end{para}

\begin{conjecture}\label{conj-geometric-langlands-3-mfds-refined}
    Let $G$ be a semisimple group.
    For $d \in \mathrm{H}^2(M, \mathrm{Z}(G)) \cong \mathrm{H}^1(M, \pi_1(G^{\vee}))^{\vee}$ and $d' \in \mathrm{H}^1(M, \pi_1(G))^{\vee} \cong \mathrm{H}^2(M, \mathrm{Z}(G^{\vee}))$,
    the conjectural isomorphism \Cref{eq-geometric-Langlands-3mfds} swaps the subspaces
    \[
        \mathrm{H}^*_{\mathrm{BPS}^{(0)}}(\mathrm{Loc}_G(M)_d)_{d'} \cong \mathrm{H}^*_{\mathrm{BPS}^{(0)}}(\mathrm{Loc}_{G^{\vee}}(M)_{d'})_d.
    \]
    
\end{conjecture}

\begin{remark}
    It was conjectured by \textcite[Conjecture 1.1]{jordan2024langlands} in his joint work with Ben-Zvi, Gunningham and Safronov, that a similar duality phenomenon holds for the generic skein modules when $G$ is semisimple.
    On the other hand, it was conjectured by \textcite[Conjecture C]{gunningham2023deformation} that the generic skein module is isomorphic to the zeroth cohomology of the DT perverse sheaves on the framed character variety.
    This suggests that the BPS cohomology is related with the skein modules and that \Cref{conj-geometric-langlands-3-mfds} might be equivalent to the duality for a derived version of the generic skein modules. 
\end{remark}

\begin{para}[Relation with topological mirror symmetry]\label{para-top-mirror-and-langlands}
    Let $\Sigma$ be a compact oriented $2$-manifold underlying a Riemann surface $C$ and $G$ be a semisimple group.
    The non-abelian Hodge correspondence \cite[Lemma 9.14]{simpson1994moduli} provides a bijection of $\mathbb{C}$-valued points of the good moduli spaces of the loop stacks
    \[
        \mathrm{LHiggs}_G(C)^{\mathrm{ss}} (\mathbb{C}) \cong_{\mathrm{bij}} \mathrm{Loc}_G(\Sigma \times S^1) (\mathbb{C}).     
    \]
    We expect that this bijection upgrades to a homeomorphism and preserves the (underlying perverse sheaf of the) BPS sheaves.
    If this is the case, the topological mirror symmetry conjecture (= \Cref{conj-top-mirror}) follows from the Langlands duality conjecture of $3$-manifolds (= \Cref{conj-geometric-langlands-3-mfds}) at the level of graded vector spaces.
\end{para}

\newpage
\phantomsection
\addcontentsline{toc}{section}{\refname}
\renewcommand*{\bibfont}{\small}
\sloppy
\printbibliography

\authorinforule

\authorinfo{Chenjing Bu}
    {bu@maths.ox.ac.uk}
    {Mathematical Institute, University of Oxford, Oxford OX2~6GG, United Kingdom}

\authorinfo{Ben Davison}
    {ben.davison@ed.ac.uk}
    {School of Mathematics and Maxwell Institute for Mathematical Sciences, University of Edinburgh, Edinburgh EH9~3FD, United Kingdom}

\authorinfo{Andrés Ibáñez Núñez}
    {andres.ibaneznunez@columbia.edu}
    {Department of Mathematics, Columbia University, New York, NY 10027, USA}

\authorinfo{Tasuki Kinjo}
    {tkinjo@kurims.kyoto-u.ac.jp}
    {Research Institute for Mathematical Sciences, Kyoto University, Kyoto 606-8502, Japan}

\authorinfo{Tudor P\u{a}durariu}
    {padurariu@imj-prg.fr}
    {Sorbonne Université and Université Paris Cité, CNRS, IMJ-PRG, F-75005 Paris, France}

\end{document}